\documentclass[12pt,letterpaper,reqno]{amsart}
\usepackage[foot]{amsaddr}
\makeatletter
\renewcommand{\email}[2][]{%
  \ifx\emails\@empty\relax\else{\g@addto@macro\emails{,\space}}\fi%
  \@ifnotempty{#1}{\g@addto@macro\emails{\textrm{(#1)}\space}}%
  \g@addto@macro\emails{#2}%
}
\makeatother
\usepackage{cmap}
\usepackage[T1]{fontenc}
\usepackage{epsfig}
\usepackage{amsmath}
\usepackage{amssymb}
\usepackage{amsthm}
\usepackage{indentfirst}
\usepackage{xspace}
\usepackage{xcolor}
\usepackage{setspace}
\usepackage{verbatim}
\usepackage[letterpaper,margin=1in,headheight=15pt]{geometry}
\usepackage{mathpazo}
\usepackage{mathtools}
\usepackage{tikz}
\usetikzlibrary{chains}
\usepackage{booktabs}
\usepackage{cancel}
\usepackage{float} 
\definecolor{shadecolor}{rgb}{0.85,0.85,0.85}
\usepackage{bibentry}
\usepackage{bm}
\usepackage{graphics} 
\usepackage[all]{xy}
\usepackage{hyperref} 
\definecolor{darkred}{rgb}{0.5,0.15,0.15}
\hypersetup{colorlinks=true,urlcolor=darkred,linkcolor=darkred,citecolor=darkred}
\allowdisplaybreaks

\newtheorem{thm}{Theorem}
\newtheorem{cor}[thm]{Corollary}
\newtheorem{conj}[thm]{Conjecture}
\newtheorem{lem}[thm]{Lemma}
\newtheorem{obs}{Observation}
\newtheorem{prop}[thm]{Proposition}

\newtheorem{ex}[thm]{Example}

\theoremstyle{remark}

\theoremstyle{definition}
\newtheorem{rem}[thm]{Remark}
\newtheorem{defn}[thm]{Definition}

\numberwithin{thm}{section}
\numberwithin{equation}{section}
\numberwithin{figure}{section}

\newcommand{\fB}{{\mathfrak B}}

\newcommand{\cG}{\ensuremath{\mathcal G}}
\newcommand{\cB}{\ensuremath{\mathcal B}}
\newcommand{\cL}{\ensuremath{\mathcal L}}

\newcommand{\cF}{\ensuremath{\mathcal F}}

\newcommand{\cM}{\ensuremath{\mathcal M}}
\newcommand{\cO}{\ensuremath{\mathcal O}}

\newcommand{\cI}{\ensuremath{\mathcal I}}

\newcommand{\cP}{\ensuremath{\mathcal P}}
\newcommand{\cT}{\ensuremath{\mathcal T}}
\newcommand{\cV}{\ensuremath{\mathcal V}}

\newcommand{\R}{\ensuremath{\mathbb R}}
\newcommand{\C}{\ensuremath{\mathbb C}}
\newcommand{\CP}{\ensuremath{\mathbb {CP}}}

\newcommand{\Z}{\ensuremath{\mathbb Z}}

\newcommand{\bfalpha}{\ensuremath{\bm \alpha}}

\newcommand{\calA}{\mathcal A}
\newcommand{\calB}{\mathcal B}
\newcommand{\calD}{\mathcal D}
\newcommand{\calG}{\mathcal G}
\newcommand{\calH}{\mathcal H}

\newcommand{\calP}{\mathcal P}
\newcommand{\calR}{\mathcal R}

\newcommand{\calM}{\mathcal M}

\newcommand{\tcalX}{\widetilde{\mathcal X}}
\newcommand{\frakB}{\mathfrak B}
\newcommand{\RR}{{\mathcal R}}

\newcommand{\cE}{{\mathcal E}}

\newcommand{\delbar}{\ensuremath{\overline{\partial}}}
\newcommand{\zbar}{{\ensuremath{\overline{z}}}}

\newcommand{\Id}{\ensuremath{\mathrm{Id}}}

\newcommand{\semif}{\ensuremath{\mathrm{sf}}}

\newcommand{\I}{{\mathrm i}}
\newcommand{\e}{{\mathrm e}}
\newcommand{\de}{\mathrm{d}}

\newcommand{\IP}[1]{\langle#1\rangle}

\newcommand{\eps}{\epsilon}

\newcommand{\calX}{\ensuremath{\mathcal X}}

\newcommand{\aff}{\mathrm{aff}}

\newcommand{\del}{{\partial}}

\newcommand{\cMam}{\mathcal{M}(\boldsymbol{\alpha}, \mathbf{m})}
 
\newcommand{\bsa}{\boldsymbol{\alpha}}
\newcommand{\bfm}{\mathbf{m}}

\newcommand{\Simplex}{\Delta}

\DeclareMathOperator{\ad}{ad}
\DeclareMathOperator{\im}{Im}
\DeclareMathOperator{\image}{im}
\DeclareMathOperator{\re}{Re}
\DeclareMathOperator{\Tr}{Tr}

\DeclareMathOperator{\End}{End}

\DeclareMathOperator{\pdeg}{pdeg}

\DeclareMathOperator{\Res}{Res}
\DeclareMathOperator{\interior}{int}

\newcommand{\fixme}[1]{{\color{blue}{\tt [#1]}}}

\DeclareMathOperator{\GL}{GL}
\DeclareMathOperator{\SL}{SL}

\DeclareMathOperator{\OO}{O}

\DeclareMathOperator{\Diff}{Diff}

\tikzset{node distance=2em, ch/.style={circle,draw,on chain,inner sep=2pt},chj/.style={ch,join},every path/.style={shorten >=4pt,shorten <=4pt},line width=1pt,baseline=-1ex}

\begin{document}
\onehalfspacing

\title{ALG gravitational instantons and Hitchin moduli spaces, I: \\ Torelli parameters}
\author{Laura Fredrickson}
\address{Department of Mathematics\\
University of Oregon, Eugene, OR 97403 USA}
\email{lfredric@uoregon.edu}

\author{Rafe Mazzeo}
\address{Department of Mathematics\\
Stanford University, Stanford, CA 94305 USA}
\email{rmazzeo@stanford.edu}

\author{Jan Swoboda}
\address{Department of Mathematics for Engineering\\
Universit\"{a}t der Bundeswehr M\"{u}nchen, 95579 Neubiberg, Germany}
\email{jan.swoboda@unibw.de}

\author{Hartmut Weiss}
\address{Department of Mathematics\\
Christian-Albrechts-Universit\"{a}t Kiel, 24118 Kiel, Germany}
\email{weiss@math.uni-kiel.de}
\date{\today}

\begin{abstract}
This is the first of two papers which together prove that the $12$-parameter family of parabolic $SU(2)$-Hitchin moduli spaces on the 
four-punctured sphere are all ALG gravitational instantons of type D4, and hence are asymptotic to $(\C \times T^2_\tau)/\Z_2$ at infinity. 
The elliptic modulus $\tau$ is determined by the cross-ratio of the four points.  In this first paper, we consider each Hitchin moduli space 
corresponding to an allowable set of parabolic data and compute its Torelli parameters. 
There is a $12$-parameter family of Hitchin moduli spaces corresponding to different parabolic data, and we show that these realize
all possible allowable Torelli parameters.  In the companion paper, we we will show there that all of the Hitchin moduli spaces studied here are indeed ALG of type $D_4$, and
consequently that every ALG-$D_4$ gravitational instanton can be realized as a Hitchin moduli space. Altogether, this will give the first verification of any 
case of the Modularity Conjecture: that all ALG gravitational instantons with tangent cone $\C/\Z_2$ can be realized as Hitchin moduli spaces
with their natural associated $L^2$ metrics.
\end{abstract}

\maketitle

\tableofcontents

\setcounter{page}{1}

\date{\today}

\section{Introduction}\label{sec:introduction}

Let $C$ be a compact Riemann surface and $D = \{p_1, \ldots, p_N\}$ a divisor on $C$ with all points of multiplicity $1$. Let $E$ be a rank $2$ 
complex vector bundle over $C$.  We shall consider parabolic structures on $(C, D, E)$, consisting of the following data: a full flag
$\{0\} \subset F_{p_i} \subset E_{p_i} := E|_{p_i}$ in the fiber of $E$ over each point $p_i$ in $D$, as well as an assignment of two real numbers 
$\alpha^{(1)}(p_i), \alpha^{(2)}(p_i) \in (0,1)$ and a complex number $m_i$ for each $p_i \in D$.  In general, the bundle $E$ may have higher rank, one does not 
require the flag to be full, and the numbers $\alpha^{(i)}(p)$ are unrelated.  In this paper we specialize by requiring $E$ to be rank $2$, as already stated,
and to have structure group $\mathrm{SU}(2)$, which is equivalent to fixing a Hermitian metric on $\det E$ and requiring all transformations below
to be compatible with this choice. In particular, this forces that $\alpha^{(1)}(p_i) + \alpha^{(2)}(p_i) = 1$, so it suffices to specify only 
$\alpha^{(1)}(p_i) \in (0, 1/2)$, which we call simply $\alpha_i$ for simplicity. 

Any such choice of parabolic data leads to three seemingly distinct collections of objects:  a space of compatible Higgs bundles, a space of gauge fields $(A,\varphi)$
which solve the Hitchin equations and have specified singularities at points of $D$, and a space of flat $\mathrm{SL}(2,\C)$ connections on $C^\times = C
\setminus D$, again with specified holonomy and singular behavior at points of $D$.  There are gauge groups acting on each of these spaces. The first main theorem 
in this subject is that restricting to suitable subsets of stable objects in each space, the quotients of each of these spaces
by the relevant gauge group defines a smooth moduli space, and the three moduli spaces determined in this way are mutually diffeomorphic, but inherit
different complex structures.  These identifications, in various directions, are called the Kobayashi-Hitchin and nonlinear Hodge correspondences. The 
moduli space itself is denoted $\cMam$, where $\boldsymbol{\alpha} = (\alpha_1, \ldots, \alpha_N)$ and $\mathbf{m} = (m_1, \ldots, m_N)$. 

In slightly more detail, a Higgs bundle consists of a pair $(\cE, \varphi)$ where $\cE = (E, \overline{\partial}_A)$ is a holomorphic structure on $E$ and $\varphi$ is a 
meromorphic section of $\mathrm{Hom}(\cE, \cE \otimes K_C)$ with simple poles at $D$. The `real mass parameters' $\alpha_j$ in this setting are extra data 
whose role will be explained later, while the `complex mass parameters' $m_j$ specify the eigenvalues of the residue of $\varphi$ at $p_j$.   Next, the gauge field $A$ is a unitary 
connection on $E$ and $\varphi$ is a section of $\mathrm{Hom}\,(E, E \otimes K_C)$; each has first order blowup at the points of $D$ with singular part 
determined by the $\alpha_j$ and $m_j$, respectively. These satisfy the Hitchin equations 
\begin{equation}
F_A^\perp + [\varphi, \varphi^*] = 0, \ \ \overline{\del}_A \varphi = 0,
\label{HE}
\end{equation}
where the adjoint is taken with respect to a fixed Hermitian metric $h_0$ on $E$.  An alternate point of view is that if we begin with the holomorphic structure
$\overline{\partial}_A$, then we seek a Hermitian metric $h_0$ such that the system \eqref{HE} is satisfied, where $F_A$ is the curvature of the associated Chern 
connection $\overline{\partial}_A + \partial_A^{h_0}$.  Finally, in terms of these fields, the flat $\mathrm{SL}(2,\C)$ connections
we consider take the form $\nabla_A + \zeta^{-1}\varphi + \zeta\varphi^*$ for $\zeta \in \C^\times$. The group $\cG$ of unitary gauge transformations acts on solutions of the Hitchin equations and
on the space of flat connections; its complexification $\cG_{\C}$ acts on the space of Higgs bundles. 

In this paper we shall primarily focus on the space of pairs $(A, \varphi)$ which solve the Hitchin equations from the Higgs bundle perspective. 
It is known that   
\begin{equation}
\dim_{\C} \cMam =  6(\mathrm{genus}(C)-1) + 2N.
\label{dimM}
\end{equation}
 There is a natural hyperK\"ahler metric on the affine space of all pairs $(A,\varphi)$, and an infinite dimensional 
version of hyperK\"ahler reduction defines a metric $g$ on this quotient moduli space.   Another key feature of this moduli space is the existence
of the Hitchin fibration $\mathrm{Hit}$, which maps a pair $(\cE, \varphi)$ (or equivalently, $(A,\varphi)$) to $-\det \varphi$, which is a meromorphic quadratic differential
on $C$, with poles of order no more than $2$ at points of $D$.  The leading singular terms of these differentials are determined by the
complex masses $m_j$. The range of $\mathrm{Hit}$ is an affine space, and this map is known to be surjective and proper. 

\bigskip

While much is known about the geometry and topology of the hyperK\"ahler spaces $(\cMam, g)$, many open questions remain.   Our goal here is to 
study some of these questions when $C= \CP^1$ 
and $|D| = 4$. In this special case, $\dim_{\C} \cMam = 2$; note that by \eqref{dimM}, the only other case where the moduli space has this dimension 
is when $g=N=1$, i.e., for the once-punctured torus\footnote{There is a $3$-parameter family of Hitchin moduli spaces on once-punctured torus. It is expected that each of these hyperK\"ahler structures can also be obtained from a $3$-parameter subfamily of Hitchin moduli spaces on the four-punctured sphere.}.  By one of the general results obtained by Hitchin in his initial investigation of these spaces \cite{hitchin87} (see \cite{simpsonnoncompact} 
for this result in the parabolic case), the metric $g_{L^2}$ is complete.  In other words, $(\cMam, g_{L^2})$ is a complete, noncompact (real) 
four-dimensional hyperK\"ahler space. It is expected that the Riemannian curvature tensor is $L^2$-integrable. We will prove this in the companion paper.

\bigskip

\noindent\emph{Gravitational Instantons:} Complete, noncompact (real) four-dimensional hyperK\"ahler spaces with integrable Riemannian curvature tensors are known as \emph{gravitational instantons}, and are fundamental objects in the study of four-dimensional 
Riemannian manifolds with special holonomy.  
Gravitational instantons have been the focus of intensive study, particularly in the last decade, and there is now a fairly complete classification.  There are six 
different types: ALE, ALF, ALG${}^*$, ALG, ALH${}^*$, and ALH, 
based on their asymptotic geometric structure, and within each type the possible topological types are classified by affine Dynkin
diagrams (which correspond to their intersection forms on middle degree integer homology). Finally, fixing both the asymptotic and 
topological type, there are still continuous moduli.  The monikers ALE and ALF are acronyms for asymptotically locally Euclidean and 
asymptotically locally flat; the remaining ones, ALG${}^*$/G/H${}^*$/H, were coined by S.\ Cherkis (with ${}^*$ additions by Hein), 
as the natural `continuation' of this series.  The most `visible' geometric differences between these spaces are based on their 
volume and quasi-isometry classes. Thus, the ALE spaces are asymptotically conical; ALF spaces are $S^1$ fibrations over three-dimensional
asymptotically conical spaces, where the $S^1$ fibers degenerate at some finite collection of points and which have lengths converging to
some finite positive number at infinity; ALG and ALG${}^*$ spaces are $T^2$ fibrations over two-dimensional asymptotically conical spaces,
with the $T^2$ fibers degenerating at some finite collection of points and with the fiber metrics converging to a flat metric on $T^2$ of fixed
area in the ALG case but asymptotic to a diverging family of flat conformal structures on $T^2$ in the ALG${}^*$ case; similarly, the 
ALH and ALH${}^*$ spaces are fibrations over a half-line with fiber a three-dimensional nilmanifold, with similar convergence and divergence
properties in the unstarred and starred cases, respectively.  There has been significant progress in understanding these gravitational 
instanton moduli spaces, starting from the celebrated thesis of Kronheimer \cite{Kronheimerconstruction, KronheimerTorelli}, where a complete classification of all ALE spaces was 
obtained. That all gravitational instantons fall into one of these six classes is the upshot of a number of papers by several different authors, leading to the final resolution in \cite{SunZhang}. The classification of ALG and ALG${}^*$ 
starts with G.\ Chen-X.\ Chen \cite{ChenChenI, ChenChenIII}. 
In the companion paper we shall provide a much more detailed description of this entire classification scheme.

Fixing $\boldsymbol{\alpha}$ and $\mathbf{m}$, the Hitchin moduli space $\cMam$ for the four-punctured sphere is a singular $T^2$ fibration over a two real 
dimensional (asymptotically) conical space.  The map to this base is the Hitchin fibration. Because of this, it is natural to suspect that these parabolic 
Hitchin moduli spaces are either ALG or ALG${}^*$. 

The goal of \emph{this paper and its companion} is to prove that these are in fact always ALG:
\begin{thm} \label{thm:mainofboth}
As the parabolic parameters $\boldsymbol{\alpha}$, $\mathbf{m}$ vary, the spaces $\cMam$ nearly exhaust the entire family of ALG spaces with affine D4 
topology. 
\end{thm}  
This directly answers one case of the so-called `Modularity Conjecture', articulated by Boalch, 
which states loosely that every possible 
gravitational instanton occurs as a gauge theoretic moduli space.  In particular, 
\begin{conj}[Modularity Conjecture \cite{aim}]
Every ALG or ALG$^*$ gravitational instanton arises as the hyperK\"ahler metric on some Hitchin moduli space.
\end{conj}
Our result will be the first of this kind for ALG${}^*$/G/H${}^*$/H moduli spaces, where the proof is considerably harder than the ALE or ALF cases.

\bigskip

In this particular case, the Torelli Theorem of \cite{CVZ2024}[Theorem 1.10]  states that marked ALG-$D_4$ gravitational are classified by the classes of $[\omega_I], [\omega_J], [\omega_K]$, or alternatively, the integrals of $\omega_I, \Omega_I = \omega_J+ i \omega_K$ over a given basis of the homology.  See Section \ref{subsec:Torelli_map} for a more precise statement.
 
 \bigskip
 
 In \emph{this paper}, we prove two main results. These are not stated precisely, since there is considerable notation we must introduce first; but we refer to the precise statement appearing later.

 \begin{thm}[cf.\, Theorem \ref{thm:Torelli}] The class $[\omega_I]$ is an certain affine linear function of $\boldsymbol{\alpha}$ and the class $[\Omega_I = \omega_J + i \omega_K]$ is a certain linear function of $\mathbf{m}$.
 \end{thm}
 
\begin{thm}[cf.\, See Theorem \ref{thm:surjective}] 
As the parabolic parameters $\boldsymbol{\alpha}$, $\mathbf{m}$ vary, the classes of $[\omega_I], [\omega_J], [\omega_K]$ nearly exhaust the allowable Torelli parameters of the family of ALG spaces with affine D4 
topology. 
\end{thm}  
We now address this annoying word ``nearly''. The statement in Theorem \ref{thm:surjective} clarifies that the issue is that the set $\widetilde{\mathcal{R}}$ of extended parameters $(\boldsymbol{\alpha}, \mathbf{m})$ is a proper subset of $\widetilde{\mathcal{R}}^{\mathrm{full}}$ (see Definition \ref{def:Rdomain}). Fundamentally, this discrepancy is for the following reason: as $(\boldsymbol{\alpha}, \mathbf{m})$ approaches one of the points in $([0,\frac{1}{2}]^4 \times \C^4) \cap (\widetilde{\mathcal{R}}^{\mathrm{full}} \setminus \widetilde{\mathcal{R}})$, necessarily $\boldsymbol{\alpha}$ is approaching the boundary $\del[0, \frac{1}{2}]^4$ and $\mathbf{m} \neq \mathbf{0}$; in this case we are degenerating from the full flag case to the not-full-flag case; in the current definition of the such weakly parabolic not-full-flags Higgs bundle moduli spaces, the dimension of these moduli spaces drops; the moduli spaces are not four-dimensional, and thus not ALG-$D_4$. 
\begin{conj}\label{conj:notfullflags} 
There is a more natural definition of the Higgs bundle moduli space when the flags are not all full, i.e. when for some $p_i$, $\alpha_i \in \{0, \frac{1}{2}\}$.
\end{conj}
If Conjecture \ref{conj:notfullflags} holds, then we can remove the word ``nearly'' in Theorem \ref{thm:mainofboth} and Theorem \ref{thm:surjective}.
\bigskip

In the \emph{companion paper}, we will prove the necessary result that the Riemannian curvature tensor is $L^2$-integrable, hence, 
\begin{thm}
The moduli space $\mathcal{M}(\bsa, \bfm)$ is ALG for every generic choice $(\bsa, \bfm)$.
\end{thm}
This will extend our result that $\mathcal{M}(\bsa, \mathbf{0})$ is ALG-$D_4$ in \cite{FMSW}, a paper about asymptotic geometry of the moduli space of weakly parabolic $SL(2,\C)$-Higgs bundles. Theorem \ref{thm:mainofboth} follows from the above results and the Torelli Theorem of \cite{CVZ2024}[Theorem 1.10].

\bigskip
Proving such a result was considerably more involved than we originally expected. In the strongly parabolic case, it is easy to compute $[\omega_I]$.  However, the weakly parabolic case is difficult. While we expected to show that $[\omega_I]$ doesn't depend on $\mathbf{m}$ using some $\C^\times$-action, we could only control the integrals of $[\omega_I]$ by considering the $\C^\times$-action of the recently analytically constructed moduli space $\mathcal{P}_0(\bsa)$ in \cite{collier2024conformallimitsparabolicslnchiggs}. (I.e. we had to justify that $\bsa, \bfm$ is a smooth parameter of the moduli space, and hence $[\omega_I]$ depend smoothly on it.)
Similarly, while we compute $[\Omega_I]$ via intersection numbers of some submanifolds, we again had to use the theory in \cite{collier2024conformallimitsparabolicslnchiggs} to justify that these were submanifolds. In practice, we show that $[\Omega_I]$ agrees with the semiflat holomorphic symplectic form in Gaiotto--Moore--Neitzke. 
 But for now we note that relevant case of ALG-$D_4$ gravitational instantons are asymptotic to $(\C \times T^2_\tau)/\Z_2$ for some torus parameter $\tau$.

\subsection{Fixed data for Hitchin moduli space}

We now say more about the Higgs bundle moduli space on the four-punctured sphere. Since the Higgs bundle moduli space depends only on the complex structure 
of $(\CP^1,D)$, we may as well use a M\"obius transformation to transform $D$ to $\{0, 1, \infty, p_0\}$ in the extended complex plane.  Fixing the first three points 
removes the M\"obius freedom, leaving $p_0$ as a free parameter.  This parameter determines the complex modulus $\tau$ of the limiting torus fibers of the Hitchin 
fibration by the equation $p_0 = \lambda(\tau)$, where $\lambda$ is the elliptic modular function. 
It is also possible to write the $j$-invariant of the elliptic curve $T_\tau$ as a rational function of $p_0$.  These explicit formul{\ae} will be omitted
since they are not needed here. 

As already noted, the parabolic parameters $\boldsymbol{\alpha}$ and $\mathbf{m}$ act as `boundary conditions' at each $p_j$.  To explain this, we use
the flag data $\{0\} \subset F_j \subset E_{p_j}$. The residue of the Higgs field is required to preserve these flags at each $p_j$ and
acts by multiplication by $m_j$ on $F_j$ (and hence, by the trace-free condition, by $-m_j$ on $\cE_{p_j}/F_j$).  Thus, the Hitchin fibration carries a 
Higgs bundle $(\cE, \varphi)$ to the meromorphic quadratic differential $q = - \det \varphi = m_j^2 z^{-2}dz^2 + \cO(z^{-1})$. If all $m_j = 0$, which
is called the strongly parabolic case, these residues map $E_{p_j}  \to F_j$ and $F_j \to 0$, or in other  words,  $\varphi$ has nilpotent residue at each $p_j$. 

The role of the $\alpha_j$ is less direct. In the Higgs bundle picture, these are simply additional parameters, but are subject to a stability condition. There is a notion of parabolic stability. The parabolic degree of the rank $2$ Higgs bundle with given real mass parameters
$\alpha^{(1)}(p_j), \alpha^{(2)}(p_j)$ at each $p_j$ is, by definition, 
\[
\mathrm{pdeg}\, (\cE(\boldsymbol{\alpha})) = \mathrm{deg}\,(\cE) + \sum (\alpha^{(1)}(p_j) + \alpha^{(2)}(p_j)).
\]
Since $\alpha^{(1)}(p_j) + \alpha^{(2)}(p_j) = 1$, this reduces to $\mathrm{deg}\,(\cE) + 4$, and we require in this paper that this parabolic degree vanishes.
We say that $(\cE(\boldsymbol{\alpha}), \varphi)$ is \emph{stable} if, for any holomorphic line subbundle $\cV \subset \cE$ preserved by $\varphi$, one has 
\[
\mathrm{pdeg}\, (\cV(\boldsymbol{\alpha})) =\mathrm{deg}\, (\cV) + \sum \alpha_j(\mathcal{V}) < \mathrm{pdeg}\, (\cE(\boldsymbol{\alpha})),
\]
where 
\[\alpha_j(\mathcal{V}) =\begin{cases} 1- \alpha_j & V|_{p_j} = F_j \\ \alpha_j & \mbox{otherwise}. \end{cases}\]
We say that $(\cE(\boldsymbol{\alpha}), \varphi)$ is \emph{polystable} if it is a direct sum of two holomorphic line subbundles, each preserved by $\varphi$.  
We say that $(\cE(\boldsymbol{\alpha}), \varphi)$ is \emph{semistable} if equality $\mathrm{pdeg}\, (\cV(\boldsymbol{\alpha})) \leq \mathrm{pdeg}\, (\cE(\boldsymbol{\alpha}))$, and note that for generic choices of $(\boldsymbol{\alpha}, \mathbf{m})$, semistability implies stability.

The parabolic parameters $\boldsymbol{\alpha}$ 
are used directly, however, to define the correspondence with solutions to the Hitchin equations in \eqref{HE}.   Indeed, first note that the flag determines a filtration $\cF$ 
of the space of germs of holomorphic sections of $\cE$ at $p_j$, namely $\cF_{j,1} = \cO \supset \cE_{j, \alpha_j} \supset \cF_{j,0} = \cO(-1)$.  Here, $\cO(-1)$ is the space of germs of holomorphic sections of $\cE$ vanishing to first order at $p_j$.
The quotient $\cF_{j,1}/\cF_{j,0}$ is identified with the fiber $\cE_{p_j}$, while by definition $\cE_{j,\alpha_j}/\cF_{j,0}$ is the subspace $\cL_j$.  We say that 
a Hermitian metric $h$ on $\cE$ is admissible if, near $p_j$, $h = \cO( |z|^{2(1-\alpha_j) - \epsilon})$ on the germs in $\cF_{j, 0 }$ and $h = \cO( |z|^{2\alpha_j - \epsilon})$ 
on $\cF_{j,1}$, for any $\epsilon > 0$.   

The moduli space $\cMam$ is now defined to be the quotient of the space of all stable pairs $(\cE, \varphi)$ by the group of all smooth $\mathrm{SL}(2,E)$ 
gauge transformations, or alternately, as the space of all solutions of \eqref{HE} moduli the group of all smooth $\mathrm{SU}(2,E)$ gauge transformations.
The aforementioned Kobayashi--Hitchin and nonlinear Hodge correspondences require, as a central hypothesis, this parabolic stability.

\subsection{HyperK\"ahler preliminaries}
Each moduli space $\cMam$ carries a natural $L^2$ metric, which is always hyperK\"ahler and complete. 
It is most naturally defined on the moduli space of solutions of the Hitchin equations \[\mathcal{M}=\{(\delbar_A, \varphi): \delbar_A \varphi=0, F^\perp_{\delbar_A + \del^{h_0}_A} + [\varphi, \varphi^{*_{h_0}}]=0\}/ \mathcal{G}\]
on a fixed complex bundle $E \to C$ with weighted flags at $p \in D$ and adapted hermitian metric $h_0$, modulo gauge transformations preserving $h_0$.
Then, the tangent space at $(\delbar_A, \varphi)$ consists of deformations $(\dot{A}^{(0,1)}, \dot{\varphi})$ satisfying the linearized equations, modulo equivalence.
In each deformation class, there is a unique representative---the \emph{harmonic representative}---which is perpendicular to the gauge orbit. This ``Coulomb gauge'' condition and the infinitesimal version of  $ F^\perp_{\delbar_A + \del^{h_0}_A} + [\varphi, \varphi^{*_{h_0}}]=0$ can be combined into a single equation 
\begin{equation}
\del_A^{h_0} \dot{A}^{0,1} + [\varphi^{*_{h_0}}, \dot{\varphi}]=0.
\label{eq:harmonic}
\end{equation} The hyperK\"ahler metric is naturally an $L^2$-metric on the harmonic representatives.

In particular, let  $v_i = (\dot A_i^{0,1}, \dot \varphi_i)$ be harmonic tangent vectors, $i=1,2$. If $z = x+ i y$ is a local holomorphic coordinate, the K\"ahler metric on $C$ is given by $e^{2u} (dx^2 + dy^2)$, hence $dA = e^{2u} dx \wedge dy$. In particular $|dz|^2 dA = |d \bar z|^2 dA = 2 dx \wedge dy = i dz \wedge d \bar z$.
The hermitian metric on the space of harmonic tangent vectors is 
\begin{align}\label{eq:hintro}
h(v_1,v_2) &= 2 \int_C \left(\langle \dot A_1^{0,1}, \dot A_2^{0,1} \rangle + \langle \dot \varphi_1, \dot \varphi_2 \rangle \right) \, dA\\
 & =  2i \int_C \Tr \left( (\dot A_2^{0,1})^* \wedge \dot A_1^{0,1} + \dot \varphi_1 \wedge (\dot \varphi_2)^* \right)
\end{align}
and the associated Riemannian metric is 
\begin{align} \label{eq:gintro}
g(v_1,v_2) = \operatorname{Re} h(v_1,v_2)  =  -2 \operatorname{Im} \int_C \Tr \left( (\dot A_2^{0,1})^* \wedge \dot A_1^{0,1} + \dot \varphi_1 \wedge (\dot \varphi_2)^* \right).\end{align}
The factor of $2$ is explained above \eqref{eq:metric3}.
The complex structures are 
\begin{align*}
 I(\dot{A}^{0,1}, \dot \varphi) &=  (\I \dot{A}^{0,1}, \I \dot \varphi) \\  
  J(\dot{A}^{0,1}, \dot \varphi) &= ( \dot \varphi^*, -(\dot{A}^{0,1})^*) \\ 
    K(\dot{A}^{0,1}, \dot \varphi) &= (\I  \dot \varphi^*, - \I (\dot{A}^{0,1})^*).
\end{align*}
The 
K\"ahler forms\footnote{Here we have used that for $\alpha = A dz \in \Omega^{1,0}$ and $\beta = B d \bar z \in \Omega^{0,1}$ we have
\[
\overline{\Tr \alpha \wedge \beta} = \overline{ \Tr A dz \wedge B d \bar z} = \Tr A^*B^* d \bar z \wedge d z = \Tr \alpha^* \wedge \beta^* = - \Tr  \beta^* \wedge  \alpha^*
\]
and therefore $\operatorname{Re} \Tr \beta^* \wedge \alpha^* = - \operatorname{Re} \Tr \alpha \wedge \beta$ and $\operatorname{Im} \Tr \beta^* \wedge \alpha^* = \operatorname{Im} \Tr \alpha \wedge \beta$.
} $\omega_{I_j}(v,w) = g(I_{j}v,w)$, $j=1,2,3$ are 
\begin{align*}
\omega_I(v_1,v_2) &=   -2 \operatorname{Im} \int_C \Tr \left( (\dot A_2^{0,1})^* \wedge i\dot A_1^{0,1} + i\dot \varphi_1 \wedge (\dot \varphi_2)^* \right) = -2 \operatorname{Re} \int_C \Tr \left( (\dot A_2^{0,1})^* \wedge \dot A_1^{0,1} + \dot \varphi_1 \wedge (\dot \varphi_2)^* \right)  \\
\omega_J(v_1,v_2) &=   -2 \operatorname{Im} \int_C \Tr \left( (\dot A_2^{0,1})^* \wedge (\dot \varphi_1)^*  - (\dot A_1^{0,1})^*  \wedge (\dot \varphi_2)^* \right) = 2 \operatorname{Im} \int_C \Tr \left( \dot \varphi_2 \wedge \dot A_1^{0,1} - \dot \varphi_1 \wedge \dot A_2^{0,1} \right)\\
\omega_K(v_1,v_2) &=  -2 \operatorname{Im} \int_C \Tr \left( (\dot A_2^{0,1})^* \wedge i(\dot \varphi_1)^*  - i(\dot A_1^{0,1})^*  \wedge (\dot \varphi_2)^* \right)   =-2 \operatorname{Re} \int_C \Tr \left( \dot \varphi_2 \wedge \dot A_1^{0,1} - \dot \varphi_1 \wedge \dot A_2^{0,1} \right)
\end{align*}
(We derive these in Appendix \ref{sec:scalings}, and note here that the conventions are $\e^{i \theta}=1$, $\lambda_1=\lambda_2=1$.)
We note that $\Omega_I = \omega_J + i \omega_K$ is simply
\begin{align*}
\Omega_I(v_1,v_2) &= -2i \int_C \Tr \left( \dot \varphi_2 \wedge \dot A_1^{0,1} - \dot \varphi_1 \wedge \dot A_2^{0,1} \right)\\
& = 2i \int_C \Tr \left(   \dot A_1^{0,1} \wedge   \dot \varphi_2  - \dot A_2^{0,1} \wedge  \dot \varphi_1\right).
\end{align*}

\subsection{Why Affine Linearity?} \label{sec:affinelinearity} This section is motivational and not at all rigorous. In this section, we'll explain why one might expect that $[\omega_I]$ is an affine linear function of $\alpha_i$, while $[\Omega_I]$ is a linear function of $m_i$. (We are indebted to Max Zimet for this observation.)

\bigskip

 In Duistermaat--Heckman \cite{DuistermaatHeckman}, one considers a symplectic manifold $(M, \omega)$ with a (free) Hamiltonian torus action $T$, with associated $T$-equivariant moment map $\mu: M \to \mathfrak{t}^*$. 
Given $\xi \in \mathfrak{t}^*$ a regular value of $\mu$, define $Y_\xi= \mu^{-1}(\xi)$, and define $M_\xi = Y_\xi/T$.  The principal $T$-bundle $\pi: Y_\xi \to M_\xi$ is given by an element $c \in H^2(M_\zeta, \Lambda)$ where $\Lambda= \{X \in \mathfrak{t}: \exp X=\mathrm{id}\}$ is the integral lattice of $T$.  

Now, let $U$ be an open neighborhood of some fixed $\xi_0$ of regular values of $\mu$. There is a $T$-invariant Ehresmann connection on $\mu^{-1}(U)$, which gives us a $T$-equivariant map $\mu^{-1}(U) \to Y_{\xi_0}$ by parallel transport over straight lines in $\mathfrak{t}^*$.  Since $c(\xi)$ is a continuous function of $\xi$  valued in a discrete lattice, we have an honest identification of $H^2(M_\xi, \Lambda_\xi)$ and $H^2(M_{\xi_0}, \Lambda_{\xi_0})$. This  induces an identification of $H^2(M_\xi,\R)$ with $H^2(M_{\xi_0}, \R)$, which allows us to compare the symplectic forms $[\omega_\xi] \in H^2(M_{\xi}, \R) , [\omega_{\xi_0}] \in H^2(M_{\xi_0}, \R)$. Duistermaat--Heckman prove that $[\omega_\xi]- [\omega_{\xi_0}]$ depends linearly on $\xi-\xi_0$.
Moreover, the constant coefficient of this linear map is an integer, namely $[\omega_\xi]- [\omega_{\xi_0}] = \IP{c, \xi-\xi_0}$, using the pairing between the two-form $c$ valued in $\Lambda_{\xi_0} \subset \mathfrak{t}$ and $\xi-\xi_0 \in \mathfrak{t}^*$.

\smallskip

There is a well-known extension of this, in which a non-abelian compact Lie group $F$ acting freely on $(M, \omega)$ 
If an $F$-equivariant map $\mu: M \to \mathfrak{f}^*$ exists\footnote{See HKLR p. 545 for a short discussion about obstructions and references.}, $\mu$ is called a moment map, and is unique up to the addition of a constant element in $\mathfrak{f}^*$ which is fixed by $F$, i.e. a character of $F$; we call this subspace $Z \subset \mathfrak{f}^*$.  If $\xi$ is in $Z$,
then $\mu^{-1}(\xi)/F$ inherits a symplectic structure\footnote{If we take a point $\xi \in \mathfrak{f}^*$ which is not fixed by $F$, but has isotropy subgroup $H$, then $\mu^{-1}(\xi)/H$ has a symplectic structure.}.

\smallskip

There is a natural hyperK\"ahler version of the above non-abelian statement, where $F$ acts freely on a hyperK\"ahler manifold $M$, preserving its hyperK\"ahler structure. We now consider the triple of symplectic structures $(\omega_{I_1}, \omega_{I_2}, \omega_{I_3})$, and consequently, if it exists, the $F$-equivariant moment map $\mu=(\mu_1, \mu_2, \mu_3): M \to  \mathfrak{f}^* \otimes \R^3$. Similarly one takes $ \xi = (\xi_1, \xi_2, \xi_3) \in (Z \otimes \R^3)^\circ$, where the superscript $\circ$ denote the regular values of $\mu$, and considers $\mu^{-1}(\xi)/F = \cap_{i=1}^3 \mu_i^{-1}(\xi_i)/F$ admitting a hyperK\"ahler structure.

As an example of this, \cite{Kronheimerconstruction}, Kronheimer constructs ALE spaces associated to $\Gamma \subset SU(2)$ as a finite hyperK\"ahler quotient, and then quotes Duistermaat--Heckman to prove in \cite[Proposition 4.1i]{Kronheimerconstruction} that the period map 
\begin{align*}
(Z \otimes \R^3)^\circ &\mapsto H^2(M_\xi, \R)^{\oplus 3} \\
\xi &\mapsto ([\omega_{I_1}(\xi)], [\omega_{I_2}(\xi)], [\omega_{I_3}(\xi)])
\end{align*}
is a linear isomorphism, and in particular that $[\omega_{I_i}(\xi)]$ is a linear function of $\xi_i$ alone.

\bigskip

 Hitchin moduli spaces arise as hyperK\"ahler quotients of an infinite-dimensional hyperK\"ahler space $\mathcal{A}$ by an infinite-dimensional gauge group $\mathcal{G}$. Consequently, Duistermaat--Heckman's result does not apply. 
 However, it is still motivational. Morally, we want to say something along the lines of ``if an analog of the Duistermaat--Heckman result held in this infinite-dimensional setting,  it would imply the affine linearity that we prove in this paper.''
 We refer the reader to Appendix \ref{sec:scalings} for further discussion of the hyperK\"ahler moment maps: the complex moment map is proportional to $\delbar_A \varphi$ while the real moment map is proportional to $F^\perp_{\delbar_A + \del_A^h} + [\varphi, \varphi^{*_h}]$.
Ideally $\mathcal{A}$ would be (1) slightly bigger than the space in Collier--Fredrickson--Wentworth so that it would be possible to vary $\boldsymbol{\alpha}$ and (2) admit a hyperk\"ahler structure (something it does not have in Collier--Fredrickson--Wentworth's construction). However, we briefly review their configuration space  which features a fixed framing near on the cylindrical ends each $p \in D$ as motivation:
\[ \mathcal{A} = \mathcal{A}_\delta \times \mathcal{D}_\delta
\]
where \begin{align*}\mathcal{A}_\delta &= d_{A_0} + L^2_{1, \delta}(\mathfrak{su}_E \otimes T^*C^\times), \qquad \qquad \qquad  d_{A_0} = d + i \begin{pmatrix} \alpha_p & \\ & 1 - \alpha_p \end{pmatrix} \de \theta\\\mathcal{D}_\delta
&= \{ \varphi \in L^2_{-\delta}(\mathfrak{sl}_E \otimes K_C) | \delbar_{A_0} \varphi \in L^2_\delta(\mathfrak{sl}_E \otimes K_C \otimes \overline{K}_C)\},\end{align*}
while the framed gauge group $\mathcal{G}=\mathcal{K}_{\delta, *}$ defined in \cite{collier2024conformallimitsparabolicslnchiggs} has Lie algebra is $\mathrm{Lie}(\mathcal{K}_{\delta, *}) = L^2_{2, \delta}(\mathfrak{su}_E)$. Notably, all elements are the identity at $p \in D$\footnote{This is imposed by the $\mathcal{G}_{\delta, *} = \{\eta \in \mathcal{G}_{\delta}: \mathbf{b}(\eta)= \mathrm{I} \}$ appearing above \cite[(3.7)]{collier2024conformallimitsparabolicslnchiggs}.}. Consequently,---being completely unrigorous---since elements of $\mathcal{G}$ are the identity at $p \in D$, then it is possible for some distributional component supported along $D$ to be fixed by $\mathcal{G}$. 
Since $\del_{\zbar} \frac{1}{z-p} = \pi \delta_p$, one can think of $\delbar_A \varphi$ and $F_{\delbar_A+ \del_A^h}^\perp + [\varphi, \varphi^{*_h}]$ as distribution-valued $2$-forms supported at $D$.
As is discussed further in Proposition \ref{prop:distributions}, the values of the complex and real moment maps are 
\begin{align*}
M_{I}(\delbar_A, \varphi) &= 2 i \delbar_A \varphi \\
&=  \underbrace{2 i  \sum_{p \in D} \begin{pmatrix} -m_p& n_p\\ 0 & m_p \end{pmatrix} \pi \delta_p d \zbar \wedge d z}_{\psi_{\C}(\mathbf{m}, \mathbf{n})}\\
\mu_I(\delbar_A, \varphi) &=-(F^\perp_D  +[\varphi, \varphi^{*}])\\ \nonumber
&= \underbrace{- \sum_{p \in D} \begin{pmatrix} \alpha_p - \frac{1}{2} & \\ & \frac{1}{2}- \alpha_p \end{pmatrix} \pi \delta_p d \zbar \wedge d z}_{\psi_{\R}(\boldsymbol{\alpha})}.
\end{align*}
In \cite{collier2024conformallimitsparabolicslnchiggs}, the leaves of the framed extended moduli space $\mathcal{P}_0(\boldsymbol{\alpha})$ each have a hyperK\"ahler structure which should come from their construction as $\mathcal{M}(\boldsymbol{\alpha}, \mathbf{m})= \mu_\C^{-1}(\psi_\C(\mathbf{m}, \mathbf{n})) \cap \mu_\R^{-1}(\psi_\R(\boldsymbol{\alpha}))/\mathcal{G}$. Note that in \cite{collier2024conformallimitsparabolicslnchiggs}, the moduli spaces for various values of $\mathbf{n}$ can be identified since $\mathbf{n}$ is an artifact of the framing.

\subsection{Torelli map for marked ALG-\texorpdfstring{$D_4$}{D4} spaces}\label{subsec:Torelli_map}

In this section, we carefully state the Torelli Theorem of \cite{CVZ2024}. Despite being in the introduction, this section contains some new results specifying the image $\mathcal{P} \subset \R^4 \times \C^4$ more explicitly (see Corollary \ref{cor:descriptionofP}), since we will need this later. 

Let $X$ be an ALG-space of type $D_4$.
Based on \cite{ChenChenI, ChenChenIII},
given $(X, g, I, J , K)$ an ALG-$D_4$ space, we may associate a \emph{complex torus parameter} $\tau \in \mathbb{H}/PSL(2,\Z)$ and \emph{length scale} $\ell$. Since $X \to \C$ is elliptically fibered, $\ell$ is defined by $\int_F \omega_I = \ell^2 \mathrm{Im} \tau$ for fiber $F$. 

\bigskip

Let $I_0$ denote the quadratic form on $\Z^5$ given by the symmetric matrix
\begin{equation}\label{eq:I0}
\begin{pmatrix}
	-2 &  1 &1 &1 &1 \\
	1 & -2 & 0 & 0 & 0\\
	1 & 0 & -2 & 0 & 0 \\
	1 & 0 & 0 & -2 & 0 \\
	1 & 0 & 0 & 0 & -2 
\end{pmatrix}.\end{equation}
\begin{defn}
An \emph{adapted frame} of $H_2(X, \Z)$ is an isomorphism $f : \Z^5  \to H_2(X,\Z)$ preserving quadratic forms and the fiber class:
\begin{enumerate}
\item $f^* I = I_0$, and
\item $f(2 e_0 + \sum_{i=1}^4 e_i) = F \in H_2(\mathcal{M}(\bfalpha,\bfm),\Z)$.
\end{enumerate}
\end{defn}

Note that since $2e_0 + \sum_{i=1}^4 e_i$ spans the kernel of $I_0$ in $\Z^5$ and $F$ the kernel of $I$ in $H_2(\mathcal{M}(\bfalpha,\bfm),\Z)$, one necessarily has $f(2 e_0 + \sum_{i=1}^4 e_i) = \pm F$ for a frame satisfying the first condition only.
\begin{defn}
Given an adapted frame, let $S_i = f(e_i)$ for $i=0, \ldots, 4$. 
We call $\mathfrak{B} = (S_0, \ldots, S_4)$  \emph{an admissible basis} of $H_2(X,\Z)$. 
\end{defn}

\bigskip

We now define the relevant moduli space.
\begin{defn}\hfill
\begin{enumerate} 
\item
Fix $X$ an ALG-$D_4$ space  with fixed $\tau \in \mathbb H$ and scale $\ell>0$. A hyperk\"ahler manifold $(M, g, I,J,K)$ together with diffeomorphism $\mu : X \to M$ is called a \emph {marked ALG-$D_4$ space}. 
\item Two marked ALG-$D_4$ spaces $(M,g, I,J,K)$ and $(M',g', I',J',K')$ are considered \emph{equivalent} if there exists a diffeomorphism $f : M \to M'$ such that $f^*g' =g$, $f^* I'=I$, $f^* J'=J$,$f^* K'=K$ and 
\[
f_* \circ \mu_* = \mu'_* : H_2(X, \Z) \to H_2(M', \Z). 
\]
\item[(3a)] Define $\calM_{ALG}(X; D_4,\tau, \ell)$ as the set of equivalence classes of marked ALG-structures.
\item[(3b)] Consider ALG-structures $(g, I,J,K)$ on the fixed 	manifold $X$. Let $\calH=H_2(X, \Z)$ and define $\Diff_\calH(X) = \{ f \in \Diff(X) : f_* = \Id : \calH \to \calH \}$. Then set
\[
\calM_{ALG}(X;D_4,\tau, \ell) = \{ (g, I, J, K) \text{ ALG-structure on $X$}\} / \Diff_\calH(X).
\]
\end{enumerate}
\end{defn}
Both ways of defining the moduli space in (3a) and (3b) are equivalent (just pull the structure on $M$ back  to $X$ by the marking diffeomorphism $\mu$).

The following Torelli theorem has been proven by \cite{CVZ2024} and the surjectivity part independently by \cite{LeeLin}.
\begin{thm}[Torelli theorem  for ALG spaces] Let $X$ be an ALG-$D_4$ space with associated parameters $\tau, \ell$ and let $\mathfrak{B}=(S_0, \ldots, S_4)$ be an admissible basis of $H_2(X,\Z)$. 
The period domain of $\mathfrak{B}$ is defined by 
\begin{equation}\label{eq:P}
 \calP = (\R \times \C)^4 \setminus\bigcup \{( (x_j, z_j)_{j=1}^4: \sum_{i=1}^4 (\lambda_0 - 2 \lambda_i)(x_i,z_i) = \lambda_0 (\ell^2 \im \tau, 0) \}
\end{equation}
where the union is over integers $\lambda_0, \ldots , \lambda_4$ for which the class $\Sigma = \lambda_0  S_0 + \sum_{i=1}^4 \lambda_i S_i$ has self-intersection $-2$.
Then the Torelli map
\begin{align}
\mathcal{T}_{\frakB} : \calM_{ALG}(X,I_0^*,\tau, \ell) &\longrightarrow \calP \\ \nonumber
 [(g, I, J, K)] &\longmapsto \left\{\left( \int_{S_i} \omega_I, \int_{S_i} \Omega_I \right)\right\}_{i=1}^4
\end{align}
is a bijection.
\end{thm}

\begin{rem}[Explanation of $\mathcal{P}$]
Observe that the class $[F]$ of the elliptic (Hitchin) fiber $F$ satisfies
\begin{equation} 
\int_F \omega_I = \ell^2 \im \tau, \qquad \int_F \Omega_I=0,\ \ i = 1, 2, 3.
\label{eq:fiberintegral}
\end{equation}
Indeed, each fiber is a `BAA brane', namely complex for the complex structure $I$ and Lagrangian with respect to $J$ and $K$.  Since
$2[S_0] + [S_1] + [S_2] + [S_3] + [S_4] = [F]$, the corresponding combination of the basis periods must also vanish.  

However, this is not the only constraint. Recall also that an integral homology class $[\Sigma]$ of self-intersection $-2$ in an ALG space can be represented by a curve
which is holomorphic with respect to some complex structure $aI + bJ + cK$, $a^2 + b^2 + c^2 = 1$, hence these
classes represent further constraints. For example, this implies that the set of periods of such a curve cannot all vanish, see \cite{LeeLin}.  To understand the
further constraints, let us first write $\Sigma = \sum_{i=0}^4 \lambda_i S_i$.  We then compute that
\begin{align*}  
\int_\Sigma \omega_I & = \sum_{i=0}^5 \lambda_i \int_{S_i} \omega_I  \\
&= \frac{\lambda_0}{2}\left(\int_F \omega_I - 
\sum_{i=1}^4 \int_{S_i} \omega_I\right) + \sum_{i=1}^4 \lambda_i \int_{S_i} \omega_I\\   &=\frac{\lambda_0}{2} \int_F \omega_I + \sum_{i=1}^4 \left(\lambda_i - \frac{\lambda_0}{2}\right)\int_{S_i}\omega_I,
\end{align*}
with a similar computation for the integrals of $\Omega_I$.  We conclude from this that 
\[
 \sum_{i=1}^4 (\lambda_0 - 2 \lambda_i)\left( \int_{S_i} \omega_I , \int_{S_i} \Omega_I\right)  \neq \lambda_0 ( \ell^2 \im \tau, 0) 
\]
This is precisely the definition of $\mathcal{P}$ in \eqref{eq:P} appearing above.
\end{rem}

\begin{rem}
	We have stated the Torelli theorem only for ALG-spaces of type $I_0^*$ since this is the only asymptotic geometry we will encounter in this article. The Torelli theorem is more generally true for ALG-structures if the order is at least 2. In the $I_0^*$-case this places no further restriction, see \cite{CVZ2024}
\end{rem}

\bigskip

As written, it appears that $\calP$ is defined relative to a particular choice of basis $\frakB$. We will see that in fact this domain is 
independent of that basis.
Let
\[
\OO(\Z^5,I_0)^+ = \left\{ A \in \OO(\Z^5,I_0): A(2e_0 + \sum_{i=1}^4 e_i) =  2e_0 + \sum_{i=1}^4 e_i \right\}.
\]
There is a natural right action of $\OO(\Z^5,I_0)^+$ on adapted frames, namely $f\cdot A =f \circ A$ for $A \in \OO(\Z^5,I_0)^+$. If $A= (a_{ij}) \in \OO(\Z^5,I_0)^+$ then the corresponding right action on admissible bases is given by $\mathfrak{B} \cdot A = A^t \cdot \mathfrak{B}$ where we think of $\mathfrak{B}$ as the column vector $(S_0, \ldots, S_4)^t$, more precisely
\[
S_j' = f'(e_j) = f(A(e_j)) = \sum_{i=0}^4 a_{ij} f(e_i) = \sum_{i=0}^4 a_{ij} S_i.
\]

\begin{lem}\label{lem:period_equivariance}
The period map satisfies the equivariance condition
\[
\mathcal{T}_{\mathfrak{B} \cdot A}([(g, I, J, K)])= \mathcal{T}_{\mathfrak{B} }([(g, I, J, K)])\cdot A
\]	
for $A \in \OO(\Z^5,I_0)^+$. 
\end{lem}
\begin{proof}
By linearity of $C \mapsto \int_C \omega_I$ (resp.\ $C \mapsto \int_C \Omega_I$)
\[
\int_{S_j'} \omega_I  = \sum_{i=0}^4 a_{ij}  \int_{S_i}  \omega_I \quad \text{and} \quad  \int_{S_j'}\Omega_I  = \sum_{i=0}^4 a_{ij}  \int_{S_i}  \Omega_I.
\]
If we set $x= (x_0, \ldots , x_4)^t$ for $x_i =\int_{S_i} \omega_I$ and $z= (z_0, \ldots , z_4)^t$ for $z_i =\int_{S_i} \Omega_I$ and $x',z'$ correspondingly  this amounts to 
\[
x' = A^t \cdot x \quad \text{and} \quad z' = A^t \cdot z.
\]
Using the homological relation $2S_0 + \sum_{i=1}^4 S_i = F$  together with $\int_F \omega_I = 4 \pi^2$ and $\int_F \Omega_I=0$ we may express $x_0= 2 \pi^2 - \frac 12 \sum_{i=1}^4 x_i$ and $z_0 = - \frac 12 \sum_{i=1}^4 z_i$. Hence, we obtain
\[
\begin{pmatrix} 
x_1' \\ x_2' \\ x_3' \\x_4'	
\end{pmatrix}
=
\hat{A}^t \cdot 
\begin{pmatrix} 
x_1 \\ x_2 \\ x_3 \\x_4	
\end{pmatrix}
+ \hat{b}
\quad \text{and} \quad
\begin{pmatrix} 
z_1' \\ z_2' \\ z_3' \\z_4'	
\end{pmatrix}
=
\hat{A} \cdot 
\begin{pmatrix} 
z_1 \\ z_2 \\ z_3 \\z_4	
\end{pmatrix}
\]
for $\hat{A} \in \GL(4,\R)$ and $\hat{b} \in \R^4$ defined by $\hat{a}_{ij} = a_{ij} - \frac 12 a_{0j}$  and $\hat{b}_j = 2 \pi^2 a_{0j}$, $i,j \in \{1, 2, 3, 4\}$. This determines a right action of $\OO(\Z^5,I_0)^+$ on $\R^4$ by affine-linear and on $\C^4$ by linear transformations which we will denote by $(\mathbf x, \mathbf z) \mapsto (\mathbf x, \mathbf z) \cdot A$. The following is then clear.
\end{proof}

We shall convince ourselves in the following that the right action of $\OO(\Z^5,I_0)^+ $ on the target $\R^4 \times \C^4$ restricts to an action on the period domain and that the period domain does not depend on the choice of admissible basis $\mathfrak{B}$.

\begin{lem}\label{lem:invariance_period_domain}
The period domain $\calP \subset \R^4 \times \C^4$ is invariant under the action of $\OO(\Z^5,I_0)^+$. More precisely, if $\mathfrak{B} = (S_0, \ldots, S_4)$ is an admissible basis, then $\mathcal{H}_{\boldsymbol{\lambda}} \cdot A = \mathcal{H}_{A^{-1}\boldsymbol{\lambda}}$ for $A \in \OO(\Z^5,I_0)^+$, $\boldsymbol{\lambda} = (\lambda_0, \ldots, \lambda_4) \in \Z^5$ and the codimension-three subspaces
\[
\mathcal{H}_{\boldsymbol{\lambda}} = \bigl\{\sum_{i=1}^4 (\lambda_0 -2 \lambda_i) (x_i,z_i) = \lambda_0( 4 \pi^2, 0) \bigr\} \subset \R^4 \times \C^4 
\]
where $[\Sigma]= \sum_{i=0}^4 \lambda_i [S_i] $ has self-intersection $-2$. 
\end{lem}

\begin{proof}
First recall that after setting $x_0 = \frac{1}{2} ( 4 \pi^2 - \sum_{i=1}^4 x_i)$ and $z_0 = - \frac{1}{2} \sum_{i=1}^4 z_i$ the inhomogeneous equation defining $\mathcal{H}_{\boldsymbol{\lambda}}$ is equivalent to the homogeneous equation $\sum_{i=0}^4 \lambda_i(x_i, z_i) = (0,0)$. If $[\Sigma] = \sum_{i=0}^4 \lambda_i[S_i]$ and $\mathfrak{B}' = A^t \cdot \mathfrak{B}$ for $A \in \OO(\Z^5,I_0)^+$, then
\[
\Sigma = \boldsymbol{\lambda}^t \cdot \mathfrak{B} =(A^{-1} \boldsymbol{\lambda})^t \cdot A^t \mathfrak{B}= (\boldsymbol{\lambda}')^t \cdot \mathfrak{B}' = \sum_{i=0}^4 \lambda_i'S_i'
\]
with $\boldsymbol{\lambda}'  = A^{-1} \boldsymbol{\lambda}$ for $\boldsymbol{\lambda}=(\lambda_0, \ldots, \lambda_4) $ and $\boldsymbol{\lambda}' = (\lambda_0', \ldots, \lambda_4')$. 

Now for $\mathbf{x}=(x_0, \ldots, x_4)$ and $\mathbf{z}=(z_0, \ldots, z_4)$ 
\begin{align*}
\sum_{i=0}^4 \lambda_i(x_i, z_i) = (0,0) &\Longleftrightarrow \boldsymbol{\lambda}^t \cdot \mathbf{x} = \boldsymbol{\lambda}^t \cdot \mathbf{z} = 0 \\
&\Longleftrightarrow (A^{-1} \boldsymbol{\lambda})^t \cdot A^{t}\mathbf{x}  = (A^{-1} \boldsymbol{\lambda})^t \cdot A^{t}\mathbf{z} = 0 \\  
&\Longleftrightarrow (\boldsymbol{\lambda}')^t \cdot \mathbf{x}'  = (\boldsymbol{\lambda}')^t \cdot \mathbf{z}' = 0 \Longleftrightarrow \sum_{i=0}^4 \lambda_i'(x_i', z_i') = (0,0)
\end{align*}
for $\boldsymbol{\lambda}'=A^{-1}\boldsymbol{\lambda}$ as above, $\mathbf{x}' =A^{t} \mathbf{x}$ and $\mathbf{z}' = A^{t}\mathbf{z}$. 
This equivalence show that $\mathcal{H}_{\boldsymbol{\lambda}} \cdot A = \mathcal{H}_{A^{-1}\boldsymbol{\lambda}}$ for $A \in \OO(\Z^5,I_0)^+$. 
\end{proof}

\begin{cor}
The period domain $\calP \subset \R^4 \times \C^4$ does not depend on the choice of admissible basis $\mathfrak{B}=(S_0, \ldots, S_4)$ of $H_2(\mathcal{M}(\bfalpha,\bfm),\Z)$.
\end{cor}

\begin{proof}
	The orthogonal group $\OO(\Z^5, I_0)^+$ acts transitively on the set of admissible bases. Now $\boldsymbol{\lambda} \in \Z^5$ gives rise to a codimension-three subspace $\mathcal{H}_{\boldsymbol{\lambda}}$ with respect to the basis $\mathfrak{B}=(S_0, \ldots, S_4)$ if and only if $C = \sum_{i=0}^4 \lambda_iS_i$ has self-intersection $-2$. Reexpressing $C = \sum_{i=0}^4 \lambda_i'S_i'$ for $\boldsymbol{\lambda}'  = A^{-1} \boldsymbol{\lambda}$ and $\mathfrak{B}'=A^t \cdot \mathfrak{B}$ for $A \in \OO(\Z^5, I_0)^+$ as above yields that $\boldsymbol{\lambda}' \in \Z^5$ gives rise to a codimension-three subspace $\mathcal{H}_{\boldsymbol{\lambda}'}$ with respect to the basis $\mathfrak{B}'=(S_0', \ldots, S_4')$. But $\mathcal{H}_{\boldsymbol{\lambda}'}$ is the translate of $\mathcal{H}_{\boldsymbol{\lambda}}$ under the element $A \in \OO(\Z^5, I_0)^+$ which proves the claim.
		\end{proof}

It remains to determine which linear combinations of the $S_i$ have self-intersection $-2$.
\begin{lem}
The classes $[\Sigma] = \lambda_0  [S_0] + \sum_{i=1}^4 \lambda_i [S_i]$ of self-intersection $-2$ in $H_2(X, \Z)$ have
coefficients $(\lambda_0, \boldsymbol{\lambda}) \in \Z^5$ given by 
\begin{enumerate}
\item $\boldsymbol{\lambda} = k (1,1,1,1)$, $k \in \Z$ and $\lambda_0 = 2k \pm 1$,
\smallskip \
\item $\boldsymbol{\lambda} = k (1,1,1,1) \pm e_i$, $k \in \Z$, $i \in \{1, \ldots, 4\}$ and $\lambda_0 = 2k, 2k \pm 1$,
\smallskip
\item $\boldsymbol{\lambda} = k (1,1,1,1) \pm (e_i+e_j)$, $k \in \Z$, $ i \neq j \in \{1, \ldots, 4\}$ and $\lambda_0 =  2k \pm 1$,
\end{enumerate}
where the $e_j$, $j = 1, \ldots, 4$ are the standard basis vectors of $\R^4$ and $\boldsymbol{\lambda} = (\lambda_1, \ldots, \lambda_4)$.
\end{lem}
\begin{proof}
Using the intersection matrix \eqref{eq:intersectionform}, 
\begin{align*}
[\Sigma]\cdot [\Sigma] = -2 \Bigl( \lambda_0^2 +\sum_{i=1}^4 \lambda_i^2 - \lambda_0 \sum_{i=1}^4 \lambda_i \Bigr) = -2
\end{align*}
hence
\[
\lambda_0^2 - \lambda_0 \bigl(\sum_{i=1}^4 \lambda_i\bigr) + \bigl( \sum_{i=1}^4 \lambda_i^2)- 1 =0 
\]
This quadratic equation has discriminant,
\[
D = \frac{1}{4} \bigl(\sum_{i=1}^4 \lambda_i\bigr)^2 +1 -  \sum_{i=1}^4 \lambda_i  = 1 - \frac{1}{4} \sum_{i<j} (\lambda_i-\lambda_j)^2 \leq 1,
\]
so there can be integer solutions only when $D \in \{0, \frac{1}{4}, 1\}$. The solutions (1), (2) and (3) correspond to these three cases.
\end{proof}
\begin{cor}\label{cor:descriptionofP}
The period domain $\calP$ is the complement of the collection of codimension-three planes
\[
\bigcup_{k \in \Z} \calH_k \cup \calH_{k,i_0} \cup \calH'_{k,i_0} \cup \calH_{k, i_1, i_2},
\]
where
\begin{enumerate}
\item[(1)] $\mathcal{H}_k = \{ \sum_{i=1}^4 x_i =(2k+1) \ell^2 \im\tau, \, \sum_{i=1}^4 z_i =0 \} $, 
\smallskip 
\item[(2)] $\mathcal{H}_{k,i_0} = \{ x_{i_0} = k \ell^2 \im \tau, \,z_{i_0} = 0 \}$, 
\smallskip
\item[] $\mathcal{H}'_{k,i_0} = \{ 2x_{i_0} - \sum_{i=1}^4 x_i = (2k+1) \ell^2 \im\tau, \, 2 z_{i_0} - \sum_{i=1}^4 z_i = 0 \}$,  $i_0 \in \{1, \ldots ,4\}$,
\smallskip
\item[(3)] $\mathcal{H}_{k,i_1, i_2}= \{ 2 (x_{i_1} + x_{i_2}) -\sum_{i=1}^4 x_i)  = (2k + 1) \ell^2 \im\tau, \, 2 (z_{i_1} + z_{i_2}) -\sum_{i=1}^4 z_i = 0 \}$, 
$i_1 \neq i_2 \in \{1, \ldots , 4 \}$.
\end{enumerate}
\end{cor}
\begin{ex}
The spheres $S_0, S_1, S_2, S_3, S_4$ themselves have self-intersection $-2$ in $H_2(X,\Z)$. The corresponding codimension-three planes are then
\begin{enumerate}
\item $\mathcal{H}_{S_0}= \{ \sum_{i=1}^4 x_i = \ell^2 \im \tau, \, \sum_{i=1}^4 z_i =0 \}$,
\smallskip
\item $\mathcal{H}_{S_i} = \{ x_{i} = 0, \, z_i = 0 \}$,\ $i = 1, \ldots, 4$.
\end{enumerate}
Note that the intersections of these codimension three planes with $\R^4 \times \{0\}$ bound 
the 4-simplex with vertices $V_0=(0,0,0,0)$ and $V_i = \ell^2 \im \tau e_i$, $i = 1, \ldots, 4$.
\end{ex}

\subsection{Organization of the Paper}

\begin{enumerate}
\item[\S\ref{sec:extended}]A key tool in this paper is the fact that the various moduli spaces $\mathcal{M}(\bsa, \bfm)$ fit together into a larger `extended' moduli space, in which the complex masses (and the real masses) vary.  The larger moduli space $\mathcal{P}_0(\bsa)$ constructed in \cite{collier2024conformallimitsparabolicslnchiggs} in which the complex masses vary is particularly useful for us, and we discuss the $\C^\times$-action on the space.
\item[\S\ref{sec:distinguished_basis}] We review the chamber structure for the space of real mass parameters $\bsa$. This leads to an explicit description of the nilpotent cone for $\mathcal{M}(\bsa, \mathbf{0})$, which in turn gives a distinguished $2$-homology basis for $\mathcal{M}(\bsa, \bfm)$.
\item[\S\ref{sec:Torelli}] The first main theorem appears in Theorem \ref{thm:Torelli}. It is a statement about the integral of $\omega_I$ and $\Omega_J$ over the distinguished $2$-homology basis for $\mathcal{M}(\bsa, \bfm)$. Notably, as we mentioned before, each integral of $\omega_I$ is an affine linear function depending on $\bsa$ while $\Omega$ is a linear function depending on $\bfm$. \item[\S\ref{sec:TorelliomegaI}] In this section, we compute the integral of $\omega_I$ over the distinguished $2$-homology basis for $\mathcal{M}(\bsa, \mathbf{0})$. We then use the $\C^\times$-action on the larger moduli space $\mathcal{P}_0(\bsa)$ to compute the integral of $\omega_I$ over the distinguished $2$-homology basis of $\mathcal{M}(\bsa, \bfm)$. The main technical result here is Proposition \ref{prop:dlambda}, which implies Proposition \ref{prop:sameperiods}.
\item[\S\ref{sec:integralasintersectionnumber}] In this section, we prove in Proposition \ref{prop:primitive} we prove that the tautological $1$-form $\tau$ is a primitive for the holomorphic symplectic form $\Omega_I$ on the complement of a divisor, that we call the polar divisor (see Definition \ref{def:polarsections}); lastly, in Corollary \ref{cor:intersection_numbers}, we prove that the integral of $\Omega_I$ over $S^2$ can be computed via from the intersection numbers of $S^2$ with the spheres $\Sigma_j = \sigma_j^+ - \sigma_j^-$ formed from the polar sections. 
\item[\S\ref{sec:concrete}] The goal of this section is to parameterize an open dense subset of the Higgs bundle moduli space, as well as the tautological $1$-form, so that the required intersection numbers can be computed. 
\item[\S\ref{sec:TorelliOmegaI}] In this section, we compute the integral of $\Omega_I$ over the distinguished $2$-homology basis for  $\mathcal{M}(\bsa, \mathbf{m})$ when $\bsa$ is in two (of the 24) chambers.  
\item[\S\ref{sec:Dehntwists}] Finally, we discuss the action of the affine $D_4$ Coxeter group on various spaces and prove that the Torelli map is equivariant with respect to this action. Using this, we prove the \emph{near} surjectivity of the Torelli map onto the period domain in Theorem\ref{thm:surjective}.
\end{enumerate}
Lastly, we note that in Appendix \ref{sec:scalings}, we write out the $\R^+ \times \R^+ \times U(1)$ family of hyperK\"ahler structures one could naturally put on the moduli space of Higgs bundles,
while in Appendix \ref{sec:homologyspheres} we discuss the location of the homology spheres in the weakly parabolic (i.e. $\mathbf{m} \neq \mathbf{0}$) setting.

\subsection{Acknowledgements}

We received substantial advice throughout this project and we wish to thank in particular  
Philip Boalch (Modularity Conjecture), 
Ben Elias (affine Coxeter groups), 
Hans-Joachim Hein (Diffeomorphisms of ALGs), 
Sebastian Heller (explaining related work \cite{HHT2025} to us),
Paul Melvin (identifying $H_2(\cM,\Z)$ inside generic weakly parabolic moduli space), 
Richard Wentworth (extended moduli space), 
Claudio Meneses (description of the nilpotent cone and the use of his images), and 
Arya Yae (simplified calculation of $U(1)$ moment map).

Each of the four of us were supported at various times by the GEAR network NSF grant DMS 1107452, 1107263, 1107367  ``RNMS: Geometric Structures and Representation Varieties'';  RM was supported by
NSF DMS-1608223, LF was supported by NSF DMS-2005258, and JS and HW were supported by DFG SPP 2026. JS was supported   by a Heisenberg grant  of the  DFG  and  within the  DFG RTG 2229 ``Asymptotic invariants and limits of groups and spaces''. This work  is supported by   DFG under  Germany's Excellence Strategy EXC-2181/1 -- 390900948 (the Heidelberg  STRUCTURES Cluster of Excellence).
This paper is based upon work supported by the National Science Foundation under Grant No. DMS-1440140 while we were all in residence at the Mathematical Sciences Research Institute (now SLMath) in Berkeley, California, during the Fall 2019 semester and again while LF, RM, and HW were in residence during the Fall 2022 semester.

\section{Extended Moduli Space}\label{sec:extended}
A key tool in this paper is the fact that the various moduli spaces $\calM(\bsa,\bfm)$ fit together into a larger `extended' moduli space $\tcalX$
by letting the real and complex masses vary. In other words, $\tcalX$ is a fibration over an appropriate open dense subset in $(0, \frac12)^4 \times \C^4$
with fiber at $(\bsa, \bfm)$ equal to $\calM(\bsa,\bfm)$.    This construction has been achieved in different categories by various authors. 
First, regarding the individual fibers, Yokogawa \cite{yokogawa93} constructed the moduli space of semistable parabolic Higgs bundles 
for a given $(\bsa, \bfm)$ as a quasi-projective variety. Analytic constructions of the parabolic moduli spaces have also been given, first
by Konno \cite{konno93} in the strongly parabolic rank $2$ case, i.e., where $\bfm = 0$, then by Nakajima \cite{Nakajima} in the weakly 
parabolic ($\bfm \neq 0$) rank $2$ case.  More recently, the first author, together with Collier and Wentworth \cite{collier2024conformallimitsparabolicslnchiggs} 
constructed $\mathcal{P}_0(\boldsymbol{\alpha})$\footnote{They use the notation $\mathcal{P}_0(\bsa)$ for this space, while they use $\mathcal{P}(\bsa) \to \mathbb{C}$ for the space of parabolic logarithmic $\lambda$-connections.}, which as a set is $\sqcup_{\bfm \in \C^4}\calM(\bsa, \bfm)$.  This analytic construction gives rise to a hyperK\"ahler metric on the smooth locus of each $\calM(\bsa,\bfm)$. We 
review the construction of \cite{collier2024conformallimitsparabolicslnchiggs} later in this section with a particular emphasis on the hyperK\"ahler structure.

In any of these constructions, $\calM(\bsa, \bfm)$ acquires the structure of a complex manifold for generic parameters $(\bsa,\bfm)$. Here genericity
of $(\bsa, \bfm) \in (0, \frac12)^4 \times \C^4$ corresponds to the condition that $\bsa$-semistability implies $\bsa$-stability (see Definition \ref{defn:gen1} below). 

In this section we describe this fibration and some of its basic properties. 

\subsection{Genericity and Nakajima's Walls}\label{subsec:period_domain}

Following Nakajima in \cite{Nakajima}, we define the subset of generic parameters $(\boldsymbol{\alpha}, \mathbf{m})$ as follows. 
\begin{defn}\label{defn:gen1}
$(\boldsymbol{\alpha}, \mathbf{m}) \in (0, \frac12)^4 \times \C^4$ is called \emph{generic} if $\boldsymbol{\alpha}$-semistability implies $\boldsymbol{\alpha}$-stability.
\end{defn}
This is the complement of a finite union of codimension-three planes.
\begin{prop}\label{prop:generic}
The set $\mathcal R$ of generic parameters $(\boldsymbol{\alpha}, \mathbf{m}) \in (0, \frac12)^4 \times \C^4$ is characterized as 
\begin{equation}\label{eq:R}
\mathcal{R} = (0, \textstyle\frac12)^4 \times \C^4 \setminus \bigcup_{d \in \Z, \mathbf{e} \in \{0, 1\}^4} \mathcal{H}_{d, \mathbf{e}}, 
\end{equation}
where 
\begin{equation}
\mathcal{H}_{d, \mathbf{e}} = \{  (\boldsymbol{\alpha}, \mathbf{m}) \in (0, \textstyle \frac12)^4 \times \C^4 :   
d + \sum_{i=1}^4 (e_i + (-1)^{e_i} \alpha_i)  =0  \ \   \&  \ \    \sum_{i=1}^4 (-1)^{e_i} m_i = 0.\}.
\label{Nakajimagenericm} 
\end{equation}
\end{prop}
This is \cite[Lemma 2.4]{Nakajima}, adapted to our setting.
\begin{cor}
If $(\bsa, \mathbf{0})$ is generic, then $(\bsa, \mathbf{m})$  is generic for all $\bfm \in \mathbb{C}^4$.
\end{cor}
This Corollary suggests that although genericity has been defined for pairs $(\bsa, \bfm)$, it makes sense to refer to a given $\bsa$ as being generic.
We note, finally, that since $\mathcal R$ is the complement of a collection of codimension-three planes, 
\begin{cor}
$\mathcal{R}$ is both connected and simply-connected.
\end{cor}

\subsection{Smoothness of the fibration} \label{sec:bundle} 
We now review the smooth bundle structure of $\tcalX$, with the added feature that the fibers are quasiprojective varieties.

\begin{prop}[Corollary of \cite{ISS19a, ISS19b}]\label{prop:fiberbundle} 
The complex manifolds $\calM(\bsa,\bfm)$, fit together to form a smooth fiber bundle 
$\pi: \tcalX = \coprod_{(\bsa,\bfm) \in \calR} \calM(\bsa, \bfm) \to \calR$.  Each fiber $\pi^{-1}( (\bsa, \bfm)) = \calM(\bsa,\bfm)$
is a smooth affine variety with completion a rational elliptic surface. The compactifying divisor is an $I_0^*$ fiber. 
\end{prop}
\begin{proof} 
Given a contractible open set $U \subset \mathcal{R}$ containing the base point $(\bsa,\bfm)$, we must produce a local trivialization $\calX|_U  \cong U \times
\calM(\boldsymbol{\alpha_0}, \mathbf{m_0})$. For the 4-punctured sphere, such a trivialization exists by the work of Ivanics, Stipsicz and 
Szab\'o \cite{ISS19a}, \cite{ISS19b} who construct the weakly parabolic moduli spaces in that case as follows.

For any set of complex masses $\bfm = (m_1, m_2, m_3, m_4) \in \C^4$, we consider the family of spectral curves $\Sigma_b \subset K_C(D)$
parameterized by $b \in \calB = \calB(\mathbf{m})$, the Hitchin base as in Proposition \ref{prop:Hitchinbase}. By construction, this family depends 
smoothly on $\bfm$, but is independent of $\bsa$. It extends to a pencil of curves in the projective compactification of $K_C(D)$ which, since 
$K_C(D) = \cO(2)$, equals the Hirzebruch surface $\mathbb{F}_2 = \mathbb{P}(\cO(2) \oplus \cO)$. The compactifying divisor curve $C_\infty$ has 
five components: the section at infinity $\sigma_\infty=\mathbb{F}_2 \setminus K_C(D)$ and the four projective lines $F_1, F_2, F_3, F_4$, 
which are the fibers of $p_1, p_2, p_3, p_4 \in C$ in the ruling $\mathbb{F}_2 \to C$. The base locus of this pencil is $F_1 \cup F_2 \cup F_3 \cup F_4$ 
since all spectral curves pass through the points $\lim_{z\to p_i}\frac{ \pm m_i}{z-p_i} dz$, regarded as points in the total space of $K_C(D)$. 
In particular there are two base points in $F_i$ if $m_i \neq 0$, and otherwise there is a single one with multiplicity 2.  This results in 8 base points,
and blowing up all of these results in a surface $X$ biholomorphic to the rational elliptic surface $\CP^2 \#9 \overline{\CP}^2$. Under the Hitchin
fibration, the pencil becomes an elliptic fibration on $X$, and $\calM(\bsa,\bfm)$ is identified with $\calM = X \setminus C_\infty$. (Here $C_\infty$ 
is identified with its proper transform in $X$). This show that the moduli spaces $\calM(\bsa,\bfm)$ are mutually diffeomorphic and form a smooth family 
over $\calR$. 

Since the fiber $C_\infty$ is a singular fiber of type $I_0^*$, this argument implies that the weakly parabolic moduli spaces compactify to a rational elliptic surface 
with $I_0^*$-fiber at infinity.
\end{proof}

\begin{rem} To foreshadow a later discussion, we observe that by a result of Hein \cite{hein43}, each open fiber admits a complete hyperK\"ahler metric of type ALG-$D4$. 
This fiber also admits Hitchin's hyperK\"ahler metric, and it is not a priori clear that these two metrics coincide. Failure to do so would be very surprising, and
could only happen if the Hitchin metric didn't have finite energy, i.e. the Riemannian curvature tensor was not $L^2$-integrable.  We proved in \cite{FMSW} that when $\bfm = \mathbf{0}$, Hitchin's metric
is ALG. In the sequel to this paper we prove that this remains true even for nonzero complex masses.
\end{rem}

\begin{rem}
The construction in Proposition \ref{prop:fiberbundle}  is equivalent to a familiar description of the strongly parabolic moduli space $\calM(\bsa,\mathbf{0})$, which has
been called the `toy model' \cite{MR1650276}.  That description proceeds by taking the product $\C \times T^2_\tau$ for some $\tau \in \mathbb{H}$,
and dividing by the $\Z_2$ action which acts by $\pm 1$ in each factor. The quotient $(\C \times T^2_\tau) / \Z_2$ has 4 orbifold singularities, all
lying in the central fiber $(\{0\} \times T^2_\tau)/\Z_2$, which is thus the pillow-case orbifold. Blowing up these 4 orbifold singularities yields the elliptic 
surface $\calM$ with a singular fiber of type $I_0^*$ over $0 \in \C/\Z_2$. This fiber is identified with the nilpotent cone of the corresponding strongly 
parabolic Higgs bundle moduli space and consists of 5 spheres $S_0,S_1,S_2,S_3,S_4$ which constitute a basis for $H_2(\cM, \Z) \cong \Z^5 $ 
for which the intersection form agrees with the intersection form for the affine $D_4$ Dynkin diagram: 
\begin{equation}\label{eq:intersectionform}
 \begin{pmatrix}
 -2 & 1 & 1  & 1 & 1 \\ 1 & -2 & 0 & 0 & 0 \\ 1 & 0 & -2 & 0 & 0 \\ 1 & 0 & 0 & -2 & 0 \\ 1 & 0 & 0 & 0 & -2
\end{pmatrix}
\end{equation}
This space can be compactified by adding another $I_0^*$ fiber at infinity, and the resulting space is a rational elliptic surface. The modulus $\tau \in \mathbb{H}$ 
of the torus fiber is determined by the cross-ratio of the points $p_1, p_2, p_3, p_4 \in D$, see  \cite{FMSW}.
\end{rem}
\subsection{Second Homology Bundle}
At each point of $\mathcal R$, the homology lattice $H_2(\cM, \Z)$ has rank $5$, and has an integral basis $S_0, S_1, S_2, S_3, S_4$ 
where the intersection form agrees with the intersection form for the affine $D_4$ Dynkin diagram. We call such bases {\it admissible}.
When $\bfm = \mathbf{0}$ and $\bsa$
lies in a certain subset of $\calR$, there is a particularly natural basis where the central sphere $S_0$ corresponds to the set of stable
parabolic bundles with Higgs field $\varphi \equiv 0$, and the other $S_i$ can also be explicitly described in terms of Higgs data, see Section \ref{sec:distinguished_basis} below.
An important point is that for any integral basis with this intersection form, the cycle $2S_0 + S_1 + S_2 + S_3 + S_4$ is homologous
to the Hitchin fiber $F$, and hence vanishes in homology relative to the fiber at infinity. 

It is not \emph{a priori} clear how to choose a basis for $H_2$ as $\bfm$ varies, but in fact, we may propagate this natural basis to all of $\calR$.
Proposition \ref{prop:fiberbundle} implies:
\begin{prop}\label{prop:localsystem}
The collection of homology lattices
$$ 
\mathcal{H} = \coprod_{(\boldsymbol{\alpha}, \mathbf{m}) \in \mathcal{R}}H_2(\mathcal{M}(\boldsymbol{\alpha}, \mathbf{m}),\Z)
$$ 
forms a local system of abelian groups over $\calR$.  In fact, this local system is trivial, so a choice of basis at 
$(\boldsymbol{\alpha_0}, \mathbf{0})$ (or indeed any $(\boldsymbol{\alpha_0}, \mathbf{m_0})$ in $\calR$) may be extended
by parallel transport over all of $\calR$.
\end{prop}
\begin{proof}
Over any contractible open set $U \ni (\boldsymbol{\alpha_0}, \mathbf{m_0})$, the local trivialization $\mathcal{X}|_U  \cong U \times \mathcal{M}(\boldsymbol{\alpha_0},
\mathbf{m_0})$ induces a canonical isomorphism
\[
H_2(\mathcal{M}(\boldsymbol{\alpha}, \mathbf{m}),\Z) \cong H_2(\mathcal{M}(\boldsymbol{\alpha_0}, \mathbf{m_0}),\Z) \quad 
\forall \,(\boldsymbol{\alpha}, \mathbf{m}) \in U,
\]
which makes $\mathcal{H}$ a local system of abelian groups.

Tensoring $\mathcal H$ with $\RR$ yields a vector bundle which has a canonical flat Gauss--Manin connection. Since $\calR$ is connected and simply-connected, 
we can extend any basis $\mathfrak{B}_0 $ of $H_2(\mathcal{M}(\boldsymbol{\alpha_0}, \mathbf{m_0}),\Z)$ at a given base point 
$(\boldsymbol{\alpha_0}, \mathbf{m_0}) \in \mathcal{R}$ to a parallel basis $\mathfrak B$ of $\mathcal{H}$ over all of $\calR$. 
\end{proof}

\subsection{Hitchin Torelli Map}

We are now in a position to define one of the main topics of this paper.
\begin{defn}[$\mathfrak{B}$-Torelli map]\label{def:Torelli}
Fix a parallel admissible basis $\mathfrak{B} = (S_0, \ldots, S_4)$ of the local system of homology lattices $\mathcal{H} \to \mathcal{R}$ in Proposition \ref{prop:localsystem}. We then define the $\frakB$-\emph{Torelli map} by
\[
\mathcal{T}_{\frakB} : \calR \longrightarrow ( \R \times \C)^4, \quad (\boldsymbol{\alpha}, \mathbf{m}) \longmapsto \left\{\left( \int_{S_i} \omega_I, \int_{S_i} \Omega_I \right)\right\}_{i=1}^4.  
\]
Each of these integrals is called a \emph{period} of the hyperK\"ahler structure of $\calM(\bsa,\bfm)$. 
\end{defn}
 


\subsection{Extended Moduli Space of Collier--Fredrickson--Wentworth}\label{sec:CFW}
We need a bit more information about the smooth fibration $\pi: \widetilde{\calX} \to \calR$.  For each $(\bsa,\bfm) \in \calR$, there is
a distinguished complex structure $I$ and holomorphic symplectic form $\Omega_I$ on $\calM(\bsa,\bfm)$. This complex structure is
characterized as the one for which this moduli space is an algebraic completely integrable system.
We seek to understand how $(I, \Omega_I)$ behaves as we vary $\mathbf{m}$.  

Fixing a generic $\bsa$, let us restrict the fibration $\widetilde{\calX}$ to a fibration $\pi_{\bsa}: \mathcal{P}_0(\bsa) \to \{\bsa\} \times \C^4$. This restricted
total space $\mathcal{P}_0(\bsa)$ was constructed by Collier, Fredrickson and Wentworth \cite{collier2024conformallimitsparabolicslnchiggs}.
As a smooth fiber bundle, $\calP_0(\bsa)$ is simply the union of fibers of $\widetilde{\calX}$ over $\{\bsa\} \times \C^4$, but this space has a richer structure 
as a smooth quasi-projective variety. To exhibit this, we describe how it is constructed analytically.

\begin{figure}[ht]
\includegraphics[height=1.5in]{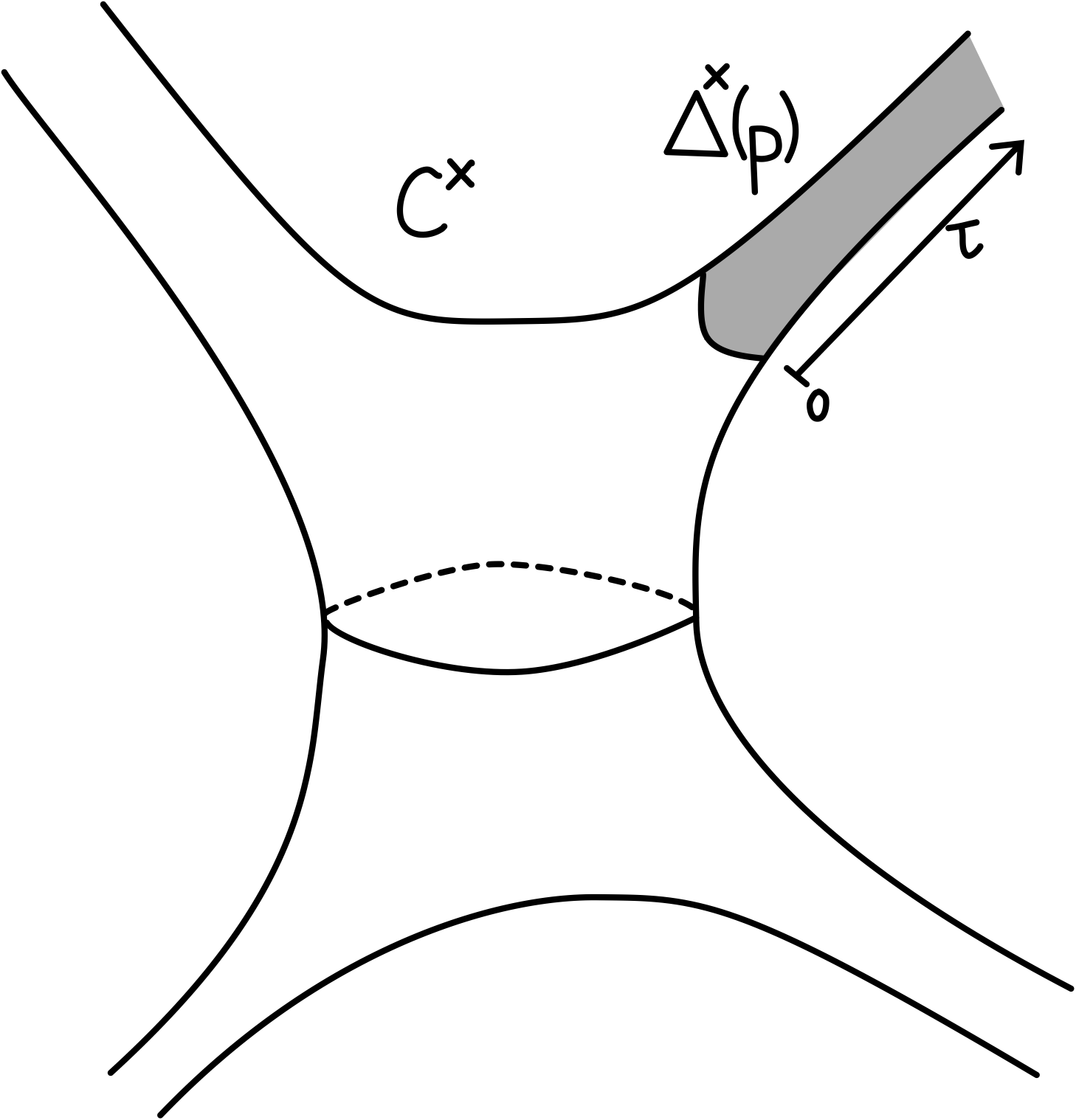}
\caption{\label{fig:CFW} }
\end{figure}
Let $C$ be a compact curve and $D=\{p_1, \ldots, p_N\}$ a divisor in $C$.  For simplicity, we restrict to rank $2$ Higgs bundles, and of course are primarily
concerned with the case where $C = \CP^1$ and $N=4$.  Endow the open Riemann surface $C^\times = C \setminus D$ with a conformally compatible
metric $g$ for which in a neighborhood of each $p_j$, $g$ takes the form $d\tau^2 + d\theta^2$, where $z$ is a local holomorphic coordinate in
a punctured disk $\Delta^\times(p_i)$, $z = e^w$ and $w = -\tau + i\theta$.  
Fix a hermitian metric $h_0$ on $E \vert_{C^\times}$ and a unitary frame $(e_1,e_2)$ over each $\Delta^\times(p_i)$. Let further $A_0$ be a fixed unitary 
connection on $E \vert_{C^\times}$ of the form
\[
d_{A_0} = d + \begin{pmatrix}
	\alpha^{(1)}(p_i) & 0 \\ 0 & \alpha^{(2)}(p_i)
\end{pmatrix} d\theta 
\]
on each punctured disk. 
We use exponentially weighted Sobolev spaces $L^2_{k,\delta}(E)$ with norm
\[
\|u\|_{L^2_{k,\delta}} = \Bigl(\int_{C^\times} e^{\tau \delta} \, \Bigl(\sum_{i=0}^k | \nabla_0^{(i)} u|^2 \Bigr)  \, dA_g \Bigr)^{1/2},\ \ \delta \in \R.
\]
This norm can be used, with obvious modifications, for sections of the various bundles below.

Now consider the space of $L^2_{k,\delta}$ connections
\[
\mathcal{A}_\delta =d_{A_0} + L^2_{1,\delta}( T^*C \otimes\mathfrak{su}(E))
\] 
and Higgs fields
\[
\calD_\delta = \{ \varphi \in L^2_{-\delta}(  T^*C^{1,0} \otimes \mathfrak{sl}(E)) : \bar\partial_{A_0} \varphi \in L^2_\delta \}.
\]
Clearly $L^2_{1, \delta}(  T^*C^{1,0} \otimes \mathfrak{sl}(E)) \subset \mathcal{D}_\delta$ and by \cite{collier2024conformallimitsparabolicslnchiggs}[Proposition 3.8],
this subspace is the kernel of the residue map
\[
\Res = r : \mathcal{D}_\delta \to \C^N, \quad \varphi \mapsto \lim_{\tau \to \infty} \varphi(\tau, \theta) = (m_1, \ldots, m_N)
\]
where $\varphi = \varphi(z)\frac{dz}{z} = \varphi(w) dw$, and we identify the matrix 
$\begin{pmatrix} 	-m_i & 0 \\ 0 & m_i \end{pmatrix}$ with the complex number $m_i$. 
Hence, $\mathcal{D}_\delta$ is a finite-dimensional 
extension of $L^2_{1, \delta}(  T^*C^{1,0} \otimes \mathfrak{sl}(E))$.

Now write
\[
\calB_\delta = \{ (A,\varphi) \in \calA_\delta \times \calD_\delta : \bar\del_A \varphi = 0 \};
\]
this is the space of parabolic Higgs pairs, and we let $\calB_\delta(\bsa)$ be the subset of $\bsa$-stable ones.
The complex gauge group $\calG_\delta$ acts on $\calB_\delta(\bsa)$ with quotient
\[
\calP_0(\bsa) = \calB_\delta(\bsa) / \calG_\delta
\]
the (extended) moduli space of parabolic Higgs bundles. This is independent of $\delta$ for $\delta > 0$ sufficiently small. 

By equivariance, the residue map descends to a map $\pi_{\bsa} : \calP_0(\bsa) \to \C^N$, and we define the space of parabolic Higgs bundles with fixed
complex masses as $\calM(\bsa,\bfm) = \pi_{\bsa}^{-1}(\bfm)$.
In particular, $\calM(\bsa, \mathbf{0})$ is the moduli space of strongly parabolic Higgs bundles.

The following is proved in \cite{collier2024conformallimitsparabolicslnchiggs}:
\begin{thm}[\cite{collier2024conformallimitsparabolicslnchiggs}]\ \label{thm:CFW}
\begin{enumerate}
\item $\calP_0(\bsa)$ is a complex manifold of complex dimension $6g-6 + 3N$;
\item $\pi_{\bsa}$ is a holomorphic submersion;
\item the fibers $\calM(\bsa,\bfm)$ are hyperK\"ahler manifolds of complex dimension $6g-6 + 2N$.
\end{enumerate}
\end{thm}
This result generalizes earlier work of Konno \cite{konno93} and Nakajima \cite{Nakajima}. 

In our setting, $g=0$ and $N=4$, so $\dim_\C\mathcal{P}_0(\bsa) = 6$ and $\dim_\C \calM(\bsa, \bfm)=2$.

\begin{rem}\label{rem:bundle}
It does not seem to be known in general if the projection $\pi_{\bsa} : \calP_0(\bsa) \rightarrow \C^N$ is a smooth fiber bundle. However, as already pointed
out in Proposition \ref{prop:fiberbundle}, this is transparent for the four-punctured sphere by the alternate construction of Ivanics, Stipsicz and Szab\'o \cite{ISS19a}, \cite{ISS19b}.
\end{rem}

\begin{rem}  This analytic construction of $\calP_0(\bsa)$ agrees with the algebraic one in \cite{yokogawa93}, and hence in particular, the map $\pi_{\bsa}$
is just the usual one. Thus, if $(\cE, \varphi)$ is a parabolic Higgs bundle with a full flag $\{0\} \subset F_p \subset E_p$ at each $p \in D$, then
the residue $\mathrm{Res}(\varphi)\Big|_p$ in the eigenline $F_p$ is the complex mass $m_p$.
\end{rem}

\subsection{Hitchin map}
The Hitchin map is defined on this extended bundle as
\[
\mathrm{Hit}: \mathcal{P}_0(\boldsymbol{\alpha}) \to \mathcal{B}=H^0(K^2(2D)), \ \ (\mathcal{E}, \varphi) \mapsto -\det \varphi,
\]
and was proved by Yokogawa \cite{yokogawa93} to be proper. 

It is compatible with the Hitchin maps on each fiber
\[ 
\mathrm{Hit}: \calM(\bsa, \bfm) \to \calB(\bfm).
\] 
This base is an affine space modeled on $H^0(K^2(D))$, which consists of elements in $H^0(K^2(2D))$ which have residue $m_i^2$ at $p_i$.
In particular, $\calB(\mathbf{0})=H^0(K^2(D))$.  The spaces $\calB(\bfm)$ together as an affine bundle $\widehat{\calB} \to \C^4$, cf.\ 
\cite[Remark 5.6]{collier2024conformallimitsparabolicslnchiggs}.  Thus, the maps $\pi_{\bsa}: \calP_0(\bsa) \to \C^4$ and 
$\mathrm{Hit}:\calP_0(\bsa) \to H^0(K^2(2D))$ factor through the map $\calP_0(\bsa) \to \widehat{\calB}$.

We can make this even more concrete and give a trivialization $\widehat{\mathcal{B}} \simeq \C^4_{\mathbf{m}} \times \C_{\beta}$:
\begin{prop} [Hitchin Base] \label{prop:Hitchinbase}
Fix $\mathbf{m} \in \C^4$. 
The Hitchin base consists of the space of quadratic differentials of the form
\begin{equation}
 \cB(\mathbf{m}) =\left\{q=\frac{f_{\mathbf{m}}(z) + \beta z(z-1)(z-p_0)}{z^2(z-1)^2 (z-p_0)^2} \de z^2,\ \ \beta \in \C\right\} 
\end{equation}
where 
\begin{align*}
& f_{\mathbf{m}}(z):= 
m_\infty^2 z^4 + \frac{1}{3} (-m_0^2+2 m_1^2-4 m_\infty^2-m_{p_0}^2-m_{0}^2 p_0-m_1^2 p_0-4 m_{\infty}^2 p_0+2 m_p^2 p_0) z^3+ \\
& \frac{1}{3}(m_0^2+m_1^2+m_\infty^2+m_{p_0}^2+5 m_0^2 p_0-4 m_1^2 p_0+5 m_\infty^2 p_0-4 m_{p_0}^2 p_0+m_0^2 p_0^2+m_1^2 p_0^2+m_\infty^2 p_0^2+m_{p_0}^2 p_0^2) z^2+\\ & -\frac{p_0}{3}  (4 m_0^2+m_1^2+m_\infty^2-2 m_{p_0}^2+4 m_0^2 p_0-2 m_1^2 p_0+m_\infty^2 p_0+m_{p_0}^2 p_0) z + m_0^2 p_0^2.
\end{align*}
\end{prop}
\begin{proof}  The fact that $q \in H^0(K^2(2D))$ implies that $q = \frac{f(z)}{z^2(z-1)^2(z-p_0)^2} dz^2$ for some quartic polynomial $f = \sum_{j=0}^4 f_j z^j$. 
The condition on the residues yields 
\begin{equation} \label{eq:complexmass}
 m_0^2 = \frac{f(0)}{p_0^2}, \quad  m_1^2 = \frac{f(1)}{(1-p_0)^2},  \quad m_{p_0}^2 = \frac{f(p_0)}{p_0^2(p_0-1)^2}, \quad m_\infty^2 = f_4.
\end{equation}
The particular polynomial $f_{\mathbf{m}}(z)$ above is the unique quartic polynomial satisfying these constraints for which the $\beta^5$-coefficient in 
the ($6^{\mathrm{th}}$ order in $\beta$) discriminant polynomial of $f_{\mathbf{m}}(z) + \beta z(z-1)(z-p_0)$ with respect to $z$ vanishes. 
\end{proof}
\begin{rem}
This quartic polynomial $f_{\mathbf{m}}(z)$ has the following geometric significance. The fiber $\mathrm{Hit}^{-1}(q) \subset: \calM(\bsa,\bfm)  \to \calB(\bfm)$
is smooth if and only if $q \in H^0(K^2(2D))$ has simple zeros, which is the same as the nonvanishing of the discriminant  of $f_{\bfm}$.  The fibers at the
zeroes of this discriminant are singular. The negative of the coefficient of $\beta^5$ in this discriminant is the sum of these six zeroes, counted 
with multiplicity. So in other words, $f_{\bfm}$ is normalized so that the center of mass of the locations of the singular fibers occurs at $\beta = 0$.
\end{rem}

\subsection{\texorpdfstring{$\C^\times$}{C*}-action and \texorpdfstring{Bia{\l}ynicki--Birula}{Bialynicki-Birula} stratification}\label{sec:BBstratification}

In this section, we will discuss the $\C^\times$ action on  $\mathcal{P}_0(\bsa)$ and its accompanying Bia{\l}ynicki--Birula stratification. 

Bia{\l}ynicki--Birula stratifications for complex manifolds equipped with holomorphic $\C^\times$ actions are
provided by Bia{\l}ynicki--Birula \cite{BialynickiBirula1973,BialynickiBirula1974}, Carrell and Sommese \cite{CarrellSommese1979}, Kirwan \cite{Kirwan1988}, and Yang \cite{Yang2008}, under additional hypotheses, e.g. (1) the manifold is a smooth algebraic variety \cite{BialynickiBirula1973,BialynickiBirula1974}, (2) the manifold is compact K\"ahler \cite{CarrellSommese1979, Kirwan1988}, or (3), in the non-compact K\"ahler setting under the assumption that the circle action $S^1 \subset \C^\times$ admits a Hamiltonian moment map that is proper and bounded below, such that the number of connected components of the fixed point set is finite \cite{Yang2008}.

In the case of non-parabolic (i.e. ordinary) Higgs bundles, there is a $\C^\times_\zeta$-action on the moduli space given by \[[(\cE, \varphi)] \mapsto [(\cE, \zeta \varphi)]. \]
In this case, the  Hitchin moduli spaces satisfy the hypotheses \cite{Yang2008}, and hence have Bia{\l}ynicki–Birula stratifications.
Since the $\C^\times$-action rescales the Higgs field, note that in its extension to parabolic Higgs bundles, the $\C^\times$-action will rescale the complex masses by $m_p \mapsto \zeta m_p$. Consequently, the Hitchin moduli space $\mathcal{M}=\mathcal{M}(\boldsymbol{\alpha}, \mathbf{m})$ has a $\C^\times$-action precisely when $\mathbf{m}=\mathbf{0}$. (Again, in this case, $\mathcal{M}(\boldsymbol{\alpha}, \mathbf{0})$ satisfies the hypotheses in \cite{Yang2008}.)  The extended moduli space $\mathcal{P}_0(\boldsymbol{\alpha})$ admits a $\C^\times$-action, but it does not satisfy the hypotheses in \cite{Yang2008} because it is not K\"ahler.   Similarly, it does not satisfy the hypotheses of \cite{Feehan2022} which make use of a non-degenerate (but not necessarily closed) $2$-form $\omega$. Consequently, we can't use any of powerful results about Bia{\l}ynicki--Birula stratifications for smooth complex manifolds off the shelf. 

We begin with the standard
Bia{\l}ynicki--Birula decomposition of non-singular semi-projective varieties, e.g. see the summary of results in \cite[\S1.2]{hausel2013cohomologylargesemiprojectivehyperkaehler}. 
 The moduli space $\mathcal{P}_0(\bsa)$ as constructed by Yokogawa is a quasi-projective variety.  Since this $\C^\times$-fixed point set is projective (we will soon see that it is the union of $\CP^1$ and four disjoint points) and all $\C^\times$-limits exist inside of $\mathcal{P}_0(\bsa)$, the space $\mathcal{P}_0(\bsa)$ is semi-projective. When $\bsa$ is generic, it is non-singular.  Consequently, it has a Bia{\l}ynicki--Birula stratification. 
 
 In particular,  by the Bia{\l}ynicki‑-Birula theory for non-singular semi‑projective varieties, (1) there is a decomposition of 
 \[\mathcal{P}_0(\bsa)= \bigsqcup_{a \in \pi_0(\mathcal{P}_0(\bsa)^{\C^\times})} \mathcal{P}_0(\bsa)_a^+; \qquad \qquad  \mathcal{P}_0(\bsa)_a^+ =\{x \in \mathcal{P}_0(\bsa) | \lim_{t \to 0} t \cdot x \in a\} \]
 by locally closed subvarieties, (2) when $\mathcal{P}_0(\bsa)$ is smooth, for each fixed component $a$,  
 \[\mathcal{P}_0(\bsa)_a^+ \to a\] 
 is a locally trivial affine bundle, (3) the Bia{\l}ynicki‑-Birula strata are ordered by a partial order determined by weights, and 
 \[\overline{\mathcal{P}_0(\bsa)_a^+} \subset \bigcup_{a' \leq a} \mathcal{P}_0(\bsa)_{a'}^+.\]
 In particular, since $\mathcal{M}(\bsa, \mathbf{0}) \subset \mathcal{P}_0(\bsa)$ contains all the $\C^\times$-fixed points and is K\"ahler, this partial ordering is from the upwards Morse flow. In this case, the central $\CP^1$ is higher than each of the four exterior $\C^\times$-points; there is no comparison between the four exterior $\C^\times$-fixed points.  As a result, for each exterior $\C^\times$-fixed point, the corresponding stratum is closed. We will call these \emph{generalized Hitchin sections}.\footnote{In fact, over the extended Hitchin base $\hat{\mathcal{B}}$, these really are sections.}  The strata for the central $\CP^1$ is open, but becomes closed once adds in the generalized Hitchin sections.

Note that since we are interested in the analytically defined moduli space rather than the algebraically defined moduli space in \cite{yokogawa93}, there is some accompanying subtlety involving the  $\C^\times$-limits addressed in 
\cite[Proposition 4.1]{collier2024conformallimitsparabolicslnchiggs}
\footnote{Moreover, note that Collier-Fredrickson-Wentworth extend this and show that $\C^\times$-limit exists on the larger space of logarithmic $\lambda$-connections, of which $\mathcal{P}_0(\bsa)$ is the fiber over $\lambda=0$\cite[Proposition 2.11]{collier2024conformallimitsparabolicslnchiggs}.}.
We can upgrade the above statements to statements about smooth submanifolds in $\mathcal{P}_0(\bsa)$ using the slices in \cite{collier2024conformallimitsparabolicslnchiggs}.
Moreover, as in \cite[Proposition 4.7]{collier2024conformallimitsparabolicslnchiggs}, these slices are compatible with the projection of the framed analytic moduli space to the residues of $\varphi$ at $p \in D$.
Consequently, the Bia{\l}ynicki--Birula stratification of $\mathcal{P}_0(\bsa)$ induces a stratification of each $\mathcal{M}(\boldsymbol{\alpha}, \mathbf{m})$ by intersecting the strata with it. 

\bigskip

We briefly discuss the structure of the $\C^\times$-fixed point set. The $\C^\times$-fixed points are known to be \emph{systems of Hodge bundles} \cite{SimpsonThesis}. The description of these Hodge bundles depends on $\boldsymbol{\alpha}$; this chamber-dependence is described in detail in Section \ref{sec:fixedpoints}, and then we can address this stratification more explicitly.
All the $\C^\times$-fixed points lie in the fiber defined by $q=-\det \varphi =0$ inside of the strongly parabolic moduli space $\mathcal{M}(\bsa, \mathbf{0})$. 
The preimage of $q=0$ in $\mathcal{M}(\bsa, \mathbf{0})$ is shown in Figure \ref{fig:fixedpoints}; it consists of five spheres in an affine $D_4$-configuration. 
Every point of the central sphere is $\C^\times$-fixed point.
Each of the four exterior spheres $\CP^1_p$ for $p \in D$ has exactly \emph{two} $\C^\times$-fixed points which we label $q_p^\pm$. We call the attaching point between exterior sphere and central sphere the \emph{interior fixed point} (or \emph{wobbly point} following Simpson) and denote it $q_p^-$. We call the other one the \emph{exterior fixed point} and denote it $q_p^+$. 
\begin{figure}[ht]
  \begin{centering} 
  \includegraphics[height=1.5in]{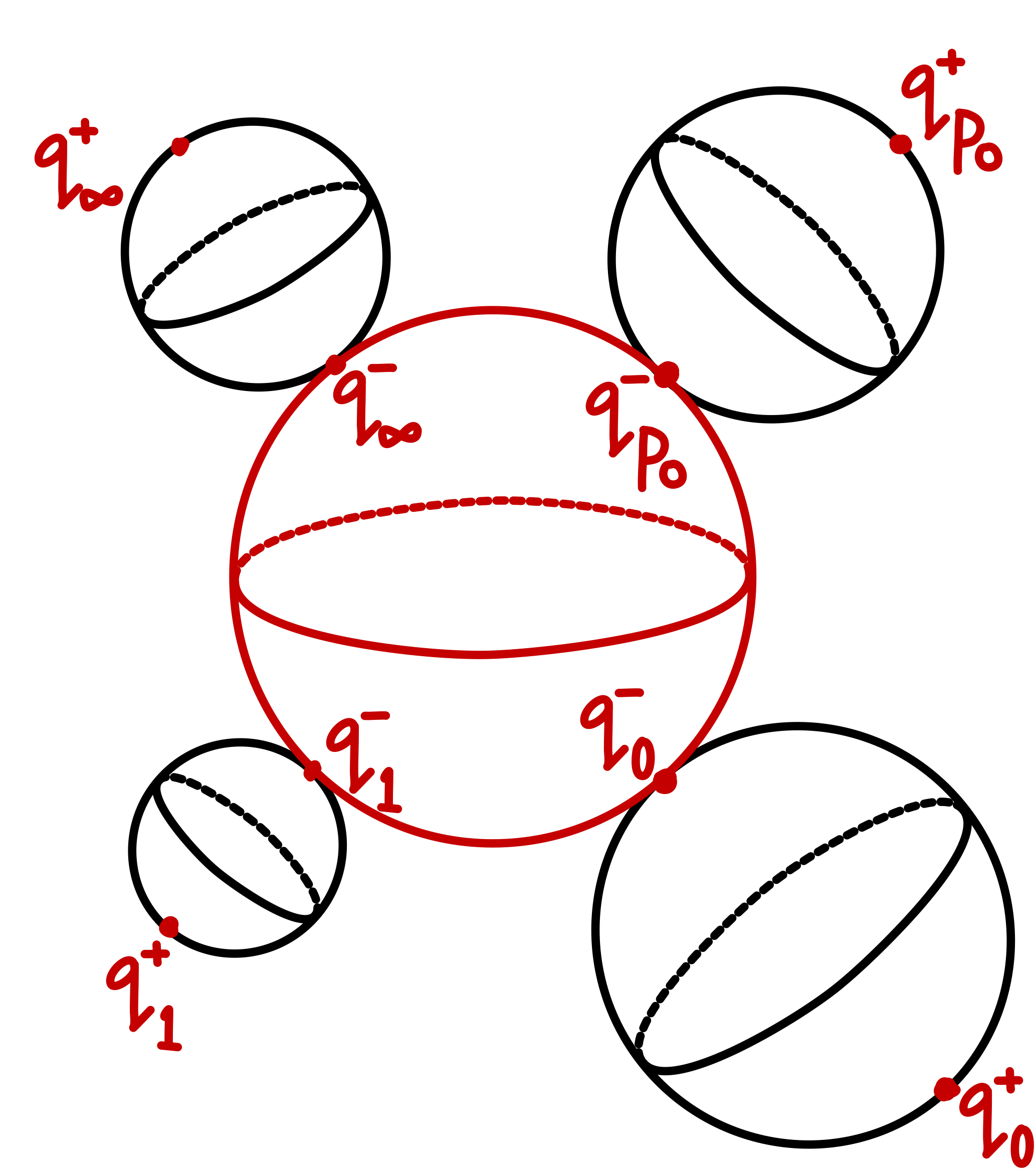}
  \caption{\label{fig:fixedpoints} $\C^\times$-fixed points are shown in red}
  \end{centering}
  \end{figure} 
Consequently, the $\C^\times$-fixed point set consists of five components when $\boldsymbol{\alpha}$ is generic: a copy of $\mathbb{P}^1$ and four additional points. 
\bigskip

\section{Chamber structure of stability conditions, the nilpotent cone and consistent homology bases}\label{sec:distinguished_basis}
In this section we review the chamber structure for the space of real mass parameters $\bsa$ for which the stability condition holds when the complex mass 
parameters $\bfm = \mathbf{0}$. This leads to an explicit description and parametrization of the nilpotent cone in these regimes. 

Recall from Definition \ref{defn:gen1}, that the parabolic data $(\bsa, \bfm) \in ((0, \tfrac{1}{2}) \times \C)^4$ is called generic if all semistable Higgs bundles with
these parameters are stable. The set of all generic parameters $\mathcal{R}$ is the complement of a collection of real-codimension three affine subspaces. 
The intersection of these affine subspaces with $\{(\bsa,\bfm): \bfm = \mathbf{0}\}$ is a set of affine hyperplanes which divide $(0,\tfrac12)^4$ into chambers. 

We first review this chamber structure in Section \ref{sec:chambers}, following \cite{Meneses}, and then in Section \ref{sec:fixedpoints} use the combinatorial 
data associated to these chambers to characterize the $\C^\times$-fixed points. Finally, Section \ref{sec:nilpotentcone} contains a full description of the nilpotent cone, 
i.e., the fiber over the quadratic differential $q=0$, in $\mathcal{M}(\bsa, \mathbf{0})$ as an affine $D_4$ pattern of five spheres, which
constitute a distinguished basis for $H_2(\mathcal{M}(\boldsymbol{\alpha}, \mathbf{0}), \Z)$.

\medskip

We pause briefly to explain how these results are used later.   This distinguished homology bases for the strongly parabolic moduli spaces $\mathcal{M}(\bsa, \mathbf{0})$ 
are used to obtain homology bases for $\mathcal{M}(\boldsymbol{\alpha}, \mathbf{m})$. The integrals of $\omega_I$ over such a homology basis
when $\bfm = \mathbf{0}$ are calculated using the explicit description of the $\C^\times$-fixed points in Section \ref{sec:fixedpoints}, but we then show that these 
integrals are independent of $\mathbf{m}$. On the other hand, we reduce the calculation of the integrals of $\Omega_I$ over this distinguished basis 
to the intersection numbers of the four outer spheres in the $D_4$ arrangement  and four other $2$-cycles related to the $\C^\times$-flow.

\medskip

For simplicity below, we simply label the distinct points $0, 1, p_0, \infty$ by $p_1, p_2, p_3, p_4$. 

\subsection{Chamber Structure}\label{sec:chambers}  We now recall Meneses' \cite{Meneses} description of the chamber structure of the space of 
real mass parameters $(0, \frac{1}{2})^4$. 

The starting point is the Biswas polytope, which by definition is the subset of parabolic weights which admit stable parabolic bundles.  When $C=\CP^1$, 
the cube $(0,1/2)^4$ is divided into 24 chambers by 12 hyperplanes. Amongst these, the $8$ ``exterior'' chambers lie outside the Biswas polytope
while the 16 ``interior'' chambers lie inside it.  These have slightly different combinatorial descriptions.

\subsubsection{Interior chambers}\label{sec:interiorchambers} 
The interior chambers are in bijective correspondence with the {\it even partition sets} of $\{1,2,3,4\}$ as
defined by \cite[Definition 1]{Meneses}: 
\begin{defn}
Consider the set of all decompositions $I \sqcup I^c = \{1,2,3,4\}$ where $I$ has even cardinality. This is a subset of the power set $2^{\{1,2,3,4\}}$
consisting of those $I \subset \{1,2,3,4\}$ such that either $I$ or its complement $I^c$ lie in $\mathtt E := \{\emptyset, \{1,2\}, \{1,3\}, \{1,4\} \}$.  
An \emph{even partition set} $\mathcal{I}$ is then a function which selects either $I$ or $I^c$ as $I$ ranges over $\mathtt E$. 
There are $16$ such even partition sets,which can be organized into four types, each corresponding to a distinguished index $i$ (with 
the remaining indices $j, k, l$ then used in their descriptions):
\begin{itemize}
\item Type A1: $\mathcal{I}_{A, i}=\{\emptyset, \{k,l\}, \{j,l\}, \{j,k\}\}$, corresponding to even subsets, none of which contain $i$;
\item Type A2: $\mathcal{I}_{A, i}^c=\{\{1, 2, 3, 4\}, \{i,j\}, \{i,k\}, \{i,l\}\}$, corresponding to even subsets all of which contain $i$;
\item Type B1: $\mathcal{I}_{B, i}=\{\emptyset, \{i,j\}, \{i,k\}, \{i,l\}\}$, where $i$ lies in the three subsets of cardinality $2$;
\item Type B2: $\mathcal{I}_{B, i}^c=\{ \{1, 2, 3, 4\}, \{k,l\}, \{j,l\}, \{j,k\}$, where only the cardinality $4$ subset contains $i$.
\end{itemize}
\end{defn}

We will see below that interior chambers associated to the even partition sets of type $A$ share no boundary with any exterior chamber, while those
associated to even partition sets of type $B$ each share a boundary with a single exterior chamber. 

\subsubsection{Description of exterior chambers}\label{sec:exteriorchambers} The combinatorial labels for the exterior chambers are the subsets 
of $\{1, 2, 3, 4\}$ of odd order.   
We use that to each such odd order subset $I_0$ we can associate an even partition set $\mathcal{I}$, which consists of all even order subsets that can obtained 
from $I_0$ by adding or removing a single element.  In the moduli space itself, each set in this corresponding $\mathcal{I}$ labels an exterior sphere, while $I_0$
itself labels the central sphere.

\begin{itemize}
\item Type E1: When $I_0=\{i\}$, the even partition set is $\cI_{B,i}=\{\emptyset, \{i, j\}, \{i, k\}, \{i, l\}\}$. 
\item Type E2: When $I_0=\{j, k, l\}$, the even partition set is $\cI_{B,i}^c=\{\{1, 2, 3, 4\}, \{k, l\}, \{j, l\}, \{j, k\}\}$. 
\end{itemize}

\subsubsection{Correspondence with parabolic weights}
The map from these combinatorial objects to the corresponding chambers is given by associating to each subset, as above, of $\{1, 2, 3,4\}$ 
a set of inequalities which determine the chamber.  Each subset $J \in \{1,2,3,4\}$ corresponds to a vertex $v_J$ in $[0,1/2]^4$ via
\begin{equation}\label{eq:vI}
v_{J,i} = \begin{cases} \frac{1}{2} & i \in J \\ 0 & i \notin J. 
\end{cases}\end{equation}
In  fact, within $(0,1/2)^4$, the interior chamber associated to the even partition set $\mathcal I$ 
is the convex hull of $(\frac{1}{4}, \frac{1}{4}, \frac{1}{4}, \frac{1}{4})$ and the vertices $\{v_I: I \in \mathcal{I}\}$.
The exterior chamber labelled by $\mathcal{I}$ is the convex hull of the vertices $\{v_I: I \in \mathcal{I}\}$ and $v_{I_0}$.
The following is then clear:

\begin{lem}
\label{lem:adjacentchambers}Chambers have a common boundary if they share four vertices. An exterior chamber is adjacent to an interior chamber if their associated even partition sets agree.
Two interior chambers 
are adjacent if their associated even partition sets satisfy
\begin{equation}\label{eq:wallcrossing}\mathcal{I}\backslash\{I\}= \mathcal{I}'\backslash \{I^c\}.\end{equation}
\end{lem}
\medskip

It will be convenient to define 
\begin{equation} 
  L_{i} =-\alpha_i + \sum_{j \neq i} \alpha_j.  \label{eq:L}\end{equation}
The Biswas polytope is cut out by the inequalities $0 < L_i < 1$. The exterior chamber labelled by $I_0=\{j, k, l\}$ is cut out by the single inequality $L_i > 1$. The exterior chamber labelled by $I_0=\{i\}$ is cut out by the single inequality $L_i < 0$.
We can also describe the interior chambers in terms of inequalities.  Given a set $I$, we define 
\begin{align}\label{eq:K} K_I : \left(0, \frac{1}{2}\right)^4 &\rightarrow \R\\ \nonumber
 (\alpha_{1}, \alpha_{2}, \alpha_{3}, \alpha_{4}) & \mapsto \sum_{i \in I}\alpha_{i}- \sum_{i \notin I} \alpha_{i} + \left\lfloor\frac{|I^c|-|I|}{4}\right\rfloor.\end{align}
 Then, by \cite[Corollary 5]{Meneses}, the interior chamber is characterized by $K_I>0$ for all $I \in \mathcal{I}$ and the Biswas polytope condition that 
 $$\{(\alpha_{1}, \alpha_{2}, \alpha_{3}, \alpha_{4})| 0 < L_i < 1, i \in \{1, 2, 3, 4\}\}.$$
 \begin{rem}\label{rem:KvsL}Note that $L_i < 0$---associated to $I_0 = \{i\}$---can be given by $K_{\{i\}} > 0 $, 
 while $L_i < 1$---associated to $I_0 =\{j, k, l \}$---can be given by $K_{\{j, k, l \}} > 0$, 
 parallel to the role of $K_I$ above.
 \end{rem}
 
 Consequently, we see that the Biswas polytope is divided into $16=2^4$ chambers by the four hyperplanes
\[0=K_{\{1,2\}}=-K_{\{3,4\}} \qquad 0=K_{\{1,3\}} = -K_{\{2,4\}} \qquad 0=K_{\{1,4\}}=-K_{\{2,3\}} \qquad 0=K_{\emptyset}=-K_{\{1,2, 3, 4\}}.\]

We call the $12$ hyperplanes from the vanishing loci of $K_I, L_i-1$, and $L_i$  \emph{walls}.

 \begin{rem}[Nakajima's Walls]
 When $\mathbf{m}=0$, we can see that Nakajima's equations in Proposition \ref{prop:generic} coincide precisely with the interior walls of $(0, \frac{1}{2})^4$.
  The walls dividing the exterior chambers from the interior chambers  are given by $L_i=0$ and $L_i=1$ for $i=1, 2, 3, 4$, i.e. $-\alpha_i + \alpha_j + \alpha_k + \alpha_l \in \{0, 1\}$ or equivalently $\alpha_i -\alpha_j - \alpha_k - \alpha_l \in \{-1, 0\}$. In Nakajima's language these correspond exactly to taking  $e_i=-1$ and $e_j, e_k, e_l=1$ with $d+ 2 \in \{0, -1\}$, respectively; or $e_i=0$ and $e_j, e_k, e_l=1$ with $d+ 2 \in \{0, 1\}$. Because each parabolic weight is in $(0, \frac{1}{2})$, there are no other choices of $d$ that correspond to a non-empty wall. The four walls dividing the Biswas polytope into $2^4$ interior chambers are given by $K_I=0$ or equivalently $K_{I^c}=0$ for
 an even subset $I \subset \{1, 2, 3, 4\}$. In Nakajima's language, we take $e_i=0$ for $i \in I$ and for $i \in I^c$ we take $e_i=1$. We take $d$ so that $d + \sum_{i=1}^4 e_i = \frac{|I^c| -|I|}{4}$. 
 When $|I|$ is even, we see that $\sum_{i=1}^4(-1)^{e_i} \alpha_i \in (-\frac{1}{2}|I^c|, \frac{1}{2} |I|)$. When $|I|=2$, $0$ is the only integer in this range; when $|I|=4$, $1$ is the only integer in this range; when $|I|=0$, $-1$ is the only integer in this range. Consequently, all other choices of $d$ correspond to empty walls.
\end{rem}

\subsection{\texorpdfstring{$\C^\times$}{C*}-fixed points}\label{sec:fixedpoints}
Recall the $\C^\times$-action $(\cE, \varphi) \mapsto (\cE, \lambda \varphi)$ on the moduli space of parabolic Higgs bundles, as described in Section \ref{sec:BBstratification}. 
In this section we describe the $\C^\times$-fixed points \emph{as Higgs bundles}. (We defer any discussion of the harmonic metric to a later section.)
We follow the succinct description in  Collier--Fredrickson--Wentworth, which historically was borrowed from a draft of this paper.  The $\C^\times$-fixed points are systems of Hodge bundles \cite{SimpsonThesis}, i.e. either (1) the Higgs vanishes and the underlying parabolic bundle is stable, or (2), the parabolic bundle splits as a direct sum of parabolic line bundles with non-zero Higgs field 
\begin{equation}
        \label{eq general rk 2 fixed}\left(\cL_1(\beta_1)\oplus\cL_2(\beta_2),\begin{pmatrix}
        0&0\\\phi_0&0
\end{pmatrix}\right). 
\end{equation}
Here, 
$\phi_0$ in the Higgs field is a parabolic map $\cL_1(\beta_1)\to
\cL_2(\beta_2)\otimes K(D).$  By ``parabolic map,'' we mean
\begin{defn}
         \label{def parabolic morph}Given two parabolic vector bundles $\cE(\alpha)$ and $\cF(\beta)$, a holomorphic bundle map $f:\cE\to \cF$ is called {\em parabolic} if, $\alpha_j(p)> \beta_{k}(p)$ implies $f(\cE_{p,j})\subset \cF_{p,k+1}$ for all $p\in D$, and {\em strongly parabolic} if $\alpha_j(p)\geq \beta_k(p)$ implies $f(\cE_{p,j})\subset \cF_{p,k+1}$ for all $p\in D$.
 \end{defn} 
 
 \bigskip
 
 For generic choices of parabolic weights, there are five components fixed by the $\C^\times$-action: a $\mathbb{CP}^1$ and four additional points. Inside the nilpotent cone, these are respectively the central sphere and the exterior $\C^\times$-fixed points of  four exterior spheres.  In interior chambers, points of the central sphere are parabolic bundles; in exterior chambers, points of the central sphere are of the second type above.  
 
A $\C^\times$-fixed point with \emph{non-zero} Higgs field determines $I\subset \{1, 2, 3, 4\}$:
\begin{lem}\label{lem:fixedtoI}
Given a $\C^\times$-fixed point as in \eqref{eq general rk 2 fixed}, there is a naturally associated set $I \subset \{1, 2, 3,4 \}$ defined by $i \in I$ if, and only if, $F_{p_i} = \mathcal{L}_2|_{p_i}$.
\end{lem}
\begin{proof}
The flag $F_p$ is either $\cL_1|_p$ or $\cL_2|_p$. Let $D_I$ and $D_{I^c}$ be the effective subdivisors of $D$ for which the subspace $F_p$ is $\cL_2|_p$ and $\cL_1|_p$, respectively. 
Here $I$ denotes the subset of $\{1,2,3,4\}$ determined by the support of $D_I,$ and $I^c$ is its complement.   
\end{proof}

 When $|I|$ is even, this can be uniquely reversed; when $|I|$ is odd, there is a $\mathbb{CP}^1$-family. 

\begin{prop}\label{prop:generalsplit}
Given (a) $I \in \mathcal{I}$ and, (b) if $|I|$ is odd, a choice of point $u \in \mathbb{CP}^1$. 

Define the following Higgs bundle fixed by the $\C^\times$-action of the shape in \eqref{eq general rk 2 fixed}:
\begin{itemize}
\item Let $\cL_2 = \cO(-1 - \lfloor \frac{|I|}{2} \rfloor)$ and $\cL_1 = \cO(-4)\cL_2^*$.
\item 
The parabolic weights $\beta_1(p_i)$ and $\beta_2(p_i)$ are given by 
\begin{equation}
        \label{eq fixedpoint weights}\beta_1(p_i)=\begin{dcases}
        \alpha_i & i \in I \\ 1-\alpha_i & i \in I^c
\end{dcases}\ \ \ \ \ \ \ \text{and}\ \ \ \ \ \ \ \beta_2(p_i)=\begin{dcases}
       1-\alpha_i & i \in I\\
       \alpha_i&i\in I^c
\end{dcases}.
\end{equation}
\item If $|I|$ is even, let $\phi_0$ be the unique section in $H^0(\cL_1^* \cL_2 K(D_I))$.  If $|I|$ is odd, let $\phi_0$ be the unique section in $H^0(\cL_1^* \cL_2 K(D_I))$ vanishing at $u \in \CP^1$.
\end{itemize}

If $|I|$ is even, this $\C^\times$-fixed point is stable when $K_I>0$, i.e. where $I \in \mathcal{I}$, the associated even partition set.
If $|I|=1$, this $\C^\times$-fixed point is stable when $L_i <0$. If $|I|=3$, this $\C^\times$-fixed point is stable when $L_i >1$.
\end{prop}
\begin{rem}
The following proof gives us an interpretation of $K_I$:
At the $\C^\times$-fixed point labelled by $I \in \mathcal{I}$,
\[ K_I = -\pdeg (\cL_2(\beta_2)),\]
for $\cL_2(\beta_2)$, the potentially destabilizing $\varphi$-invariant parabolic subbundle. See \eqref{eq:stab2}.
\end{rem}
 \begin{proof}
 Following Lemma \ref{lem:fixedtoI}, the parabolic weights are given as above.

 \medskip
 
Now, we consider $\phi_0$ appearing in the Higgs field. The map $\phi_0$ is a meromorphic section of $\cL_1^*\otimes\cL_2\otimes K$
with at worst simple poles at $p\in D$. When $p_i \in D_{I^c}$,  
$\beta_1(p_i)> \beta_2(p_i)$, hence $\phi_0(F_p) \subset \{0\}$---i.e. the residue is zero.
Thus,
\[\phi_0\in H^0(\cL_1^* \cL_2 K(D_I)).\]

 \medskip
 
The $\C^\times$-fixed point is stable whenever $\phi_0\in H^0(\cL_1^*\cL_2K(D_I))$ is nonzero, i.e. 
\begin{equation} \deg(\cL_1^*\cL_2K(D_I))=2+2\deg(\cL_2)+\deg(D_I)\geq 0 \label{eq:stab1}
\end{equation}
 and also $\pdeg(\cL_2(\beta_2))<0$, i.e. 
\begin{align}
\pdeg(\cL_2(\beta_2))&=\deg(\cL_2)+\deg(D_I)-\sum_{i\in I}\alpha_i+\sum_{j\in I^c}\alpha_j \nonumber \\
&=-K_I<0.\label{eq:stab2}
\end{align}
A straightforward computation shows that stability forces the degrees of $\cL_1$ and $\cL_2$ to be $-3,-2,-1$. Hence, $\cL_2$ is isomorphic to $\cO(-1),\cO(-2)$ or $\cO(-3)$ and $\cL_1\cong\cO(-4)\cL_2^*.$ There are five cases determined by the degree of $D_I.$ The following table gives the conditions which are direct consequences of \eqref{eq:stab1}, \eqref{eq:stab2}.

\begin{equation}
\label{eq:tablefixedpoints}
        \begin{tabular}{|c|c|c|c|}\hline
                $\deg(D_I)$&$\cL_2$&condition on $(\alpha_1,\alpha_2,\alpha_3,\alpha_4)$& $\cL_1^* \cL_2 K(D_I)$\\\hline
                 $0$& $\cO(-1)$& $K_{\emptyset}= -\sum_{i=1}^4\alpha_i+1>0$&$\cO$ \\\hline 
                 $1$ &$\cO(-1)$& $L_i=-\alpha_i+\sum_{j\in I^c}\alpha_j<0$, where $D_I=p_i$& \cO(1)\\\hline
                 $2$ &$\cO(-2)$&$K_I = \sum_{i\in I}\alpha_i-\sum_{j\in I^c}\alpha_j>0$&$\cO$\\\hline
                 $3$&$\cO(-2)$& $L_i = -\alpha_i+\sum_{j\in I}\alpha_j>1$, where $D_{I^c}=p_i$&$\cO(1)$\\\hline
                 $4$&$\cO(-3)$& $K_{\{1,2,3,4\}}=\sum_{i=1}^4\alpha_i-1>0$&$\cO$\\\hline
        \end{tabular}
\end{equation}
Note that the diagonal gauge group action rescales $\phi_0$ by non-zero constant. Hence, when $|I|$ is odd, there is a $\mathbb{CP}^1$-family of non-zero sections of $\cO(1)$, up to rescaling gauge, parameterized by $u \in \mathbb{CP}^1$, the zero of $\phi_0$.
Concretely, 
the Higgs field $\phi_{0,u}$ is (up to scale) 
\begin{equation} \phi_{0,u} = \frac{(z-u)}{\prod_{i \in I} (z-p_i)} \de z 
\end{equation}
Similarly, when $|I|$ is even, there is a single non-zero section of $\cO$, up to rescaling. 
The Higgs field $\phi_0$ is (up to scale) 
\begin{equation} \label{eq:extfixed}
\phi_0 = \frac{\de z}{\prod_{i \in I} (z-p_i)} .
\end{equation}

\end{proof}

\subsection{Full description of nilpotent cone}\label{sec:nilpotentcone}
Given $\boldsymbol{\alpha}$, we now describe the nilpotent cone of the strongly parabolic moduli space $\mathcal{M}(\boldsymbol{\alpha}, \mathbf{0})$: five spheres in an affine $D_4$ configuration.
We already described the $\C^\times$-fixed points in the nilpotent cone with non-zero Higgs field, since Proposition \ref{prop:generalsplit} implies
\begin{cor}
In an interior chamber labelled by $\mathcal{I}$, the four $\C^\times$-fixed points labelled by $I \in \mathcal{I}$ are stable.
In an exterior chamber labelled by $I_0$ with associated even partition set $\mathcal{I}$, the four $\C^\times$-fixed points labelled by $I \in \mathcal{I}$ are stable; moreover, the $\CP^1_u$-family of $\C^\times$-fixed points labelled by $I_0$ are stable.
\end{cor}

To fully describe the nilpotent cone, we need to describe (1) the central sphere in the interior chambers and (2) the $\C^\times$-family in each exterior sphere, which is not fixed by the $\C^\times$-action, and (3) the gluing maps between these pieces. The first thing to note is that the four exterior spheres to the central sphere $\CP^1_u$ at the points of the divisor $D$. These four points are known as ``wobbly points'' or the 
``wobbly locus'' \footnote{According to \cite{Laumon}, 
a bundle $\cE$ is said to be \emph{very stable} if the only nilpotent Higgs field $\varphi$ on $\cE$ is $\varphi=0$. The wobbly locus is the set of non-very-stable bundles. In the case without pictures, Hausel and Hitchin prove in \cite[Theorem 1.1]{HauselHitchinverystable} that bundle is very stable if and only if upward flow is closed. 
}.  
Since the gluing maps were carefully considered in \cite{Meneses}, we focus on (1) and (2), making only passing mention of (3).   ``Nilpotent cone assembly kits'' are provided in \cite[Figure 7-9]{Meneses}, and we include with permission those figures below in Figure \ref{fig:assemblyA}, \ref{fig:assemblyB}, \ref{fig:assemblyE}. Notably, these give the information about the underlying \emph{holomorphic} bundle type in the various chambers.

 \begin{figure}[!ht]
  \begin{centering} 
  \includegraphics[height=3in]{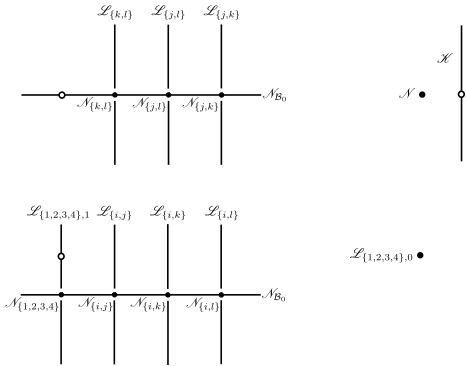}
  \end{centering}
  \caption{\label{fig:assemblyA} Nilpotent cone assembly kit for interior chambers of Type A1 (first row) and A2 (second row), as appearing in Figures 7-9 of Meneses. The first column describes the pieces of bundle type $\cO(-2) \oplus \cO(-2)$ and the second column describes the pieces of bundle type $\cO(-3) \oplus \cO(-1)$.   The labels of the $\CP^1$ are Meneses.}

  \end{figure}
  
 \begin{figure}[!ht]
  \begin{centering} 
  \includegraphics[height=3in]{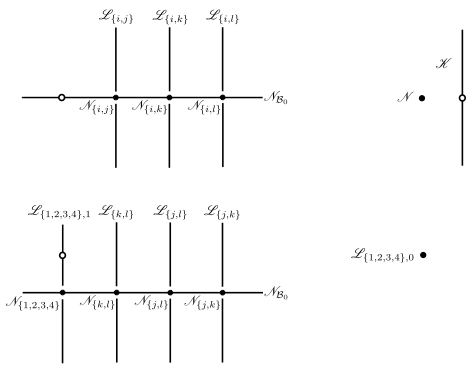}
  \end{centering}
  \caption{\label{fig:assemblyB} Nilpotent cone assembly kit for interior chambers of Type B1 (first row) and B2 (second row), as appearing in Figures 7-9 of Meneses.  See comments on Figure \ref{fig:assemblyA}.}
  \end{figure}

\begin{figure}[!ht]
  \begin{centering} 
  \includegraphics[height=3in]{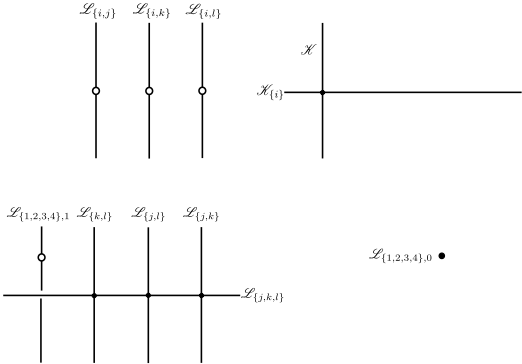}
  \end{centering}
  \caption{\label{fig:assemblyE} Nilpotent cone assembly kit for exterior chambers of Type E1 (first row) and E2 (second row), as appearing in Figures 7-9 of Meneses.   See comments on Figure \ref{fig:assemblyA}.
  }
  \end{figure}

\subsubsection{Central Sphere in Interior Chambers} \label{sec:centralsphereinint}

 The central sphere is the moduli space of stable parabolic bundles. In Types A1 and B1, where $\mathcal{I}$ contains $\emptyset$, there is one point of the central sphere with underlying bundle type $\cO(-3) \oplus \cO(-1)$, while all others have underlying bundle type $\cO(-2) \oplus \cO(-2)$. In Types A2 and B2, where $\mathcal{I}$ contains $\{1, 2, 3, 4\}$, all points of the central sphere have underlying bundle type $\cO(-2) \oplus \cO(-2)$. When the underlying bundle type is $\cO(-2) \oplus \cO(-2)$, it is meaningful to talk about the cross-ratio of the flags (interpreted as points in  $\mathbb{P}(\C^2)$) since $\mathrm{Aut}(\cO(-2) \oplus \cO(-2))$ is the group of constant $SL(2,\C)$ matrices, which acts by M\"obius transform on the flags.  In interior chambers of Types A2 and B1, one has that $F_j , F_k $ and $F_l$ are all distinct\footnote{Otherwise, say if $F_j=F_k$, there is a destabilizing parabolic line bundle intersecting both flags with parabolic degree $-2 + (1-\alpha_j)+(1-\alpha_k) + \alpha_i + \alpha_l = K_{\{i,l\}}>0$.}, so fixing them, the central sphere is parameterized by the flag $F_i$; three wobbly points occur where $F_i$ is equal to $F_j, F_k$ and $F_l$, while the distinguished exterior sphere, corresponding either to $I=\{1, 2, 3, 4\}$ or $I=\emptyset$, attaches where a certain cross-ratio is equal to $p_0$.  In interior chambers of Types A1 and B2, the central sphere is still parameterized by a cross-ratio of the four flags, but describing it is slightly more difficult. We refer the reader to \cite{Meneses}.

\subsubsection{Holomorphic coordinate on central sphere}
There are holomorphic coordinates on the central sphere such that exterior spheres (associated $I \in \mathcal{I}$) attach at $p \in D$. In this section, we want to make a few comments about the map
\begin{align}\label{eq:ItoD} \mathcal{I} &\to D \\ \nonumber 
I &\mapsto p \end{align}
relating these.

While we know that the central sphere can be holomorphically identified with $\CP^1$ and the wobbly locus with $D$, we remark that there are $4$ conformal transformations of the spheres which preserve the set $D$, up to permutations of its elements.  If we define the cross ratio of four points to be 
\[(p_1, p_2;p_3, p_4) = \frac{(p_3-p_1)(p_4-p_2)}{(p_3-p_2)(p_4-p_1)},\]
then we observe that 
\[(p_1, p_2;p_3, p_4)=(p_2, p_1; p_4, p_3)=(p_3, p_4; p_1, p_2)=(p_4, p_3; p_2, p_1). \]
This means that ``the natural coordinate on the central sphere'' is not well-defined. Two of the reasonable maps for \eqref{eq:ItoD} are as follows

\medskip
\noindent\emph{Convention \#1:} In this first convention, the distinguished point of $D$ is dependent on the chamber.  Consequently, the natural holomorphic coordinate jumps as we cross these interior hyperplanes between chambers.

Given an even partition set $\cI$, the natural coordinate $u$ on $\CP^1_{\mathrm{central}}$ is described by matching up $I \in \cI$ with points of $D$.
\begin{itemize}
  \item Type A1: The even partition set is $\mathcal{I}_{A, i}=\{\emptyset, \{k,l\}, \{j,l\}, \{j,k\}\}$. The sphere $\CP^1_i$ is labelled by $\emptyset$, $\CP^1_j$ is labelled by $\{k,l\}$, etc.
  \item Type A2: $\mathcal{I}_{A, i}^c=\{\{1, 2, 3, 4\}, \{i,j\}, \{i,k\}, \{i,l\}\}$. The sphere $\CP^1_i$ is labelled by $\{1, 2, 3, 4\}$, $\CP^1_j$ is labelled by $\{i,j\}$, etc.
  \item Type B1: $\mathcal{I}_{B, i}=\{\emptyset, \{i,j\}, \{i,k\}, \{i,l\}\}$. The sphere $\CP^1_i$ is labelled by $\emptyset$, $\CP^1_j$ is labelled by $\{i,j\}$, etc.
  \item  Type B2: $\mathcal{I}_{B, i}^c=\{ \{1, 2, 3, 4\}, \{k,l\}, \{j,l\}, \{j,k\} \}$. The sphere $\CP^1_i$ is labelled by $\{1, 2, 3, 4\}$, $\CP^1_j$ is labelled by $\{k,l\}$, etc.
  \item For each exterior chamber, use the same coordinate $u$ for the adjacent interior chamber.
  \end{itemize}
Consequently, one sees that if $\mathcal{I}$ and $\mathcal{I}'$ are adjacent interior chambers with respective distinguished points $i$ and $i'$ in $\{1, 2, 3, 4\}$ then we have the associated conformal transformation.

\smallskip

\noindent\emph{Convention \#2:} In this second convention, we pick a distinguished point of $D$ once and for all and put coordinates on the central sphere so that the distinguished exterior sphere attaches at this distinguished point of $D$.
In this convention, the holomorphic coordinate does not jump as we cross the hyperplanes between chambers.  
Taking $D=\{0, 1, p_0, \infty\}$,
we could, for example, take $\infty$ to be our distinguished point in $D$.  There is an interior chamber with even partition set
$$\cI=\{\{0, 1, p_0, \infty\}\}, \{0, \infty\}, \{1, \infty\}, \{\infty, \infty\} \}$$
which allows us to label the spheres as 
\begin{align*}
\CP^1_\infty &\leftrightarrow \{0, 1, p_0, \infty\} \mbox{ or } \emptyset\\
\CP^1_0 &\leftrightarrow \{0, \infty\} \mbox{ or } \{0, \infty\}^c \\
\CP^1_1 &\leftrightarrow \{1, \infty\} \mbox{ or } \{1, \infty\}^c\\
\CP^1_{p_0} &\leftrightarrow \{p_0, \infty\} \mbox{ or } \{p_0, \infty\}^c\\
\end{align*}

\subsubsection{Exterior Spheres} Looking at the expressions for the exterior $\C^\times$-fixed points for the exterior spheres in \eqref{eq:extfixed}, we recall that stability prescribed that $\pdeg (\cL_2(\beta_2))<0$, and hence the flags $F_{p_i}$ for $i \in I$ had to line up with $\cL_2(\beta_2)$. However, it was merely the ``$\C^\times$-fixed point condition'' that constrained the flags for $i \in I^c$ to line up with $\cL_1(\beta_1)$. We call the flags $F_i$ for $i \in I^c$ ``moveable flags''. In this section, we show how we can fill out the associated exterior sphere minus the wobbly point, by taking extensions of the subbundle $\cL_2(\beta_2)$ by the quotient bundle $\cL_1(\beta_1)$. 

\begin{prop} \label{prop:extension}
Fix a set $I$ of even cardinality and assume the parabolic weights $\boldsymbol{\alpha}$ satisfy $K_I>0$.  Let 
\[ \left(\cL_1(\beta_1)\oplus\cL_2(\beta_2),\begin{pmatrix}
        0&0\\\phi_0&0
\end{pmatrix}\right) \]
be the associated $\C^\times$-fixed point from Proposition \ref{prop:generalsplit}.
There is a $\C$-family of stable nilpotent Higgs bundles $(\mathcal{E}(\alpha), \varphi)$ in $\mathcal{M}(\boldsymbol{\alpha}, \mathbf{0})$, such that $\mathcal{E}(\alpha)$ is an extension
\[ 0 \to \cL_2(\beta_2) \to \mathcal{E}(\alpha) \to \cL_1(\beta_1) \to 0\]
and $\varphi$ vanishes on $\cL_2(\beta_2)$ and  thus induces a map $\mathcal{E}(\alpha)/\cL_2(\beta_2) \to \cL_2(\beta_2) \otimes K(D)$ agreeing with $\phi_0$.
\end{prop}

We delay the proof, since we first need to discuss extension classes of parabolic line bundles. 

The isomorphism classes of extensions of parabolic bundles are in one-to-one correspondence with elements of
the group $\mathrm{Ext}^1(\cL_1(\beta_1), \cL_2(\beta)_2)$. In what follows,  we will use $\mathrm{Hom}(\cL(\beta), \cL'(\beta'))$ for both parabolic homomorphisms and strongly parabolic homomorphisms
between line bundles, unambiguously, since the distinction doesn't matter for the line bundles where $\beta(p) \neq \beta'(p)$. Here, 
\begin{align*}\mathrm{Ext}^1(\cL_1(\beta_1), \cL_2(\beta_2))&=H^1(\CP^1, \mathrm{Hom}(\cL_1(\beta_1), \cL_2(\beta_2)))\\
& \simeq H^0(\CP^1, \mathrm{Hom}(\cL_1(\beta_1), \cL_2(\beta_2))^* \otimes K)^*\\
&= H^0\left(\CP^1, 
\cL_2(\beta_2)^* \otimes \cL_1(\beta_1) \otimes K(D)\right)^*. \end{align*}
We remind the reader that in the definition of tensor products of parabolic line bundles, the parabolic weights are taken in $[0,1)$, i.e.  \[(\beta_1 -\beta_2)- \lfloor \beta_1 - \beta_2 \rfloor = \begin{cases} 2 \alpha_i & i \in I\\ 1-2 \alpha_i & i \in I^c \end{cases}\] and underlying holomorphic bundle
is $\cL_2^* \otimes \cL_1 \otimes K(D_{I^c})$. The underlying holomorphic bundle has degree $0$. Thus, 
\begin{lem}\label{lem:extensionclasses}
For $\cL_1(\beta_1), \cL_2(\beta_2)$ as in Proposition \ref{prop:extension},
the vector space $\mathrm{Ext}^1(\cL_1(\beta_1), \cL_2(\beta_2))$ is one-dimensional.
\end{lem}
\begin{proof}[Proof of Proposition \ref{prop:extension}] The proof follows from Lemma \ref{lem:extensionclasses}.
\end{proof}

Concretely, when $|I|$ is even, these one-dimensional extension classes represent the following parabolic bundles with parabolic subbundle $\cL_1$:
\begin{itemize}
\item When $|I|=2$, the underlying holomorphic bundle is $\cO(-2) \oplus \cO(-2)$. Configurations of flags can be thought of as points in $\CP^1$, with holomorphic gauge transformations acting by M\"obius transformations.  The flags at $i \in I$ line up with $\cL_2=\cO(-2)$ (thought of as ``$\infty \in \CP^1=\mathbb{P}(\C^2)$''), while the flags at $i \in I^c$ do not line up with $\cL_2 =\cO(-2)$ (i.e. they are valued in $\C=\CP^1-\{\infty\}$). The extension class is trivial if the two flags at $i \in I^c$ agree, and by M\"obius transformation we can take them both to line up with $\cL_1$; the extension class is non-trivial if the two flags at $i \in I^c$ are distinct.
\item When $|I|=0$, the underlying holomorphic bundle is $\cO(-1) \oplus \cO(-3)$. No flags line up with $\cO(-1)$. The extension is trivial if there is a degree $-3$ subbundle interpolating all flags; otherwise, it is non-trivial.
\item When $|I|=4$, a different phenomenon occurs.  Because there is a non-trivial \emph{holomorphic} bundle extension $0 \rightarrow \cO(-3) \to \cE \rightarrow \cO(-1) \rightarrow 0$, the underlying bundle type is $\cO(-3) \oplus \cO(-1)$ for the trivial extension (i.e. the $\C^\times$-fixed point), while it is $\cO(-2) \oplus \cO(-2)$ for all non-trivial extensions. Regardless, all flags line up with an $\cO(-3)$ subbundle. 
\end{itemize}

\begin{ex}[The exterior sphere when $I=\{1, 2, 3, 4\}$] We now give slightly more concrete expressions for $|I|=4$. Suppose $D=\{p_1, p_2, p_3, p_4\}$ with $p_i \in \C$ distinct.
Let $s^{-1}$ be the deformation parameter in the extension. When $s^{-1} \in \C^\times$, the underlying holomorphic bundle type is $\cO(-2) \oplus \cO(-2)$.
In the trivialization $\de z \oplus \de z$, the Higgs field can be written 
\[ \varphi_{\{1, 2, 3, 4\}, s}= \frac{s \de z}{\prod_{i \in \{1, 2, 3, 4\}}(z-p_i)} \begin{pmatrix}(z-z_a)(z-z_b)& -(z-z_b)^2 \\ (z-z_a)^2 & -(z-z_a)(z-z_b) \end{pmatrix},\]
for $z_a, z_b \in \CP^1$ with  $z_a \neq z_b$. All flags line up with the line bundle $\begin{pmatrix} z-z_b\\z-z_a \end{pmatrix}$ of degree $-3$.
Note that the constants $z_a, z_b$ are not geometrically meaningful since the constant gauge transformation $g=\begin{pmatrix} a & b \\ c & d \end{pmatrix}$ acts by M\"obius transform $z_a \mapsto \frac{a z_a - c z_b}{a-c}, z_b \mapsto \frac{b z_a - d z_b}{b -d}$. We can use the freedom of $z_a, z_b$ in order to show that the $s \to 0$ limit is a certain $\C^\times$-fixed point; likewise, for the $s \to \infty$ limit. 

We consider the $s \to \infty$ limit and produce a gauge transformation $g_s$ on $\C \subset \CP^1$ such that 
\[ g_s^{-1} \varphi_{\{1, 2, 3, 4\}, s}  g_s =  \frac{\de z}{\prod_{i=1}^4 (z-p_i)} \begin{pmatrix} 0 & 0 \\ 1 & 0 \end{pmatrix},\]
matching the expression for the $\C^\times$-fixed point associated to $I=\{1, 2, 3, 4\}$.  
For this limit, it is convenient to take $z_a=\infty$ and $z_b =0$.
Recall the proof that non-trivial extensions $0 \rightarrow \cO(-1) \rightarrow \cE_t \rightarrow \cO(1) \rightarrow 0$ all satisfy $\cE_t \simeq \cO \oplus \cO$ for $t \in \C^\times$ based on the Birkhoff--Grothendieck decomposition 
\[ \begin{pmatrix} z & t \\ 0 & z^{-1} \end{pmatrix} = \begin{pmatrix} 0 & t \\ -1 & \frac{1}{z} \end{pmatrix} \begin{pmatrix}t & 0 \\ -z & 1 \end{pmatrix}^{-1}. \]
Tensoring everything by $\cO(-2)$, we get the equivalent result that non-trivial extensions  $0 \rightarrow \cO(-3) \rightarrow \cE_t \rightarrow \cO(-1) \rightarrow 0$ satisfy $\cE_t \simeq \cO(-2) \oplus \cO(-2)$. 
Taking $t=s^{-1}$,
modify the change of trivialization (by conjugation by a permutation matrix) so that $\cO(-1)$ is listed first and (by multiplication by $s^{-1}$) so the determinant is one: \[g_s=s^{-1/2} \begin{pmatrix} 0 & 1 \\ 1 & 0 \end{pmatrix}  \begin{pmatrix}s^{-1}& 0 \\ -z & 1 \end{pmatrix}^{-1} \begin{pmatrix} 0 & 1 \\ 1 & 0 \end{pmatrix}.\]
We compute that in the new trivialization of $\cE_{\frac{1}{s}}$, the Higgs field is 
\[ g_s^{-1} \varphi_{\{1, 2, 3, 4\}, s}  g_s =  \frac{\de z}{\prod_{i \in I} (z-z_i)} \begin{pmatrix} 0 & 0 \\ 1 & 0 \end{pmatrix}\]
as desired, even without taking the limit as $s^{-1} \to 0$.
The distinguished $\cO(-3)$ corresponding to $e=\begin{pmatrix} z\\ 1 \end{pmatrix}$ get mapped to 
\[ \lim_{s^{-1} \to 0} g_s \cdot e = \begin{pmatrix} 1 \\ 0 \end{pmatrix}, \]
in $\mathcal{O}(-3) \oplus \mathcal{O}(-1)$.

Conversely, at the wobbly point corresponding to $s=0$, the description of the $\C^\times$-fixed point depends on the chamber. Noting that the flag $F_{p_i} =\frac{p_i - z_b}{p_i-z_a}$, we note that the cross-ratio of the flags $(F_{p_1}, F_{p_2}; F_{p_3}, F_{p_4})$ of $\varphi_{\{1, 2, 3, 4\}, s}$ for $s \in \C^\times$ is equal to $(p_1, p_2; p_3, p_4)$. In interior chambers, this determines the $\C^\times$-fixed point. In exterior chambers, $\{1, 2, 3, 4\} \in \mathcal{I}$ forces $|I_0|=3$. In the exterior chamber of Type $E2_i$, it is useful to take $z_a=p_i$ and $z_b$ arbitrary and note that for $s \in \C^\times$, $\varphi_{\{1, 2, 3, 4\}, s}$ is gauge equivalent to\[  \frac{\de z}{\prod_{i \in I} (z-z_i)}  \begin{pmatrix} s(z-p_i)(z-z_b) & -s^2(z-z_b)^2 \\ (z-p_i)^2 & -s(z-p_i)(z-z_b)  \end{pmatrix}\]
which converges as $s \to 0$ to the following desired attaching point in the central sphere:
\[\cE=\cO(-2) \oplus \cO(-2) \qquad \varphi = \frac{(z-u) \de z}{\prod_{j \neq i} (z-z_j)} \begin{pmatrix} 0 & 0 \\ 1 & 0 \end{pmatrix} \Big|_{u=p_i}.\]

\end{ex}

\begin{ex}[The exterior sphere when $|I|=2$]
The exterior sphere when $|I|=2$ is almost trivial to describe.
For $s \in \CP^1-\{0\}$, we  
\begin{equation}
\mathcal{E} = \cL_1 \oplus \cL_2, \varphi =\begin{pmatrix}
        0&0\\\phi_0&0
\end{pmatrix},
\end{equation}
where $\cL_1 \simeq \cO(-2) \simeq \cL_2$, $F_{p_i} = \cL_2|_{p_i}$ for $i \in I$, and from \eqref{eq:extfixed}
\begin{equation} \phi_0 = \frac{\de z 
}{\prod_{i \in I} (z-p_i)}. \end{equation}
Designating the two points of $I^c$, $I^c=\{i_{vary}, i_{fixed}\}$, we let $F_{p_{i_{fixed}}} = \cL_1|_{p_{i_{fixed}}}$, while letting 
$F_{p_{i_{fixed}}} = \begin{pmatrix} 1\\ s^{-1} \end{pmatrix}$ in the above decomposition. 

We recover the exterior $\C^\times$-fixed point, by taking $s=\infty$. We refer to Meneses for the attachment at the interior $\C^\times$-fixed point, since it is chamber dependent. 
\end{ex}

\section{Torelli theorem for Hitchin moduli spaces}\label{sec:Torelli}

We now state the Torelli theorem for weakly parabolic Hitchin moduli spaces on the four-punctured sphere. The proof will occupy the remaining part of the paper.

Recall from Section \ref{sec:distinguished_basis} that each chamber determines a distinguished basis for the second homology of the moduli space. More precisely, if $\boldsymbol{\alpha}$ lies in an interior chamber labelled by the even partition set $\mathcal{I}$, then a basis $\mathfrak{B}$ of $H_2(\cM(\boldsymbol{\alpha}, \mathbf{0}),\Z)$ is provided by the central sphere $S_0$ consisting of stable parabolic bundles and the exterior spheres $S^2_J$ for $J \in \mathcal{I}$. If $\boldsymbol{\alpha}$ lies in an exterior chamber labelled by the odd order subset $I_0$ and associated even partition set $\mathcal{I}$, then a basis $\mathfrak{B}$ of $H_2(\cM(\boldsymbol{\alpha}, \mathbf{0}),\Z)$ is provided by the central sphere $S_0=S_{I_0}$ and the exterior spheres $S^2_J$ for $J \in \mathcal{I}$.

Each such basis may be extended to a parallel basis of the local system $\mathcal{H}$ in Proposition \ref{prop:localsystem} and hence to bases of $H_2(\cM(\boldsymbol{\alpha}, \mathbf{m}),\Z)$ for all $(\boldsymbol{\alpha}, \mathbf{m}) \in \mathcal{R}$ simultaneously. We will make no notational distinction between a basis $\mathfrak{B}$ and its parallel extension in the following. This applies as well to the homology classes of the individual spheres $S_0$ and $S^2_J$ in that basis.

\begin{thm}[Torelli theorem for Hitchin moduli spaces] \label{thm:Torelli}
For an interior chamber labelled by even partition set $\mathcal{I}$, the integrals of $\omega_I$ and $\Omega_I$ over the exterior sphere labelled by the subset $J \in \mathcal{I}$ are given by 
\begin{align}\label{eq:Torelli1}\int_{S^2_J} \omega_I &= 4 \pi^2 K_J\\ \nonumber
\int_{S^2_J} \Omega_I &= 2 \pi M_J,
\end{align}
where (see original definition of $K_J$ in \eqref{eq:K})
\begin{align}\label{eq:Torelli2}K_J &= \sum_{j \in J} \alpha_{p_j} -\sum_{j \notin J} \alpha_{z_j} + \left \lfloor \frac{|J^c| - |J|}{4} \right\rfloor\\ \nonumber
M_J &= \sum_{j \in J} m_{p_j} -\sum_{j \notin J} m_{p_j}
\end{align}

For an exterior chamber with central sphere labelled by $I_0$ and associated even partition set $\mathcal{I}$, the integrals of $\omega_I$ and $\Omega_I$ of the exterior sphere labelled by the subset $J \in \mathcal{I}$ are given by 
\begin{align}\label{eq:Torelli3}
\int_{S^2_J} \omega_I &= 4 \pi^2 \left(K_J- K_{I_0} \right)\\ \nonumber
\int_{S^2_J} \Omega_I &= 2 \pi \left(M_J- M_{I_0}\right)
\end{align}
\end{thm}

\begin{rem}
We emphasize the following features of these formul\ae:
\begin{enumerate}
\item The $\omega_I$-periods $\int_{S^2_J} \omega_I$ are affine-linear functions of the parabolic weights $\boldsymbol{\alpha} \in  (0, \frac 12 )^4$.
\item The $\Omega_I$-periods $\int_{S^2_J} \Omega_I$ are linear functions of the complex masses $\mathbf{m} \in  \C^4$. 	
\end{enumerate}
	In particular, the $\omega_I$-periods are independent of the complex masses and the $\Omega_I$-periods are independent of the parabolic weights.
\end{rem}
 
The proof of Theorem \ref{thm:Torelli} will be carried out in the remaining sections of the paper. In the rest of this section we will discuss a number of consequences, mostly concerning the $\omega_I$-periods and the geometry of strongly parabolic Hitchin moduli spaces $\cM(\boldsymbol{\alpha}, \mathbf{0})$.
 
\medskip
 
The first two corollaries collect some useful formul\ae\ for the integral of $\omega_I$ over the central and the exteriors spheres in various chambers. These will be used later to determine the range of the first component of the Torelli map $\mathcal{T}_{\mathfrak B}$ in Definition \ref{def:Torelli}. Note that since these spheres are holomorphic with respect to complex structure $I$ in $\cM(\boldsymbol{\alpha}, \mathbf{0})$, the integral of $\omega_I$ precisely computes the area of that sphere with respect to the Hitchin metric on $\cM(\boldsymbol{\alpha}, \mathbf{0})$.

\begin{cor}\label{cor:centralvolumes}
The integral of $\omega_I$  over the central sphere $S_0$ is as follows for each of the chamber types:
\smallskip
  \begin{itemize}
    \item Type A1: $\mathcal{I}_{A, i}=\{\emptyset, \{k,l\}, \{j,l\}, \{j,k\}\}$. 
    \[\int_{S_0} \omega_I=4 \pi^2 \cdot 2\alpha_i\]
    \item Type A2: $\mathcal{I}_{A, i}^c=\{\{1, 2, 3, 4\}, \{i,j\}, \{i,k\}, \{i,l\}\}$.  \[\int_{S_0} \omega_I=4\pi^2\left(1- 2\alpha_i\right) \]
    \item Type B1: $\mathcal{I}_{B, i}=\{\emptyset, \{i,j\}, \{i,k\}, \{i,l\}\}$.   \[\int_{S_0} \omega_I= 4 \pi^2 \left( -\alpha_i + \alpha_j +\alpha_k + \alpha_l \right)= 
    4 \pi^2 \left( - K_{\{i\}}\right)  \]
    \item  Type B2: $\mathcal{I}_{B, i}^c=\{ \{1, 2, 3, 4\}, \{k,l\}, \{j,l\}, \{j,k\}\}$.   \[\int_{S_0} \omega_I=4 \pi^2 \left( 1+ \alpha_i -  \alpha_j -  \alpha_k - \alpha_l\right) 
    = 4\pi^2 \left(- K_{\{j, k, l\}} \right)\]
    \item Type E1: $\mathcal{I}_{E, i}=\mathcal{I}_{B, i}=\{\emptyset, \{i,j\}, \{i,k\}, \{i,l\}\}$ and $I_0=\{i\}$.
    \[ \int_{S_0} \omega_I = 4 \pi^2 \, K_{I_0}\]
    \item Type E2: $\mathcal{I}_{E, i}^c=\mathcal{I}_{B, i}^c=\{ \{1, 2, 3, 4\}, \{k,l\}, \{j,l\}, \{j,k\}\}$ and $I_0=\{j, k, l\}$.
    \[ \int_{S_0} \omega_I = 4 \pi^2 \,K_{I_0} \]
    \end{itemize}
\end{cor}
    \begin{rem}
      Note in particular that in the exterior chamber labelled by $I_0$, $K_{I_0}$ is positive, while in the adjacent interior chamber $K_{I_0}$ is negative. Thus, it is clear that these integrals are all positive. This is expected since they compute the area of the central sphere in $\cM(\boldsymbol{\alpha}, \mathbf{0})$. Furthermore, these integrals vanish on the boundary face of the respective chamber which is given by the intersection of the closed chamber with the boundary of the Biswas polytope. This behavior is predicted by the description of wall-crossing in \cite{Meneses}: when crossing a wall from an interior chamber to an adjacent exterior chamber, the central sphere given by the moduli space of stable parabolic bundles disappears; it reappears as the sphere $S_{I_0}$ labelled by the odd order subset $I_0$.  
    \end{rem}
 
 \begin{rem}[Witten's Volume Formula]
 In the 16 interior chambers, where the central sphere $S_0$ represents the moduli space of stable parabolic bundles on $\CP^1 \setminus D$, the integral of $\omega_I$ over $S_0$ is known as the symplectic volume of the moduli space. Witten gives a formula for the symplectic volume in arbitrary genus for any number of punctures  in \cite{Witten91}, which has been rederived by Heller and Heller for the four-punctured sphere in \cite{hellerheller}. We find it noteworthy that the formula in Corollary \ref{cor:centralvolumes} precisely matches Witten's formula as presented in Theorem 4 in \cite{hellerheller} (where the shifted weights $\rho_i = \frac 12 - \alpha_i \in (0, \frac 12)$ are used  to state the formula).
  \end{rem}
   
\begin{cor}
    The integral of $\omega_I$ over the exterior spheres are as follows for each of the exterior chambers:
    \smallskip
   \begin{itemize}
       \item Type E1: $\mathcal{I}_{E, i}=\mathcal{I}_{B, i}=\{\emptyset, \{i,j\}, \{i,k\}, \{i,l\}\}$ and $I_0=\{i\}$.
    \[ \int_{S_\emptyset} \omega_I = 4 \pi^2(1-2\alpha_i), \quad \int_{S_{\{i,j\}}} \omega_I= 4 \pi^2 \cdot 2 \alpha_j,  \quad \int_{S_{\{i,k\}}} \omega_I = 4 \pi^2 \cdot 2 \alpha_k,  \quad \int_{S_{\{i,l\}}} \omega_I= 4 \pi^2 \cdot 2 \alpha_l  \]
    \item Type E2: $\mathcal{I}_{E, i}^c=\mathcal{I}_{B, i}^c=\{ \{1, 2, 3, 4\}, \{k,l\}, \{j,l\}, \{j,k\}\}$ and $I_0=\{j, k, l\}$.
    \[ \int_{S_{\{1, 2, 3, 4\}}} \omega_I = 4 \pi^2 \cdot 2\alpha_i, \; \int_{S_{\{k,l\}}} \omega_I= 4 \pi^2(1-2 \alpha_j),  \; \int_{S_{\{j,l\}}} \omega_I = 4 \pi^2(1-2 \alpha_k),  \; \int_{S_{\{j,k\}}} \omega_I= 4 \pi^2(1-2 \alpha_l) \]
    \end{itemize}
\end{cor}

\begin{rem}
It is evident that the integrals are positive on the open chambers. Each integral vanishes on one of the boundary faces of the chamber which are given as the intersection of the closed chamber with the boundary of the weight cube $(0,\frac 12)^4$. 
\end{rem}

  \begin{rem} [Orbifold Points] \label{rem:orbifold} In the closure of each chamber there is a unique point $\boldsymbol{\alpha} \in [0, \frac 12]^4$  at which the exterior spheres have volume zero and the central sphere has volume $2 \pi^2$. While this is not a generic point, so that the moduli space $\mathcal{M}(\boldsymbol{\alpha}, \mathbf{0})$ is singular and the Hitchin metric is formally not defined, it follows from the injectivity part of the Torelli theorem \cite{CVZ2024} and the Kummer construction for gravitational instantons in \cite{BM11} that the Hitchin metric degenerates to an appropriate scaling of the flat orbifold $(\C \times T^2_\tau)/\Z_2$ at the distinguished point. This orbifold has singularities modeled on $\C^2/\Z_2$ at the four fixed points of the $\Z_2$ action. In the interior chambers, this distinguished point occurs at $\boldsymbol{\alpha}=(\frac{1}{4}, \frac{1}{4}, \frac{1}{4},\frac{1}{4})$.
Each exterior chamber with central sphere labelled by the odd order subset $I_0$ contains the vertex $v_{I_0}$ defined in \eqref{eq:vI}. This point is a vertex on the boundary of the weight cube and is the distinguished point for the exterior chamber.
  \end{rem}
  
  \begin{rem}[Comparison with recent work on degeneration from ALG-$D_4$ to ALE-$D_4$ gravitational instantons]   Recently, Heller--Heller--Traizet \cite{HHT2025}\footnote{   In \cite{HHT2025}, Heller--Heller--Traizet use twistorial methods to describe the path of orbifold metrics corresponding to the line of parabolic weights when $\boldsymbol{\beta}=(\beta, \beta, \beta, \beta)$. This line segment connects the vertex $(\frac 12, \frac 12,\frac 12,\frac 12)$ on the boundary of the weight cube to the orbifold point $(\frac 14, \frac 14, \frac 14, \frac 14)$ in the interior. It lies inside the Biswas polytope and has 3 interior chambers adjacent to it.  Along the interior of the segment three exterior spheres are collapsed. Indeed, the segment lies in the intersection of the three hyperplanes $\{K_{\{1,2\}}=0\}$, $\{K_{\{1,3\}}=0\}$ and $\{K_{\{1,4\}}=0\}$. At the vertex $(\frac 12, \frac 12,\frac 12,\frac 12)$, the central sphere collapses as well.  The authors prove that in the blow-up limit when the area of the central sphere is rescaled to be constant along the segment, the family of hyperK\"ahler orbifolds converges as $R \to 0$ to the quotient of an isometric action of $\Z_2 \times \Z_2$  on $T^*\CP^1$ equipped with the Eguchi--Hanson metric. This action has 6 fixed points on $\CP^1 \subset T^*\CP^1$ giving rise to 3 orbifold points in the limit. The methods used in \cite{HHT2025} are quite different from the methods used in this paper, namely they use loop group methods to construct solutions of Hitchin's self-duality equation.
}, Heller--Heller--Meneses \cite{HellerHellerMeneses}\footnote{Building on \cite{HHT2025}, Heller--Heller--Meneses again use loop methods. Their proof is for strongly parabolic Higgs bundles on the $n$-punctured sphere.}, and Fredrickson--Yae \cite{FredricksonYae}\footnote{Fredrickson--Yae use constructive analytic techniques building a hermitian metric $h_R^{\mathrm{app}}$ which is close to solving the Hitchin equations, then perturb to the true harmonic metric $h_R$ for $R\ll 0$. Their proof is for strongly parabolic Higgs bundles on the $n$-punctured sphere.} have considered the behavior of the  hyperK\"ahler metric  on a family of Hitchin moduli spaces $\mathcal{M}_R(\boldsymbol{\alpha}(R))$ as $R \to 0$. Here, the subscript $R$ indicates a certain rescaling of the metric by $R^{-1/2}$ on its torus fibers and $R^{1/2}$ in the base directions (see Appendix \ref{sec:scalings}\footnote{For comparison with Appendix \ref{sec:scalings}, take $\lambda_2=R^2, \lambda_1=R^{-1}$.}).  With this tuning, the holomorphic symplectic form $\Omega_{I, R}$ is unchanged along this family. The parabolic weights depend linearly on $R$ via $\alpha_i(R)=\frac{1}{2} - R \beta_i$ for weights $\boldsymbol{\beta}$ sufficiently close to $\mathbf{0}$. They proved that on the $n$-punctured sphere, the limiting hyperK\"ahler metric  agreed (up to a scale) with the hyperK\"ahler metric on $n$-hyperpolygon space with $(\mathfrak{u}(1))^n$-valued real moment map determined by $\boldsymbol{\beta}$.

Our theorem gives another interpretation of their choice of $\boldsymbol{\alpha}(R)$ when $n=4$: it is precisely the choice so that the integral of $\omega_{I,R}$ over the exterior spheres is constant in $R$. (In fact, early computations of $\omega_I$ and $\Omega_I$ from our project were the motivation for \cite{FredricksonYae}'s work.)  Moreover, while all three of those papers restrict to the case $\mathbf{m}=\mathbf{0}$, from our perspective, one should also be able to fix $\mathbf{m}(R)=\mathbf{m}$, and land on the hyperpolygon space with value $\left(\mathfrak{gl}(1)\right)^n$-valued complex moment map determined by $\mathbf{m}$.
\end{rem}

\section{Integral of \texorpdfstring{$\omega_I$}{the real symplectic form}}\label{sec:TorelliomegaI}

We separately integrate $\omega_I$ in the strongly parabolic moduli spaces ($\mathbf{m}=\mathbf{0}$) and the weakly parabolic moduli spaces ($\mathbf{m}\neq \mathbf{0}$). The computation in the strongly parabolic moduli spaces relies on our description of the $\C^\times$-fixed points in each chamber. The extension to the weakly parabolic moduli spaces depends on the structure of the fibration $\mathcal{P}_0(\boldsymbol{\alpha}) \to \C^4_{\mathbf{m}}$.

\subsection{Integrating \texorpdfstring{$\omega_I$}{the real symplectic form} in strongly parabolic Hitchin moduli spaces} 

In this section, we will use the notation $D=\{p_1, p_2, p_3, p_4\}$ rather than $D=\{0, 1, p_0, \infty\}$, and we will label spheres by $I \in \mathcal{I}$, using the notation of Section \ref{sec:nilpotentcone}.

\begin{prop}\label{prop:stronglyparabolicvolumesv2}
In an interior chamber (see Section \ref{sec:interiorchambers}) labelled by even partition set $\mathcal{I}$, the area of the exterior sphere labelled by the subset $J \in \mathcal{I}$ is given by 
\[\int_{S^2_J} \omega_I = 4 \pi^2 K_J>0,\]
where (see original definition of $K_J$ in \eqref{eq:K})
\[K_J = \sum_{j \in J} \alpha_{p_j} -\sum_{j \notin J} \alpha_{p_j} +\left\lfloor \frac{|J^c| - |J|}{4}\right\rfloor.\]

In an exterior chamber (see Section \ref{sec:exteriorchambers}) labelled $I_0=\{i\}$ or $\{j, k, l\}$, an odd order subset of $\{1, 2, 3, 4\}$,  with associated even partition set $\mathcal{I}$, 
the area of the exterior sphere labelled by the subset $J \in \mathcal{I}$  is given by 
\[\int_{S^2_J} \omega_I = 4 \pi^2 \left(K_J- K_{I_0} \right)>0.\]
\end{prop}

The proof of Proposition \ref{prop:stronglyparabolicvolumesv2} is based on the following two observations.
\begin{obs}\label{obs:1}
If $S_0$ denotes the central sphere and $S_1, S_2, S_3, S_4$ the exterior spheres, then
\[
	2 \int_{S_0} \omega_I + \sum_{i=1}^4 \int_{S_i} \omega_I  = 4 \pi^2.
\]
\end{obs}

\begin{proof}[Proof of Observation \ref{obs:1}]
If $F$ denotes the regular fiber, then $2[S_0] + \sum_{i=1}^4 [S_i] = [F]$ in homology. Hence,
\[
	2 \int_{S_0} \omega_I + \sum_{i=1}^4 \int_{S_i} \omega_I  = \int_{F} \omega_I
\]
since $\omega_I$ is closed. But the fiber may be pushed off to infinity and along the family of rescalings $m_t[(\bar\partial_A,\varphi)] = [(\bar\partial_A,t\varphi)]$ one has $ m_t^*(\omega_I - \omega_I^{\mathrm{sf}} ) = O(e^{- \varepsilon t})$ for some $\varepsilon>0$ by Theorem 7.2 in \cite{FMSW}. Hence, 
\[
\int_{F} \omega_I =\int_{F} \omega^{\mathrm{sf}}_I = \mathrm{Area}_{\mathrm{sf}}(T^2_\tau) = 4 \pi^2 
\]
as computed by Proposition 8.4 in \cite{FMSW}.
\end{proof}
\begin{obs}\label{obs:2}
If $\mu: S^2 \to \R$ is the moment map of a Hamiltonian $S^1$-action on $(S^2,\omega)$ with exactly two critical points $p_{max}$ and $p_{min}$, then
	\[
	\int_{S^2} \omega = 2 \pi ( \mu_{max} - \mu_{min})
	\]
	for $\mu_{max} = \mu(p_{max})$ and $\mu_{min} = \mu(p_{min})$.
\end{obs}

\begin{proof}[Proof of Observation \ref{obs:2}]
If $(\mu, \theta)$ are action-angle coordinates on $S^2 \setminus \{ p_{max}, p_{min}\}$, then $\omega = d\theta \wedge d \mu  $	
and 
\[
\int_{S^2} \omega = \int_{ S^1 }  d\theta  \int_{\mu_{min}}^{\mu_{max}}d \mu = 2 \pi ( \mu_{max} - \mu_{min})
\]
as claimed.
\end{proof}

The relevant $S^1$-action will be from restricting 
the $\C^\times$-action given by $[(\cE, \varphi)] \mapsto [(\cE, \zeta \varphi)]$ for $\zeta \in \C^\times$ to an $S^1$-action by taking $|\zeta|=1$.

\begin{lem}\label{lem:mu}
The moment map is 
\[\mu =  i \int_C \Tr \varphi \wedge \varphi^{*_h}.\]
 \end{lem}
\begin{proof}[Proof of Lemma \ref{lem:mu}]
We use the fixed hermitian metric perspective. To see that 
    \begin{align}
    \mu(\delbar_A, \varphi, h):=  i \int_C \Tr \varphi \wedge \varphi^*
    \end{align}
    is a moment map for the $U(1)$-action, i.e.\  $\iota_X \omega =d \mu$, using the convention that $\iota_X$ inserts  $X=(0, i \varphi)$ in the first slot,
    we compute that
    \[\de \mu(\dot{A}^{0,1}, \dot \varphi) =  i \int_C \Tr (\dot \varphi \wedge \varphi^* + \varphi \wedge \dot \varphi^* ) = 2 \im\int_C \Tr(\dot\varphi \wedge \varphi^*)  \]
    and 
    \[ \iota_X \omega_I (\dot{A}^{0,1}, \dot \varphi)= -2 \re \int_C \Tr((i \varphi) \wedge \dot\varphi^*) =-2 \re \int_C \Tr(i \dot \varphi  \wedge \varphi^*) = 2 \im \int_C \Tr(\dot\varphi \wedge \varphi^*)\]
    proving the claim.
   \end{proof}

\begin{proof}[Proof of Proposition \ref{prop:stronglyparabolicvolumesv2}]
    From Observation \ref{obs:2} we can compute the size of the exterior sphere labelled by $I \in \mathcal{I}$, by evaluating the moment map at the two $\C^\times$-fixed points on the sphere: the exterior $\C^\times$-fixed point and at the wobbly point.  From Lemma \ref{lem:mu}, the moment map is 
\[\mu =  i \int_C \Tr \varphi \wedge \varphi^{*_h}.\]
     It is clear that $\mu$ vanishes on the central sphere in interior chambers since $\varphi \equiv 0$. We will consequently compute $\mu$ at other $\C^\times$-fixed points labelled by $I \subset \{1, 2, 3,4 \}$, assuming the expression in \eqref{eq general rk 2 fixed} and Proposition \ref{prop:generalsplit}. The key observation\footnote{We acknowledge Arya Yae for this observation, which considerably shortened our original proof.} is that because $\varphi$ is strictly lower triangular and $h$ respects the splitting, then 
     \[
     [ \varphi, \varphi^{*_h}] = \begin{pmatrix} - \phi_0^{*_h} \wedge \varphi_0 & 0 \\ 
 0 & \phi_0 \wedge \phi_0^{*_h}	
 \end{pmatrix}
     \]
     and the Hitchin equations break into two scalar-valued-$2$-form equations 
     \[
      (1) \;\;F^\perp_{\delbar_A + \del_A^h}|_{\cL_1} - \phi_0^{*_h} \wedge \phi_0=0 \qquad \text{and} \qquad (2) \;\; F^\perp_{\delbar_A + \del_A^h}|_{\cL_2} + \phi_0 \wedge \phi_0^{*_h}=0.  
    \] 
     Using
     \[
      \varphi \wedge \varphi^{*_h} = \begin{pmatrix} 0 & 0 \\ 
 0 & \phi_0 \wedge \phi_0^{*_h}	
 \end{pmatrix}
     \]
     we thus obtain
    \begin{equation}
      \mu(\delbar_A, \varphi, h) =  i \int_C \phi_0 \wedge \phi_0^{*_h}
      = -i \int_C F^\perp_{\delbar_A + \del_A^h}|_{\cL_2} 
      =-2 \pi  \pdeg \cL_2(\beta_2)   
      = -2 \pi K_I.
    \end{equation}
Thus, we've proved Proposition \ref{prop:stronglyparabolicvolumesv2}.
    \end{proof}

  \begin{proof}[Proof of Corollary  \ref{cor:centralvolumes}]
In the interior chambers, we use Observation \ref{obs:1} to see that the integral of $\omega_I$ over the central sphere is simply $2\pi^2 (1- \sum_{I \in \mathcal{I}} K_I)$. Expanding this, we get 
  \[2 \pi^2 \left(1+ \sum_{I \in I}\left(\sum_{i \in I^c} \alpha_i - \sum_{i \in I} \alpha_i + \frac{|I|-2}{2}\right)\right)\]
The formulas above simply come by counting how many times each element of $\{1, 2, 3, 4\}$ appears in an element of $\mathcal{I}$.
In the exterior chambers, we see that the integral of $\omega_I$ over the central sphere is $2\pi^2 (1 - \sum_{I \in \mathcal{I}} K_I) +  8 \pi^2 K_{I_0}$, so we can obtain the integral for the exterior chambers by adding $8 \pi^2 K_{I_0}$ to the integral for the adjacent interior chamber of type $B$.
\end{proof}

\subsection{Integrating \texorpdfstring{$\omega_I$}{the real symplectic form} in weakly parabolic Hitchin moduli spaces} 

In this section we show that the periods of $\omega_I$ on the weakly parabolic moduli spaces $\calM(\bsa,\bfm)$ do not depend at all on the complex masses $\bfm \in \C^4$ and hence are given by the values on the corresponding strongly parabolic moduli spaces $\calM(\bsa,\mathbf{0})$. We first prove that the periods of $\omega_I$ are invariant under the $\R_+$-action 
\begin{align*}
m_t : \mathcal{M}(\boldsymbol \alpha,\mathbf m) &\to \mathcal{M}(\boldsymbol \alpha, t \mathbf m)\\ 
  (\cE, \varphi) &\mapsto (\cE, t \varphi) 
  \end{align*}
by showing that $\omega_I$ and $\omega_{I,t}= m_t^* \omega_I$ are cohomologous on $\calM(\bsa,\bfm)$  for $t \in \R_+$. Then we employ a continuity argument for $t \to 0$. 

As a preliminary we treat the case without punctures first and then describe the necessary modifications in the parabolic case.

\subsubsection{The case without punctures} 
In this case the maps $m_t : \mathcal M \to \mathcal M$ are self-diffeomorphisms of the moduli space which are isotopic to the identity. Then clearly $\omega_{I,t}$ and $\omega_I$ are cohomologous. We recall the construction of a relative primitive. If we set $\phi_s = m_t$ for $t=e^s$, then $\phi_s$ is the flow of the vector field $X = \frac{\partial}{\partial s}\bigr\vert_{s=0} \phi_s =   \frac{\partial}{\partial t}\bigr\vert_{t=1} m_t$. Since $\omega_I$ is closed, Cartan's magic formula yields 
\[
\frac{\partial}{\partial t}  m_t^* \omega_I =  \frac{1}{t}\frac{\partial}{\partial s} \phi_s^* \omega_I =\frac{1}{t} \, m_t^* \, \mathcal{L}_X  \omega_I =  \frac{1}{t} \, m_t^* \, d(  \iota_X  \omega_I).
 \]
Setting $\lambda_I =  \iota_X \omega_I$ and $\lambda_{I,t} = \frac{1}{t} m_t^* \lambda_I$  we finally obtain
 \[
\omega_{I,t} - \omega_I =m_t^* \omega_I - \omega_I  = d \Bigl(\int_1^t \lambda_{I,t} \, d\tau \Bigr).
\]

Working on the configuration space $\mathcal C$ of pairs $(A,\varphi)$ we wish to make this relative primitive as explicit as possible. Recall that the  K\"ahler form with respect to complex structure $I$ is given by
\[
\omega_I \bigl( (\dot A_1^{0,1}, \dot \varphi_1), (\dot A_2^{0,1}, \dot \varphi_2) \bigr) = - \operatorname{Im} \int_C \IP{\dot A_1^{0,1}, \dot A_2^{0,1}} + \IP{\dot \varphi_1, \dot \varphi_2}.
\]

\begin{rem} In this remark we show why the naive expression for this relative primitive doesn't work.
It is tempting to set $\hat X(A,\varphi) =\frac{\partial}{\partial t} \bigr\vert_{t=1} (A,t\varphi) = (0,\varphi)$ and define $\hat\lambda_I=\iota_{\hat X} \omega_I$, i.e.
\[
\hat \lambda_I(\dot A^{0,1}, \dot \varphi) = - \operatorname{Im} \int_C \IP{\varphi,\dot \varphi}.
\]
Then 
\[
d \hat \lambda_I \bigl( (\dot A_1^{0,1}, \dot \varphi_1), (\dot A_2^{0,1}, \dot \varphi_2) \bigr) = - \operatorname{Im} \int_C \IP{\dot \varphi_1, \dot \varphi_2} + \operatorname{Im} \int_C \IP{\dot \varphi_2, \dot \varphi_1} = -2 \operatorname{Im} \int_C \IP{\dot \varphi_1, \dot \varphi_2}.
\]
With 
\[
\hat\omega_{I,t} \bigl( (\dot A_1^{0,1}, \dot \varphi_1), (\dot A_2^{0,1}, \dot \varphi_2) \bigr) = - \operatorname{Im} \int_C \IP{\dot A_1^{0,1}, \dot A_2^{0,1}} + t^2 \IP{\dot \varphi_1, \dot \varphi_2}
\]
we obtain $\hat{\omega}_{I,t}= \hat{m}_t^* \omega_I$ and
\[
\frac{\partial}{\partial t} \hat\omega_{I,t} = -2t \operatorname{Im} \int_C \IP{ \dot \varphi_1, \dot \varphi_2 } = \hat{m}_t^* d\hat\lambda_I/t
\]
for the naive $\R_+$-action $\hat{m}_t(A,\varphi)=(A,t\varphi)$. However, this action  doesn't preserve the zero set of the real moment map 
\[
\mu_\R^{-1}(0) = \{ (A,\varphi) \in \mathcal C :  \mathcal{H}(A,\varphi) := F_A + [\varphi , \varphi^* ] = 0\}.
\]
In particular, the vector field $\hat X(A,\varphi) = (0,\varphi)$ is not tangent to this level set, and as an effect $\hat\lambda_I$ as defined above won't descend to the moduli space. Similarly, $\hat \omega_{I,t}$ is not the correct lift of the K\"ahler form $\omega_{I,t}$ of the $t$-rescaled Hitchin metric.
\end{rem}

To rectify the issue addressed in the above remark,  we project the flow to $\mu_\R^{-1}(0)$. In order to do so we determine complex gauge transformations $g_t$  such that $\tilde{m}_t(A, \varphi):=(A,t\varphi)^{g_t} \in \mu_\R^{-1}(0)$ for $t \in \R_+$.  We may assume that the $g_t$ are hermitian and will do so from now on. This $\R_+$-action now lifts the $\R_+$-action on the Higgs bundle moduli space to the solution space of the Hitchin equations. We then set $\tilde{\omega}_{I,t}= \tilde{m}_t^* \omega_I$ and obtain using $D\tilde{m}_t(\dot A,\dot \varphi) = (g_t^{-1}\dot{A}^{0,1}g_t - (g_t^{-1}\dot{A}^{0,1}g_t)^*, t g_t^{-1} \dot \varphi g_t)$ the following expression
\begin{align}\label{tresckahler}
\tilde{\omega}_{I,t}\bigl( (\dot A_1^{0,1}, \dot \varphi_1), (\dot A_2^{0,1}, \dot \varphi_2) \bigr) & =  - \operatorname{Im} \int_C \IP{g_t^{-1}\dot A_1^{0,1}g_t, g_t^{-1}\dot A_2^{0,1}g_t} + t^2 \IP{g_t^{-1}\dot \varphi_1 g_t, g_t^{-1}\dot \varphi_2 g_t}\\
     &= - \operatorname{Im} \int_C \IP{\dot A_1^{0,1}, \dot A_2^{0,1}}_{h_t} + t^2 \IP{\dot \varphi_1 , \dot \varphi_2 }_{h_t}
\end{align}
for $h_t=(g_t^{-1})^*g_t^{-1}$.
We further set 
\[
\tilde X(A,\varphi) = \frac{\partial}{\partial t} \Bigr\vert_{t=1} (A,t\varphi)^{g_t}
\]
 and then $\tilde \lambda_I = \iota_{\tilde X} \omega_I$ and $\tilde\lambda_{I,t} = \frac{1}{t} \tilde m_t^* \tilde \lambda_I$. 
 
 In the following we give a more concrete expression for $\tilde X$ and in particular for $\tilde \lambda_I$.
 
 \begin{lem} For a solution of the Hitchin equations $(A,\varphi)$ and infinitesimal deformation $(\dot A, \dot \varphi)$    
 \begin{equation}\label{inf_action}
 \tilde X(A,\varphi) = \bigl( \bar\partial_A \dot g - (\bar\partial_A \dot g)^*, \varphi + [\varphi, \dot g] \bigr)
 \end{equation}
 and
 \begin{equation}\label{relprim}
 \tilde\lambda_I (\dot A, \dot \varphi) = - \im \int_C \langle \bar\partial_A \dot g, \dot A^{0,1} \rangle + \langle \varphi + [\varphi, \dot g] , \dot \varphi \rangle
\end{equation}
where $\dot g = \dot g(A,\varphi) =-2 L_{A,\varphi}^{-1}( [\varphi,\varphi^*])$.
\end{lem}

\begin{proof}
We observe first that
\[
(A, t\varphi)^{g_t} = \bigl( A+g_t^{-1}(\bar\partial_A g_t) - (g_t^{-1}(\bar\partial_A g_t))^*, t g_t^{-1}\varphi g_t \bigr)
\]
and therefore
\begin{equation*}
\tilde X(A,\varphi)=\frac{\partial}{\partial t} \Bigr \vert_{t=1} (A, t\varphi)^{g_t} 
= \bigl( \bar\partial_A \dot g - (\bar\partial_A \dot g)^*, \varphi + [\varphi, \dot g] \bigr).
\end{equation*} 
Secondly, if $(A,t\varphi)^{g_t} \in \mu_\R^{-1}(0)$ for all $t$, then $\frac{\partial}{\partial t} \bigr \vert_{t=1} (A, t\varphi)^{g_t} \in \ker \mathcal{L}_{A,\varphi}$, where  $\mathcal{L}_{A,\varphi} := D\mathcal{H}(A,\varphi)$ denotes the linearization of the Hitchin equations which is given by
 \[
  \mathcal{L}_{A,\varphi}(\dot A, \dot \varphi) = d_A \dot A + [\dot \varphi, \varphi^*] + [ \varphi, \dot \varphi^*].
 \]
Putting this together we obtain that $\dot g$ satisfies the equation
\[
d_A ( \bar\partial_A \dot g - (\bar\partial_A \dot g)^* ) + [ \varphi + [\varphi, \dot g], \varphi^*] + [ \varphi, \varphi^* + [\varphi, \dot g]^* ] = 0
\]
or, if $L_{A,\varphi}$ denotes the linearization along the complex gauge orbit, the equation
\[
L_{A,\varphi} \dot g + 2 [\varphi,\varphi^*] = 0
\]
whereas usual $L_{A,\varphi} = i \ast \Delta_A + M_\varphi$. We conclude that $\tilde \lambda_I$  is given by 
\begin{equation}
\tilde\lambda_I (\dot A, \dot \varphi) = - \im \int_C \langle \bar\partial_A \dot g, \dot A^{0,1} \rangle + \langle \varphi + [\varphi, \dot g] , \dot \varphi \rangle
\end{equation}
for  $\dot g = \dot g(A,\varphi) =-2 L_{A,\varphi}^{-1}( [\varphi,\varphi^*])$.
\end{proof}

\begin{rem}
If $(\dot A, \dot \varphi)$ is in Coulomb gauge, then it's easily checked that $\tilde \lambda_I(\dot A, \dot \varphi) = \hat\lambda_I(\dot A, \dot \varphi)$.
\end{rem}

Following Simpson, we define the operators $D' = \partial_A + \ad \varphi^*$ and $D'' = \bar\partial_A + \ad \varphi$ on differential forms with values in $\mathfrak{sl}(E)$.  Then according to  \cite{simpsonhiggs},  one has the K\"ahler identities
\[
(D')^* = i \,[ \Lambda, D''] \quad  \text{and} \quad (D'')^* = - i \,[ \Lambda, D']
\]
where $\Lambda$ denotes contraction with the K\"ahler form on $C$, i.e.\ $\Lambda = \ast$ on $\Omega^2(\mathfrak{sl}(E)) = \Omega^{1,1}(\mathfrak{sl}(E))$ and $\Lambda = 0$ on $\Omega^0(\mathfrak{sl}(E)) \oplus \Omega^1(\mathfrak{sl}(E)) $. If we set 
\[
d_1 = D''\vert_{\Omega^0}: \Omega^0(\mathfrak{sl}(E)) \to \Omega^{0,1}(\mathfrak{sl}(E)) \oplus \Omega^{1,0}(\mathfrak{sl}(E))\] 
and  
\[
d_2= D''\vert_{\Omega^1}: \Omega^{0,1}(\mathfrak{sl}(E)) \oplus \Omega^{0,1}(\mathfrak{sl}(E)) \to \Omega^2(\mathfrak{sl}(E))
\]
then $d_1(\gamma) = (\bar\partial_A \gamma, [ \varphi, \gamma])$ for $\gamma \in \Omega^0(\mathfrak{sl}(E))$ and $d_2(\dot A^{0,1}, \dot \varphi) = \bar\partial_A \dot \varphi + [\varphi,\dot A^{0,1}]$ for $\dot A^{0,1} \in \Omega^{0,1}(\mathfrak{sl}(E)) $ and $\dot \varphi \in  \Omega^{1,0}(\mathfrak{sl}(E))$. The deformation complex at $(A,\varphi)$ assumes the shape 
\[
0 \to \Omega^0(\mathfrak{sl}(E)) \overset{d_1}{\longrightarrow}  \Omega^{0,1}(\mathfrak{sl}(E)) \oplus \Omega^{0,1}(\mathfrak{sl}(E)) \overset{d_2}{\longrightarrow} \Omega^2(\mathfrak{sl}(E)) \to 0
\]
where we use the slightly unusual indexing of the differentials in \cite{collier2024conformallimitsparabolicslnchiggs}. The formal adjoints are given by $d_1^* = - i \ast (\partial_A +  \ad \varphi^*)$ and $d_2^* = i (\partial_A +  \ad \varphi^*) \ast$. 

Finally, recall from \cite{MSWW14} the formula $- i \ast L_{A,\varphi} = 2 d_1^* d_1$ for the linearization. With these preparations we can prove the following:
\begin{lem}
Let $\tilde \lambda_I$ be defined by \eqref{relprim}. Then the following holds: 
	\begin{enumerate}
	\item $R_g^* \tilde \lambda_I = \tilde \lambda_I $ for all unitary gauge transformations $g$.
	\smallskip
	\item $\tilde \lambda_I$ is horizontal, i.e.\ $\tilde \lambda_{I}(X_\gamma) =0 $ for $X_\gamma = (d_A \gamma, [\varphi, \gamma])$ and $\gamma$ an infinitesimal unitary gauge transformation.
	\smallskip
	\item $\frac{\partial}{\partial t}\bigr\vert_{t=1} \tilde\omega_{I,t} = d \tilde \lambda_I$ ($\Longrightarrow$ $\frac{\partial}{\partial t} \tilde \omega_{I,t} = \frac{\partial}{\partial t}  \tilde m_t^* \omega_I  =  \frac{1}{t} \, \tilde m_t^* \, d \tilde\lambda_{I} = d \tilde \lambda_{I,t}$ by naturality).
\end{enumerate}
\label{lem:primitive}
\end{lem}

\begin{proof}
(1) If $R_g(A,\varphi)=(A,\varphi)^g$ denotes the right action by unitary gauge transformations, then $DR_g(\dot A,\dot \varphi) = (g^{-1}\dot A g, g^{-1}\dot \varphi g)$. Further note that $L_{(A,\varphi)^g} g^{-1}\dot g g = g^{-1}(L_{A,\varphi}\dot g)g$, i.e.\ $\dot g((A,\varphi)^g) = -2 L_{(A,\varphi)^g}^{-1}(g^{-1}[ \varphi, \varphi^*]g) = -2 g^{-1}L_{A,\varphi}^{-1}([\varphi,\varphi^*]) g =g^{-1} \dot g(A,\varphi) g$. With this at hand we obtain
\begin{align*}
(R_g^*\tilde\lambda_I) (\dot A, \dot \varphi)  &= - \im\int_C \langle  \bar\partial_{A^g} g^{-1}\dot g g ,g^{-1}\dot A^{0,1} g \rangle + \langle g^{-1} \varphi g + [g^{-1}\varphi g, g^{-1} \dot g g], g^{-1}\dot \varphi g \rangle\\
&=- \im \int_C \langle g^{-1} (\bar\partial_{A} \dot g)g  ,g^{-1}\dot A^{0,1} g \rangle + \langle g^{-1} (\varphi  + [\varphi,  \dot g ] )g, g^{-1}\dot \varphi g \rangle = \tilde\lambda_I(\dot A,\dot \varphi)
\end{align*}
as claimed.

(2) We compute using $- i \ast L_{A,\varphi} =2(D'')^*D''$ on $\Omega^0(\mathfrak{sl}(E))$ 
\begin{align*}
\tilde\lambda_I(X_\gamma) &= - \im \int_C  \langle \bar\partial_A \dot g, \bar\partial_A \gamma \rangle + \langle [\varphi, \dot g] + \varphi , [\varphi, \gamma] \rangle   = - \im \int_C  \langle D'' \dot g, D'' \gamma \rangle  +\langle \varphi, [\varphi, \gamma] \rangle \\
& = - \im \int_C  \langle  (D'')^*D'' \dot g,  \gamma \rangle  +\langle \varphi, [\varphi, \gamma] \rangle = - \im \int_C \langle - \frac{i}{2} \ast L_{A,\varphi} \dot g, \gamma \rangle + \langle \varphi, [\varphi, \gamma] \rangle \\
& = - \im \int_C \langle i \ast [\varphi,\varphi^*] , \gamma\rangle + \langle \varphi, [\varphi, \gamma] \rangle = 0
\end{align*}
since  $(\ad \varphi)^* = - i \ast\ad \varphi^*$.

(3) We use the formula $d \tilde \lambda_I(X_1,X_2) = X_1(\tilde \lambda_I(X_2))-X_2(\tilde \lambda_I(X_1))-\tilde \lambda_I([X_1,X_2])$ to compute the exterior differential. To that end we think of $\tilde \lambda_I$ as defined on a whole neighborhood of $(A,\varphi)$ in the space of configurations. If $X_i = (\dot A_i, \dot \varphi_i)$, $i=1,2$ are constant vector fields, they will, in particular, commute and we obtain from \eqref{relprim}
\begin{align*}
	&d \tilde \lambda_I \bigl( (\dot A_1^{0,1}, \dot \varphi_1), (\dot A_2^{0,1}, \dot \varphi_2) \bigr) \\
	  =&  - \im  \int_C \langle  [\dot A_1^{0,1}, \dot g] + \bar\partial_A (D \dot g(\dot A_1, \dot \varphi_1)), \dot A_2^{0,1} \rangle   + \langle [ \dot \varphi_1, \dot g]+ [ \varphi, D\dot g(\dot A_1, \dot \varphi_1)], \dot \varphi_2 \rangle   +  \langle \dot \varphi_1, \dot \varphi_2 \rangle \\
	 & +\im \int_C \langle  [\dot A_2^{0,1}, \dot g] + \bar\partial_A (D \dot g(\dot A_2, \dot \varphi_2)), \dot A_1^{0,1} \rangle   + \langle [ \dot \varphi_2, \dot g]+ [ \varphi, D\dot g(\dot A_2, \dot \varphi_2)], \dot \varphi_1 \rangle   +  \langle \dot \varphi_2, \dot \varphi_1 \rangle 
\end{align*}
where $D \dot g$ is the linearization of the map $(A,\varphi) \mapsto \dot g(A,\varphi) = -2 L_{A,\varphi}^{-1}([\varphi,\varphi^*])$ at $(A,\varphi)$.
 
 On the other hand, using \eqref{tresckahler} we get
\begin{align*}
& \frac{\partial}{\partial t}\Bigr\vert_{t=1} \tilde\omega_{I,t} \bigl( (\dot A_1^{0,1}, \dot \varphi_1), (\dot A_2^{0,1}, \dot \varphi_2) \bigr)\\
=&  - \im  \int_C \langle  [\dot A_1^{0,1}, \dot g], \dot A_2^{0,1} \rangle + \langle  \dot A_1^{0,1}, [ \dot A_2^{0,1},  \dot g] \rangle + \langle [ \dot \varphi_1, \dot g], \dot \varphi_2 \rangle + \langle \dot \varphi_1, [\dot \varphi_2, \dot g] \rangle  + 2 \langle \dot \varphi_1, \dot \varphi_2 \rangle
\end{align*}
Together these imply that
\begin{align*}
&\frac{\partial}{\partial t}\Bigr\vert_{t=1} \tilde\omega_{I,t} \bigl( (\dot A_1^{0,1}, \dot \varphi_1), (\dot A_2^{0,1}, \dot \varphi_2) \bigr)-d \tilde \lambda_I \bigl( (\dot A_1^{0,1}, \dot \varphi_1), (\dot A_2^{0,1}, \dot \varphi_2) \bigr)\\
=&  \im \int_C \langle D'' (D \dot g(\dot A_1, \dot \varphi_1)), (\dot A_2^{0,1},\dot\varphi_2) \rangle  - \im \int_C \langle D'' (D \dot g(\dot A_2, \dot \varphi_2)), (\dot A_1^{0,1},\dot\varphi_1) \rangle  \\
=&  \im \int_C \langle D \dot g(\dot A_1, \dot \varphi_1),  (D'')^*(\dot A_2^{0,1},\dot\varphi_2) \rangle  - \im \int_C \langle  D \dot g(\dot A_2, \dot \varphi_2), (D'')^*(\dot A_1^{0,1},\dot\varphi_1) \rangle  \\
=&  - \im \int_C \langle D \dot g(\dot A_1, \dot \varphi_1), i \ast ( \partial_A \dot A_2^{0,1} + [\varphi^*, \dot \varphi_2] ) \rangle  + \im \int_C \langle D \dot g(\dot A_2, \dot \varphi_2), i \ast (\partial_A \dot A_1^{0,1} + [\varphi^*, \dot \varphi_1] )\rangle  
\end{align*}
where we have used that $d_1^* = - i \ast (\partial_A +  \ad \varphi^*)$. 
Now if $(\dot A_i, \dot \varphi_i)$ are infinitesimal deformations, then  $2 \re (\partial_A \dot A_i^{0,1} + [\varphi^*, \dot \varphi_i]) = d_A\dot A_i + [\dot \varphi_i, \varphi^*] + [\varphi, \dot \varphi_i^*]=0$ (see Proposition 2.2 in \cite{Fredricksonasygeo}). This implies that $ i \ast (\partial_A \dot A_i^{0,1} + [\varphi^*, \dot \varphi_i])$ has values in skew-hermitian endomorphisms. Since $D \dot g(\dot A_i, \dot \varphi_i)$ is hermitian, this implies that the last two integrals are zero and we obtain $\frac{\partial}{\partial t}\Bigr\vert_{t=1} \tilde\omega_{I,t} = d \tilde \lambda_I $ when restricted to the space of solutions to the Hitchin equations.
\end{proof}

\begin{cor}
$\tilde \lambda_I$ descends to a 1-form $\lambda_I$ on $\mathcal M$ satisfying $\frac{\partial}{\partial t}  m_t^* \omega_I = \frac{1}{t} \, m_t^* \, d \lambda_I$. 
\end{cor}
\subsubsection{The parabolic case}

In this case $m_t : \mathcal{M}(\boldsymbol{\alpha}, \mathbf{m}) \to \mathcal{M}(\boldsymbol{\alpha}, t \mathbf{m}) $ is not a self-diffeomorphism of a moduli space with fixed complex masses. Nevertheless, we can still define $\tilde \lambda_I$ via formula \eqref{relprim} for $\dot g = -2 L_{A,\varphi}^{-1}( [\varphi,\varphi^*])$ and check the desired properties. Here we have to be slightly more careful about the function spaces which $\varphi$, $\dot \varphi$ and $\dot g$ live in and, in particular, invertibility of $L_{A,\varphi}$ between those.

Recall from Section \ref{sec:CFW} that $\varphi \in \mathcal{D}_\delta$ and $\dot \varphi \in L^2_{1,\delta}$ for $\delta >0$. Note that since $(\dot A, \dot \varphi)$ is tangent to $\mathcal{M}(\boldsymbol{\alpha}, \mathbf{m})$, a moduli space with fixed complex masses, 
 $\dot A \in L^2_{1, \delta}$ and $\dot \varphi \in L^2_{1,\delta} \subset \mathcal{D}_\delta$. Now $\varphi \in \mathcal{D}_\delta \subset L^2_{-\delta}$ and $\dot \varphi \in L^2_{1,\delta} \subset L^2_\delta$ implies that
\[
\int_C \langle \varphi, \dot \varphi \rangle < \infty.
\]
In order to investigate the regularity of the term $[\varphi, \varphi^*]$ we express  
\[
\varphi =\frac{1}{z} \begin{pmatrix}
 m & 0 \\ |z|^{\alpha^{(2)} - \alpha^{(1)}} c & -m	
 \end{pmatrix}
dz + O(1)
\]
in the unitary frame $\tilde e_i = |z|^{-\alpha^{(i)}}$ and compute
\begin{align*}
[\varphi, \varphi^*] &= \frac{1}{|z|^2} \begin{pmatrix}
 - |z|^{2(\alpha^{(2)} - \alpha^{(1)})}	|c|^2 & 2 |z|^{\alpha^{(2)}-\alpha^{(1)}} m \bar c\\
 2 |z|^{\alpha^{(2)}-\alpha^{(1)}} \bar m c  &|z|^{2(\alpha^{(2)} - \alpha^{(1)})}	|c|^2
 \end{pmatrix} dz \wedge d \bar z
+O(|z|^{-1})\\
&=  \begin{pmatrix}
 - e^{-2 \tau (\alpha^{(2)}- \alpha^{(1)})}	|c|^2& 2 e^{-\tau(\alpha^{(2)}-\alpha^{(1)})}  m \bar c\\
 2 e^{-\tau(\alpha^{(2)}-\alpha^{(1)})}  \bar m c & 	e^{-2 \tau (\alpha^{(2)} - \alpha^{(1)})} |c|^2
 \end{pmatrix} d w \wedge d \bar w
+O(e^{-\tau})
\end{align*}
where $z=e^w$ for $w = - \tau + i \theta$, $\tau = - \log r$. We conclude that $[\varphi, \varphi^*] \in L^2_\delta$. Alternatively this can be obtained from Proposition 3.13 in \cite{collier2024conformallimitsparabolicslnchiggs}. Now set 
\[
\mathcal{R}_\delta = \{ \dot g \in L^2_{2,-\delta} : \|d_A \dot g \|_{L^2_{1, \delta}} < \infty \}. 
\]
Then $\mathcal{R}_\delta = L^2_{2, \delta} \oplus \mathcal{H}_\delta$ for a finite-dimensional complement $\mathcal{H}_\delta$, see Proposition 3.3 in \cite{collier2024conformallimitsparabolicslnchiggs}. According to Proposition 3.24 in \cite{collier2024conformallimitsparabolicslnchiggs} the operator
\[
(D'')^*D'' : \mathcal{R}_\delta= L^2_{2, \delta} \oplus \mathcal{H}_\delta  \longrightarrow L^2_\delta
\]
is Fredholm of index $0$ and $\ker (D'')^*D'' = \ker D''$. By Lemma 3.25 in \cite{collier2024conformallimitsparabolicslnchiggs}  the kernel of $d_1$ is trivial on $\mathcal{R}_\delta$. This implies that 
\[
- i \ast L_{A,\varphi} = 2 d_1^* d_1 : \mathcal{R}_\delta \longrightarrow L^2_\delta
\]
is invertible, hence there exists a unique $\dot g \in \mathcal{R}_\delta$ satisfying $L_{A,\varphi}\dot g = -2 [\varphi,\varphi^*]$. Then $\bar\partial_A \dot g$ and $[\varphi, \dot g]$ both lie in $L^2_{\delta}$ by the proof of Proposition 3.24 in \cite{collier2024conformallimitsparabolicslnchiggs}. It follows that
\[
\int_C \langle \bar\partial_A \dot g, \dot A^{0,1} \rangle < \infty \quad \text{and} \quad \int_C \langle [\varphi, \dot g] , \dot \varphi \rangle < \infty.
\]
We conclude that the integral expressions defining $\tilde \lambda_I$ in the following definition are finite and hence $\tilde \lambda_I$ is well-defined.

\begin{defn}
	For $(A,\varphi) \in \calB_\delta(\bsa)$, $(\dot A, \dot \varphi) \in L^2_{1,\delta} \oplus L^2_{1,\delta}$ we set
	\[
	\tilde\lambda_I (\dot A, \dot \varphi) = - \im \int_C \langle \bar\partial_A \dot g, \dot A^{0,1} \rangle + \langle \varphi + [\varphi, \dot g] , \dot \varphi \rangle.
	\]
\end{defn}

As in the case without punctures we need to prove that $\tilde\lambda_I$ descends to the moduli space to provide a relative primitive. 
\begin{prop}\label{prop:dlambda}
$\tilde\lambda_I$ descends to a 1-form $\lambda_I$ on $\mathcal{M}(\boldsymbol{\alpha}, \mathbf{m})$ such that with  $\lambda_{I,t} = m_t^* \frac{\lambda_I}{t} \in \Omega^1(\mathcal{M}(\boldsymbol{\alpha}, \mathbf{m}))$ the relation
\[
\frac{\partial}{\partial t} \omega_{I,t}= \de \lambda_{I,t}
\]
holds in $\Omega^2(\mathcal{M}(\boldsymbol{\alpha}, \mathbf{m}))$ for all $t>0$. In particular, $\omega_{I,t}$ and $\omega_I$ are cohomologous for all $t>0$. 
\end{prop}

\begin{proof}
We need to establish properties (1), (2) and (3) of Lemma \ref{lem:primitive} in the parabolic setting. The proof of (1) goes through without essential change. In the proof of properties (2) and (3) we in addition need to justify integration by parts. In (2) this amounts to checking that 
\[
\int_C \langle D'' \dot g, D'' \gamma \rangle = \int_C \langle (D'')^* D'' \dot g, \gamma \rangle
\]
for $\dot g, \gamma \in \mathcal{R}_\delta$, whereas in (3) we need to check that
\[
\int_C \langle D'' D \dot g (\dot A_1, \dot \varphi_1), (\dot A_2, \dot \varphi_2) \rangle = \int_C \langle D \dot g (\dot A_1, \dot \varphi_1), (D'')^*(\dot A_2, \dot \varphi_2) \rangle 
\]
for $(\dot A_i, \dot \varphi_i) \in L^2_{1, \delta} \oplus L^2_{1, \delta}$. Note that $D'' \dot g \in L^2_{1, \delta}$ and $D \dot g (\dot A_1, \dot \varphi_1) \in \mathcal{R}^0_\delta$, so both can be dealt with using the same argument as in the proof of Proposition 3.24 in \cite{collier2024conformallimitsparabolicslnchiggs}.
\end{proof}

\begin{rem}
Since $\omega_I$ is clearly invariant under the $S^1$-action, the cohomology class of $\omega_I$ is invariant under the full $\C^\times$-action, i.e.\ $[m_\lambda^* \omega_I] = [\omega_I]$ for all $\lambda \in \C^{\times}$.
\end{rem}

\begin{cor}\label{cor:const}
For $[S] \in H_2(\cM(\bsa, \bfm), \Z)$  the value $\int_S \omega_{I,t}$ does not depend on $t \in \R_+$.
\end{cor}

\begin{proof}
By Proposition \ref{prop:dlambda},
the relation 
\[
\frac{\partial}{\partial t} \omega_{I,t}= \de \lambda_{I,t}
\]
holds in $\Omega^2(\mathcal{M}(\boldsymbol{\alpha}, \mathbf{m}))$. Consequently, for $[S] \in H_2(\cM(\bsa, \bfm), \Z)$
\[\frac{\partial}{\partial t}\int_{S^2} \omega_{I,t}= \int_{S^2} \de \lambda_{I,t}= 0,\]
i.e.\ $t \mapsto \int_{S} \omega_{I,t}$ is a constant function on $\R_+$. 
\end{proof}

Recall that if $(\bsa, \bfm)$ is generic, then $(\bsa, \mathbf{0})$ is generic and so the whole ray $(\bsa, t\bfm)$ for $t \in \R_{\geq 0}$ consists of generic parameters.  We can thus extend the homology class $[S] \in H_2(\cM(\bsa, \bfm), \Z)$ to a parallel family of classes $[S_t]  \in H_2(\cM(\bsa, t\bfm), \Z)$ for $t \in \R_{\geq 0}$.

\begin{prop}\label{prop:sameperiods}
The periods of $\omega_I$ on $\calM(\bsa,\bfm)$ coincide with the periods of $\omega_I$ on $\calM(\bsa,\mathbf{0})$. More precisely, for $[S] \in H_2(\cM(\bsa, \bfm), \Z)$ 
 \[
 \int_{S_{t=0}} \omega_I = \int_{S_{t=1}} \omega_I
 \]
where $[S_t]$ is the parallel extension in the local system $\mathcal H$. In particular, the value of $\int_{S^2_J} \omega_I$ on $\calM(\bsa,\bfm)$ is precisely given by the value when $t=0$ which appears in the previous section.
\end{prop}

\begin{proof}
We first observe that $[S_t] =[m_t(S)]$ for all $t>0$ by continuity, since $[S_t]=[S]$ for $t=1$ and the fiber of the local system $\mathcal{H}$ is discrete.  Then, by naturality
\[
\int_{S_t} \omega_I = \int_{m_t(S)} \omega_I = \int_{S} m_t^* \, \omega_I = \int_{S} \omega_{I,t}  
\]
for all $t >0$ and the right-hand side is constant in $t$ by Corollary \ref{cor:const}. 
Since $\omega_I$ extends continuously as  $t \to 0$ by Theorem 3.33 in \cite{collier2024conformallimitsparabolicslnchiggs} and so does the family $[S_t]$ of homology classes, we obtain 
\[
\int_{S_{t=0}} \omega_I = \lim_{t \to 0} \int_{S_t} \omega_I = \int_{S_{t=1}} \omega_I
\]
as claimed.
\end{proof}

\begin{rem}
Compare this to the complex symplectic form
\[
\Omega_I \bigl( (\dot A_1^{0,1}, \dot \varphi_1), (\dot A_2^{0,1}, \dot \varphi_2) \bigr) = 2i \int_C \Tr (\dot \varphi_2 \wedge \dot A_1^{0,1} - \dot \varphi_1\dot \wedge A_2^{0,1})
\]
for which one has $\Omega_I =  d \Lambda_I$ with
\[
\Lambda_I(\dot A^{0,1}, \dot \varphi) = -2i \int_C \Tr \varphi \wedge\dot A^{0,1}.
\]
The finiteness of this integral is checked as above.

This form is clearly invariant and for an infinitesimal gauge transformation $\gamma$ and corresponding vertical vector field $X_\gamma=(d_A \gamma, [\varphi,\gamma])$ one has
\[
i \,\Lambda_I(X_\gamma) =  2\int_C \Tr (\varphi \wedge \bar\partial_A \gamma) = 2\int_C d  (\Tr \varphi\gamma  ) =  2\lim_{\varepsilon \to 0}\int_{\partial C_\varepsilon} \Tr \varphi \gamma
\]
where $C_\varepsilon = C \setminus \bigcup_i B_\varepsilon(p_i)$. Now with respect to the unitary frame $\tilde e_i = |z|^{-\alpha^{(i)}}e_i$
\[
\varphi =\frac{1}{z} \begin{pmatrix}
 m & 0 \\ |z|^{\alpha^{(2)} - \alpha^{(1)}} c & -m	
 \end{pmatrix}
dz + O(1) = \begin{pmatrix}
 m & 0 \\ e^{-\tau(\alpha^{(2)} - \alpha^{(1)})} c & -m 	
 \end{pmatrix}
dw + O(e^{-\delta z})
\]
and 
\[
\gamma = \begin{pmatrix}
 	i \alpha & b \\ - \bar b & - i \alpha
 \end{pmatrix} \in U_\delta = \{ \eta \in \mathcal{R}_\delta : \eta = - \eta^* \}
\]
one has
\[
\Tr \varphi \gamma = \frac{1}{z} ( 2 mi\alpha + |z|^{\alpha^{(2)} -\alpha^{(1)}} bc) dz + O(1) = (2m i \alpha + e^{-\tau (\alpha^{(2)} -\alpha^{(1)})} bc ) dw +  O(e^{-\delta z}).
\]
Now according to Proposition 3.3 in \cite{collier2024conformallimitsparabolicslnchiggs} $\lim_{\tau \to \infty} \gamma (\tau, \theta)$ exists. In particular, $\lim_{\tau \to \infty} \Tr \varphi \gamma = 2 m i \lim_{\tau \to \infty} \alpha(\tau, \theta)$. In particular, we see
\[
m = 0 \Longrightarrow \lim_{\varepsilon \to 0}\int_{\partial C_\varepsilon} \Tr \varphi \gamma = 0,
\]
i.e.\ for strongly parabolic Higgs bundles $\Lambda_I$ indeed descends to a primitive for $\Omega_I$ on $\mathcal{M}$. 

We will later see that for non-zero complex masses $\Omega_I$ is not exact at all but rather has a singular primitive with first order poles transverse to a divisor.
\end{rem}

\section{Integral of \texorpdfstring{$\Omega_I$}{the holomorphic symplectic form} is an intersection number}\label{sec:integralasintersectionnumber}

In this section, we prove the integral of $\Omega_I$ over the exterior spheres reduces to a computation of some intersection numbers.  Our main result is Corollary \ref{cor:intersectionnumbers}. We delay the computation of these intersection numbers to Section \ref{sec:TorelliOmegaI}.
Along the way, we prove a few other results of interest. 
\begin{itemize}
\item In Proposition \ref{prop:semiflatsymplectic}, we prove that the holomorphic symplectic form $\Omega_I$ agrees with a simpler ``semiflat holomorphic symplectic form.'' We note that while it is clear in the unpunctured case that $\Omega_I$ depends on the Higgs bundle $(\delbar_A, \varphi)$ but not on the hermitian metric $h$ solving the Hitchin equations, there is a potential contribution from $D$ that one needs to prove vanishes.
\item 
We use this to semiflat holomorphic form to prove that the tautological $1$-form is a primitive for $\Omega_I$ in Proposition \ref{prop:primitive}.
\end{itemize}

\begin{rem}[Importance of framing]
One might suppose that one does not need to feed harmonic representatives into $\Omega_I$. This is essentially correct.
However, we want to explain the following subtlety \emph{in the parabolic (or wild) setting, one must continually think about the framing}:
Given a Higgs bundle deformation $(\dot{A}^{0,1}, \dot{\varphi})$ one convenient way to find the harmonic representative $(\dot{A}^{0,1}-\delbar_A \dot{\nu}, \dot{\varphi} - [\varphi, \dot{\nu}])$ is to solve the PDE 
\begin{equation}\label{eq:ingaugetriple}
 \del_A^h \delbar_A \dot{\nu} - \del_A^h \dot{A}^{0,1} - \left[\varphi^{*_h}, \dot{\varphi} + [\dot{\nu}, \varphi]\right]=0
\end{equation}
in the appropriate space (see \cite{FMSW}).
In this perspective, one can think of  $\dot{\nu}$ is the unique  $\mathfrak{sl}(E)$-valued section defining one-parameter deformations of Hermitian metrics expressed in terms of an $\mathfrak{sl}(E)$-valued section $\dot{\nu}$ as
\begin{equation*}
 h_\eps(v,w) = h(\e^{\eps \dot{\nu}} v, \e^{\eps \dot{\nu}} w).
\end{equation*}
We can derive that 
the expression for the holomorphic symplectic form on the Hitchin moduli space is 
\begin{align*}
   & \Omega_I \left((\dot{A}^{0,1}_1, \dot{\varphi}_1), (\dot{A}^{0,1}_2, \dot{\varphi}_2)\right)\\  \nonumber
  &=2\I \int_{\CP^1 \setminus D} \Tr \left( (\dot{A}^{0,1}_1 - \delbar_A \dot{\nu}_1) \wedge (\dot{\varphi}_2 - [\varphi, \dot{\nu}_2]) -  
   (\dot{A}^{0,1}_2 - \delbar_A \dot{\nu}_2) \wedge (\dot{\varphi}_1 - [\varphi, \dot{\nu}_1])
  \right)\\
   &=2\I \int_{\CP^1 \setminus D} \Tr \left( (\dot{A}^{0,1}_1 \wedge \dot{\varphi}_2 
   -\dot{A}^{0,1}_2 \wedge \dot{\varphi}_1)
   - (\dot{A}^{0,1}_1 \wedge [\varphi, \dot{\nu}_2] - \delbar_A \dot{\nu}_2 \wedge \dot{\varphi}_1) \right.\\
   &  \qquad \left.+(\dot{A}^{0,1}_2\wedge [\varphi, \dot{\nu}_1] - \delbar_A \dot{\nu}_1 \wedge \dot{\varphi}_2) + (\delbar_A \dot{\nu}_1 \wedge [\varphi, \dot{\nu}_2] - \delbar_A \dot{\nu}_2 \wedge [\varphi, \dot{\nu}_1]) \right)\\
     &= 2\I \int_{\CP^1 \setminus D} \Tr \left( (\dot{A}^{0,1}_1 \wedge \dot{\varphi}_2 
   -\dot{A}^{0,1}_2 \wedge \dot{\varphi}_1)
   -   \left(\dot{\nu}_2  \left(\cancel{ \delbar_A \dot{\varphi}_1 + [\dot{A}^{0,1}_1,\varphi] } \right)- \delbar_A(\dot{\nu}_2 \dot{\varphi}_1)\right)\right.\\
   &  \qquad \left.+
 \left(\dot{\nu}_1 \left(\cancel{\delbar_A \dot{\varphi}_2 +  [\dot{A}^{0,1}_2, \varphi]} \right) - \delbar_A(\dot{A}^{0,1}_1 \dot{\varphi}_2)\right)
    + \delbar_A \left([\dot{\nu}_2, \dot{\nu}_1] \right) \wedge  \varphi  \right)\\
     &= 2\I \int_{\CP^1 \setminus D} \Tr \left( (\dot{A}^{0,1}_1 \wedge \dot{\varphi}_2 
   -\dot{A}^{0,1}_2 \wedge \dot{\varphi}_1)
   + \delbar_A\left(\dot{\nu}_2 \dot{\varphi}_1- \dot{\nu}_1 \dot{\varphi}_2+ [\dot{\nu}_2, \dot{\nu}_1]  \varphi \right) \right)
  \end{align*}
  In the last line we used that $\delbar_A \varphi$ is zero on $\CP^1 \setminus D$. The crossed out terms vanish since $\delbar_A \dot{\varphi}_i +[\dot{A}^{0,1}_i, \varphi]=0$ on $\CP^1 \setminus D$ since it is the linearization of $\delbar_A \varphi=0$. 
  
The framing shows up in the fact that this correction term can be non-zero. It is certainly zero if $(\dot{A}^{0,1}, \dot{\varphi})$ are in the hyperk\"ahler configuration space of allowable, i.e. \emph{framed}, deformations $\mathcal{A}$ used to construct the moduli space as a hyperK\"ahler quotient. Otherwise,  the boundary term can contribute.
\end{rem}

\subsection{Darboux coordinates for \texorpdfstring{$\Omega_I$}{the holomorphic symplectic form}}

 We recall the construction of the semiflat metric on $\cM'$ according to Gaiotto, Moore and Neitzke. For $b \in \cB'$ consider the spectral curve $\Sigma_b \subset \mathrm{Tot}(K_C(D))$ and the branched cover $\rho : \Sigma_b \to C$. Denote  by $\sigma$ the involution switching the sheets. Let $D' = \rho^{-1}(D)$ and $\Sigma_b' = \Sigma_b \setminus D'$ the punctured spectral curve. Note that $\Sigma_b' \subset \mathrm{Tot}(K_C)$ and hence carries the restriction of the tautological 1-form, denoted here by $\lambda$, which is meromorphic with simple poles at $D'$. Consider the local systems of lattices  $\Gamma_b = H_1(\Sigma_b, \Z)_-$ and $\Gamma_b' =H_1(\Sigma'_b , \Z)_-$, where $( \,\cdot\, )_-$ denotes the odd part under the involution $\sigma$. These fit into the short exact sequence
 \[
 0 \longrightarrow \Gamma^{\circ} \longrightarrow \Gamma' \longrightarrow \Gamma \longrightarrow 0 
 \]
 where $\Gamma^{\circ} \cong \bigoplus_{p \in D} \Z$ is generated by odd lifts of loops around the punctures. The lattice $\Gamma^{\circ}$ doesn't undergo any monodromy and is the kernel of the intersection pairing
 \[
 \langle \cdot \,  , \cdot \rangle : H_1(\Sigma'_b , \Z)_- \times H_1(\Sigma'_b , \Z)_- \to \Z.
 \]
 In particular, the intersection pairing gives rise to a non-degenerate pairing $\Gamma_b \times \Gamma_b \to \Z$.
 
 There is a natural special K\"ahler metric on $\cB'$ from its interpretation as a family of spectral curves $\{\Sigma_b\}_{b \in \cB'}$. Locally, choose a symplectic basis $\{\alpha_i(b), \beta_i(b)\}_{i=1}^N$ of the lattice $\Gamma_b = H_1(\Sigma_b, \Z)_-$ of rank $N$ with the intersection pairing, i.e.\ $\IP{\alpha_i, \beta_j} = \delta_{ij}$ and $\IP{\alpha_i, \alpha_j}=0$ and $\IP{\beta_i, \beta_j}=0$. There are special conjugate holomorphic coordinates $Z_{\alpha_i}(b), Z_{\beta_j}(b)$ defined by integrating the tautological $1$-form $\lambda$ over cycles of $\Sigma'_b \subset \mathrm{Tot}(K_C)$, more precisely $Z_{\gamma}(b)=\int_{\gamma(b)} \lambda$ for $\gamma \in H_1(\Sigma_b',\Z)_-$. If $\gamma$ is the local section of $\Gamma^{\circ}$ given by the odd lift of a loop around $p \in D$, then
 \[
 Z_\gamma(b) = \int_{\gamma(b)} \lambda = \pm  4 \pi i  \, m_p
 \]
 where $m_p \in \C$ is the complex mass attached to $p \in D$. In particular, $dZ_\gamma =0$ for any local section $\gamma$ of $\Gamma^{\circ}$. Hence, $dZ_\gamma$ is well-defined for any local section $\gamma$ of $\Gamma \cong \Gamma'/\Gamma^{\circ}$. Then the special K\"ahler metric on the base is given by 
  \[\omega_{\mathrm{sK}} = \sum_{i=1}^N \de Z_{\alpha_i} \wedge \de \overline{Z}_{\beta_i} + \de \overline{Z}_{\alpha_i} \wedge \de Z_{\beta_i}.\]
  More compactly\footnote{To relate this to the previous formula, note that $\langle\alpha^i(b), \beta^i(b)\rangle=1$ and $\langle\beta^i(b), \alpha^i(b)\rangle=-1$ and all other pairings vanish.}, if we take $\{\gamma_i\}$ to be a basis of the lattice $H_1(\Sigma_b, \Z)_-$ and $\{\gamma^i\}$ to be the dual basis of the dual lattice $H_1(\Sigma_b, \Z)^*_-$, we can write  \[\omega_{\mathrm{sK}}=\langle \de Z \wedge \de \overline{Z}\rangle,\]
  where $\langle \de Z \wedge \de \overline{Z}\rangle= \sum_{i,j}\langle\gamma^i, \gamma^j\rangle \de Z_{\gamma_i} \wedge \de \overline{Z}_{\gamma_j}$. 

 The cotangent bundle $T^*\cB'$ carries a canonical hyperK\"ahler metric \cite[Theorem 2.1]{freedsemiflat} which descends to the quotient $\cM' = T^* \cB'/\Gamma$. 
  Just as $Z_\gamma$ are coordinates on the base, $\theta_\gamma$ are coordinates on the fiber from 
  \[\mathrm{Hom}(H_1(\Sigma_b, \Z)_-, U(1)) \simeq \mathrm{Jac}(\Sigma_b). \]
  
 More precisely: Fiber coordinates $\theta_\gamma$ can be obtained as follows. Using the spectral correspondence the fiber $\cM'_b = \pi^{-1}(b)$ for $b \in \cB'$ is given by $\mathrm{Pic}_{d_L}(\Sigma_b)$ in the $\GL(2,\C)$-case and $\mathrm{Prym}_{d_L}(\Sigma_b)$ for $d_L=\deg \cE + 2(g-1+ \frac{n}{2})$ in the $\SL(2,\C)$-case. Both are torsors, over $\mathrm{Jac}(\Sigma_b)$ in the first case and $\mathrm{Prym}(\Sigma_b)$ in the second. If $[\dot \xi] \in H^{0,1}(\Sigma_b)$ represents a tangent vector to $\mathrm{Pic}_{d_L}(\Sigma_b)$, then
\[
 d \theta_\gamma ( \dot{A}^{0,1}, \dot \varphi) = \int_\gamma \im (\dot \xi - \bar\partial \dot \nu)
 \]
 where $\dot \xi - \bar\partial \dot \nu$ is the harmonic representative of $[\dot \xi]$, i.e.\ $\dot \nu \in C^\infty(\Sigma_b,\C)$ solves $\partial\bar\partial \dot \nu = \partial\dot \xi$. The same reasoning applies if $[\dot \xi] \in H^{0,1}(\Sigma_b)_-$ is tangent to $\mathrm{Prym}_{d_L}(\Sigma_b)$ in the $\SL(2,\C)$-case. Finally, the coordinates $\theta_\gamma$ are obtained by integrating the closed 1-forms $d \theta_\gamma$. 
 \medskip
 
 The resulting holomorphic symplectic form on $\cM'$ can be written 
 \[
 \Omega_I^{\semif}= \langle \de Z \wedge \de \theta \rangle = \sum_{i=1}^N d Z_{\alpha_i}\wedge d \theta _{\beta_i} - d Z_{\beta_i}\wedge d \theta_{\alpha_i}
 \]
where $\{\alpha_i, \beta_i\}_{i=1}^N$ is a local symplectic basis of  $\Gamma_b = H_1(\Sigma_b, \Z)_-$.
\begin{prop}\label{prop:semiflatsymplectic} 
$\Omega_I = -4 \,\Omega_I^{\semif}$. 
\end{prop}

\begin{proof}
 If $v_i = (\dot A^{0,1}_i, \dot \varphi_i)$ are infinitesimal deformations of the parabolic Higgs bundle $(\bar\partial_A,\varphi)$, i.e.\ $\dot A^{0,1}_i \in \Omega^{0,1}(\mathfrak{sl}(E))$ are nonsingular, $\dot\varphi_i \in \Omega^{1,0}(\mathfrak{sl}(E))$ have at most simple poles at $D$, and satisfy $\bar\partial_A \dot\varphi_i + [\dot{A}^{0,1}_i,\varphi]=0$ for $i=1,2$, then 
\begin{align*}
\Omega_I \bigl( (\dot{A}^{0,1}_1, \dot\varphi_1), (\dot{A}^{0,1}_2, \dot\varphi_2) \bigr) &= 2i \int_C \Tr \bigl(\dot{A}^{0,1}_1 \wedge \dot \varphi_2 - \dot{A}^{0,1}_2 \wedge \dot \varphi_1 \bigr) \\
&= i \int_{\Sigma_b} \Tr \bigl(\pi^*\dot{A}^{0,1}_1 \wedge \pi^*\dot \varphi_2 - \pi^*\dot{A}^{0,1}_2 \wedge \pi^*\dot \varphi_1 \bigr)   \\
&= 2i \int_{\Sigma_b} \dot\xi_1 \wedge \dot \tau_2 - \dot\xi_2 \wedge \dot \tau_1 = 2i \int_{\Sigma_b} \dot \tau_1\wedge  \dot\xi_2 -  \dot \tau_2 \wedge \dot\xi_1
\end{align*}
where $\dot\tau_i$ denotes the variation of the tautological 1-form $\tau$ on $\Sigma_b$. Note that $\dot\tau_i$ is actually nonsingular on $\Sigma_b$, which can be seen as follows: If $b \in \mathcal{B}$ is represented by $q \in H^0(K_C(D)^2)$, then differentiating the relation $\tau^2 = q$ yields $2 \tau \dot\tau_i = \dot q_i$, where since the second order residue of $q$ is determined by the complex masses, $\dot q_i \in H^0(K_C(D))$. This means that $\dot\tau_i = \frac 12 \frac{\dot q_i}{\tau} \in H^0(K_{\Sigma_b})$.

In the last integral above we may as before replace $\dot\xi_i$ by the harmonic representative $\dot \xi_i - \bar\partial \dot \nu_i$ of $[\dot \xi_i]$ where $\dot \nu_i \in C^\infty(\Sigma_b,\C)$ solves $\partial\bar\partial \dot \nu_i = \partial\dot \xi_i$. Recall the Riemann bilinear identity
\[
\int_S \eta \wedge \xi  = \sum_{j=1}^g \left( \int_{\alpha_j} \eta  \int_{\beta_j} \xi - \int_{\alpha_j} \xi  \int_{\beta_j} \eta \right)
\]
for closed 1-forms $\eta, \xi$ on the Riemann surface $S$ of genus $g$. In particular, if $\eta \in \mathcal{H}^{1,0}$ is holomorphic and $\xi \in\mathcal{H}^{0,1}$ is anti-holomorphic, then $\eta \wedge \bar \xi = 0$ and using $2i \im \xi = \xi - \bar \xi$
\begin{align*}
	\int_S \eta \wedge \xi = 2i \int_S \eta \wedge \im \xi = 2i\sum_{j=1}^g \left( \int_{\alpha_j} \eta \int_{\beta_j} \im \xi - \int_{\alpha_j} \im \xi  \int_{\beta_j} \eta \right).
\end{align*}
Then
\begin{align*}
 2i \int_{\Sigma_b} \dot \tau_1\wedge  \dot\xi_2 &= -4 \sum_{j=1}^g \left( \int_{\alpha_j} \dot\tau_1 \int_{\beta_j} \im  (\dot\xi_2 - \bar\partial\dot\nu_2)       - \int_{\alpha_j} \im (\dot\xi_2 - \bar\partial\dot\nu_2)  \int_{\beta_j} \dot\tau_1 \right)\\
 & = -4\sum_{j=1}^g \left( dZ_{\alpha_j}(v_1)d\theta_{\beta_j}(v_2) - d\theta_{\alpha_j}(v_2)dZ_{\beta_j}(v_1)   \right)
\end{align*}
and similarly
\[
2i \int_{\Sigma_b} \dot \tau_2\wedge  \dot\xi_1 = -4 \sum_{j=1}^g \left( dZ_{\alpha_j}(v_2)d\theta_{\beta_j}(v_1) - d\theta_{\alpha_j}(v_1)dZ_{\beta_j}(v_2)   \right).
 \]
Altogether we obtain 
\begin{align*}
\Omega_I = -4 \sum_{j=1}^g d Z_{\alpha_j}\wedge d \theta _{\beta_j} - d Z_{\beta_j}\wedge d \theta_{\alpha_j} = -4 \, \Omega_I^{\semif}
\end{align*}
as claimed. 
\end{proof}

\begin{rem}
In contrast to $\omega_I$, the complex symplectic form $\Omega_I$ is homogeneous of degree 1 with respect to the $\C^\times$-action, i.e.\ 
$m_\lambda^* \Omega_I = \lambda \Omega_I$ for $\lambda \in \C^\times$. In particular, 	$m_t^*\Omega_I = t \Omega_I$ for $t \in \R_+$. \end{rem}

If $\Sigma_b = T^2_\tau$ is an elliptic curve  one has $H_1(\Sigma_b, \Z)_- = H_1(\Sigma_b, \Z)$. Furthermore, in this case $\mathrm{Prym}_{d_L}(\Sigma_b) = \mathrm{Pic}_{d_L}(\Sigma_b)$. This gives the following:
\begin{cor}\label{cplx_symp_ell}
If 	$\Sigma_b = T^2_\tau$ is an elliptic curve and $\{\gamma, \delta\}$ a local symplectic basis of $H_1(\Sigma_b, \Z)$, then
\[
\Omega_ I = -4 (dZ_\gamma \wedge d\theta_\delta - dZ_\delta \wedge d\theta_\gamma).
\]
\end{cor}

\subsection{A singular primitive for \texorpdfstring{$\Omega_I$}{the holomorphic symplectic form}} 
We show that the the tautological 1-form $\tau$ on $K_C(D)$ provides a primitive for the complex symplectic form $\Omega_I$ on  the complement of a divisor. The components of this divisor consist of eight sections. These will be called polar sections since the primitive has first order poles transverse to the divisor. We determine the residues of $\tau$ transverse to the polar sections.

Recall that for $q$ not in the discriminant locus the spectral curve $\Sigma_q$ is the non-singular curve given by
\[
\Sigma_q = \{ \lambda \in K_C(D) : \lambda^2 = q(\pi(\lambda))\} 
\]
where $q \in H^0(C, K_C(D)^2)$ and $\pi : K_C(D) \to C$ the projection. With $\tau \in H^0(K_C(D), \pi^* K_C(D))$ denoting the tautological section this means that $\Sigma_q$ is simply given as the zero locus of the section $\tau^2 - \pi^* q \in H^0(K_C(D), \pi^* (K_C(D)^2)$. The tautological section gives rise to a meromorphic 1-form on $K_C(D)$, the so-called tautological 1-form. We will denote the tautological 1-form again by $\tau$ as well as its restriction to $\Sigma_q$.

\begin{prop}\label{prop:primitive}
Let $\tau$ be the tautological $1$-form, and for $p \in D$ and choice of $\pm$, let sections $\sigma_p^{\pm}: \mathcal{B} \to \mathcal{M}$ be defined (see Definition \ref{def:polarsections}) as the point in $\pi^{-1}(p) \subset \Sigma_b$ such that $\tau$ has residue $\pm m_p$. Then on the complement of these sections, 
\[ \Omega_I = i \de \tau. \]
\end{prop}
The rest of this section contains definitions, lemmas, and corollaries necessary to prove this.

In the case at hand, i.e.\ $C = \CP^1$ and $D=\{0,1,p_0,\infty\}$ we will introduce local coordinates on $K_C(D)$ and derive coordinate expressions for $\tau$. On $U =\CP^1\setminus \{\infty\}$ we pick the standard coordinate $z \in \C$ . Then $K_C(D)$ is trivialized over $U$ by the local section $\sigma_U=\frac{dz}{z(z-1)(z-p_0)}$. If $\tilde w$ is the fiber coordinate relative to this trivialization, then
\[
\tau = \frac{\tilde{w} \, dz}{z(z-1)(z-p_0)}
\]

Over $V = \CP^1 \setminus \{0\}$ with local coordinate $\hat z = z^{-1}$ we trivialize $K_C(D)$ by the local section $\sigma_V = \frac{d \hat z}{\hat z}$. If $\hat{\tilde w}$ is the fiber coordinate relative to this trivialization, then
 \[
\tau = \frac{\hat{\tilde{w}} \, d \hat z}{\hat z}
\]
and the change of coordinates is given by $(z, \tilde w) \mapsto  (z^{-1}, - \frac{\tilde w}{(z-1)(z-p_0)})$. These expressions are compatible with the standard coordinate representation of $\tau$ on $K_C$, where the trivializing sections are given by $dz$ and $d \hat{z}$. If $w$ and $\hat{w}$ are the corresponding fiber coordinates, then the change of coordinates is given by $(z,w) \mapsto (z^{-1}, -z^2w)$, and  we have
\[
 w dz = \frac{\tilde{w} \, dz}{z(z-1)(z-p_0)} \quad \text{and} \quad \hat{w} d \hat z =  \frac{\hat{\tilde{w}} \, d \hat z}{\hat z}
\]
for $\tilde w = w z(z-1)(z-p_0)$ and $\hat{\tilde{w}} = \hat{w}\hat{z} = - wz$.

\medskip

Since $\Sigma_q$ is an elliptic curve, $K_{\Sigma_q}$ is  trivialized by a holomorphic section $\omega$ which is unique up to scale. We wish to compute $g = \tau / \omega$ for the tautological 1-form $\tau$ on the spectral curve $\Sigma_q$ for $q \in H^0(K_C(D)^2)$, $C= \CP^1$ and $D= \{0,1, p_0, \infty\}$. In order to do so we fix a choice of $\omega$ as follows.

\begin{lem}\label{invariant_1-form}
If $q$ does not lie in the discriminant locus and $m_\infty \neq0$, then the 1-form 	
\[
\omega =  \frac{dz}{ \tilde w} = \frac{dz}{ z (z-1)(z-p_0)w}
\]
defines a holomorphic trivialization of $K_{\Sigma_q}$.
\end{lem}

\begin{proof}
Using the equation $\tilde w^2 = \beta z( z-1)(z-p_0) + f(z) =: \tilde f(z)$ we compute
\[
2 \tilde w d \tilde w = \Bigl( \beta ((z-1)(z-p_0)  +z (z-p_0) + z (z-1)) + \frac{\partial f}{\partial z} \Bigr) dz = \frac{\partial \tilde f}{\partial z}dz
\]
and hence
\[
\omega = \frac{dz}{2 \tilde w} = \frac{ d \tilde w}{\frac{\partial \tilde f}{\partial z}},
\] 
which implies (since $\tilde f$ does not have double zeros by assumption) that $\omega$ is nonsingular at the zeros of $\tilde w$ (resp.\ $\tilde f$).

Near infinity, we introduce coordinates $\hat z = z^{-1}$ and $\hat{\tilde w} = - \frac{\tilde w}{(z-1)(z-p_0)} = - wz$ as above. Then the equation
\[
\hat{\tilde{w}}^2 = \frac{\beta \hat z (1 - \hat z)(1 - \hat z p_0) + \hat z^4 f}{(1-\hat z)^2(1 - \hat z p_0)^2}
\]
shows that $\lim_{\hat z \to 0} \hat{\tilde w}^2 = m_\infty^2 \neq 0$ and therefore the expression
\[
\omega = \frac{dz}{\tilde w} = - \frac{d \hat z}{  (1- \hat z)(1 - \hat z p_0) \hat{\tilde w}}
\]
is nonsingular at $\hat z =0$.
\end{proof}

\begin{lem}\label{residues}
The tautological 1-form is given by $\tau = g \omega$ for
\[
g(z) = \frac{f(z) + \beta z(z-1)(z-p_0)}{z(z-1)(z-p_0)}
\]	
and in particular has simple poles at $\pi^{-1}(D)$. The residues of $\tau$ are given by the complex masses
\[
\Res_{z=0} \tau = \pm m_0, \quad \Res_{z=1} \tau = \pm m_1, \quad \Res_{z=p_0} \tau = \pm m_{p_0}, \quad \Res_{z=\infty} \tau = \pm m_\infty. 
\]
\end{lem}

\begin{proof}
The trivializing differential is given by
\[
\omega =  \frac{dz}{ \tilde w} = \frac{dz}{ z (z-1)(z-p_0)w}
\]	
and the tautological 1-form by $\tau = w dz$. Hence,
\[
\tau/ \omega = w^2z(z-1)(z-p_0) = \frac{f(z) + \beta z(z-1)(z-p_0)}{z(z-1)(z-p_0)}
\]
using the equation
\begin{equation}\label{quadratic}
w^2  = \frac{f(z) + \beta z( z-1)(z-p_0)}{z^2(z-1)^2(z-p_0)^2}. 
\end{equation}
The representation \[
\tau = \frac{ \tilde w dz}{z(z-1)(z-p_0)} = - \hat{\tilde w} \frac{d \hat z}{\hat z}
\]
in coordinates $(z,\tilde w)$ resp.\ $(\hat z, \hat{\tilde w})$ (cf.\ the proof of Lemma \ref{invariant_1-form}) 
 shows that
 \[
 \Res_{z=0} \tau = \pm \frac{\sqrt{f(0)}}{p_0} = \pm m_0 , \quad  \Res_{z=1} \tau = \pm \frac{\sqrt{f(1)}}{1-p_0}=\pm m_1 , \quad \Res_{z=p_0} \tau = \frac{\sqrt{f(p_0)}}{p_0(p_0-1)}=\pm m_{p_0} 
 \]
 and 
 \[
 \Res_{z=\infty} \tau = - \lim_{\hat z \to 0} \hat{\tilde w} = \pm m_\infty.
 \]
 This concludes the computation of the residues.
\end{proof}

Write $D = \{p_1, p_2, p_3, p_4\}$. For $q \in \mathcal B$ and $\pi : \Sigma_q \to C$ we write $\pi^{-1}(p_i)=\{p_i^+,p_i^-\}$, where
$p_i^\pm$ are determined by
\[
\operatorname{Res}_{p_i^+ }\tau =  m_i  \quad \text{and} \quad \operatorname{Res}_{p_i^-} \tau = -m_i .
\]
In the generic case when all the complex masses $m_i$ are different from zero, this determines sections $\sigma_i^+ : \mathcal B \to \mathcal M$ and $\sigma_i^- : \mathcal B \to \mathcal M$
by demanding that $\sigma_i^{\pm}(q)= p_i^{\pm}$ for $q \in \mathcal B$.
\begin{defn}\label{def:polarsections}
The sections $\sigma_i^+ : \mathcal B \to \mathcal M$ ($i=1, \ldots, 4$) are called \emph{polar sections}. Setting $\sigma_i = \sigma_i^+ \cup \sigma_i^-$ and $D_\infty =\sigma_1 \cup \sigma_2 \cup \sigma_3 \cup \sigma_4$ we will call $D_\infty$ \emph{polar divisor}.
\end{defn}

A priori $\tau$ is a fiberwise 1-form on the elliptic fibration $\mathcal M \to \mathcal B$. The representation $\tau = w dz$ determines an extension to a 1-form on  $\mathcal M$ which is vertical with respect to coordinates $(\beta, z,w)$. The same applies to $\omega$. While $\omega$ is holomorphic on $\calM$, $\tau$ is meromorphic with poles transverse to the polar divisor with residues as in Lemma \ref{residues}.

\begin{lem}
$ d \tau = d\beta \wedge \omega$.	
\end{lem}
\begin{proof}
Using equation \eqref{quadratic} we compute
\[
2 w dw = \frac{\partial}{\partial z} \left\{\frac{f(z) + \beta z( z-1)(z-p_0)}{z^2(z-1)^2(z-p_0)^2}\right\}dz  + \frac{d\beta}{z(z-1)(z-p_0)}
\]
so that 
\[
d \tau  = dw \wedge dz = \frac{d\beta \wedge dz}{2w z (z-1)(z-p_0)} = d \beta \wedge \omega
\]
as claimed.
\end{proof}

We claim that $i d \tau = \Omega_I$. Towards that end, consider special coordinates on $\mathcal B$
\[
Z_\gamma(q) = \int_\gamma \tau
\]
for $\gamma \in H_1(\Sigma_q,\Z)$. Then using $\dot \tau = \frac{1}{2} \frac{\dot q}{\tau}$  we obtain $\dot{Z}_\gamma = \frac{1}{2} \int_\gamma \frac{\dot q}{\tau}$, or equivalently
\[
dZ_\gamma = \frac{1}{2} \left( \int_\gamma \frac{ \frac{\partial q}{\partial \beta}}{\tau} \right) d\beta .
\]
With
\[
q =  \frac{f(z) + \beta z( z-1)(z-p_0)}{z^2(z-1)^2(z-p_0)^2}
\]
we obtain
\[
\frac{\partial q}{\partial \beta} =\frac{dz^2}{z(z-1)(z-p_0)}
\]
and together with $ \tau = w dz$
\[
\frac{ \frac{\partial q}{\partial \beta}}{\tau} = \frac{dz}{w z(z-1)(z-p_0)} \Longleftrightarrow \dot \tau = \frac{1}{2} \omega .
\]
Hence,
\[
d Z_\gamma = \frac{1}{2} \left( \int_\gamma \frac{dz}{w z(z-1)(z-p_0)} \right) d\beta\\
= \frac 12 \left(\int_\gamma \omega\right) d\beta
\]
which proves the following lemma.
\begin{lem}
$
d Z_\gamma = \frac 12\left(\int_\gamma \omega \right) d\beta.
$
\end{lem}

Recall from Corollary \ref{cplx_symp_ell}
\[
\Omega_ I =  -4 (dZ_\gamma \wedge d\theta_\delta - dZ_\delta \wedge d\theta_\gamma)
\]
if $\{\gamma, \delta\}$ is a symplectic basis of $H_1(\Sigma_q,\Z)$.

\begin{cor}
$
\Omega_I = -2 \left(d \beta \wedge \bigl( \left( \int_\gamma \omega  \right) d\theta_\delta - \left( \int_\delta \omega \right) d\theta_\gamma \bigr) \right). 
$	
\end{cor}

\begin{lem}
$\omega = 2i\left( \int_\gamma \omega \right )d\theta_\delta - 2i\left( \int_\delta \omega \right)d\theta_\gamma $.
\end{lem}

\begin{proof}
We identify $\Sigma_q$ with its Jacobian $\mathrm{Jac}(\Sigma_q) = H^0(K_{\Sigma_q})^*/ H_1(\Sigma_q; \Z)$ and use that via Serre duality $H^0(K_{\Sigma_q})^* \cong H^{0,1}(\Sigma_q)$. For $\dot \xi \in H^{0,1}(\Sigma_q)$ we obtain using the Serre duality pairing and the Riemann bilinear identity as in the proof of Proposition \ref{prop:semiflatsymplectic}
\begin{align*}
\omega(\dot\xi) &= \int_{\Sigma_q} \omega \wedge \dot \xi = \int_{\Sigma_q} \omega \wedge (\dot\xi - \bar\partial\dot\nu)\\ 
&= 2i \int_\gamma\omega  \int_\delta \im ( \dot\xi - \bar\partial\dot\nu) - 2i \int_\gamma \im( \dot\xi - \bar\partial\dot\nu) \int_\delta \omega\\
&= 2i \left( \int_\gamma \omega \right )d\theta_\delta(\dot\xi) - 2i \left( \int_\delta \omega \right)d\theta_\gamma(\dot \xi).
\end{align*}
Here we have used that $d\theta_\gamma(\dot \xi) = \int_\gamma \im ( \dot\xi - \bar\partial\dot\nu)$ for $\dot \nu \in C^\infty(\Sigma_q,\C)$ solving $\partial\bar\partial \dot \nu = \partial\dot \xi$.
 \end{proof}
 
 We finally obtain the desired relationship.
 
 \begin{cor}\label{cor:primitive}
$\Omega_I =  i \,d \beta \wedge \omega =  i \,d\tau$.
 \end{cor}

\subsection{Computing complex periods via intersection numbers}\label{sec:intersectionnumbers}
The aim of this section is to show that integrals of the complex symplectic form $\Omega_I = \omega_J + i \omega_K$ over the exterior spheres are linear functions of the complex masses $m_i$, and that the coefficients can be interpreted as intersection numbers (see Corollary \ref{cor:intersectionnumbers}). 

Let $S$ be a cycle in $\mathcal M$. Since $\sigma_j$ (resp.\ $\sigma_j^\pm$) are relative cycles (relative to the boundary at infinity, which is a torus bundle over the circle at infinity), the intersection product
\[
I : H_2(\mathcal{M}) \times H_2(\mathcal{M},\partial_\infty \mathcal{M}) \to \Z
\]
and hence the intersection numbers $I(S,\sigma_j)$ (resp.\ $I(S,\sigma_j^\pm)$) are defined. 
We mainly want to apply this to $S=S_i$ for one of the exterior spheres $S_i=\CP^1_i$ in the following.
\begin{cor}\label{cor:intersectionnumbers}
For any cycle $S$ we have
\[	
\int_{S} \Omega_I =  -2 \pi  \sum_{j=1}^4 m_j \bigl( I(S,\sigma_j^+)- I(S,\sigma_j^-) \bigr).
\]
\end{cor}

\begin{proof}
Make $S$ transverse to $\sigma_1 \cup \ldots \cup \sigma_4$ by an isotopy and decompose $\sigma_j = \sigma_j^+ \cup \sigma_j^-$ as above. Then $S$ intersects $\sigma_j^\pm$ in finitely many points. Since $\sigma_j^{\pm}$ is a holomorphic submanifold we can choose local holomorphic coordinates $(z,w)$ near a point of intersection such that locally $\sigma_j^{\pm} = \{ w=0\}$ and
\[
\tau = \pm dz \wedge \frac{ m_j dw}{w}.
\]
By a further isotopy, we can achieve that locally $S = \{z=0\}$ near an intersection point. Note that the intersection number is the sum of local intersection numbers, which are either $1$ or $-1$, depending on whether the orientation of $S$ agrees with the orientation of the local curve $\{z=0\}$ at the intersection point or disagrees. In the former case, Stokes' theorem gives $2 \pi $ times the residue as a local contribution, in the latter its negative (taking into account that $\Omega_I = i d \tau$).
\end{proof}

In homology, we may express $S_i = (\mathcal{H}_i +S_i) - \mathcal{H}_i$, where $\mathcal{H}_i + S_i$ and $\mathcal{H}_i$ are relative cycles with the same boundary in $H_1(\del_\infty \cM)$ (hence their difference defines an absolute class). Let 
\[
\tau_i^+ = \mathcal{P}_0(\bsa)_{q_p^+}  \cap \mathcal{M}(\boldsymbol{\alpha}, \mathbf{m}) \quad \text{and} \quad \tau_i^- = \overline{\mathcal{P}_0(\bsa)_{q_p^-}}  \cap \mathcal{M}(\boldsymbol{\alpha}, \mathbf{m}) 
\]
where $\mathcal{P}_0(\bsa)_{q_p^+} $ is the (closed) upward flow of the exterior fixed point and $\overline{\mathcal{P}_0(\bsa)_{q_p^-}}$ is the closure of the upward flow of the interior fixed point. Both $\tau_i^+$ and $\tau_i^-$ are relative cycles, just like $\sigma_j^+$ and $\sigma_j^-$. Note that  \[\overline{\mathcal{P}_0(\bsa)_{q_p^-}}  \cap \mathcal{M}(\boldsymbol{\alpha}, \mathbf{m}) \neq \overline{\mathcal{P}_0(\bsa)_{q_p^-}  \cap \mathcal{M}(\boldsymbol{\alpha}, \mathbf{m})},\]
e.g. see Figure \ref{fig:tau}. 
In particular, note that for $\mathbf{m}=\mathbf{0}$, the left set is homologous to $S_i + \tau_i^+$ while the right set is $S_i$;  $\tau_i^+$ is homologous to $\mathcal{H}_i$.

  \begin{figure}[!ht]
  \includegraphics[height=2.0in]{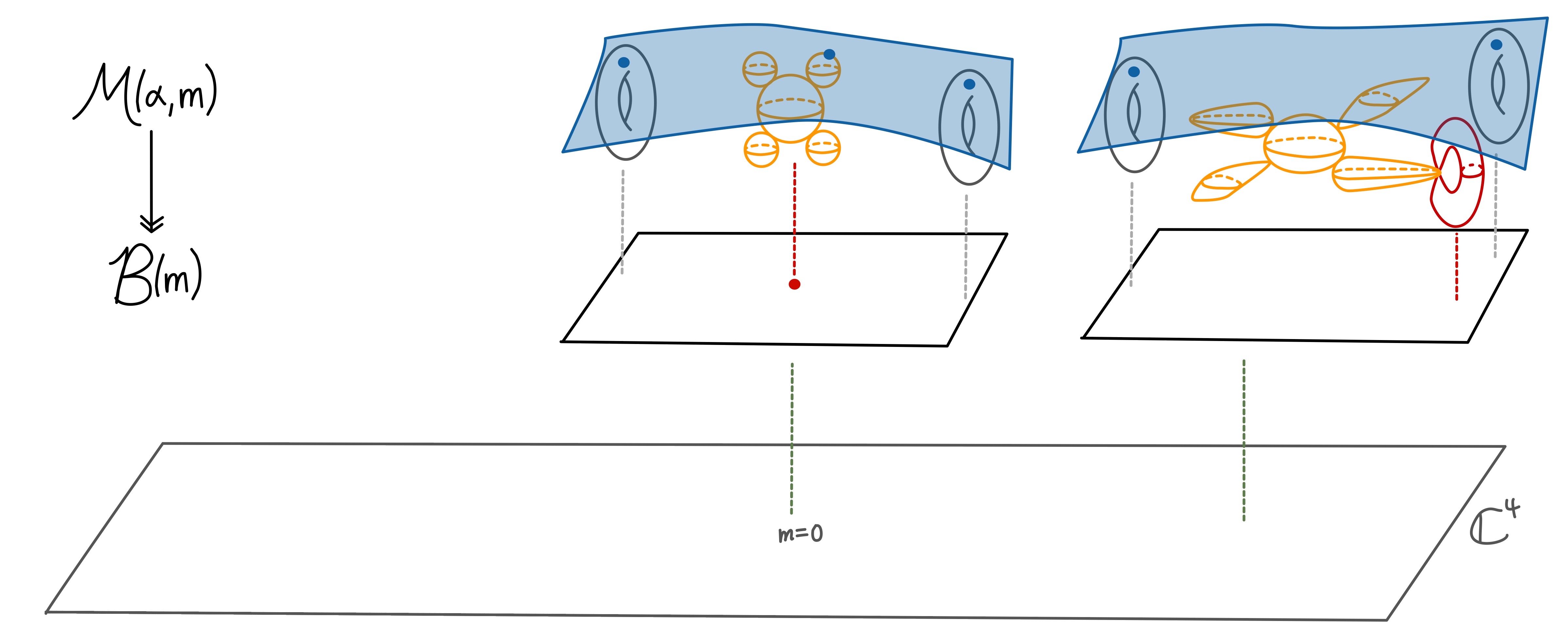}\\
    \includegraphics[height=2.0in]{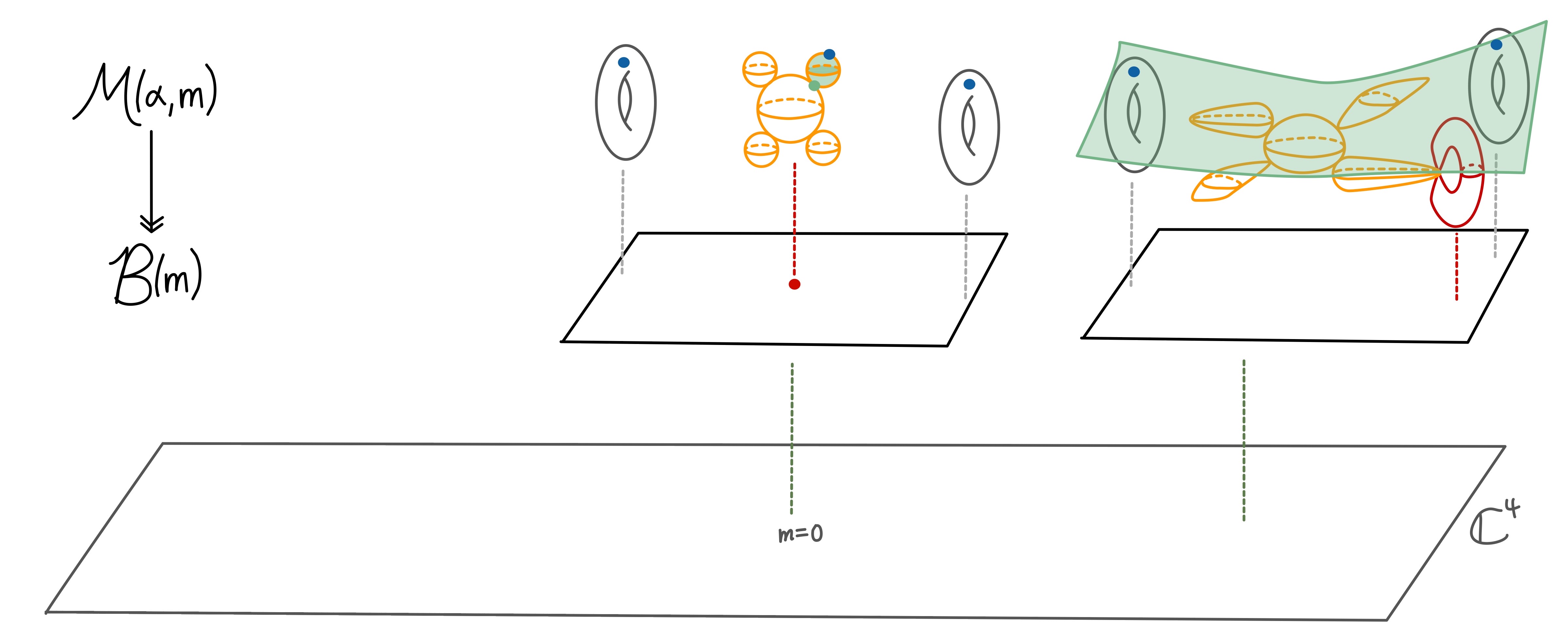}
  \caption{The top image is $\mathcal{P}_0(\bsa)_{q_p^+}  \cap \mathcal{M}(\boldsymbol{\alpha}, \mathbf{m})$, while the bottom is $\mathcal{P}_0(\bsa)_{q_p^-}  \cap \mathcal{M}(\boldsymbol{\alpha}, \mathbf{m})$. Note that $\overline{\mathcal{P}_0(\bsa)_{q_p^-}}  \cap \mathcal{M}(\boldsymbol{\alpha}, \mathbf{m}) \neq \overline{\mathcal{P}_0(\bsa)_{q_p^-}  \cap \mathcal{M}(\boldsymbol{\alpha}, \mathbf{m})}$. In particular, for $\mathbf{m}=\mathbf{0}$, $\tau_i^-$ is the union of the green and blue flows. \label{fig:tau}}
  \end{figure}  
  
\begin{rem}
In general, the intersection product of two relative classes $a, b \in H_2(\calM, \partial_\infty\calM)$ is not defined in a meaningful way, since the cup product of their Poincar\'e-dual cohomology classes lies in $H^4(\calM)= H_0(\calM, \partial_\infty \calM)  = 0$. However, looking at the long exact sequence of the pair $(\mathcal M, \partial_\infty \mathcal M)$
\[
\ldots  \to H_2(\partial_\infty \mathcal M)\overset{i_*}{\longrightarrow} H_2(\mathcal M) \overset{j_*}{\longrightarrow} H_2(\mathcal M, \partial_\infty \mathcal M) \overset{\partial_*}{\longrightarrow} H_1(\partial_\infty \mathcal M)  \to \ldots
\]
we can define the intersection product of  $a, b \in H_2(\calM,\partial_\infty\calM)$ with $\partial_* a = \partial_*b$ by setting
\[
I(a,b) = I(c,c)
\]
for any $c \in H_2(\calM)$ with $j_* c = a-b$. This is well-defined since any other choice differs from $c$ by an element in $\ker j_* =  \image i_*$ and the image of $H_2(\partial_\infty \mathcal M)$ in $H_2(\mathcal M)$ lies in the radical of the intersection form $I : H_2(\mathcal M) \times H_2(\mathcal M) \to \Z$. This can be seen by representing 2-dimensional homology classes by embedded surfaces, noting that any surface in $\partial_\infty \calM$ can be made disjoint by isotopy from any other surface in $\calM$. More concretely, this image is generated by the class of the smooth fiber $F=2S_0 + \sum_{i=1}^4 S_i$, which clearly lies in the radical. 
\end{rem}

The polar sections $\sigma_j^\pm$ are properly embedded disks which are disjoint in the interior of $\calM$. We claim that $\partial_\infty \sigma_j^+ = \partial_\infty\sigma_j^- \subset \partial_\infty \calM$. 

To see this, observe that for $q= \frac{f(z) + \beta z(z-1)(z-p_0)}{z^2(z-1)^2(z-p_0)^2} dz^2$  one has 
\[
\frac{q}{|\beta|} = \frac{\frac{1}{|\beta|}f(z) + \frac{\beta} {|\beta|}z(z-1)(z-p_0)}{z^2(z-1)^2(z-p_0)^2}dz^2 \underset{|\beta|\to \infty}{\longrightarrow} \frac{\beta} {|\beta|}\frac{dz^2}{z(z-1)(z-p_0)}
\]
and 
\[
p_j^\pm \underset{|\beta|\to \infty}{\longrightarrow}  \tilde{p_j}
\]
where $\tilde p_j$ is the unique preimage of $p_j$ in the spectral cover of the differential $\frac{\beta} {|\beta|}\frac{dz^2}{z(z-1)(z-p_0)}$. In particular, $\partial_\infty \calM$ is identified with the boundary at infinity of strongly parabolic moduli space for any choice of complex masses. Further, $\partial_\infty \sigma_j^+ = \partial_\infty\sigma_j^-  = \partial_\infty \mathcal{H}_j$ with respect to this identification. 

\medskip

Having established this claim, it follows that  $\Sigma_j = \sigma_j^+ - \sigma_j^-$ defines a spherical class in $H_2(\calM)$. We decompose  $S_i=(\mathcal{H}_i + S_i)-S_i$ and use that $ \mathcal{H}_i +S_i \sim \tau_i^- $ and $ \mathcal{H}_i \sim \tau_i^+ $ to obtain
\[
I(S_i,\sigma_j^+)- I(S_i,\sigma_j^-) =  I(S_i, \Sigma_j) = I(\mathcal{H}_i +S_i,\Sigma_j) - I(\mathcal{H}_i,\Sigma_j) = I (\tau_i^-, \Sigma_j) - I(\tau_i^+, \Sigma_j)
\]
which yields

\begin{cor}\label{cor:intersection_numbers} For the exterior spheres $S_i$ and the spheres $\Sigma_j = \sigma_j^+ - \sigma_j^-$ we have
\[
\int_{S_i} \Omega_I =  -2 \pi \sum_{j=1}^4 m_j  I(S_i, \Sigma_j)  =  2 \pi  \sum_{j=1}^4 m_j \bigl( I (\tau_i^+, \Sigma_j) - I(\tau_i^-, \Sigma_j)\bigr).
\]	
\end{cor}

\section{Interlude: Concrete description of \texorpdfstring{$\mathcal{M}$}{the Higgs bundle moduli space}}\label{sec:concrete}

In this section, we will explicitly write out the Higgs bundles associated to polar sections, so that we can compute the intersection numbers appearing in Corollary \ref{cor:intersectionnumbers} in Section \ref{sec:integralasintersectionnumber}. In fact, 
we give a detailed and explicit description of an open dense subset of the Hitchin moduli space.

\subsection{Moduli space of generic Higgs bundles}
\label{eq:genericHiggs}

We now turn to describing the moduli space of parabolic Higgs bundles. Generically, the Higgs field $\varphi$ and \emph{holomorphic} bundle 
$\cE$ determine the flag structure, and hence the parabolic Higgs bundle. This is true for all parabolic Higgs bundles in the regular locus
$\mathcal{M}'(\boldsymbol{\alpha}, \mathbf{m})=\mathrm{Hit}^{-1}(\mathcal{B}'(\mathbf{m}))$, the preimage of the regular locus under the Hitchin map; but it actually is true for 
a larger space $\mathcal{M}^0(\boldsymbol{\alpha}, \mathbf{m})=\mathrm{Hit}^{-1}(\mathcal{B}^0(\mathbf{m}))$\footnote{We note that if $q$ is a perfect square, then $\mathbf{m}$ satisfies $\sum (-1)^{e_i} m_i=0$ for some $e_i \in \{0, 1\}$. We note that if $q$ has a non-simple zero at $p_i \in D$, then $m_i = 0$.}, where
\begin{equation}
\mathcal{B}^{0}(\mathbf{m})=
\left\{q(\beta)
\in \mathcal{B}\Big| q(\beta)=f_{\mathbf{m}}(z) +\beta z(z-1)(z-p_0) \mbox{
\begin{minipage}{0.4\linewidth}
    is not a perfect square \\ \textbf{and} doesn't have a non-simple zero at some $p \in D$
  \end{minipage}
}\right\}.
\label{eq:B0}
\end{equation}
The following description is true in all chambers.  However, it is best-suited for the exterior chamber E1$_\infty$ and its adjacent interior chamber B1$_\infty$. In both, the $\C^\times$-action will be particularly simple. (See Remark \ref{rem:preference}.)
\begin{prop}[The moduli space of Higgs bundles]\label{prop:higgsbundles}
Given $\beta \in \cB^0(\mathbf{m})$, the fiber  $\mathrm{Hit}^{-1}(\beta)$ is identified with the projective curve $\{F(u,w,x) = 0\} \subset \mathbb P(\C^3)$, where
\begin{equation}\label{eq:projective}
F(u, w, x):=(f_{\mathbf{m}}(ux^{-1})-m_\infty^2(ux^{-1})^4)x^3  +\beta u(u-x)(u-p_0 x) - 2 m_\infty u^2 w - w^2 x.
\end{equation}
The locus $\cM^0(\mathbf{m}) = \mathrm{Hit}^{-1}(\mathcal{B}^0(\mathbf{m}))$ of the $SL(2,\C)$ moduli space on the four-punctured sphere is stratified by the holomorphic bundle type of $\cE$,
which is either $\cO(-2) \oplus \cO(-2)$ (``the big stratum'') or $\cO(-1) \oplus \cO(-3)$ (``the small stratum''). 

The big stratum is largely parameterized by $(\beta, u, w)$ which solve the affine \underline{cubic} equation (corresponding to $x=1$ in the 
projective curve) 
\begin{equation}\label{eq:cubic}
f_{\mathbf{m}}(u) + \beta u (u-1)(u-p_0) - (m_\infty u^2 + w)^2=0.
\end{equation} 
 The corresponding Higgs bundle is
 \begin{align}\label{eq:bigstrata}
 \cE &\simeq \cO(-2) \oplus \cO(-2)  \\ \nonumber
 \varphi_{\beta, u, w} &= \frac{\de z}{z(z-1)(z-p_0)}\begin{pmatrix}  -(m_\infty z^2 + w) & \frac{f(z) + \beta z (z-1)(z-p_0) - (m_\infty z^2 + w)^2}{z-u} \\  z-u & m_\infty z^2 + w \end{pmatrix},
\end{align}
where the flag $F_p$ is the eigenvector of $\mathrm{Res}\, \varphi|_p$ with associated eigenvalue $m_p \frac{\de z}{(z-p)}$ for each $p \in D$. (See Corollary \ref{cor:flags} below).

When $m_\infty \neq 0$, there is an additional point in the big stratum (corresponding to $(u:x:w)=(m_\infty: 0: \frac{f_3+\beta}{2})$) given by 
\begin{align} \label{eq:extrapoint}
  \cE &\simeq \cO(-2) \oplus \cO(-2)  \\ \nonumber
 \varphi_\beta &=  \frac{\de z}{z(z-1)(z-p_0)}\begin{pmatrix} -(m_\infty z^2 +
            \frac{f_3 + \beta}{2 m_\infty} z) & f(z) + \beta z(z-1)(z-p_0) 
            -\left(m_\infty z^2 +
            \frac{f_3 + \beta}{2 m_\infty} z\right)^2 \\ 1 & m_\infty z^2 +
            \frac{f_3 + \beta}{2 m_\infty} z.
           \end{pmatrix}.
\end{align}
(Note that the upper right entry of $\varphi$ is in fact quadratic.) 

The small stratum (corresponding to $(u:x:w)=(0:0:1)$) is
\begin{align}\label{eq:smallstrata}
 \cE \simeq&  \cO(-1) \oplus \cO(-3) \\ \nonumber
 \varphi_\beta  =& \frac{\de z}{z(z-1)(z-p_0)}\begin{pmatrix} 0 & f(z) + \beta z (z-1)(z-p_0) \\ 1 & 0 \end{pmatrix} 
\end{align}
\end{prop}

The proof of this proposition is a somewhat lengthy computation, and so is delayed until later in this section.  
We first give some remarks and a corollary.

\begin{rem}\label{rem:preference} As mentioned at the beginning of this section, this description is convenient in chambers E1$_\infty$ (and less so for B1$_\infty$). There are other more natural descriptions in other $E$-type chambers. In the proof of the  Proposition \ref{prop:higgsbundles}, observe that we make a choice to prefer $\infty \in D$, by picking a frame in which $F_\infty=<e_1>$.
 \end{rem}

From Proposition  \ref{prop:higgsbundles}, we can write out the flags on an open dense subset of $\mathcal{M}^0(\boldsymbol{\alpha}, \mathbf{m})$:
\begin{cor} [Flags] \label{cor:flags}
For the Higgs bundles in the big strata \eqref{eq:bigstrata}, the flags are 
     \begin{equation}\label{eq:flags}
    F_0 = \begin{pmatrix} \frac{-m_0 p_0 +w}{u} \\ 1 \end{pmatrix}, \quad 
      F_1 = 
  \begin{pmatrix} \frac{m_1(p_0-1) + m_\infty +w}{u-1} \\ 1 \end{pmatrix}, \quad 
        F_{p_0} = \begin{pmatrix} \frac{-m_{p_0} p_0(p_0-1) + m_\infty p_0^2 +w}{u-p_0}\\ 1 \end{pmatrix}, \quad 
          F_\infty = \begin{pmatrix} 1\\ 0 \end{pmatrix};
   \end{equation}
 for the Higgs bundles in the additional point of the big stratum \eqref{eq:extrapoint} these are   \begin{align*}
    F_0 &= \begin{pmatrix} m_0 p_0  \\ 1 \end{pmatrix}, \quad 
      F_1 = 
  \begin{pmatrix} - \frac{f_3 + 2 m_\infty (m_\infty + m_1(p_0-1)) + \beta}{2 m_\infty} \\ 1 \end{pmatrix}, 
        F_{p_0} &= \begin{pmatrix} - \frac{f_3 p_0+ 2 m_\infty m_{p_0}(p_0 - 1) p_0 + p_0(2 m_\infty^2 p_0 + \beta)}{2m_\infty}\\ 1 \end{pmatrix}, \quad 
          F_\infty = \begin{pmatrix} 1\\ 0 \end{pmatrix},
   \end{align*}
   where $f_3$ is shorthand for the coefficient of $z^3$ in $f_{\mathbf{m}}(z)$;
 and for the Higgs bundles in the small stratum \eqref{eq:smallstrata}
  \begin{equation*}
    F_0 = \begin{pmatrix} m_0 p_0 \\ 1 \end{pmatrix}, \quad 
      F_1 = 
  \begin{pmatrix} m_1(1-p_0)  \\ 1 \end{pmatrix}, \quad 
        F_{p_0} = \begin{pmatrix} m_{p_0} p_0(p_0-1) \\ 1 \end{pmatrix}, \quad 
          F_\infty = \begin{pmatrix} - m_\infty\\ 1 \end{pmatrix}.
   \end{equation*}
\end{cor}

\begin{proof}[Proof of Proposition \ref{prop:higgsbundles}]
As already noted, since $\pdeg(\cE)=0$ and $|D| = k = 4$, the bundle $\cE$ itself has degree $-4$.  As a holomorphic bundle over $\CP^1$, it must
split as $\cE \simeq \cO(m) \oplus \cO(-4-m)$, for $m \in \Z$ (and we may as well take $m\geq -2$). For any $m$,
\begin{equation}
\End \cE \simeq \begin{pmatrix} \cO & \cO(4+2m) \\ \cO(-4-2m) & \cO \end{pmatrix}.
\end{equation}
The set $\cM^0(\mathbf{m})=\mathrm{Hit}^{-1}(\mathcal{B}^0(\mathbf{m}))$ is chosen as precisely the locus where the Higgs field $\varphi$ uniquely determines the flags. 
Since $\varphi$ has at most simple poles at $0, 1, p_0$, we can write
 \begin{equation}
  \varphi = \frac{\de z}{z(z-1)(z-p_0)}\begin{pmatrix} a(z) & b(z) \\ c(z) & -a(z) \end{pmatrix} \qquad \mbox{for some } a(z), b(z), c(z) \in \C[z].
\label{genvarphi}
 \end{equation}
Passing to the coordinate $\tilde{z}=z^{-1}$ gives
\begin{equation}
\varphi = -\frac{\de \tilde{z}}{\tilde{z}(1-\tilde{z})(1-\tilde{z}p_0)} \cdot \tilde{z}^2 \begin{pmatrix} \tilde{z}^0 a(\tilde{z}^{-1}) & 
\tilde{z}^{4+2m} b(\tilde{z}^{-1}) \\ \tilde{z}^{-4-2m} c(\tilde{z}^{-1}) & -\tilde{z}^0 a(\tilde{z}^{-1}) \end{pmatrix},
\end{equation}
hence since $\varphi$ also has at most a simple pole at $z=\infty$ (i.e. $\tilde{z}=0$), we see that
\begin{equation}
 \deg a \leq 2, \quad  \deg b \leq 6+2m, \quad \deg c \leq -2-2m.
\end{equation}
In addition, the expression for $-\det \varphi$ gives
\begin{equation} \label{eq:determinant}
a(z)^2 + b(z) c(z) = f(z)+\beta z(z-1)(z-p_0).
\end{equation}

If $m \geq 0$, then $\deg c\leq -2$, so in fact $c \equiv 0$, and hence
\begin{equation}
 a(z)^2 =f(z) + \beta z(z-1)(z-p_0).
\end{equation}
But this means that $f(z) + \beta z(z-1)(z-p_0)$ is a perfect square, which is not allowed since $q \in \mathcal{B}^0$.
We conclude from all of this that only two cases remain: either $\cE \simeq \cO(-2) \oplus \cO(-2)$ or else $\cE \simeq \cO(-1) \oplus \cO(-3)$.

\smallskip

Suppose first that $\cE \simeq \cO(-1) \oplus \cO(-3)$. In this case, 
\begin{equation}
\End \cE \simeq \begin{pmatrix} \cO & \cO(2) \\ \cO(-2) & \cO \end{pmatrix},
\end{equation}
and the polynomial entries $a(z), b(z), c(z)$ of the Higgs field have $\deg a \leq 2$, $\deg b \leq 4$ and $\deg c =0$. These also satisfy
\eqref{eq:determinant}, which shows in particular that $c$ is nonzero, else $f(z) + \beta z(z-1)(z-p_0)$ would again be a perfect square. 
Using the diagonal gauge freedom, we then arrange that $c \equiv 1$; the remaining gauge freedom then comes from the transformations
\begin{equation}
 g=\begin{pmatrix} 1 & s(z) \\ 0 & 1 \end{pmatrix}, \qquad \deg s(z) =2,
\end{equation}
and taking $s(z) = a(z)$ gives
\begin{equation}
 g^{-1} \varphi g = \frac{\de z}{z(z-1)(z-p_0)}\begin{pmatrix} 0 & f(z) + \beta z (z-1)(z-p_0) \\ 1 & 0 \end{pmatrix}.
\end{equation}

\smallskip

Next suppose that $\cE \simeq \cO(-2) \oplus \cO(-2)$.  Now the degrees of $a, b$ and $c$ are all less than or equal to $2$, so we write
\begin{equation}
\varphi = \frac{\de z}{z(z-1)(z-p_0)}\begin{pmatrix} a_2 z^2 + a_1 z + a_0 & b_2 z^2 + b_1 z + b_0 z \\ c_2 z^2 + c_1 z + c_0 & - a_2 z^2 - a_1 z - a_0 \end{pmatrix}.
\end{equation}
The complex gauge group is the group of constant $SL(2,\C)$ gauge transformations. We use this to write explicit representatives in each equivalence class. 
Choose a constant gauge transformation so that the flag $F_\infty$ is spanned by $\begin{pmatrix} 1\\ 0 \end{pmatrix}$.  From this we see that $a_2=-m_\infty$. 
Since $\left.\mathrm{Res}\,\varphi\right|_\infty$ respects this flag, $c_2=0$. But then, since $c(z) \not\equiv 0$, so $c_1$ and $c_0$ cannot simultaneously vanish,
hence we can regard $[c_1:c_0] \in \CP^1$.

\smallskip

\noindent \underline{If $c_1 \neq 0$}, a further diagonal gauge transformation makes $c_1=1$; the residual gauge freedom is to conjugate by 
$\begin{pmatrix} 1 & d \\ 0 & 1 \end{pmatrix}$, and using this we make $a_1=0$.  Now rename $w=-a_0$ since it will serve as one of the coordinates 
on the elliptic curve. In summary, we have now found a representative of the form 
\begin{align*}
 \varphi &= \frac{\de z}{z(z-1)(z-p_0)}\begin{pmatrix}  -(m_\infty z^2 + w) & b_2 z^2 + b_1 z + b_0 z \\  z-u & m_\infty z^2  + w \end{pmatrix}\\
& = \frac{\de z}{z(z-1)(z-p_0)}\begin{pmatrix}  -(m_\infty z^2 + w)  & \frac{f(z) + \beta z (z-1)(z-p_0) - (m_\infty z^2 + w)^2}{z-u} \\  z-u & m_\infty z^2  + w  \end{pmatrix},
\end{align*}
where, by definition, $u$ solves the equation
\begin{equation}
f(u) + \beta u (u-1)(u-p_0) - (m_\infty u^2 + w)^2=0.
\end{equation}

\smallskip

\noindent \underline{If $c_1 = 0$}, a diagonal gauge transformation makes $c_0=1$, and then by an upper triangular gauge transformation 
$g= \begin{pmatrix} 1 & d \\ 0 & 1 \end{pmatrix}$, we also arrange that $a_0=0$.  In summary, 
\begin{equation}
 \varphi = \frac{\de z}{z(z-1)(z-p_0)}\begin{pmatrix}  -m_\infty z^2 + a_1 z & b_2 z^2 + b_1 z + b_0 z \\  1 & m_\infty z^2  - a_1 z \end{pmatrix}.
\end{equation}
Writing $-\det \varphi$ as in \eqref{genvarphi}  
with $f(z) = f_4 z^4 + f_3 z^3 + f_2 z^2 + f_1 z + f_0$, we note again that $f_4=m_\infty^2$, and find that $-a_1=\frac{f_3 + \beta}{2 m_\infty}$.
Thus, for each choice of $\beta$, there is a single point in the fiber, and the corresponding Higgs field is
\begin{equation}
 \varphi = \frac{\de z}{z(z-1)(z-p_0)}\begin{pmatrix} -(m_\infty z^2 +
            \frac{f_3 + \beta}{2 m_\infty} z)& f(z) + \beta z(z-1)(z-p_0) 
            -\left(m_\infty z^2 +
            \frac{f_3 + \beta}{2 m_\infty} z\right)^2 \\ 1 & m_\infty z^2 +
            \frac{f_3 + \beta}{2 m_\infty} z.
           \end{pmatrix}.
\end{equation}

\medskip

Finally, let us consider whether this Higgs field determines the flag at $0$. We have
\begin{equation}
\left.\mathrm{Res}(\varphi_{\beta, u,w})\right|_{z=0}= \frac{\de z}{p_0 z}  \begin{pmatrix} 
-w& \frac{w^2-f(0)}{u}\\  -u &w \end{pmatrix},
\end{equation}
and this vanishes, so that the flag is not defined, when both
\begin{enumerate}
 \item $(\beta, u, w)=(\beta, 0, 0)$ solves $f(u)+\beta u(u-1)(u-p_0)-(m_\infty u^2 + w)^2=0$ and 
 \item $\frac{w^2-f(0)}{u}=0$.
\end{enumerate}
Condition (1) is equivalent to $f(u)=0$, i.e., the complex mass $m_0=0$. For Condition (2), observe that by L'H\^{o}pital's rule, $\lim_{u \to 0}\frac{w^2-f(0)}{u}$ 
vanishes only if $f(z) + \beta z(z-1)(z-p_0) \in \C[z]$ has a double root at $z=0$. These are the precisely two conditions defining $\mathcal{B}^0(\mathbf{m})$ in \eqref{eq:B0}.

A similar calculation gives the same result at the other residues.
\end{proof}

\subsection{Spectral Data}

 In Proposition \ref{prop:higgsbundles}, we saw that the fiber $\Sigma_\beta$ over $\beta \in \mathbb{C}$ inside of $\mathbb{CP}^2$ is cut out by the equation $F(u:x:w)=0$. It is helpful to note that 
\begin{equation} \frac{1}{x^3} F(u:x:w)= f_{\mathbf{m}}\left(\frac{u}{x}\right) + \beta \frac{u}{x}\left(\frac{u}{x}-1\right)\left(\frac{u}{x} - p_0\right) -\left(m_\infty \left(\frac{u}{x}\right)^2  + \left(\frac{w}{x}\right)\right)^2 \label{eq:F2}\end{equation}
 There is a clear map $\pi: \mathbb{CP}^2-\{(0:0:1)\} \to \mathbb{CP}^1; (u:x:w) \mapsto (u:x)$.
 \begin{lem}\label{lem:pi}
 There is a holomorphic map $\pi:\Sigma_\beta \to \CP^1$ which is generically $2:1$.
 \end{lem}
 \begin{proof}
 We note that the map $\pi: \mathbb{CP}^2-\{(0:0:1)\} \to \mathbb{CP}^1$ immediately gives a map $\pi: \Sigma_\beta-\{(0:0:1)\} \to \CP^1$. Note, however, that $(0:0:1) \in \Sigma_\beta$. 
 
 The map $\pi$ continuously extends to $\Sigma_\beta$ by defining $\pi(0:0:1)=(1:0)$. To see this, $F(u:x:1)=0$ implicitly defines $u(x)$; by L'H\^{o}pital's rule, $\lim_{x \to 0, u \to 0}\frac{u}{x} = \lim_{x \to 0, u\to 0} \frac{\partial u}{\partial x} = -  \frac{\partial_x F}{\partial_u F}|_{(0:0:1)}=\infty$.

 From \eqref{eq:F2}, it is clear that for $(u:x)$ fixed, $F(u:x:w)=0$ is quadratic in $(w:x)$, i.e. $\pi$ is generically two-to-one; moreover, it is ramified at the four zeros of $f_{\mathbf{m}}(\frac{u}{x}) + \beta \frac{u}{x}(\frac{u}{x}-1)(\frac{u}{x} - p_0)$, counted with multiplicity. Small standard modifications to above argument are made when $x=0$ or $m_\infty = 0$. In particular, if $m_\infty=0$, then $(1:0)$ is a branch point, and the preimage of $\pi^{-1}(1:0)$ consists of the single point $(0:0:1)$; if $m_\infty \neq 0$, then $\pi^{-1}(1:0)$ consists of $(0:0:1)$ and $(m_\infty: 0: \frac{f_3+\beta}{2})$, where $f_3$ is the coefficient of the cubic term in $f_{\mathrm{m}}$.
 \end{proof}

 Now, we can embed $\Sigma_\beta \hookrightarrow K_{\CP^1} \to \CP^1$, by giving the value of the tautological $1$-form $\lambda$.

 \begin{lem}\label{lem:tautological}
 The tautological $1$-form $\lambda$ on $\Sigma_\beta$ is 
\begin{equation}
\lambda =   \frac{m_\infty (\frac{u}{x})^2 + w}{\frac{u}{x}(\frac{u}{x}-1)(\frac{u}{x}-p_0)} d\left(\frac{u}{x}\right).
\end{equation}
 \end{lem}
\begin{proof}
Clearly, \begin{equation}\lambda^2 = -\det \varphi = \frac{f_{\mathrm{m}}(\frac{u}{x})+ \beta \frac{u}{x}(\frac{u}{x}-1)(\frac{u}{x}-p_0)}{(\frac{u}{x})^2 (\frac{u}{x}-1)^2(\frac{u}{x}-p_0)^2} d\left(\frac{u}{x}\right)^2.\end{equation}
Then $F(u:x:w)$ gives us a square root of this.
\end{proof}

Lastly, we conclude this entire discussion, by considering the limiting Hitchin fiber as $\beta \to \infty$. (This is not to be confused with the fiber at $\beta=\infty$ in the compactification of an ALG instanton as a rational elliptic surface, which here would be $I_0^*$.)

\begin{prop}[Limiting Hitchin Fiber]\label{prop:jinv}
The limiting torus fiber as $|\beta| \to \infty$ is isomorphic to $\C/(\Z \oplus \tau_0 \Z)$ where, 
in terms of the elliptic modular lambda function, $p_0=\lambda(\tau_0)$. 
Moreover,  the $\tau$ invariant of the fiber over any $\beta$ in the Hitchin base has asymptotics $\tau \sim \tau_0 + O(\frac{1}{\beta})$ 
as $\beta \to \infty$.
\end{prop}

\begin{proof}The spectral curve 
\[
(\lambda')^2 = f(u) + \beta u (u-1)(u-p_0)
\]
is a branched double cover of $\CP^1_u$.  Changing coordinate by $\lambda' \mapsto \sqrt{\beta} y$ gives the equivalent form
\[ 
y^2 = \frac{1}{\beta} f(u) + u (u-1)(u-p_0).
\]
Thus, the family of spectral curves is a perturbation with small parameter $\frac{1}{\beta}$ of the spectral curve 
\[
y^2 =  u (u-1)(u-p_0),
\]
which is branched at $0, 1, p_0, \infty$.  
The parameter $\tau$ is determined in $\mathbb H^2/\mathrm{SL}(2,\Z)$ by computing $\lambda^{-1}(p)$, where $p$ is the cross ratio of the 
branching points and $\lambda$ is the modular lambda function.  The roots of $\frac{1}{\beta} f(u) + u (u-1)(u-p_0)$ are distinct when
$\beta$ is sufficiently small, hence analytic in $\frac{1}{\beta}$.  In fact, if $r_p(\beta)$ is the root converging to $p$, for $p \in D$,

then 
\[
r_p(\beta) = p - \frac{1}{\beta} f(p) \cdot \left.\frac{(u-p)}{u (u-1)(u-p_0)} \right|_{u=p} + \mathcal O\left(\frac{1}{\beta^2}\right).
\]
Recall that $f(p)$ is proportional to $m_p^2$, the square of the associated complex mass.
Since $\lambda$ is holomorphic, $p-p_0 = \mathcal O(\frac{1}{\beta})$ implies $\tau - \tau_0 = \mathcal O(\frac{1}{\beta})$ as $\beta \to \infty$.
\end{proof}

\section{Computation of integral of \texorpdfstring{$\Omega_I$}{the holomorphic symplectic form}}\label{sec:TorelliOmegaI}

The preferred basis of $H_2(\cM,\Z)$ depends on parabolic weight chamber, and consequently the intersection numbers
appearing in Corollary \ref{cor:intersection_numbers} are chamber dependent.
In this section, we compute intersection numbers---and hence the integral of $\Omega_I$--- in two chambers: one exterior chamber in Proposition \ref{prop:exteriorintersection} and one interior chamber in Proposition \ref{prop:interiorintersection}.  Looking forward, the statement in \emph{all} chambers appears in Proposition \ref{prop:integralofOmegaI} in Section \ref{sec:Torelliproof}; the proof 
follows directly from formula in a \emph{single} chamber. However, we separately compute the intersection numbers in these two types of chambers because the behavior is so different.

\subsection{Computation of \texorpdfstring{$\Omega_I$}{the holomorphic symplectic form} in an exterior chamber}
 If $D=\{p_1, p_2, p_3, p_4\}$, given $I \subset \{1, 2, 3, 4\}$, let $D_I = \{p_i: i \in I\}$.
We consider the exterior chamber labelled by odd order set $D_{I_0}=\{0, 1, p_0\}$. This is 
 of type $E2$ with distinguished point $\infty \in D$.
  If $D=\{p_1, p_2, p_3, p_4\}$, given $I \subset \{1, 2, 3, 4\}$, let $D_I = \{p_i: i \in I\}$.
We will use the convention that 
\begin{equation}
\begin{array}{ccc}
D_I & & p \in D\\
\hline
\{1, p_0 \} &\leftrightarrow& 0 \\
\{0, p_0 \} &\leftrightarrow& 1 \\
\{0, 1 \} &\leftrightarrow& p_0 \\
\{0, 1 ,p_0, \infty\} &\leftrightarrow& \infty 
\end{array}
\end{equation}

\begin{prop}\label{prop:exteriorintersection}
In the exterior chamber labelled by $D_{I_0} = \{0, 1, p_0\}$ (see Section \ref{sec:exteriorchambers}) 
  \[\int_{\CP^1_p} \Omega_I = \begin{cases} -4 \pi  m_p &\quad p \in \{0, 1, p_0\}\\
  4 \pi  m_p & \quad p=\infty\end{cases},  \]
  i.e. $\int_{\CP^1_J} \Omega_I = 2 \pi (M_J -M_{I_0}).$
 \end{prop}

 \begin{proof}
 We wish to employ Corollary \ref{cor:intersection_numbers}, and therefore we need to determine the image of the polar sections $\sigma_p^\pm$ under the $\C^\times$-flow in the limit $\zeta \to 0$; then, for each polar section, we'll determine the intersection of this limiting set with the exterior and interior fixed points corresponding to points of $D$. This gives us the intersection numbers between $S^2_p = \tau_p^- - \tau_p^+$ and $\Sigma_{p'} = \sigma^+_{p'} - \sigma^-_{p'}$.

Recalling the fiber $\mathrm{Hit}^{-1}(\beta)$ in Lemma \ref{lem:pi},
the map from the spectral curve to base curve is 
\begin{align}
\Sigma_\beta &\rightarrow \mathbb{CP}^1 \\ \nonumber 
(u, w, x) &\mapsto \frac{u}{x}.
\end{align}
The tautological $1$-form in Lemma \ref{lem:tautological} is 
\[ \tau = -\frac{(m_\infty u^2 + w)}{z(z-1)(z-p_0)} dz.\]

We observe that the eight polar sections $\sigma^\pm_p$ are:
\begin{itemize}
\item[$\sigma^+_\infty$ \quad ] This is the small stratum, i.e. underlying bundle type $\cO(-1) \oplus \cO(-3)$. 

\item[$\sigma^-_\infty$ \quad ]  This section corresponds consists of the additional point of the big stratum corresponding to $(u: x: w)=(m_\infty: 0: \frac{f_3 + \beta}{2 m_\infty})$ for $\beta$ arbitrary.

\item [$\sigma^+_0$ \quad ] This section within the big stratum  ($x=1$) corresponds to $u=0, w=m_0p_0, \beta \in \C$. By way of explanation for the value of $w$, note that $(m_\infty u^2 + w)^2\Big|_{u=0}=f_{\mathbf{m}}(u) + \beta u(u-1)(u-p_0) \Big|_{u=0} = f_{\mathbf{m}}(0) = m_0^2 p_0^2$, hence $w=\pm m_0p_0$. Then, the residue of $\tau$ is $\frac{m_0p_0}{(-1)(-p_0)} = m_0$. 

\item [$\sigma^-_0$ \quad ] This section within the big stratum  ($x=1$) corresponds to $u=0, w=-m_0p_0, \beta \in \C$. 

\item[$\sigma^+_1$ \quad ]  This section within the big stratum  ($x=1$) corresponds to $u=1, w=-m_\infty + m_1(1-p_0), \beta \in \C$. By way of explanation for the value of $w$, note that $(m_\infty u^2 + w)^2\Big|_{u=1}=f_{\mathbf{m}}(u) + \beta u(u-1)(u-p_0) \Big|_{u=1} = f_{\mathbf{m}}(1) = m_1^2(1-p_0)^2$, hence $w=-m_\infty \pm m_1(1-p_0)$. Then, the residue of $\tau$ is $m_1$.

\item[$\sigma^-_1$ \quad ] This section within the big stratum  ($x=1$) corresponds to $u=1, w=-m_\infty - m_1(1-p_0), \beta \in \C$. 

\item[$\sigma^+_{p_0}$ \quad ] This section within the big stratum ($x=1$) corresponds to $u=p_0, w=-m_\infty p_0^2 + m_p p_0(p_0-1), \beta \in \C$. By way of explanation for the value of $w$, note that $(m_\infty u^2 + w)^2\Big|_{u=p_0}=f_{\mathbf{m}}(u) + \beta u(u-1)(u-p_0) \Big|_{u=p_0} = f_{\mathbf{m}}(p_0) = p_0^2(p_0-1)^2 m_{p_0}^2$, hence $w=-m_\infty p_0^2 \pm m_{p_0} p_0(p_0-1)$. Then, the residue of $\tau$ is $m_{p_0}$.

\item[$\sigma^-_{p_0}$ \quad ] This section within the big stratum ($x=1$) corresponds to $u=p_0, w=-m_\infty p_0^2 - m_p p_0(p_0-1), \beta \in \C$. 

\end{itemize}

In the exterior chamber of type E2 labelled by $D_{I_0}=\{0, 1, p_0\}$, 
the associated subsets $\{D_I: I \in \mathcal{I}\}$ are 
$$\{\{0, 1, p_0, \infty\}, \{0, 1\}, \{0, p_0\}, \{1, p_0\}\}$$
 and hence $K_I>0$ for each $I \in \mathcal{I}$. In this chamber, all but one point of the nilpotent cone has underlying bundle type $\cO(-2) \oplus \cO(-2)$, as shown in Figure \ref{fig:assemblyE}.
 We refer the reader to the description of all $\C^\times$-fixed points in \eqref{eq:tablefixedpoints}. 

We chose this exterior chamber because the choices of coordinates we made when we wrote the Higgs fields in Proposition \ref{prop:Hitchinbase} match up nicely. I.e. by taking $g_\zeta=\mathrm{diag}(\zeta^{1/2}, \zeta^{-1/2})$ the following limit exists and is $\boldsymbol{\alpha}$-stable $\C^\times$-fixed point: $\lim_{\zeta \to 0} g_\zeta \zeta\varphi g_{\zeta}^{-1}$.
Since the $\C^\times$-flow preserves $u$, the only subtlety is determining whether $\sigma^\pm_u$ lands on the interior or exterior $\C^\times$-fixed point corresponding to $u$ on the central sphere. To determine this, we consider the limit of the flags (described in Corollary \ref{cor:flags}) under the $\C^\times$-flow. 

From the description of the $\C^\times$-fixed points, observe that on the central sphere, the flags are given by $F_0=F_1=F_{p_0}= \langle e_2 \rangle$ and $F_\infty = \langle e_1 \rangle$. The exterior $\C^\times$-fixed point corresponding to $u \in \{0, 1, p_0\}$ is obtained from changing the flag $F_u$ from $\langle e_2 \rangle$ to $\langle e_1 \rangle$ and leaving all the other flags the same. The exterior $\C^\times$-fixed point corresponding to $u=\infty$ has bundle type $\cE=\cO(-1) \oplus \cO(-3)$. 

Altogether we obtain:

 \begin{itemize}
  \item  $\sigma_0^+$ all flows to the interior $\C^\times$-fixed point with $u=0$. One can indeed check that the flags are $F_0=F_1=F_{p_0}=\langle e_2\rangle$ and $F_\infty=\langle e_1\rangle$.
   \item $\sigma_0^-$ all flows to the the exterior $\C^\times$-fixed point with $u=0$. One can indeed check that $F_1=F_{p_0}=\langle e_2\rangle$ and $F_0=F_\infty=\langle e_1\rangle$.
   \item $\sigma_1^+$ all flows to the interior $\C^\times$-fixed point with $u=1$. One can indeed check that the flags are $F_0=F_1=F_{p_0}=\langle e_2\rangle$ and $F_\infty=\langle e_1\rangle$.
    \item$\sigma_1^-$ all flows to the exterior $\C^\times$-fixed point with $u=1$. One can indeed check that $F_0=F_{p_0}=\langle e_2\rangle$ and $F_1=F_\infty=\langle e_1\rangle$.  
       \item $\sigma_{p_0}^+$ all flows to the interior $\C^\times$-fixed point with $u=p_0$. One can indeed check that the flags are $F_0=F_1=F_{p_0}=\langle e_2\rangle$ and $F_\infty=\langle e_1\rangle$. 
        \item $\sigma_{p_0}^-$ all flows to the exterior $\C^\times$-fixed point with $u=p_0$. One can indeed check that $F_0=F_1=\langle e_2\rangle$ and $F_{p_0}=F_\infty=\langle e_1\rangle$. 
        \item $\sigma_{\infty}^+$ all flows to the exterior $\C^\times$-fixed point at $u=\infty$.
        \item $\sigma_{\infty}^-$ all flows to the interior $\C^\times$-fixed point at $u=\infty$. 
        \end{itemize}
\smallskip
The computations above show that $\tau_p^+ = \sigma_p^-$ and $\tau_p^- = \sigma_p^+$ for $p \in \{ 0,1,p_0\}$ as well as $\tau_\infty^+ = \sigma_\infty^+$ and $\tau_\infty^-= \sigma_\infty^-$.  
For $\Sigma_p = \sigma_p^+ - \sigma_p^-$, note
therefore that  intersection numbers are given as follows:
\begin{center}
\begin{tabular}{l c r}
	
\begin{tabular}{  c | c | c |c|c }
    I & $\Sigma_0$ & $\Sigma_1$ & $\Sigma_{p_0}$ & $\Sigma_\infty$ \\ \hline
    $\tau_0^+$ & -1 & 0 & 0 & 0\\ \hline
    $\tau_1^+$ & 0 & -1 & 0 & 0\\ \hline
    $\tau_{p_0}^+$ & 0 & 0  & -1 & 0\\ \hline
    $\tau_\infty^+$ & 0 & 0  & 0 & 1\\
\end{tabular}
&  \quad and \quad &

\begin{tabular}{  c | c | c |c|c }
    I & $\Sigma_0$ & $\Sigma_1$ & $\Sigma_{p_0}$ & $\Sigma_\infty$ \\ \hline
    $\tau_0^-$ & 1 & 0 & 0 & 0\\ \hline
    $\tau_1^-$ & 0 & 1 & 0 & 0\\ \hline
    $\tau_{p_0}^-$ & 0 & 0  & 1 & 0\\ \hline
    $\tau_\infty^-$ & 0 & 0  & 0 & -1\\
\end{tabular}.

\end{tabular}

\end{center}
Hence,  for $S_i = \tau_i^- - \tau_i^+$,
\begin{center}

 \begin{tabular}{  c | c | c |c|c }
    I & $\Sigma_0$ & $\Sigma_1$ & $\Sigma_{p_0}$ & $\Sigma_\infty$ \\ \hline
    $S_0$ & 2 & 0 & 0 & 0\\ \hline
    $S_1$ & 0 & 2 & 0 & 0\\ \hline
    $S_{p_0}$ & 0 & 0  & 2 & 0\\ \hline
    $S_\infty$ & 0 & 0  & 0 & -2 \\
\end{tabular}.
\end{center}
This proves the claim.
 \end{proof}

 \subsection{Computation of the integral of \texorpdfstring{$\Omega_I$}{the holomorphic symplectic form} in an interior chamber}
 
 Note that while the integral of $\Omega_I$ over exterior spheres in exterior chambers are some scalar multiple of a single complex mass, the integral of $\Omega_I$ over exterior spheres in the interior chambers features \emph{all} of the complex masses.  We illustrate this underlying complicated geometry in the case of the interior chamber labelled by even partition set $D_{\mathcal{I}}=\{\{0, 1, p_0, \infty\}, \{1, p_0\}, \{0, p_0\}, \{0, 1\}\}$.
This is the adjacent interior chamber to the exterior chamber considered in the previous section, we have the same identification of subset $D_I$ with points of the divisor.

\begin{prop}\label{prop:interiorintersection} 
In the interior chamber of Type B2 with $\mathcal{I}=\{\{0, 1, p_0, \infty\}, \{0, 1\}, \{0, p_0\}, \{1, p_0\}\}$ then we have  
\begin{equation}
\begin{array}{llll}
\int_{S^2_\infty} \Omega_I &=\int_{S^2_{\{0, 1, p_0, \infty \}}} \Omega_I &= -2 \pi (-m_0 - m_1 - m_{p_0}- m_\infty)&=2 \pi M_{\{0, 1, p_0, \infty\}}\\
\int_{S^2_1}  \Omega_I&= \int_{S^2_{\{0, p_0\}}} \Omega_I &= -2 \pi (-m_0 +m_1 - m_{p_0}+ m_\infty)&=2 \pi M_{\{0,p_0\}}\\
\int_{S^2_0}  \Omega_I&=\int_{S^2_{\{1, p_0\}}} \Omega_I &= -2\pi (m_0 -m_1 - m_{p_0}+ m_\infty)&=2 \pi M_{\{1,p_0\}}\\
\int_{S^2_{p_0}}  \Omega_I&=\int_{S^2_{\{0, 1\}}} \Omega_I &= -2 \pi (-m_0 -m_1 + m_{p_0}+ m_\infty)&= 2 \pi M_{\{0,1\}}
\end{array}
\end{equation}
\end{prop}

\begin{proof} 
We observe that the eight polar sections $\sigma^\pm_p$ are the same as in the proof of Proposition \ref{prop:exteriorintersection}.

As one passes from the exterior chamber of Proposition \ref{prop:exteriorintersection} to the adjacent interior chamber, all exterior $\C^\times$-fixed points remain stable. However, the description of the central sphere changes entirely. From this we see that the flows to the exterior $\C^\times$-fixed points remain unchanged:
 \begin{itemize}
   \item $\sigma_0^-$ all flows to the exterior $\C^\times$-fixed point with $u=0$. One can indeed check that $F_1=F_{p_0}=\langle e_2\rangle$ and $F_0=F_\infty=\langle e_1\rangle$.
    \item$\sigma_1^-$ all flows to the exterior $\C^\times$-fixed point with $u=1$. One can indeed check that $F_0=F_{p_0}=\langle e_2\rangle$ and $F_1=F_\infty=\langle e_1\rangle$.  
        \item $\sigma_{p_0}^-$ all flows to the exterior $\C^\times$-fixed point with $u=p_0$. One can indeed check that $F_0=F_1=\langle e_2\rangle$ and $F_{p_0}=F_\infty=\langle e_1\rangle$. 
        \item $\sigma_{\infty}^+$ all flows to the exterior $\C^\times$-fixed point at $u=\infty$.
        \end{itemize}
In fact, $\sigma_0^- = \tau_0^+, \sigma_1^-=\tau_1^+, \sigma_{p_0}^-=\tau_{p_0}^+, \sigma_\infty^+ =\tau_\infty^+$. 

It remains to consider the flows of these other four polar sections under the $\C^\times$-flow.  Since each point of these four polar sections flows somewhere on the central sphere, the $\C^\times$-limit is simply the underlying stable parabolic bundle with zero Higgs field. In our particular chamber, this is determined by the cross-ratio of the flags (see \ref{sec:centralsphereinint}).
From Corollary \ref{cor:flags}, we recall that for the big strata, the flags are    \begin{equation*}
    F_0 = \begin{pmatrix} \frac{-m_0 p_0 +w}{u} \\ 1 \end{pmatrix}, \quad 
      F_1 = 
  \begin{pmatrix} \frac{m_1(p_0-1) + m_\infty +w}{u-1} \\ 1 \end{pmatrix}, \quad 
        F_{p_0} = \begin{pmatrix} \frac{-m_{p_0} p_0(p_0-1) + m_\infty p_0^2 +w}{u-p_0}\\ 1 \end{pmatrix}, \quad 
          F_\infty = \begin{pmatrix} 1\\ 0 \end{pmatrix}.
   \end{equation*}
   Note that the formula for $F_p$ is not defined on $\sigma^+_p$ for $p \in \{0, 1, p_0\}$, however, we can use L'H\^{o}pital's Rule to extend the formula for $F_p$ to $\sigma^+_p$
\begin{itemize}
\item On $\sigma^+_0$, \[
F_0 = \begin{pmatrix} \frac{f'(0) + \beta p_0}{2 m_0 p_0} \\ 1.\end{pmatrix}
\]
To see this observe by L'H\^{o}pital's Rule
\begin{align*} F_0\Big|_{\sigma^+_0} &= \frac{-m_0 p_0 + w}{u}\Big|_{\sigma^+_0} \\
& =\frac{\frac{\partial}{\partial \beta}(m_0 p_0 + w)}{\frac{\partial}{\partial \beta}(u)}\Big|_{\sigma^+_0}\\
&=\frac{\partial_\beta w}{\partial_\beta u} \Big|_{\sigma^+_0}\\
&\overset{(1)}{=}\frac{f'_{\mathbf{m}}(0)  + \beta p_0}{2m_0 p_0}
\end{align*}
In (1), we use $\partial_\beta$ of the defining equation of the elliptic curve:
\[ 0=f'_{\mathbf{m}}(u) \partial_\beta u + u(u-1)(u-p_0)+\beta \partial_u \left( u(u-1)(u-p_0) \right) \partial_\beta u -2 (m_\infty u^2 + w)( 2 m_\infty u \partial_\beta u + \partial_\beta w).\]
Evaluated on $\sigma^+_0$ ($u=0$, $w=m_0 p_0$), this gives 
\[ 0 = \left( f'_{\mathbf{m}}(0)  + \beta p_0 \right)\partial_\beta u    - 2 (m_0p_0)\partial_\beta w. \]
   \item On $\sigma^+_1$,
   we compute 
   \begin{align*}
F_1\Big|_{\sigma^+_1} &= 
\frac{\partial_\beta w}{\partial_\beta (u-1)} \Big|_{\sigma^+_1}\\  
&=  \frac{f'_{\mathbf{m}}(1) + \beta(1-p_0)- 2 m_1(1-p_0) 2 m_\infty}{2m_1(1-p_0)} 
 \end{align*}
   \item On $\sigma^+_{p_0}$, we compute
      \begin{align*}
F_1\Big|_{\sigma^+_{p_0}} &= 
\frac{\partial_\beta w}{\partial_\beta (u-p_0)} \Big|_{\sigma^+_{p_0}}\\  
&=  \frac{f'_{\mathbf{m}}(p_0) + \beta p_0(p_0-1)- 2 m_{p_0} p_0(p_0-1)2 m_\infty p_0}{2 m_{p_0} p_0(p_0-1)} 
 \end{align*}
   \end{itemize}
   Now, the $\beta$ dependence of $F_{p}$ on $\sigma^+_p$ for $p \in \{0, 1, p_0\}$ is manifest. The other flags are independent of $\beta$.
   
Now, we look at the image under the $\C^\times$-limit of the polar sections $\sigma^+_0, \sigma^+_1, \sigma^+_{p_0}, \sigma_\infty^-$ , by considering the cross-ratio of the points.
Note that since $F_\infty=\infty$, the cross-ratio is 
\[ (F_\infty, F_0; F_1, F_{p_0})= \frac{F_{p_0} - F_0}{F_1-F_0}.\]
\begin{itemize}
\item Let's first consider $\sigma^+_0$, which we note is parameterized by $\beta \in \C$. Note that $F_1, F_{p_0}, F_\infty$ are all fixed on $\sigma^-_0$, so the cross-ratio of the flag only sees $\beta$ through $F_0$.
The important thing is that the cross-ratio of the flags has the shape of a conformal transformation:
\begin{equation} (F_\infty, F_0; F_1, F_{p_0}) = \frac{a_0 \beta + a_1}{a_2 \beta + a_3}\label{eq:conformal}\end{equation}
for fixed values $a_0, a_1, a_2, a_3$. 
Consequently, its image is all of $\mathbb{CP}^1$, minus the point 
\[ \lim_{\beta \to \infty} (F_\infty, F_0; F_1, F_{p_0}) = \frac{a_0}{a_2}, \]
Since $\beta$ only appears in $F_0$, we can see that 
\[ \lim_{\beta \to \infty} (F_\infty, F_0; F_1, F_{p_0}) =\frac{-F_0}{-F_0} = 1.\]
\item Now, let's consider $\sigma^+_1$. Since $\beta$ only appears in $F_1$, we see
\[ \lim_{\beta \to \infty} (F_\infty, F_0; F_1, F_{p_0}) =\frac{0}{F_1} = 0.\]
\item Now, let's consider $\sigma^+_{p_0}$. Since $\beta$ only appears in $F_{p_0}$, we see 
\[ \lim_{\beta \to \infty} (F_\infty, F_0; F_1, F_{p_0}) =\frac{F_p}{0}= \infty.\]
\item 
Lastly, we consider $\sigma_\infty^-$, the additional section in the big stratum. Note that from Corollary \ref{cor:flags}, $\beta$ appears in both $F_1$ (with coefficient $-\frac{1}{2m_\infty}$) and $F_{p_0}$ (with coefficient $-\frac{p_0}{2m_\infty}$), but not in $F_0, F_\infty$. Thus, 
\[ \lim_{\beta \to \infty} (F_\infty, F_0; F_1, F_{p_0}) = \frac{-\frac{p_0}{2m_\infty}}{-\frac{1}{2m_\infty}} = p_0. \]
\end{itemize}

We must match these flag arrangements to a value of $u$. 
We observe that when $\mathbf{m}=\mathbf{0}$, the cross-ratio of the flags is independent of $\beta$ and is simply 
\[(F_\infty, F_0; F_1, F_{p_0} ) = \frac{p_0(u-1)}{(u-p_0)}. \]
We indeed see that $u=0$ corresponds to the flag $1$; $u=1$ corresponds to the flag $0$; $u=p_0$ corresponds to the flag $\infty$; $u=\infty$ corresponds to the flag $p_0$. 
Thus, we see that we get the following intersections as the polar sections (each modeled on $\beta \in \C$) get wrapped around the central sphere minus a single point of $D$.
\begin{equation}
\begin{tabular}{  c | c | c |c|c }
    I & $\sigma^+_0$ & $\sigma^+_1$ & $\sigma^+_{p_0}$ & $\sigma^-_\infty$ \\ \hline
    $\tau_0^-$ & 0 & 1 & 1 & 1\\ \hline
    $\tau_1^-$ & 1 & 0 & 1 & 1\\ \hline
    $\tau_{p_0}^-$ & 1 & 1  & 0 & 1\\ \hline
    $\tau_\infty^-$ & 1 & 1  & 1 & 0\\
\end{tabular}
\end{equation}
The intersection numbers for $\Sigma_p = \sigma_p^+- \sigma_p^-$ are thus given as follows:

\begin{center}

\begin{tabular}{l c r}
	
\begin{tabular}{  c | c | c |c|c }
    I & $\Sigma_0$ & $\Sigma_1$ & $\Sigma_{p_0}$ & $\Sigma_\infty$ \\ \hline
    $\tau_0^+$ & -1 & 0 & 0 & 0\\ \hline
    $\tau_1^+$ & 0 & -1 & 0 & 0\\ \hline
    $\tau_{p_0}^+$ & 0 & 0  & -1 & 0\\ \hline
    $\tau_\infty^+$ & 0 & 0  & 0 & 1\\
\end{tabular}

&  and &

\begin{tabular}{  c | c | c |c|c }
    I & $\Sigma_0$ & $\Sigma_1$ & $\Sigma_{p_0}$ & $\Sigma_\infty$ \\ \hline
    $\tau_0^-$ & 0 & -1 & -1 & 1\\ \hline
    $\tau_1^-$ & -1 & 0 & -1 & 1\\ \hline
    $\tau_{p_0}^-$ & -1 & -1  & 0 & 1\\ \hline
    $\tau_\infty^-$ & -1 & -1  & -1 & 0.\\
\end{tabular}

\end{tabular}

\end{center}

Thus, the intersection numbers for $S_p = \tau_p^--\tau_p^+$ are 

\begin{center}

 \begin{tabular}{  c | c | c |c|c }
    I & $\Sigma_0$ & $\Sigma_1$ & $\Sigma_{p_0}$ & $\Sigma_\infty$ \\ \hline
    $S_0$ & 1 & -1 & -1 & 1\\ \hline
    $S_1$ & -1 & 1 & -1 & 1\\ \hline
    $S_{p_0}$ & -1 & -1  & 1 & 1\\ \hline
    $S_\infty$ & -1 & -1  & -1 & -1. \\
\end{tabular}

\end{center}

\end{proof}
\section{Affine \texorpdfstring{$D_4$}{D4}-Coxeter group action and Surjectivity} \label{sec:Dehntwists}

We have seen in Section \ref{subsec:Torelli_map} that there are natural right actions of the orthogonal group $\OO(\Z^5,I_0)^+$ on the set of admissible bases of $H_2(\calM(\bfalpha,\bfm),\Z)$ and on the target $\R^4 \times \C^4$  such that the period domain $\calP \subset \R^4 \times \C^4$ and the period map satisfies the equivariance condition
\[
\mathcal{T}_{\mathfrak{B} \cdot A}( \boldsymbol{\alpha}, \mathbf{m})= \mathcal{T}_{\mathfrak{B} }( \boldsymbol{\alpha}, \mathbf{m})\cdot A
\]	
for $A \in \OO(\Z^5,I_0)^+$. 

In this section we will define for fixed chamber $\Delta$ and associated homology basis $\mathfrak B(\Delta)$ a left action of the affine Weyl group $W_{\aff}$ on the domain $\R^4 \times \C^4$ and on the target  $\R^4 \times \C^4$ such that the period map satisfies the equivariance condition 
\[
\mathcal{T}_{\mathfrak{B}(\Delta)}( g \cdot (\boldsymbol{\alpha}, \mathbf{m}))= g \cdot \mathcal{T}_{\mathfrak{B}(\Delta) }( \boldsymbol{\alpha}, \mathbf{m})
\]
for $g \in W_{\aff}$. The affine Weyl group $W_\aff$ is the affine $D_4$-Coxeter group abstractly defined by the presentation
\[W_{\mathrm{aff}}=\IP{r_0, r_1, r_2, r_3, r_4| (r_i r_j)^{m_{ij}}},\]
  where the Coxeter matrix $M=(m_{ij})$ is given by 
  \[\begin{pmatrix} 1 &3 & 3 & 3 & 3\\ 3 & 1 & 2 & 2 & 2\\ 3 & 2 & 1 & 2 & 2\\ 3 & 2 & 2 & 1 & 2\\  3 & 2 & 2 & 2 & 1\end{pmatrix}.\]
  Note that $m_{ii}=1$ encodes that each 
  $r_i$ is squaring to the identity; for distinct indices $i$ and $j$, $m_{ij}=2$ is equivalent to the fact that $r_i$ and $r_j$ commute. This associated Coxeter diagram is somewhat unsurprisingly the affine $D_4$-diagram.\footnote{To create the Coxeter diagram from the Coxeter matrix, connect vertex $i$ and $j$ if $m_{ij} \geq 3$ and label the edge with the value of $m_{ij}$ if $m_{ij} \geq 4$.}
  \begin{figure}[!ht]
  \includegraphics[height=1.5in]{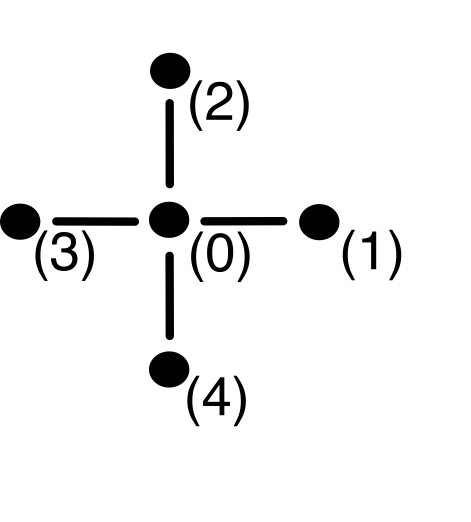}
  \caption{Coxeter diagram, with central node labelled `0' and exterior nodes labelled `1', `2', `3', `4'. \label{fig:Coxeter}}
  \end{figure}  
  The associated Schl\"afli matrix is obtained by $C_{ij}=-\cos(\pi/M_{ij})$. Here, it is $-\frac12$ times
  the usual intersection form on $H_2(\cM(\boldsymbol{\alpha}, \mathbf{m}),\Z)$\[
  \begin{pmatrix} 1 & - \frac12 & - \frac12 & - \frac12 & - \frac12 \\ - \frac12 &1 & 0 & 0 & 0 \\ - \frac12 & 0 & 1 & 0 & 0 \\
    - \frac12 & 0 & 0 & 1 & 0 \\
    - \frac12 & 0 & 0 & 0 & 1 \end{pmatrix}.
  \]The fact that this matrix has all non-negative eigenvalues, but at least one zero eigenvalue implies that this is a Coxeter group of affine type, and in particular infinite.

\smallskip

Recall that in \eqref{eq:R} we defined $\mathcal{R} \subset (0, \frac{1}{2})^4 \times \C^4$ to be the set of generic parameters. For these values, we have an associated Hitchin moduli space $\mathcal{M}(\boldsymbol{\alpha}, \mathbf{m})$. Given a fixed chamber $\Simplex $ and associated basis $\mathfrak{B}(\Simplex)$, we computed the map
\begin{align*}
\mathcal{T}_{\mathfrak{B}(\Simplex)}: (\Simplex \times \C^4) \cap \mathcal{R} &\longrightarrow \mathbb{R}^4 \times \C^4\\ \nonumber 
(\boldsymbol{\alpha}, \mathbf{m})   &\longmapsto \mathcal{T}_{\mathfrak{B}(\Simplex)}(\boldsymbol{\alpha}, \mathbf{m})=\left\{\left( \int_{S_i} \omega_I, \int_{S_i} \Omega_I \right)\right\}_{i=1}^4
\end{align*}
in Section \ref{sec:TorelliomegaI} and  Section \ref{sec:TorelliOmegaI} in some special cases. In particular, we computed the integral of $\omega_I$ for all chambers; however, we only directly computed the integral of $\Omega_I$ for two chambers. Furthermore, we established some structural results for the period map: the periods of the real symplectic form $\omega_I$ only depend on the parabolic weights $\bfalpha=(\alpha_1, \alpha_2, \alpha_3, \alpha_4)$ and the periods of the complex symplectic form $\Omega_I$ only depend on the complex masses $\bfm = (m_1, m_2, m_3, m_4)$ such that we can write
\[
\mathcal{T}_{\mathfrak{B}}(\bfalpha, \bfm) = \bigl(\mathcal{T}_{\mathfrak{B}}^{\omega_I}(\bfalpha), \mathcal{T}_{\mathfrak{B}}^{\Omega_I}(\bfm) \bigr)
\]
for maps $\mathcal{T}_{\mathfrak{B}}^{\omega_I}: (\R^4 \times \{\mathbf{0}\}    ) \cap \mathcal{R} \to \R^4$ and  $\mathcal{T}_{\mathfrak{B}}^{\Omega_I}: ( \{\mathbf{0}\} \times \C^4    ) \cap \mathcal{R} \to \C^4$ if $\mathfrak{B}$ is an admissible homology basis. If we used the chamber-dependent basis $\mathfrak{B}(\Simplex)$ we find by inspection of the concrete formul{\ae} in Proposition \ref{prop:stronglyparabolicvolumesv2} that the real period map $\mathcal{T}_{\mathfrak{B}(\Simplex)}^{\omega_I}$ extends to an affine-linear map $\R^4 \to \R^4$ for all chambers. Similarly, by the formul{\ae} in Proposition \ref{prop:exteriorintersection} and Proposition \ref{prop:interiorintersection} the complex period map  $\mathcal{T}_{\mathfrak{B}(\Simplex)}^{\Omega_I}$ extends to a $\C$-linear map $\C^4 \to \C^4$ for the two sample chambers.

\smallskip

The main objective of this section is to use the action of the affine $D_4$-Coxeter group and the equivariance properties of the Torelli map  to extend these results to all chambers. We will see that $\mathcal{T}_{\mathfrak{B}}^{\omega_I}$ is affine-linear and $\mathcal{T}_{\mathfrak{B}}^{\Omega_I}$ is $\C$-linear for any admissible  homology basis $\mathfrak{B}$ and we we will derive formul{\ae} for $\mathcal{T}_{\mathfrak{B}(\Simplex)}^{\Omega_I}$ in any chamber $\Delta$. This will prove our first  main result, Theorem \ref{thm:Torelli}, the Torelli theorem for Hitchin moduli spaces on the four-punctured sphere.

\smallskip

 We now define two different extensions of $\mathcal{R}$  to $\R^4 \times \C^4$, i.e.\ both will reproduce $\mathcal{R}$ when intersected with $(0, \frac{1}{2})^4 \times \C^4$.

\begin{defn}\label{def:Rdomain}
Let
\[
\widetilde{\mathcal{R}}^{\mathrm{full}} = \R^4 \times \C^4 \setminus \left( 
\bigcup_{d \in \Z,\, \mathbf{e} \in \{0, 1\}^4} \mathcal{H}_{d, \mathbf{e}} \cup \bigcup_{d \in \Z, \, i \in   \{1, \ldots, 4\}} \mathcal{H}_{d,i}
\right)
\]
where
    \[
  \mathcal{H}_{d, \mathbf{e}} = \Bigl\{  (\boldsymbol{\alpha}, \mathbf{m}) \in \R^4 \times \C^4 :   d + \sum_{i=1}^4 e_i +  (-1)^{e_i} \alpha_i  =0  \;  \&  \;   \sum_{i=1}^4 (-1)^{e_i} m_i = 0 \Bigr\} \]
 for $d \in \Z$, $\mathbf{e} = (e_1, e_2, e_3, e_4) \in \{0, 1\}^4$, and
    \[
  \mathcal{H}_{d, i} = \bigl \{  (\boldsymbol{\alpha}, \mathbf{m}) \in \R^4 \times \C^4 :   2\alpha_i  =  d  \;  \&  \;    m_i = 0 \bigr \} \]
for $d \in \Z$, $i \in   \{1, \ldots, 4\} $. Let further
\begin{align*}
\widetilde{\mathcal{R}} =    \widetilde{\mathcal{R}}^{\mathrm{full}} \setminus  \Bigl\{(\boldsymbol{\alpha}, \mathbf{m}) \in \R^4 \times \C^4 : & -\alpha_i + \alpha_j + \alpha_k + \alpha_\ell \in \Z \mbox{ for } i,j,k,\ell \mbox{ distinct} \\ & \;  \&  \; 2\alpha_i \in \Z  \mbox{ for some } i \in   \{1, \ldots, 4\} \Bigr\}.
\end{align*}
\end{defn}

The reason for introducing these two extensions is that $\widetilde{\mathcal{R}}^{\mathrm{full}}$ turns out to be the full preimage of the period domain $\mathcal{P}$ under the Torelli map $\mathcal{T}_\mathfrak{B}$, see Theorem \ref{thm:surjective} below. However, not all parameters $(\bfalpha, \bfm) \in \widetilde{\mathcal{R}}^{\mathrm{full}}$ correspond to Hitchin moduli spaces. Clearly, parameters $(\bfalpha, \bfm) \in \mathcal{R}$ do correspond to Hitchin moduli spaces. Now $\widetilde{\mathcal{R}}$ can alternatively be characterized as the saturation of the set of regular parameters $\mathcal{R}$ with respect to the action of $W_\aff$, i.e.\
\[
\widetilde{\mathcal{R}} = \bigcup_{g \in 	W_{\mathrm{aff}}} g \cdot \mathcal{R} = \bigl \{ (\boldsymbol{\alpha}, \textbf{m}) \in \R^4 \times \C^4 : \exists g \in W_{\mathrm{aff}} \mbox{ such that } g \cdot (\boldsymbol{\alpha}, \textbf{m}) \in \mathcal{R} \bigr\}
\]
according to Lemma \ref{lem:saturation} below. In other words, $\widetilde{\mathcal{R}}$ consists precisely of those elements in $\widetilde{\mathcal{R}}^{\mathrm{full}}$ whose $W_{\mathrm{aff}}$-orbits intersect non-trivially with $(0, \frac{1}{2})^4 \times \C^4$. To achieve this, we throw out the points in $\widetilde{\mathcal{R}}$ such that $\boldsymbol{\alpha}$ is in the $W_{\mathrm{aff}}$-orbits of the intersection of the Biswas polytope and the boundary of the cube weight $[ 0, \frac{1}{2}]^4$, leading to the above definition. The point now is that $g \in W_\aff$ gives rise to the isotopy class of a diffeomorphism (in fact, a composition of Dehn twists), such that $g \cdot (\bfalpha, \bfm)$ can be asociated with a (re-)marking of the hyperK\"ahler manifold $\calM(\bfalpha, \bfm)$.

\smallskip

We now state our second main theorem.

\begin{thm}\label{thm:surjective}
The Torelli map $\cT_\fB: \widetilde{\mathcal{R}}^{\mathrm{full}} \to \calP$ is bijective. The image $\mathcal{T}_{\mathfrak{B}}(\widetilde{\mathcal{R}})\subset \calP$ of the open and dense subset $\widetilde{\mathcal{R}}\subset \widetilde{\mathcal{R}}^{\mathrm{full}}$ precisely corresponds to periods of Hitchin moduli spaces on the four-punctured sphere.
\end{thm}

As a consequence of Theorem \ref{thm:surjective}, the period map is not quite surjective on Hitchin moduli spaces since these correspond to the open and dense subset $\widetilde{\mathcal{R}} \subset \widetilde{\mathcal{R}}^{\mathrm{full}}$. As we write in Conjecture \ref{conj:notfullflags}, we believe that this is because there is a more natural definition of the Higgs bundle moduli space when the parabolic weights are such that the flags are not full, i.e. for some $p_i$, $\alpha_i \in \{0, \frac{1}{2}\}$ .

\subsection{Affine \texorpdfstring{$D_4$}{D4} Coxeter group acting on homology}\label{sec:homologyaction}

In this section, we describe an action of the affine $D_4$ Coxeter group $W_\aff$ on $H_2(\mathcal{M}(\bfalpha,\bfm), \mathbb{Z})$. This action is generated by certain orientation-preserving diffeomorphisms, so-called Dehn twists, and will hence in particular preserve the intersection form.

\subsubsection{The 4-dimensional Dehn twist}
If $X$ is an oriented 4-manifold and $S$ an embedded 2-sphere with self-intersection $-2$, the 4-dimensional Dehn twist along $S$ is an orientation-preserving diffeomorphism \[\mathrm{DT}_S : X \to X\] which is supported in a neighborhood of $S$ and acts on $S$ as the antipodal map. It can be constructed as follows \cite{Arnold,Seidel}: Since $I(S,S)=-2$, at tubular neighborhood of $S$ is diffeomorphic to the cotangent bundle $T^*S$. To obtain a model Dehn twist on $T^*S^2$ identify $TS^2 \cong T^*S^2$ using the round metric and consider the geodesic flow 
\[
\phi_t : UTS^2 \to UTS^2
\]
on the unit tangent bundle. Pick a smooth function $h: [0, \infty) \to \R$ satisfying $h(t) = 0$ for $t \gg 0$ and $h(t)= \pi$  for  $t$ near $0$. Then define $\tau : TS^2 \to TS^2$ by $\tau (tv) =t\phi_{h(t)}(v)$ for $v \in UTS^2$, $t \in [0, \infty)$. Clearly $\tau$ restricts to the antipodal map on $S^2$ if identified with the zero section in $TS^2$. Since $\tau$ is compactly supported it can be transplanted back into the manifold $X$ yielding $\mathrm{DT}_S$.

\subsubsection{Induced action on homology}

The action of a Dehn twist on $H_2(X,\Z)$ is described by the Picard-Lefschetz theorem, namely
\[
(\mathrm{DT}_S)_* (\alpha) = \alpha  +I(\alpha, S)[S]
\]
for $\alpha \in H_2(X,\Z)$, in particular $(\mathrm{DT}_S)_*[S] =- [S]$. If $\Sigma$ is another embedded 2-sphere with $I(\Sigma,S)=1$, then $(\mathrm{DT}_S)_*[\Sigma] = [\Sigma] + [S]$.

\subsubsection{Specialization to \texorpdfstring{$4$}{4}-punctured sphere}
Let $X= \mathcal{M}(\boldsymbol{\alpha}_0, \mathbf{0})$ for some fixed $(\boldsymbol{\alpha}_0, \mathbf{0}) \in \mathcal{R}$ and let $S_0, S_1, S_2, S_3, S_4$
be the canonically associated basis of $H_2(\cM(\boldsymbol{\alpha}_0, \mathbf{0}), \Z)$ with intersection form given by \eqref{eq:intersectionform}. Since all of these spheres have self-intersection $-2$, we can perform a Dehn twist along each of them. We describe the induced action on homology using the Picard-Lefschetz theorem as above. For $i=1, 2, 3, 4$ the Dehn twist acts as the reflection
  \begin{align*}
    r_i: \qquad [S_0] &\mapsto [S_0] + [S_i]\\
    [S_i] &\mapsto -[S_i]\\
    [S_j] &\mapsto [S_j], \qquad j \neq i.
  \end{align*}
   and for $i=0$, i.e.\ along the central sphere, the Dehn twist acts as the reflection
  \begin{align*}
    r_0: \qquad [S_0] &\mapsto -[S_0]\\
    [S_j] &\mapsto [S_0]+[S_j].
  \end{align*}
on $H_2(\cM(\boldsymbol{\alpha}_0, \mathbf{0}), \Z)$. It is easy to check that $r_0, r_1, r_2, r_3, r_4$ satisfy the relations defining $W_\aff$. Hence, we obtain a left action of $W_\aff$ on $H_2(\cM(\boldsymbol{\alpha}_0, \mathbf{0}), \Z)$ and by parallel extension on $H_2(\cM(\boldsymbol{\alpha}_, \mathbf{m}), \Z)$ for all $(\bfalpha, \bfm) \in \mathcal{R}$. This action depends on the choice of open chamber containing the initial $\bfalpha_0$ and induces a left action on bases $\mathfrak{B}$ of $H_2(\cM(\boldsymbol{\alpha}, \mathbf{m}), \Z)$. Since the action preserves the fiber class $2[S_0] + \sum_{i=1}^4 [S_i]$, it transforms admissible bases into admissible bases. 
\begin{rem}\label{rem:compatibility}
The right action of $\OO(\Z^5,I_0)^+$ on admissible bases of $H_2(\cM(\boldsymbol{\alpha}_, \mathbf{m}), \Z)$ on the other hand does not depend on any choices and is simply transitive. Now fix a chamber $\Delta$ determining the basis $\mathfrak{B}(\Delta)$ and a left action of $W_\aff$ on admissible bases as above. Then for $g \in W_\aff$ there will be a unique element $A^g \in \OO(\Z^5,I_0)^+$ such that $g \cdot \mathfrak{B}(\Delta) = \mathfrak{B}(\Delta) \cdot A^g$. The assignment $g \mapsto A^g$ gives rise to a representation $W_\aff \to \OO(\Z^5, I_0)^+$.
\end{rem}

\subsection{Affine \texorpdfstring{$D_4$}{D4} Coxeter group acting on domain \texorpdfstring{$\R^4$}{R4}}\label{sec:R4action}

In the previous section, we saw how $W_{\mathrm{aff}}$ acted on admissible bases of $H_2(\cM(\bfalpha,\bfm),\Z)$. In this section we will also define an action of $W_{\mathrm{aff}}$ on $\R^4$ parameterized by $\boldsymbol{\alpha}=(\alpha_1,\alpha_2,\alpha_3,\alpha_4)$. We will see that this action can be extended to an action of $W_\aff$ on $\R^4_{\bfalpha} \times \C^4_\bfm$.

    There is a classical construction of a space $\mathcal{U}(W_{\mathrm{aff}}, \Simplex)$ with a $W_{\mathrm{aff}}$-action by pasting together copies of a space $\Simplex$ --- one for each element of $W_{\mathrm{aff}}$. Taking $W_{\mathrm{aff}}$ to be the affine $D_4$ Coxeter group, it is a fact (see the general result in \cite[Theorem 6.4.3]{Davis}) that if the matrix of dihedral angles of $\Simplex$ is the Coxeter matrix of $W_{\mathrm{aff}}$, then the natural map $\iota: \mathcal{U}(W_{\mathrm{aff}}, \Simplex) \to \R^4$ is a homeomorphism. Here, we take $\Simplex$  to be  one of the chambers of $[0, \frac{1}{2}]^4 \ni \boldsymbol{\alpha}$. In particular, take the interior chamber of type B1 whose vertices are $v_0 =(\frac{1}{4}, \frac{1}{4}, \frac{1}{4}, \frac{1}{4})$ and $v_I$, as defined in \eqref{eq:vI}, for $I \in \mathcal{I}=\{\emptyset, \{1, 2\}, \{1, 3\}, \{1, 4\}\}$. Then each of the $24$ chambers is the image of $g \cdot \Simplex$ for some $g \in W_{\mathrm{aff}}$. 
   
  \begin{prop}\label{prop:homeo}
  Consider the simplex $\Simplex$ in $\R^4$ with vertices formed by $v_1=(0, 0, 0, 0)=v_{\emptyset}, v_2=(\frac{1}{2}, \frac{1}{2}, 0, 0)=v_{\{1, 2\}}, v_3=(\frac{1}{2}, 0, \frac{1}{2}, 0)=v_{\{1, 3\}}, v_4=(\frac{1}{2}, 0, 0, \frac{1}{2})=v_{\{1, 4\}}$, and $v_0=(\frac{1}{4}, \frac{1}{4}, \frac{1}{4}, \frac{1}{4})$. 
  Let $f_i$ be the face of the simplex containing all vertices except $v_i$. Let $r_0, r_1, r_2, r_3, r_4$ be the reflections in $f_0, f_1, f_2, f_3, f_4$. These generate an action of the affine $D_4$ Coxeter group $W_{\mathrm{aff}}$ on $\R^4$. The natural map
$\iota: \mathcal{U}(W_{\mathrm{aff}}, \Simplex) \to \R^4$ is a homeomorphism. 
  \end{prop}
 
 \begin{rem}
 The action of $W_{\mathrm{aff}}$ on $\R^4$ depends on the choice of a chamber $\Delta \subset [0, \frac 12]^4$ together with an ordering of its faces. The actions corresponding to different chambers $\Delta$ and $\Delta'$ are conjugate within the group of affine transformations of $\R^4$.
\end{rem}
   
 \begin{rem} \label{rem:Nakajima} The faces of the model simplex $\Delta$ are restrictions of Nakajima walls, namely
\begin{enumerate}
\item $f_0=\{\bfm = \mathbf{0}\} \cap \{ -\alpha_1+\alpha_2+\alpha_3+\alpha_4 = 0, -m_1+m_2+m_3+m_4=0\} $	
\item $f_1=\{\bfm = \mathbf{0}\}\cap \{1- \alpha_1-\alpha_2-\alpha_3 - \alpha_4 = 0, -m_1-m_2-m_3-m_4=0\}$
\item $f_2=\{\bfm = \mathbf{0}\} \cap \{\alpha_1+\alpha_2-\alpha_3-\alpha_4=0, m_1+m_2-m_3-m_4 =0\}$
\item $f_3=\{\bfm = \mathbf{0}\} \cap\{\alpha_1-\alpha_2+\alpha_3-\alpha_4=0, m_1-m_2+m_3-m_4 =0\}$
\item $f_4=\{\bfm = \mathbf{0}\} \cap\{\alpha_1-\alpha_2-\alpha_3+\alpha_4=0, m_1-m_2-m_3+m_4 =0\}$
\end{enumerate}
 
\end{rem}

\begin{proof}[Remarks on Homeomorphism in \ref{prop:homeo}]\hfill
  \begin{itemize}
    \item Each ``alcove'' of $\mathcal{U}(W_{\mathrm{aff}}, \Simplex)$ is then labelled by an element of $W_{\mathrm{aff}}$. The induced action of $W_{\mathrm{aff}}$ on $\mathbb{R}^4$ is the natural one: if $\bfalpha \in \R^4$ is in the alcove $\Simplex_g$ labelled by $g \in W_{\mathrm{aff}}$, and $g^{-1}\bfalpha$ is the corresponding point in $\Simplex=\Simplex_{\mathrm{id}}$, then for any $g' \in W_{\mathrm{aff}}$, $g'\cdot \bfalpha := g'g \cdot (g^{-1} \bfalpha)$.
    \item 
The faces $f_0, f_2, f_3, f_4$ are respectively perpendicular to the positive simple roots
\begin{align*}
\alpha_{(0)}&=\left(-\frac{1}{2},  \frac{1}{2},  \frac{1}{2}, \frac{1}{2}\right)\\
\alpha_{(2)}&=\left(\frac{1}{2},  \frac{1}{2}, - \frac{1}{2}, -\frac{1}{2}\right)\\
\alpha_{(3)}&= \left(\frac{1}{2}, - \frac{1}{2}, \frac{1}{2}, -\frac{1}{2}\right)\\
\alpha_{(4)}&=\left(\frac{1}{2}, - \frac{1}{2}, - \frac{1}{2}, \frac{1}{2}\right),
\end{align*}
while $f_1$ is relatively perpendicular to the highest root
\begin{align}
\alpha_{(1)}:= 2 \alpha_{(0)} + \alpha_{(2)} +\alpha_{(3)}+ \alpha_{(4)}=\left(\frac{1}{2},  \frac{1}{2},  \frac{1}{2}, \frac{1}{2}\right),
\end{align}
and intersects at its midpoint. One can easily check that
 the dihedral angles between $f_i$ and $f_j$  are $\frac{\pi}{m_{ij}}$ for $m_{ij}$ given by the Coxeter matrix $M$ of the affine $D_4$ Coxeter group above, as required by the \cite[Theorem 6.4.3]{Davis}. 
  
\item In the associated naming convention, we've labelled the central node of the affine $D_4$ Dynkin diagram $`0'$ and then labelled the exterior nodes $`1', `2', `3', `4'$, as shown in Figure \ref{fig:Coxeter}. 
 One finite subgroup of $W_{\mathrm{aff}}$ is the group $W_{\mathrm{fin}}$ generated by $r_0, r_2, r_3, r_4$. Note that each of these reflections preserves the origin in $\R^4$. This is the Coxeter group associated to the ordinary $D_4$ Dynkin diagram, which sits inside the affine $D_4$ Dynkin diagram as the set of nodes $`0', `2, `3', `4'$. It has $192$ elements.

 A corollary of $\mathcal{U}(W_{\mathrm{aff}}, \Simplex) \simeq \R^4$ is that 
there is a bijection between elements of $W_{\mathrm{fin}}$ and alcoves of $\mathcal{U}(W_{\mathrm{aff}}, \Simplex) \simeq \R^4$ which contain the origin in $\R^4$. 
This count is related to our count of chambers of $[0, \frac{1}{2}]^4$ as follows. Observe that within the cube $[0, \frac{1}{2}]^4$, exactly half of the $24$ chambers contain the vertex $(0, 0, 0, 0)$. 
 The group of reflections perpendicular to the usual axes in $\R^4$ is a subgroup of $W_{\mathrm{fin}}$, consequently, the planes $x_i=0$ for $i=1, 2, 3, 4$ are all walls of the alcoves. 
Consequently, in total, there are $2^4 \cdot 12=192$ alcoves meeting at $(0,0, 0, 0)$, and indeed this is equal to $|W_{\mathrm{fin}}|$.
\end{itemize}
\begin{ex}\label{ex:action}
  Given a chamber in $[0, \frac{1}{2}]^4$ the corresponding element of $W_{\mathrm{aff}}$ can be determined by changing out the vertices one by one, staying within the interior chambers as long as possible for ease. Consider the exterior chamber labelled by $I_0=\{1, 2, 3\}$. The even partition set is $\mathcal{I}'=\{\{1, 2\}, \{1, 3\}, \{2,3\}, \{1, 2, 3,4\}\}$.
We find that $g= r_0 r_1 r_4$ via the following steps: 
\[
\begin{array}{c|ccccc}
  & r_0 & r_1 & r_2 & r_3 & r_4 \\
  \hline
 \mathrm{id} & \mathrm{central}  &\emptyset & \{1, 2\} & \{1, 3\} & \{1, 4\}  \\
  r_4 & \mathrm{central} & \emptyset &  \{1, 2\} & \{1, 3\} & \{2, 3\}  \\  
  r_1r_4 & \mathrm{central} &
  \{1, 2, 3, 4\} & \{1, 2\} & \{1, 3\} & \{2, 3\}  \\
    r_0r_1r_4 & \{1, 2, 3\}  &
  \{1, 2, 3, 4\}  & \{1, 2\} & \{1, 3\} & \{2, 3\}  \\
\end{array}
\]
The column entries for $r_1, r_2, r_3, r_4$ are the sets in the even partition set $\mathcal{I}$. In particular if $I$ is in column $r_i$, then action by $r_i$ will replace $I$ by $\{1, 2, 3, 4\}-I$, leaving all others alone, representing the reflection replacing the vertex $v_I$ with $v_{\{1, 2, 3, 4\} - I}$. The column entries for $r_0$ are either $\mathrm{central}$,  corresponding to the distinguished vertex $(\frac{1}{4}, \frac{1}{4}, \frac{1}{4}, \frac{1}{4})$ in interior chambers, or the odd subset $I_0$ corresponding to the distinguished vertex $v_{I_0}$ in exterior chambers.
\end{ex}
\end{proof}

\begin{cor}\label{cor:orbits} 
There are $5$ orbits of the vertices of $[0, \frac{1}{2}]^4$ under the $W_{\mathrm{aff}}$ action. We label these by the sets $I \subset \{1, 2, 3, 4\}$ corresponding to $v_I$ in \eqref{eq:vI}.
\begin{itemize}
\item $\{1, 2\} \leftrightarrow \{3, 4\}$
\item $\{1, 3\} \leftrightarrow \{2, 4\}$
\item $\{2, 3\} \leftrightarrow \{1, 4\}$
\item $\emptyset \leftrightarrow \{1,2 3, 4\}$
\item The $8$ subsets of odd cardinality
\end{itemize}
\end{cor}
\begin{cor}
The faces of $[0, \frac{1}{2}]^4$ are identified by $\boldsymbol{\alpha} \mapsto (\frac{1}{2}, \frac{1}{2}, \frac{1}{2}, \frac{1}{2})-\boldsymbol{\alpha}$.
\end{cor}

\begin{rem}\label{rem:Wallextension}
The above corollary means that we can use $W_{\mathrm{fin}}$ to fill out the box $[-\frac{1}{2}, \frac{1}{2}]^4$ and then tile $\R^4$ by unit shifts. 
\end{rem}

\begin{rem}[Extension to $\R^4 \times \C^4$]\label{rem:extension}
The group $W_{\mathrm{aff}}$ acts on $\R^4 \times \C^4$ as follows: Given a reflection $r_i: \R^4 \to \R^4$, define $r'_i$ to be its $\R$-linear part, i.e.\ the reflection in the parallel $3$-plane through the origin. Complex-linear extension yields a $\C$-linear map $\C^4 \to \C^4$, again denoted by $r_i'$. Then the action is given by $(r_i, r_i')$ for $i=0, \ldots, 4$. The subset $\Simplex \times \C^4$ provides a fundamental domain for this action. Note that by linearity the origin in $\C^4$ is preserved, i.e.\ strongly parabolic bundles are distinguished. 
\end{rem}

Using the formula $\boldsymbol{\alpha} \mapsto (\mathrm{Id} - 2 \mathbf{n} \mathbf{n}^T) \boldsymbol{\alpha} + 2 \mathbf{n} \mathbf{n}^T \mathbf{b}$ for the reflection in the hyperplane $(\boldsymbol{\alpha} - \mathbf{b})^T \cdot  \mathbf{n}=0$ through parallel vector $\mathbf{b}$ with unit normal $\mathbf{n}$ we obtain 

\begin{lem}\label{lem:reflections}
The reflections $r_0, r_1, r_2, r_3, r_4$ in the faces $f_0, f_1, f_2, f_3, f_4$ of the model chamber $\Delta$ are given as follows:
\[
\begin{gathered}
	r_0(\bfalpha) = \frac{1}{2} \begin{pmatrix}
 1 & 1 &1 & 1\\
 	1 & 1 & -1 & -1\\
 	1 & -1 &1 & 1\\
 	1 & -1 &-1 & 1
 \end{pmatrix}
\begin{pmatrix}
\alpha_1 \\ \alpha_2 \\ \alpha_3 \\ \alpha_4	
\end{pmatrix}, \quad 
r_2(\bfalpha) = \frac{1}{2} \begin{pmatrix}
 1 & -1 &1 & 1\\
 	-1 & 1 & 1 & 1\\
 	1 & 1 &1 & -1\\
 	1 & 1 &-1 & 1
 \end{pmatrix}
\begin{pmatrix}
\alpha_1 \\ \alpha_2 \\ \alpha_3 \\ \alpha_4	
\end{pmatrix},\\
r_3(\bfalpha) = \frac{1}{2} \begin{pmatrix}
 1 & 1 & -1 & 1\\
 	1 & 1 & 1 & -1\\
 	-1 & 1 &1 & 1\\
 	1 & -1 &1 & 1
 \end{pmatrix}
\begin{pmatrix}
\alpha_1 \\ \alpha_2 \\ \alpha_3 \\ \alpha_4	
\end{pmatrix}, \quad 
r_4(\bfalpha) = \frac{1}{2} \begin{pmatrix}
 1 & 1 &1 & -1\\
 	1 & 1 & -1 & 1\\
 	1 & -1 &1 & 1\\
 	-1 & 1 &1 & 1
 \end{pmatrix}
\begin{pmatrix}
\alpha_1 \\ \alpha_2 \\ \alpha_3 \\ \alpha_4	
\end{pmatrix}
\end{gathered}
\]
and
\[
r_1(\bfalpha) = \frac{1}{2} \begin{pmatrix}
 1 & -1 &-1 & -1\\
 	-1 & 1 & -1 & -1\\
 	-1 & -1 &1 & -1\\
 	-1 & -1 &-1 & 1
 \end{pmatrix}
\begin{pmatrix}
\alpha_1 \\ \alpha_2 \\ \alpha_3 \\ \alpha_4	
\end{pmatrix}
+
\begin{pmatrix}
\frac12 \\ \frac12 \\ \frac12 \\ \frac12	
\end{pmatrix}
\]
\end{lem}

\subsection{Affine \texorpdfstring{$D_4$}{D4} Coxeter group acting on target \texorpdfstring{$\R^4$}{R4}}\label{subsec:target_action}

We now describe the $W_{\mathrm{aff}}$ action on the target $\R^4$ parameterized by $\mathbf{x}=(x_1,x_2,x_3,x_4)$. Again we will see that this action can be extended to an action of $W_\aff$ on $\R^4_\mathbf{x} \times \C^4_\mathbf{z}$.

Recall that we defined the action of $W_{\mathrm{aff}}$ on the domain $\R^4$ starting from a fixed interior chamber $\Simplex$ corresponding to 
$\mathcal{I}=\{\emptyset, \{1, 2 \}, \{1, 3\}, \{1,4 \} \}$. Observe that the associated Torelli map is 
\[ \mathcal{T}^{\omega_I}_{\mathfrak{B}(\Simplex)}(\boldsymbol{\alpha})=\begin{pmatrix} 4 \pi^2 K_{\emptyset}(\boldsymbol{\alpha})\\ 4 \pi^2 K_{\{1,2\}}(\boldsymbol{\alpha})\\4 \pi^2 K_{\{1, 3\}}(\boldsymbol{\alpha})\\4 \pi^2 K_{\{1, 4\}}(\boldsymbol{\alpha})\end{pmatrix}, \]
with order determined by the labelling of the vertices of $\Simplex$. The images of the five vertices are
\begin{equation}
\begin{array}{c  l l c }
v_0&:=\left(\frac{1}{4}, \frac{1}{4}, \frac{1}{4}, \frac{1}{4} \right) &\mapsto ( 0, 0, 0, 0)&=:V_0\\
v_1&:=v_{\emptyset} &\mapsto  (4\pi^2, 0, 0, 0)&=:V_1\\
v_2&:=v_{\{1, 2\}} &\mapsto (0, 4\pi^2, 0, 0)&=:V_2 \\
v_3&:=v_{\{1, 3\}} &\mapsto (0, 0, 4\pi^2, 0)&=:V_3\\
v_4&:=v_{\{1, 4\}} &\mapsto (0, 0, 0, 4\pi^2)&=:V_4
\end{array}
\end{equation}
As before, after checking the angles we again have: 
\begin{prop}\label{Prop:target_action}
Consider the simplex $\Simplex_{\mathrm{target}}$ formed by $V_0, V_1, V_2, V_3, V_4$ in the target $\R^4$. Let $R_i$ by the reflection in the face $F_i$ of the simplex $\Simplex_{\mathrm{target}}$ that does not contain $V_i$. The reflections $R_0, R_1, R_2, R_3, R_4$ generate an action of the affine $D_4$ Coxeter group $W_{\mathrm{aff}}$ on $\R^4$. The natural map $\iota: 
\mathcal{U}(W_{\mathrm{aff}},\Simplex_{\mathrm{target}}) \simeq \R^4$ is a homeomorphism.
\end{prop}

\begin{rem}
Note that $\Simplex_{\mathrm{target}}$ is the image of the Torelli map for any choice of chamber $\Simplex$ if we compute with respect to the chamber basis $\mathfrak{B}(\Simplex)$. In particular, the action of $W_\aff$ on the target $\R^4$ does not depend on the choice of a chamber.
\end{rem}
\begin{rem}[Extension to $\R^4 \times \C^4$] 
As before, the action of $W_\aff$
extends to the target $\R^4 \times \C^4$ by extending the linear part of $R_i$ to a $\C$-linear map $R_i' :\C^4 \to \C^4$. The action is then given by $(R_i, R_i')$ for $i=0, \ldots, 4$. The subset $\Simplex_\mathrm{target} \times \C^4$ provides a fundamental domain for this action.
\end{rem}
 
We can map the abstract generators $r_i \in W_{\mathrm{aff}}$ to the matrices  $A_i \in \OO(\Z^5,I_0)^+$ representing the Dehn twist in the basis vector $e_i \in \Z^5$. More concretely,
\[
\begin{gathered}
A_0 = \begin{pmatrix} -1 & 1 & 1 & 1 & 1 \\
 0 & 1 & 0 & 0 & 0 \\
 0 & 0& 1 & 0 & 0 \\
 0 & 0& 0 & 1 & 0 \\
 0 & 0& 0& 0&1 	
 \end{pmatrix}, \quad
 A_1 = \begin{pmatrix} 1 & 0 & 0 & 0 & 0 \\
 1 & -1 & 0 & 0 & 0 \\
 0 & 0& 1 & 0 & 0 \\
 0 & 0& 0 & 1 & 0 \\
 0 & 0& 0& 0&1 	
 \end{pmatrix}, \quad
 A_2 = \begin{pmatrix} 1 & 0 & 0 & 0 & 0 \\
 0 & 1 & 0 & 0 & 0 \\
 1 & 0& -1 & 0 & 0 \\
 0 & 0& 0 & 1 & 0 \\
 0 & 0& 0& 0&1 	
 \end{pmatrix}\\
 A_3 = \begin{pmatrix} 1 & 0 & 0 & 0 & 0 \\
 0 & 1 & 0 & 0 & 0 \\
 0 & 0& 1 & 0 & 0 \\
 1 & 0& 0 & -1 & 0 \\
 0 & 0& 0& 0&1 	
 \end{pmatrix}, \quad
 A_4 = \begin{pmatrix} 1 & 0 & 0 & 0 & 0 \\
 0 & 1 & 0 & 0 & 0 \\
 0 & 0& 1 & 0 & 0 \\
 0 & 0& 0 & 1 & 0 \\
 1 & 0& 0& 0& -1 	
 \end{pmatrix}.
\end{gathered}
\]
This yields a representation $W_{\rm{aff}} \to \OO(\Z^5,I_0)^+$, $g \mapsto A^g$ and hence a right action of $W_{\rm{aff}}$ on bases via $\mathfrak{B} \mapsto \mathfrak{B} \cdot A^g$ and on $\R^4 \times \C^4$ via $(\mathbf{x},\mathbf{z}) \mapsto (\mathbf{x},\mathbf{z}) \cdot A^g$. Recall that the right action of   $\OO(\Z^5,I_0)^+$ on $\R^4 \times \C^4$ is given by $(\mathbf{x},\mathbf{z}) \cdot A = (\hat{A}^t \cdot \mathbf{x} + \hat b, \hat{A}^t \cdot \mathbf{z})$ for $\hat{A} \in \GL(4,\R)$ and $\hat{b} \in \R^4$ as defined earlier in Section \ref{subsec:Torelli_map}.

\begin{lem}\label{lem:affine}
	The affine transformations $\mathbf{x} \mapsto \hat{A}_i^t \cdot \mathbf{x} + \hat{b}_i$ are precisely the reflections $R_i$ in the faces $F_0, F_1, F_2, F_3, F_4$ of the simplex $\Delta_{\mathrm{target}}$.
\end{lem}

\begin{proof}
From the matrices $A_0, A_1, A_2, A_3, A_4$ above we obtain
\[
\begin{gathered}
\hat{A}_0 = \frac{1}{2}\begin{pmatrix}
 1 & -1 & -1 & -1 \\
 -1 & 1 & -1 & -1 \\
 -1 & -1 & 1 & -1 \\
 -1 & -1 & -1 & 1 	
\end{pmatrix},\quad
	\hat{A}_1 = \begin{pmatrix}
 -1 & 0 & 0 & 0 \\
 0 & 1 & 0 & 0 \\
 0 & 0 & 1 & 0 \\
 0 & 0 & 0 & 1 	
\end{pmatrix}, \quad
\hat{A}_2 = \begin{pmatrix}
 1 & 0 & 0 & 0 \\
 0 & -1 & 0 & 0 \\
 0 & 0 & 1 & 0 \\
 0 & 0 & 0 & 1 	
\end{pmatrix}\\
\hat{A}_3 = \begin{pmatrix}
 1 & 0 & 0 & 0 \\
 0 & 1 & 0 & 0 \\
 0 & 0 & -1 & 0 \\
 0 & 0 & 0 & 1 	
\end{pmatrix}, \quad
\hat{A}_4 = \begin{pmatrix}
 1 & 0 & 0 & 0 \\
 0 & 1 & 0 & 0 \\
 0 & 0 & 1 & 0 \\
 0 & 0 & 0 & -1 	
\end{pmatrix}
\end{gathered}
\]
and
\[
 \hat{b}_0 = 2 \pi^2 \begin{pmatrix}
 1 \\ 1 \\ 1 \\1 	
 \end{pmatrix}, \quad \hat{b}_1 = \hat{b}_2= \hat{b}_3 = \hat{b}_4=\begin{pmatrix}
 0 \\ 0 \\ 0 \\0 	
 \end{pmatrix}.
\]
From this we read off that $R_i(\mathbf{x})= \hat{A}_i^t \cdot \mathbf{x} + \hat{b}_i$ for $i=0, \ldots, 4$. 	
\end{proof}

\begin{cor}\label{cor:compatibility}
We recover the left action of $W_{\mathrm{aff}}$	on $\R^4 \times \C^4$ defined in Section \ref{subsec:target_action} via $g \cdot (\mathbf{x}, \mathbf{z}) = (\mathbf{x}, \mathbf{z}) \cdot A^{g^{-1}}$ for all $g \in W_{\mathrm{aff}}$.
\end{cor}

\subsection{Equivariance of the Torelli map} 
As stated in the introduction, we computed the full Torelli map $\mathcal{T}_{\mathfrak{B}(\Simplex)}: (\Simplex \times \C^4) \cap \mathcal{R}  \to \R^4 \times \C^4$ only for certain chambers $\Delta$. We have introduced (left) actions of $W_{\aff}$ on the domain $\R^4 \times \C^4$, on the target $\R^4 \times \C^4$ and on the set of admissible homology bases in the previous section. In this section we will establish various equivariance properties of the Torelli map with respect to these actions.

\begin{figure}[ht] 
\begin{centering} 
\includegraphics[height=2.0in]{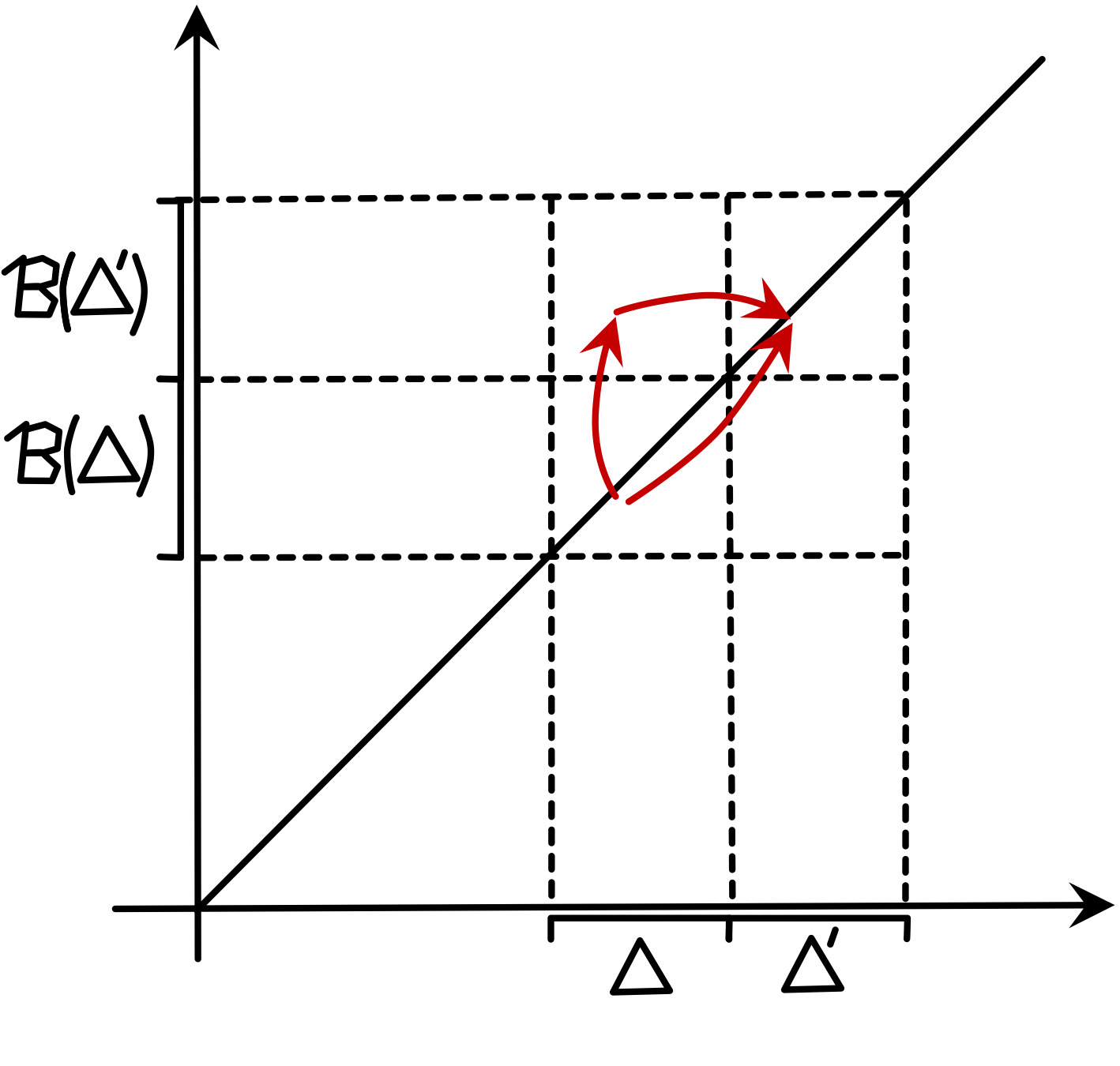}
\caption{\label{fig:groupaction}  $W_{\mathrm{aff}}$ action.}
\end{centering}
\end{figure}

\subsubsection{\texorpdfstring{$W_{\mathrm{aff}}$}{Affine D4 Coxeter group}-invariance for the integral of \texorpdfstring{$\omega_I$}{the real symplectic form}}

In this section we establish the invariance of the Torelli map on strongly parabolic moduli spaces under the action of the affine Coxeter group $W_{\mathrm{aff}}$ acting on the weight cube $[0, \frac{1}{2}]^4 \subset \R^4$. 
We have 
\[ \mathcal{T}^{\omega_I}: \mathrm{Bases} \times \R^4 \to \R^4, \quad (\mathfrak{B}, \bfalpha) \mapsto  \mathcal{T}^{\omega_I}_{\mathfrak{B}}(\bfalpha)\]
and each of these three spaces has a $W_{\mathrm{aff}}$-action.

\bigskip

\begin{prop}\label{prop:equivariance}
Let $\Delta$ be the model chamber with associated basis $\mathfrak{B}(\Delta)$ and ensuing  left action of $W_\aff$ on admissible bases and on $\R^4_{\bfalpha}$. The period map $\mathcal{T}^{\omega_I}$ has the equivariance properties:
\begin{enumerate}
\item $\mathcal{T}_{\mathfrak{B}(\Delta)}^{\omega_I}(g \cdot \bfalpha) = g \cdot \mathcal{T}_{\mathfrak{B}(\Delta)}^{\omega_I}(\bfalpha)$ 
\smallskip
\item $\mathcal{T}_{g \cdot \mathfrak{B}(\Delta)}^{\omega_I}(\bfalpha) = g^{-1} \cdot \mathcal{T}_{\mathfrak{B}(\Delta)}^{\omega_I}(\bfalpha)$
\end{enumerate}
for $g \in W_\aff$ and $\bfalpha \in \R^4$.
\end{prop}

\begin{proof}
(1) In order to prove the first equivariance of the period map, we need to show that it is equivariant for all $g \in W_{\mathrm{aff}}$ and for every element of the domain.  It suffices to check that it is equivariant for the generators of the group action. This is clear, since $\mathcal{T}^{\omega_I}_{\mathfrak{B}(\Simplex)}$ maps the faces $f_i$ of $\Simplex$ to the faces $F_i$ of $\Simplex_{\mathrm{target}}$ and the reflections $r_i$ and $R_i$ in these faces generate the respective $W_\aff$-actions on $\R^4_{\bfalpha}$ and on $\R^4_{\mathbf{x}}$.
Alternatively, using Lemma \ref{lem:reflections} and Lemma \ref{lem:reflections} we may directly compute that 
\[ \mathcal{T}^{\omega_I}_{\mathfrak{B}(\Simplex)}(\boldsymbol{\alpha})=\begin{pmatrix} 4 \pi^2 K_{\emptyset}(\boldsymbol{\alpha})\\ 4 \pi^2 K_{\{1,2\}}(\boldsymbol{\alpha})\\4 \pi^2 K_{\{1, 3\}}(\boldsymbol{\alpha})\\4 \pi^2 K_{\{1, 4\}}(\boldsymbol{\alpha})\end{pmatrix}
\]
satisfies the relation
\[ 
\mathcal{T}^{\omega_I}_{\mathfrak{B}(\Simplex)} \circ r_i = R_i \circ \mathcal{T}^{\omega_I}_{\mathfrak{B}(\Simplex)}
\]
for all $i=0, \ldots, 4$.
\smallskip

\noindent (2) Now we show the second equivariance. By Remark \ref{rem:compatibility} we have $g \cdot \mathfrak{B}(\Delta) = \mathfrak{B}(\Delta) \cdot A^g$  for $A^g \in  \OO(\Z^5, I_0)^+$. Then \[
\mathcal{T}_{g \cdot \mathfrak{B}(\Delta)}^{\omega_I}(\bfalpha) = \mathcal{T}_{\mathfrak{B}(\Delta) \cdot A^g}^{\omega_I}(\bfalpha)=\mathcal{T}_{\mathfrak{B}(\Delta)}^{\omega_I}(\bfalpha) \cdot A^g = g^{-1} \cdot \mathcal{T}_{\mathfrak{B}(\Delta)}^{\omega_I}(\bfalpha)
\]
by Lemma \ref{lem:period_equivariance} and Corollary \ref{cor:compatibility}. 
\end{proof}

Recall from Section \ref{subsec:period_domain} that the domain $\mathcal R \subset (0, \frac 12 )^4 \times \C^4$ of generic parameters $(\boldsymbol{\alpha}, \mathbf{m})$ is the complement of an arrangement of codimension-three planes which intersect $(0,\frac12)^4 \times \{0\}$ precisely in the chamber walls. Hence any basis of $H_2(\mathcal{M}(\boldsymbol{\alpha}, \mathbf{m}),\Z)$ may be moved by parallel transport to a basis of $H_2(\mathcal{M}(\boldsymbol{\alpha}', \mathbf{m}'),\Z)$ for any other generic $(\boldsymbol{\alpha}', \mathbf{m}')$.

\begin{lem}\label{lem:basiscomparison}
If $\mathfrak{B}= \mathfrak{B}(\Simplex)$ is the preferred homology basis for chamber $\Simplex$ moved by parallel transport to chamber $\Simplex'=g \cdot \Simplex$ for $g \in W_{\mathrm{aff}}$, then $g \cdot \mathfrak{B} = \mathfrak{B}'$ where $\mathfrak{B}'= \mathfrak{B}(\Simplex')$ is the preferred homology basis for chamber $\Simplex'$. In particular, if $\mathfrak{B}$ is a parallel admissible homology basis, then the Torelli map \[\mathcal{T}_{\mathfrak{B}}^{\omega_I}: (\R^4 \times \{ \mathbf{0}\}) \cap \mathcal{R} \to \R^4\] is affine-linear.
\end{lem}

\begin{proof}
We first show that the periods of $\omega_I$ in chamber $\Simplex'$ computed with respect to bases $g \cdot \mathfrak{B}$ and $\mathfrak{B}'$ are the same, i.e.\ $\mathcal{T}^{\omega_I}_{g \cdot \mathfrak{B}}(\bfalpha) = \mathcal{T}^{\omega_I}_{ \mathfrak{B}'}(\bfalpha)$ for all $\bfalpha \in \Simplex'$. Using Proposition \ref{prop:equivariance} this amounts to proving the identity $\mathcal{T}^{\omega_I}_{\mathfrak{B}} = g \cdot \mathcal{T}^{\omega_I}_{\mathfrak{B}'}$ for all $g \in W_\aff$. It is clearly enough to check this identity for the generators $r_0, \ldots, r_4$ of $W_\aff$. Then $\Simplex$ and $\Simplex'$ are adjacent and we need to examine how the formula for the periods of $\omega_I$  in Proposition \ref{prop:stronglyparabolicvolumesv2} (see also Corollary \ref{cor:centralvolumes}) behave under wall-crossing. 

If $\Simplex$ and $\Simplex'$ are adjacent interior chambers then the associated even partition sets satisfy $\mathcal{I}\backslash\{I\}= \mathcal{I}'\backslash \{I^c\}$ according to Lemma \ref{lem:adjacentchambers}. Then $K_{I^c} = - K_I$ and so $\int_{S_{I^c}} \omega_I = - \int_{S_I} \omega_I$. Since all other periods remain unaffected by this type of wall-crossing, we have shown that $\mathcal{T}^{\omega_I}_{\mathfrak{B}} = R_i \cdot \mathcal{T}^{\omega_I}_{\mathfrak{B}'}$ for $i=1, \ldots, 4$. If $\Delta$ is an interior chamber with adjacent exterior chamber $\Delta'$ labelled by the odd order subset $I_0$, then $\int_{S_0} \omega_I = - 4 \pi^2 K_{I_0} = - \int_{S_{I_0}} \omega_I$. It follows that $\mathcal{T}^{\omega_I}_{\mathfrak{B}} = R_0\cdot \mathcal{T}^{\omega_I}_{\mathfrak{B}'}$.

We claim that this implies that $g \cdot \mathfrak{B} = \mathfrak{B}'$. This and the final claim follow from the following general observation: Two local parallel admissible bases $\mathfrak{B}=(S_0, \ldots, S_4)$ and $\mathfrak{B'}=(S_0', \ldots, S_4')$ differ by an element $A \in \OO(\Z^5, I_0)^+$, and hence by Lemma \ref{lem:period_equivariance}
\[
\mathcal{T}_{\mathfrak{B}'}^{\omega_I}( \boldsymbol{\alpha})= \mathcal{T}_{\mathfrak{B}}^{\omega_I}( \boldsymbol{\alpha})\cdot A
\Longleftrightarrow 
\begin{pmatrix}
\int_{S_0'} \omega_I \\ \vdots \\ \int_{S_4'} \omega_I  
\end{pmatrix}
= A^t \cdot
\begin{pmatrix}
\int_{S_0} \omega_I \\ \vdots \\ \int_{S_4} \omega_I 	
\end{pmatrix} 
\]
for all $\boldsymbol{\alpha}$ in a given chamber. If in addition 
$\int_{S_i'} \omega_I = \int_{S_i} \omega_I$ for $i=0, \ldots ,4$ and all $\boldsymbol{\alpha}$, then $A$ has to be the identity matrix since the image of the period map with respect to the basis $\mathfrak{B}$ on the whole chamber contains a basis of $\R^5$. 
\end{proof}

\subsubsection{\texorpdfstring{$W_{\mathrm{aff}}$}{Affine D4 Coxeter}-invariance for the integral of \texorpdfstring{$\Omega_I$}{the holomorphic symplectic form}}
In contrast to the real symplectic form $\omega_I$, where we first computed the periods in all 24 chambers and then observed certain equivariance properties with respect to the action of $W_\aff$, we have only computed the periods of the complex symplectic form $\Omega_I$ in two chambers so far. However, we found that the periods of $\Omega_I$ are independent of $\bfalpha$ and can be computed from intersection numbers. More precisely, if $\mathfrak{B}=(S_0, S_1, S_2, S_3, S_4)$ is an admissible homology basis, then by Corollary \ref{cor:intersection_numbers} 
\begin{equation}\label{eq:linear}
\int_{S_i} \Omega_I = -2 \pi \sum_{j=1}^4 m_j (I(S_i, \sigma_j^+) - I(S_i, \sigma_j^-)) =  -2 \pi \sum_{j=1}^4 m_j  I(S_i, \Sigma_j). 
\end{equation}
The polar sections $\sigma_j^\pm$ and hence the homology classes $\Sigma_j \in H_2( \cM(\bfalpha,\bfm),\Z)$ are again independent of $\boldsymbol{\alpha}$. The periods $\int_{S_i} \Omega_I$ are thus linear in the complex masses $m_j$ with coefficients given by the intersection numbers $I(S_i, \Sigma_j)$. These are locally constant on $\mathcal R$, hence constant since $\mathcal R$ is connected. Consequently, we obtain the following

\begin{lem}\label{lem:linear}
If $\mathfrak{B}$ is a parallel admissible homology basis, the Torelli map \[\cT^{\Omega_I}_{\mathfrak{B}}:( \{\mathbf{0}\} \times \C^4) \cap \mathcal{R} \to \C^4\] is complex-linear. The coefficients are given by intersection numbers as in \eqref{eq:linear}.
\end{lem}

We will now use the transformation behaviour of the distinguished basis for each chamber under the action of $W_\aff$ in Lemma \ref{lem:basiscomparison}  to derive an expression for any such basis. The preferred basis of $H_2(\cM(\boldsymbol{\alpha}, \mathbf{m}),\Z)$ is the same as the preferred basis of $H_2(\cM(\boldsymbol{\alpha}, \mathbf{0}),\Z)$. Consequently, using the result about the integral of $\Omega_I$ in two chambers in Proposition \ref{prop:exteriorintersection} and Proposition \ref{prop:interiorintersection}, we can conclude that

\begin{prop}\label{prop:integralofOmegaI}
In an interior chamber labelled by even partition set $\mathcal{I}$, the integral of $\Omega_I$ over the exterior sphere labelled by the subset $J \in \mathcal{I}$ is given by 
\begin{equation}\label{eq:periodinterior}
\int_{S^2_J} \Omega_I = 2\pi M_J,
\end{equation}
where 
\[M_J = \sum_{j \in J} m_{z_j} -\sum_{j \notin J} m_{z_j}.\]
In an exterior chamber labelled by one of the eight inequalities $L_i \leq 0$ or $L_i \geq 1$ ($L_i$ was defined in \eqref{eq:L} earlier) with associated even partition set $\mathcal{I}$, 
the central sphere is labelled by $I_0=\{i\}$ or $I_0=\{j,k,l\}$ respectively.  
The integral of $\Omega_I$ over the exterior sphere labelled by the subset $J \in \mathcal{I}$ is given by 
\begin{equation}\label{eq:periodexterior}
\int_{S^2_J} \Omega_I = 2\pi \left(M_J- M_{I_0}\right)
\end{equation}
 \end{prop}

\begin{proof}
We fix a parallel admissible homology basis  $\mathfrak{B}=(S_0, \ldots ,S_4)$ of the local system. Such a basis may be obtained by choosing $\boldsymbol{\alpha_0}$ in a fixed chamber and moving around the distinguished basis of $H_2(\mathcal{M}(\boldsymbol{\alpha_0}, 0),\Z)$ by parallel transport. Since $\mathcal{H}$ is a local system (i.e.\ the Gauss-Manin connection on $\mathcal{H}$ is flat) and $\mathcal{R}$ is simply-connected, this can be done without ambiguity. With respect to such a parallel basis, the functions
\[
\mathbf{m}=(m_1, \dots, m_4) \mapsto \int_{S_i} \Omega_I =  -2 \pi \sum_{j=1}^4 m_j  I(S_i, \Sigma_j)
\]
are linear in the complex masses $m_j$ with coefficients given by the intersection numbers $I(S_i, \Sigma_j)$. As observed earlier, these are locally constant on $\mathcal R$, hence constant since $\mathcal R$ is connected. It is hence enough to determine these coefficients in a single point $(\boldsymbol{\alpha_0}, \mathbf{m_0}) \in \mathcal{R}$. We put $\boldsymbol{\alpha_0}$ in one of the two chambers where we have already done the computation, e.g.\ the exterior chamber of type E2 in Proposition \ref{prop:exteriorintersection}. We only need to check that when computing with respect to the distinguished basis in any other chamber the expression obtained in Proposition \ref{prop:exteriorintersection} reduces to \eqref{eq:periodinterior} or \eqref{eq:periodexterior}. In view of Lemma \ref{lem:basiscomparison} this is tantamount to checking that the formulas in adjacent chambers are related by the Dehn twist associated to the wall separating them. This is done as in the proof of Lemma \ref{lem:basiscomparison}. At walls within the Biswas polytope separating two interior chambers we use that $M_{J^c} = - M_J$ and at walls separating an interior chamber from an exterior chamber that $ \sum_{J \in \mathcal{I}} M_J  = 2M_{I_0} = - \sum_{J \in \mathcal{I}} (M_J - M_{I_0})$. 
\end{proof}

\begin{rem}
The proof also shows that the computations in the two chambers are consistent in the sense that the expression of Proposition \ref{prop:integralofOmegaI} reduces to either of the expressions in the two chambers.
\end{rem}

From the proof of Proposition \ref{prop:integralofOmegaI} we also obtain
\begin{cor}
The period map $\mathcal{T}^{\Omega_I}$ has the equivariance properties 
\begin{enumerate}

\item $\mathcal{T}_{\mathfrak{B}(\Delta)}^{\Omega_I}(g \cdot \bfm) = g \cdot \mathcal{T}_{\mathfrak{B}(\Delta)}^{\Omega_I}(\bfm)$ 
\smallskip
\item $\mathcal{T}_{g \cdot \mathfrak{B}(\Delta)}^{\Omega_I}(\bfm) = g^{-1} \cdot \mathcal{T}_{\mathfrak{B}(\Delta)}^{\Omega_I}(\bfm)$
\end{enumerate}
for $g \in W_\aff$ and $\bfm \in \C^4$.

\end{cor}

\subsubsection{Proof of Torelli theorem}

\begin{proof}[Proof of Theorem \ref{thm:Torelli}]
In Proposition \ref{prop:stronglyparabolicvolumesv2}, we proved the statement in \eqref{eq:Torelli1} about the integral of $\omega_I$. In Proposition \ref{prop:integralofOmegaI}, we finished the proof of the statement in \eqref{eq:Torelli1} about the integral of $\Omega_I$.
This completes the proof of Theorem \ref{thm:Torelli}.
\end{proof}

\subsection{Surjectivity of the Torelli map}  \label{sec:Torelliproof}

In this section we finish the proof of Theorem \ref{thm:surjective}.

\medskip

In Lemma  \ref{lem:basiscomparison} and Lemma \ref{lem:linear}, we saw that the integrals of $\omega_I$ and $\Omega_I$ over some chamber-dependent natural basis of $H_2(\cM(\bfalpha,\bfm), \Z)$ are, respectively, affine and linear functions of, respectively, the parabolic weights and the complex masses. In this section we show that we can fix a manifold $X$ and a basis $\mathfrak{B}_0$ of $H_2(X,\Z)$ and define a marked hyperK\"ahler manifold $\widetilde{\cM}(\boldsymbol{\alpha}, \mathbf{m})$ for $(\boldsymbol{\alpha}, \mathbf{m})$ in the subset $\widetilde{\mathcal{R}} \subset \R^4 \times \C^4$. Recall from Section \ref{subsec:Torelli_map} that
a marked hyperK\"ahler manifold is a hyperK\"ahler manifold $M$ together with a diffeomorphism $\mu : X \to M$ where two marked hyperK\"ahler manifolds $M$ and $M'$ are considered equivalent if there exists a hyperK\"ahler isometry $f: M \to M'$ such that $f_* \circ \mu_* = \mu'_* : H_2(X,\Z) \to H_2(M',\Z)$. 
Fixing a basis $\mathfrak{B}_0$ of $H_2(X,\Z)$ a marked hyperK\"ahler manifold $M$ comes with the distinguished basis $\mathfrak{B}=\mu_*(\mathfrak{B}_0)$ of $H_2(M,\Z)$.  

\medskip

We have seen how to associate with each parameter $(\boldsymbol{\alpha}, \mathbf{m}) \in \mathcal{R} \subset (0,\frac 12 )^4 \times \C^4$ a marked hyperK\"ahler manifold, namely the Hitchin moduli space $\cM(\boldsymbol{\alpha}, \mathbf{m})$.
More precisely, fix $\boldsymbol{\alpha}_0$ in the interior of some model chamber $\Simplex \subset [0,\frac 12 ]^4$. Set $X=\cM(\boldsymbol{\alpha}_0, \mathbf{0})$ and let $\mathfrak{B}_0$ the preferred basis of $H_2(X,\Z)$ associated with that chamber. Then, using the fiber bundle structure in Section \ref{sec:bundle}, for any $(\boldsymbol{\alpha}, \mathbf{m}) \in \mathcal{R}$ we obtain a diffeomorphism $\mu: X \to \cM(\boldsymbol{\alpha}, \mathbf{m})$ which is uniquely determined up to isotopy. Setting $\mathfrak{B} = \mu_*(\mathfrak{B}_0)$ precisely yields the parallel transport as described before. Now if $\boldsymbol{\alpha}' = g \cdot \boldsymbol{\alpha}$ for $\boldsymbol{\alpha} \in \interior(\Simplex)$ and $g \in W_{\mathrm{aff}}$, then $\mathfrak{B}' = g \cdot \mathfrak{B}$ is the preferred basis associated with chamber $\Simplex' = g \cdot \Simplex$.

\medskip

For values of $\boldsymbol{\alpha} \not\in (0,\frac 12)^4$, we first observe that the homeomorphism $\mathcal{U}(W_{\mathrm{aff}}, \Simplex) \simeq \R^4$ displays the simplex $\Simplex$ as a fundamental domain for the action of $W_{\mathrm{aff}}$ on $\R^4$, which means that  $\bigcup_{g \in W_{\mathrm{aff}}} g \cdot \Simplex  = \R^4$ and $g(\interior(\Simplex)) \cap h(\interior(\Simplex)) = \emptyset$ for $g\neq h$. Using the extension of the action of $W_\aff$ to $\R^4 \times \C^4$ as in Remark \ref{rem:extension} we obtain
\begin{lem}\label{lem:saturation}\
\begin{enumerate}
\item $\widetilde{\mathcal{R}}^{\mathrm{full}} \subset \R^4 \times\C^4$ is invariant under the action of $W_\aff$.

\smallskip

\item $\widetilde{\mathcal{R}}=\bigcup_{g \in W_{\mathrm{aff}}} g \cdot \mathcal{R}  \subset \R^4 \times\C^4$.
\end{enumerate}
	\end{lem}

\begin{proof}
(1) It is enough to prove this for the generators of $W_\aff$ acting on $\R^4 \times \C^4$ via $(r_i, r_i')$ for the affine-linear maps $r_i$ in Lemma \ref{lem:reflections} and $r_i'$ the complex-linear extension of their respective linear parts. Using these formul{\ae} it is straightforward to check that under the substitution $(\bfalpha, \bfm)  \mapsto (r_i(\bfalpha), r_i'(\bfm))$ a codimension-three plane of type $\mathcal{H}_{d, \mathbf{e}}$ is transformed into a codimension-three plane either of type $\mathcal{H}_{d', \mathbf{e}'}$ or of type $\mathcal{H}_{d', i'}$. Similarly, a codimension-three plane of type $\mathcal{H}_{d, i}$ is transformed into a codimension-three plane either of type $\mathcal{H}_{d', i'}$  or of type $\mathcal{H}_{d', \mathbf{e}'}$.

\smallskip

\noindent (2) Recall that the Biswas polytope is bounded by boundary facets of the weight cube $[0, \frac 12]^4$ and by hyperplanes $\{L_i=0\}$ or $\{L_i=1\}$ where $L_i(\bfalpha)= -\alpha_i + \sum_{j \neq i} \alpha_j$. Points in the intersection of these hyperplanes with the boundary of the weight cube are exactly those points in $[0, \frac 12]^4$ that cannot be moved by an element of $W_\aff$ into the open weight cube $(0, \frac 12)^4$. The $W_\aff$-orbit of these points is the set that we are removing from $\widetilde{\mathcal{R}}^{\mathrm{full}}$ to obtain $\widetilde{\mathcal{R}}$.
\end{proof}

\begin{rem}
In particular, the complement of $\widetilde{\mathcal{R}}^{\mathrm{full}}$ is an arrangement of codimension-three planes which is invariant under the action of $W_\aff$, and $\widetilde{\mathcal{R}}$ is an invariant subset of $\widetilde{\mathcal{R}}^{\mathrm{full}}$. We are removing additional codimension-two planes to obtain $\widetilde{\mathcal{R}}$ from $\widetilde{\mathcal{R}}^{\mathrm{full}}$, hence $\widetilde{\mathcal{R}}^{\mathrm{full}}$ and $\widetilde{\mathcal{R}}$ are both open and dense.
\end{rem}

\begin{defn}
For $(\boldsymbol{\alpha}',\mathbf{m}') \in \widetilde{\mathcal{R}}$ we set $\widetilde{\mathcal{M}}(\boldsymbol{\alpha}',\mathbf{m}') = \mathcal{M}(\boldsymbol{\alpha},\mathbf{m})$ with basis $\mathfrak{B}' = g \cdot \mathfrak{B}$ of $H_2(\mathcal{M}(\boldsymbol{\alpha},\mathbf{m}),\Z)$ if $(\boldsymbol{\alpha}',\mathbf{m}') = g \cdot (\boldsymbol{\alpha},\mathbf{m})$ for $g \in W_{\mathrm{aff}}$ and $(\boldsymbol{\alpha},\mathbf{m}) \in\mathcal{R}$. If we consider $g$ as an element in $\Diff_+(X)$, then we obtain a marked hyperK\"ahler manifold via the diffeomorphism $X \overset{g}\rightarrow X \overset{\mu}\rightarrow \mathcal{M}(\boldsymbol{\alpha},\mathbf{m})$.

\end{defn}

As before, we choose our initial simplex $\Delta$ to be the interior chamber of type B1 with even partition set $\mathcal{I}=\{\emptyset, \{1, 2\}, \{1, 3\}, \{1, 4\}\}$. We observe that
inside our given initial simplex,
the sphere $S_i$ collapses at $f_i$, with the naming convention that 
$S_1$ is the sphere associated to $I=\emptyset$, $S_2$ is the sphere associated to $I=\{1, 2\}$, $S_3$ is the sphere associated to $I=\{1, 3\}$, $S_4$ is the sphere associated to $I=\{1, 4\}$ and $S_0$ is the central sphere.
 The integrals of $\omega_I$ are respectively \begin{align*}
\int_{S_1} \omega_I &=4 \pi^2 K_{\emptyset}(\boldsymbol{\alpha})= 4 \pi^2 (1 - \alpha_1- \alpha_2-\alpha_3 - \alpha_4)\\
\int_{S_2} \omega_I&= 4 \pi^2 K_{\{1,2\}}(\boldsymbol{\alpha})=4 \pi^2 (\alpha_1+ \alpha_2 - \alpha_3-\alpha_4)\\
\int_{S_3} \omega_I &=4 \pi^2 K_{\{1, 3\}}(\boldsymbol{\alpha})=4 \pi^2 (\alpha_1- \alpha_2 + \alpha_3 -\alpha_4)\\
\int_{S_4} \omega_I&=4 \pi^2 K_{\{1, 4\}}(\boldsymbol{\alpha})=4 \pi^2 (\alpha_1- \alpha_2-\alpha_3+ \alpha_4 ).
\end{align*}
Similarly, the integrals of $\Omega_I$ are respectively
\begin{align*}
\int_{S_1} \Omega_I &=2\pi M_{\emptyset}(\bfm) = 2\pi  ( - m_1- m_2-m_3 - m_4)\\
\int_{S_2} \Omega_I&= 2\pi M_{\{1,2\}}(\bfm)= 2\pi (m_1+ m_2 - m_3-m_4)\\
\int_{S_3} \Omega_I &=-2\pi M_{\{1, 3\}}(\bfm)=2\pi (m_1- m_2+ m_3 -m_4)\\
\int_{S_4} \Omega_I&= 2\pi M_{\{1, 4\}}(\bfm)=2\pi (m_1- m_2-m_3 + m_4 ).
\end{align*}
The integral of $\omega_I, \Omega_I$ over $S_0$ are determined by $2 [S_0] + \sum_{i=1}^4 [S_i]=[T^2]$, and the facts that $\int_{T^2} \omega_I = 4 \pi^2$
and $\int_{T^2} \Omega_I = 0$.
This computation shows that the Torelli map
extends to an affine-linear map $\mathcal{T}_{\mathfrak{B}(\Simplex)} : \R^4 \times \C^4 \to (\R \times \C)^4$. The main result is the following:

\begin{prop}\label{prop:Torellisurjectivity}
The extended Torelli map $\mathcal{T}_{\mathfrak{B}(\Simplex)} : \R^4 \times \C^4 \to (\R \times \C)^4$ is a bijection. It restricts to a bijection between $\widetilde{\mathcal{R}}^{\mathrm{full}} \subset  \R^4 \times \C^4$  and the period domain $\calP \subset (\R \times \C)^4$. 
\end{prop}

\begin{proof}
The integrals on the parameter space $\R^4_{\boldsymbol{\alpha}} \times \C^4_{\mathbf{m}}$ are given by the above formul{\ae}, so the result for the extended Torelli map follows from invertibility of the matrix
\[
\begin{pmatrix}
-1 & -1 & -1 & -1 \\
1 & 1 & -1 & -1 \\
1 & -1 & 1 & -1 \\
1 & -1 & -1 & 1	
\end{pmatrix}
\]
over the reals (it has determinant 16). 
Finally, we need to prove that $\mathcal{T}_\mathfrak{B}(\widetilde{\mathcal{R}}^{\mathrm{full}})=\cP$. We have already observed that $\mathcal{T}_{\mathfrak{B}(\Simplex)}^{\omega_I}$ maps the faces $f_0, f_1, f_2, f_3, f_4$ of the model simplex $\Simplex$ in $\R^4_{\bfalpha}$ to the faces $F_0, F_1, F_2, F_3, F_4$ of the simplex $\Simplex_{\mathrm{target}}$ in $\R^4_{\mathbf{x}}$. Further, we have seen in Remark \ref{rem:Nakajima} that the faces $f_0, f_1, f_2, f_3, f_4$ are restrictions of the following Nakajima walls
\begin{enumerate}
\item $\{ -\alpha_1+\alpha_2+\alpha_3+\alpha_4 = 0, -m_1+m_2+m_3+m_4=0\} $	
\item $\{1- \alpha_1-\alpha_2-\alpha_3 - \alpha_4 = 0, -m_1-m_2-m_3-m_4=0\}$
\item $\{\alpha_1+\alpha_2-\alpha_3-\alpha_4=0, m_1+m_2-m_3-m_4 =0\}$
\item $\{\alpha_1-\alpha_2+\alpha_3-\alpha_4=0, m_1-m_2+m_3-m_4 =0\}$
\item $\{\alpha_1-\alpha_2-\alpha_3+\alpha_4=0, m_1-m_2-m_3+m_4 =0\}$
\end{enumerate}
By the formul{\ae} above for the period map these map to the period-domain walls 
\begin{enumerate}
	\item $\{\sum_i x_i = 4 \pi^2, \sum_i z_i=0\}$ 
	\item $\{x_1=z_1=0\}$
	\item $\{x_2=z_2=0\}$
	\item $\{x_3=z_3=0\}$
	\item $\{x_4=z_4=0\}$
\end{enumerate}
respectively. The desired statement then follows from $W_\aff$-equivariance of $\mathcal{T}_{\mathfrak{B}(\Simplex)}$.
\end{proof}

\begin{rem}\label{rem:Torellisurjectivity}
The subset $\mathcal{T}_\mathfrak{B}(\widetilde{\mathcal{R}}) \subset \cP$ consists of periods of marked Hitchin moduli space and is open and dense. In that sense the period is nearly surjective on Hitchin moduli spaces as claimed.
\end{rem}

\subsubsection{Proof of surjectivity}

\begin{proof}[Proof of Theorem \ref{thm:surjective}]
Proposition \ref{prop:Torellisurjectivity} and Remark \ref{rem:Torellisurjectivity} yield the proof of Theorem \ref{thm:surjective}
\end{proof}

\appendix

\section{Scalings of \texorpdfstring{HyperK\"ahler}{HyperKahler} Metric}\label{sec:scalings}

In this section, once and for all,  we write out the $\mathbb{R}^+ \times \mathbb{R}^+ \times U(1)$ family of hyperK\"ahler structures one could naturally put on the moduli space of Higgs bundles. Various conventions in the literature fit into this universal scheme, so it's helpful to write it out. Our treatment here is similar to \cite{Neitzkecourse}, but with $\mathbb{R}^+ \times \mathbb{R}^+ \times U(1)$ parameters spelled out.

\subsection{Affine spaces \texorpdfstring{$\mathcal{A}^{h_0}$ and $\mathcal{A}_I^{\C}$}{}}
Fix a Riemann surface $(C, g)$. Fix a complex vector bundle $E \to C$ with hermitian metric $h_0$. In this section, we will omit the trace-free conditions; they can easily be put back in. 
\begin{defn}
A \emph{doubled connection} on $(E, h_0)$ is a pair $(D, \Phi)$ where $D$ is a unitary connection on $(E,h_0)$ and $\Phi \in \Omega^1( \mathfrak{u}(E))$.
\end{defn}

The space of double connections $\mathcal{A}^{h_0}$ is naturally an affine space over the \emph{real} vector space $\Omega^1(\mathfrak{u}(E)) \oplus \Omega^1(\mathfrak{u}(E))$.
The group of unitary gauge transformations $\mathcal{G}$ acts on the space of doubled connections $\mathcal{A}^{h_0}$ by
\[(D, \Phi) \mapsto (D^g= g \circ D \circ g^{-1}, \Phi^g = g \Phi g^{-1}). \]
Soon, we will equip $\mathcal{A}^{h_0}$ with a hyperK\"ahler structure, then perform the hyperK\"ahler quotient to get $\mathcal{M}$.

If one wants to view  $\mathcal{M}$ as a holomorphic symplectic manifold obtained via a holomorphic symplectic quotient, it is most natural to instead consider doubled version of the space of $\delbar$-operators on $E$, which has a natural complex structure. 
\begin{defn}
Define $\mathcal{A}^{\mathbb{C}}= T^* \mathcal{A}^{\delbar} = \mathcal{A}^{\delbar} \times \Omega^{1,0}( \mathfrak{g}_\C(E))$
\end{defn}
Note that at $(\delbar_E, \varphi) \in \mathcal{A}^{\mathbb{C}}$, and  $(\dot{A}^{0,1}, \dot{\varphi}) \in T_{(\delbar_E, \varphi)} \mathcal{A}^{\mathbb{C}}$, the natural complex structure $I$ is 
\[I(\dot{A}^{0,1}, \dot{\varphi}) =(i \dot{A}^{0,1}, i \dot{\varphi}). \]
The space $\mathcal{A}^{\mathbb{C}}$ has a natural action of $\mathcal{G}_{\C}$ given by 
\[(\delbar_E, \varphi) \mapsto (\delbar_E^g= g \circ \delbar_E \circ g^{-1}, \varphi^g = g \Phi g^{-1}). \]

There is an isomorphism of the underlying real affine spaces 
\begin{align}
\mathcal{A}^{\mathbb{C}} &\to \mathcal{A}^{h_0} \\ \nonumber 
(\delbar_E, \varphi) &\mapsto (D = \delbar_E + \del_A^{h_0}, \Phi = \varphi - \varphi^{*_{h_0}}). 
\end{align}
Then naturally\footnote{For completeness, we note that $\dot{A} \in \Omega^1(C, \mathfrak{u}(E))$, so locally $\dot{A} = \dot{A}_x \de x + \dot{A}_y \de y$ for $\dot{A}_x, \dot{A}_y$ satisfying $\dot{A}_\bullet^{*_{h_0}} + \dot{A}_\bullet =0$. Then note that $\dot{A}$ can be viewed as a complex $1$-form 
\[\dot{A} = \underbrace{ \frac{1}{2}(\dot{A}_x +i \dot{A}_y) \de \zbar}_{\dot{A}^{0,1}} + \underbrace{ \frac{1}{2}(\dot{A}_x - i \dot{A}_y) \de z}_{\dot{A}^{1,0}}  
 \]
 and indeed that $\dot{A}^{1,0} = - (\dot{A}^{0,1})^{*_{h_0}}$.
}, there is an isomorphism of the real tangent spaces 
\begin{align}
T_{(\delbar_E, \varphi)} \mathcal{A}^{\mathbb{C}} &\to T_{(D, \Phi)} \mathcal{A}^{h_0} \\ \nonumber 
(\dot{A}^{0,1}, \dot \varphi) &\mapsto (\dot{A} = \dot{A}^{0,1} - (\dot{A}^{0,1})^{*_{h_0}}, \dot \Phi = \dot \varphi - \dot \varphi^{*_{h_0}}). 
\end{align}
The complex structure $I$ on $\mathcal{A}^{\mathbb{C}}$ induces a complex structure on $\mathcal{A}^{h_0}$ which is most naturally written\[ I (\dot{A}, \dot{\Phi})  = ( (i \dot{A}^{0,1}) - (i \dot{A}^{0,1})^{*_{h_0}}, i \dot{\varphi} - (i \dot{\varphi})^{*_{h_0}}) = (i \dot{A}^{0,1} - i \dot{A}^{1,0}, i \Phi^{1,0} - i \Phi^{0,1}) =(\star \dot{A}, -\star \dot{\Phi}) \]
since $\star \de z = \star (\de x + i \de y) = \de y  - i \de x = -i \de z$ and $\star \de \zbar = i \de \zbar$.

\subsection{Families of \texorpdfstring{hyperK\"ahler}{hyperKahler} structures} 

Now, we start introducing the various scales.  There is an $\mathbb{R}_{\lambda_1}^+ \times \mathbb{R}_{\lambda_2}^+ \times U(1)_{e^{i \vartheta}}$ family of choices of hyperK\"ahler structure on $\mathcal{A}^{h_0}$ (or $\mathcal{A}_I^{\C}$) once we've fixed $I$ as above. The $\mathbb{R}_{\lambda_1}^+ \times \mathbb{R}_{\lambda_2}^+$ freedom appears in horizontal and vertical scale of the Riemannian metric, while the $U(1)$ comes from the freedom to rotate the circle through complex structures $J, K$.

\subsubsection{Riemannian metric}
First, we put a $\mathbb{C}$-valued hermitian metric $h$ on the complex tangent space $T\mathcal{A}_I^{\C}$. Let $v_i = (\dot A_i^{0,1}, \dot \varphi_i)$ be tangent vectors, $i=1,2$. The Hermitian metric is 
\begin{equation} \label{eq:metric1}
 h \left((\dot{A}^{0,1}_1, \dot{\varphi}_1), (\dot{A}^{0,1}_2, \dot{\varphi}_2)\right)
= 2  \lambda_1 \left(\int_C \lambda_2 \IP{\dot{A}^{0,1}_1, \dot{A}^{0,1}_2}_{h_0} + \IP{\dot\varphi_1, \dot\varphi_2}_{h_0} \right), 
\end{equation}
where $\IP{\alpha, \beta} = \Tr\left(\alpha \wedge \star \beta^{*_{h_0}} \right)$ is a two-form. Note that $\star$ has conjugation built in: if $\beta = \beta_z \de z + \beta_{\zbar} \de \zbar$, then $\star \beta^{*_{h_0}}
\I \beta_z^{*_{h_0}} \de \zbar -  \I \beta_{\zbar}^{*_{h_0}} \de z$. (Note that because $C$ is K\"ahler, 
the Hodge star on $1$-forms depends only on the conformal class of the metric $g_C$.)
Thus, if $\dot A_i^{0,1} = \dot{A}_{i,\zbar}^{0,1} d \zbar$ and $\dot \varphi_i = \dot \varphi_{i z} d z$, then in local coordinates 
\begin{align} 
 h \left((\dot{A}^{0,1}_1, \dot{\varphi}_1), (\dot{A}^{0,1}_2, \dot{\varphi}_2)\right)
&= 2   \left( \lambda_1 \int_C \Tr  \left( \lambda_2 \dot{A}^{0,1}_{1,\zbar}  (\dot{A}^{0,1}_{2, \zbar} )^{*_{h_0}} + \varphi_{1,z} \varphi_{2,z}^{*_{h_0}}  \right) \right) i \de z \wedge \de \zbar \nonumber \\
& = 2i   \left( \lambda_1 \int_C \Tr  \left( \lambda_2 (\dot{A}^{0,1}_{2} )^{*_{h_0}} \wedge \dot{A}^{0,1}_{1}   + \varphi_{1} \wedge \varphi_{2}^{*_{h_0}}  \right) \right).
\end{align}
This last expression is what we wrote before in \eqref{eq:hintro} with $\lambda_1=\lambda_2=\e^{i \theta}=1$.

The Riemannian metric is simply the real part of this:
\begin{align} \label{eq:metric2}
 g \left((\dot{A}^{0,1}_1, \dot{\varphi}_1), (\dot{A}^{0,1}_2, \dot{\varphi}_2)\right)
&= 2  \mathrm{Re} \left(  \lambda_1 \int_C \lambda_2 \IP{\dot{A}^{0,1}_1, \dot{A}^{0,1}_2}_{h_0} + \IP{\dot\varphi_1, \dot\varphi_2}_{h_0} \right) \nonumber \\
&=-2 \mathrm{Im}  \left( \lambda_1 \int_C \Tr  \left( \lambda_2 (\dot{A}^{0,1}_{2} )^{*_{h_0}} \wedge \dot{A}^{0,1}_{1}   + \varphi_{1} \wedge \varphi_{2}^{*_{h_0}}  \right) \right).
\end{align}
This strange factor of $2$ appears because of the nice expression for the Riemannian metric in terms of deformations in $\mathcal{A}^{h_0}$:
\begin{align}\label{eq:metric3}
 g \left((\dot{A}_1, \dot \Phi_1), (\dot{A}_2, \dot{\Phi}_2)\right)
&=  \mathrm{Re}\left( \lambda_1 \int_C \lambda_2 \IP{\dot{A}_1, \dot{A}_2}_{h_0} + \IP{\dot\Phi_1, \dot\Phi_2}_{h_0} \right) \nonumber \\
&=  - \lambda_1 \mathrm{Re} \left(\int_C \Tr \left( \lambda_2 \dot{A}_1 \wedge \star \dot{A}_2 + \dot \Phi_1 \wedge \star  \dot\Phi_2 \right) \right).
\end{align}
\subsubsection{Complex structures}

Now, we begin on the complex structures. First, we note that $I$ clearly squares to $-1$ and moreover is compatible with $g$ in the sense that $g(I v_1, I v_2)=g(v_1,v_2)$.
This is apparent in either perspective: in the $\mathcal{A}_I^{\C}$ perspective, these follows because $i^2=-1$ and $i \cdot  \overline{i}=1$; in the $\mathcal{A}^{h_0}$ perspective, $\star^2 =-1$ on $\Omega^1(C)$. 

\begin{lem}
On $\mathcal{A}^{h_0}$, define 
 \begin{align} J(\dot{A}, \dot{\Phi}) &= (- \lambda_2^{-1/2}\dot{\Phi},  \lambda_2^{1/2}\dot{A}). \nonumber \\
K(\dot{A}, \dot{\Phi}) &= I \circ J (\dot{A}, \dot{\Phi}) = ( -\lambda_2^{-1/2} \star \dot{\Phi},  -\lambda_2^{1/2} \star \dot{A}). \end{align}
 \begin{itemize}
\item Then $I,J,K$ satisfy quaternionic relations.
\item Then $g(J \cdot, J \cdot) =g(\cdot, \cdot)=g(K\cdot, K \cdot)$.
 \item On $\mathcal{A}_I^{\C}$, $J$ acts as 
 \begin{align*} J( \dot{A}^{0,1}, \dot \varphi) &= (\lambda_2^{-1/2} \dot{\varphi}^{*_{h_0}}, -\lambda_2^{1/2} (\dot{A}^{0,1})^{*_{h_0}})\\
   K( \dot{A}^{0,1}, \dot \varphi) &= (\lambda_2^{-1/2} i\dot{\varphi}^{*_{h_0}}, -\lambda_2^{1/2} i(\dot{A}^{0,1})^{*_{h_0}}).
  \end{align*}
 \end{itemize}
\end{lem}

\begin{proof}
We compute the action on $\mathcal{A}^{\C}_I$:
\begin{align}\dot{A}^{0,1} &= \frac{\dot{A}_x + i \dot{A}_y}{2} d \zbar  \quad \overset{J}{\mapsto}  \quad  - \lambda_2^{-1/2}\frac{\dot{\Phi}_x + i \dot{\Phi}_y }{2} d \zbar = - \lambda_2^{-1/2} \dot \Phi^{0,1} = \lambda_2^{-1/2} \dot \varphi^{*_{h_0}}  \nonumber \\  \dot \varphi&= \frac{\dot{\Phi}_x  - i \dot{\Phi}_y }{2} d z  \quad \overset{J}{\mapsto}  \quad   \lambda_2^{1/2}\frac{\dot{A}_x  - i \dot{A}_y }{2} d z=  \lambda_2^{1/2} \dot A^{1,0} =-\lambda_2^{1/2} (\dot{A}^{0,1})^{*_{h_0}}.   \end{align}
 \end{proof}

Now, we observe that $J$ and $K$ are not special. We can rotate one into another and more generally write
\[ J_{\vartheta} + i K_{\vartheta} = e^{i \vartheta}(J + i K) \]
for any phase $\vartheta \in \R/ 2 \pi \Z$.
Then, 
\begin{align}
J_{\vartheta} ( \dot{A}^{0,1}, \dot \varphi) &= (e^{i \vartheta} \lambda_2^{-1/2} \dot{\varphi}^{*_{h_0}}, -e^{i \vartheta} \lambda_2^{1/2} (\dot{A}^{0,1})^{*_{h_0}})\\ \nonumber 
K_{\vartheta} ( \dot{A}^{0,1}, \dot \varphi) &= (i e^{i \vartheta} \lambda_2^{-1/2} \dot{\varphi}^{*_{h_0}}, -i e^{i \vartheta} \lambda_2^{1/2} (\dot{A}^{0,1})^{*_{h_0}}).
\end{align}
This is somewhat useful since while various papers generally have the same convention for $I$, they may pick a rotated $J_\vartheta, K_\vartheta$ instead.

\subsubsection{Symplectic and holomorphic symplectic forms}
Now, we compute the symplectic and holomorphic symplectic forms. We use the convention that $\omega_{I}(v,w) = g(Iv,w)$.
\begin{lem}

Then, 
\begin{align*}
\omega_I\left((\dot{A}^{0,1}_1, \dot{\varphi}_1), (\dot{A}^{0,1}_2, \dot{\varphi}_2)\right)=-2 \mathrm{Re}  \left( \lambda_1 \int_C \Tr  \left( \lambda_2 (\dot{A}^{0,1}_{2} )^{*_{h_0}} \wedge \dot{A}^{0,1}_{1}   +\varphi_{1} \wedge \varphi_{2}^{*_{h_0}}  \right) \right)\\\end{align*}
and 
\begin{align*}
\Omega_{I, \vartheta}\left((\dot{A}^{0,1}_1, \dot{\varphi}_1), (\dot{A}^{0,1}_2, \dot{\varphi}_2)\right)&= (\omega_{J_\vartheta} + i \omega_{K_\vartheta})\left((\dot{A}^{0,1}_1, \dot{\varphi}_1), (\dot{A}^{0,1}_2, \dot{\varphi}_2)\right)\\
&=-2i e^{-i \vartheta}   \lambda_1 \lambda_2^{1/2} \int_C  \Tr  \left(\dot{\varphi}_2  \wedge \dot{A}_{1}^{1,0}  - \dot{\varphi}_1 \wedge \dot{A}^{0,1}_{2}  \right) 
\end{align*}
\end{lem}

\begin{proof} 
We begin with the following expression for the Riemannian metric:
\begin{equation*} 
 g \left((\dot{A}^{0,1}_1, \dot{\varphi}_1), (\dot{A}^{0,1}_2, \dot{\varphi}_2)\right)
=-2 \mathrm{Im}  \left( \lambda_1 \int_C \Tr  \left( \lambda_2 (\dot{A}^{0,1}_{2} )^{*_{h_0}} \wedge \dot{A}^{0,1}_{1}   + \varphi_{1} \wedge \varphi_{2}^{*_{h_0}}  \right) \right).
\end{equation*}
Then, we compute: 
\begin{align*}
\omega_I\left((\dot{A}^{0,1}_1, \dot{\varphi}_1), (\dot{A}^{0,1}_2, \dot{\varphi}_2)\right)&= g \left(I (\dot{A}^{0,1}_1, \dot{\varphi}_1), (\dot{A}^{0,1}_2, \dot{\varphi}_2)\right)\\
&=
-2 \mathrm{Im}  \left( \lambda_1 \int_C \Tr  \left( \lambda_2 (\dot{A}^{0,1}_{2} )^{*_{h_0}} \wedge i\dot{A}^{0,1}_{1}   +i\varphi_{1} \wedge \varphi_{2}^{*_{h_0}}  \right) \right)\\
&=-2 \mathrm{Re}  \left( \lambda_1 \int_C \Tr  \left( \lambda_2 (\dot{A}^{0,1}_{2} )^{*_{h_0}} \wedge \dot{A}^{0,1}_{1}   +\varphi_{1} \wedge \varphi_{2}^{*_{h_0}}  \right) \right)\\
\end{align*}
Similarly, we compute
\begin{align*}
\omega_{J_\vartheta} \left((\dot{A}^{0,1}_1, \dot{\varphi}_1), (\dot{A}^{0,1}_2, \dot{\varphi}_2)\right)
&= g \left(J_{\vartheta} (\dot{A}^{0,1}_1, \dot{\varphi}_1), (\dot{A}^{0,1}_2, \dot{\varphi}_2)\right)\\
&=-2 \mathrm{Im}  \left( \lambda_1 \int_C \Tr  \left( \lambda_2 (\dot{A}^{0,1}_{2} )^{*_{h_0}} \wedge (e^{i \vartheta} \lambda_2^{-1/2} \dot{\varphi}_1^{*_{h_0}})  - e^{i \vartheta} \lambda_2^{1/2}(\dot{A}_1^{0,1})^{*_{h_0}} \wedge \dot \varphi_{2}^{*_{h_0}}  \right) \right)\\
&\overset{(1)}{=}2 \mathrm{Im}  \left( \lambda_1 \lambda_{1/2} e^{-i \vartheta} \int_C \Tr  \left(\dot\varphi_{2} \wedge \dot{A}_1^{0,1}   - \dot{\varphi}_1  \wedge \dot{A}^{0,1}_{2}   \right) \right).
\end{align*}
In (1), we conjugated the entire expression.
Likewise,
\begin{align*}
\omega_{K_\vartheta} \left((\dot{A}^{0,1}_1, \dot{\varphi}_1), (\dot{A}^{0,1}_2, \dot{\varphi}_2)\right)
&= g \left(K_{\vartheta} (\dot{A}^{0,1}_1, \dot{\varphi}_1), (\dot{A}^{0,1}_2, \dot{\varphi}_2)\right)\\
&=2 \mathrm{Im}  \left( \lambda_1 \lambda_{1/2} (-i)e^{-i \vartheta} \int_C \Tr  \left(\dot\varphi_{2} \wedge \dot{A}_1^{0,1}   - \dot{\varphi}_1  \wedge \dot{A}^{0,1}_{2}   \right) \right)\\
&=-2 \mathrm{Re}  \left( \lambda_1 \lambda_{1/2} e^{-i \vartheta} \int_C \Tr  \left(\dot\varphi_{2} \wedge \dot{A}_1^{0,1}   - \dot{\varphi}_1  \wedge \dot{A}^{0,1}_{2}   \right) \right)
\end{align*}
Hence,
\begin{align*}
&\Omega_{I, \vartheta}\left((\dot{A}^{0,1}_1, \dot{\varphi}_1), (\dot{A}^{0,1}_2, \dot{\varphi}_2)\right)\\
&=(\omega_{J_\vartheta} + i \omega_{K_\vartheta})\left((\dot{A}^{0,1}_1, \dot{\varphi}_1), (\dot{A}^{0,1}_2, \dot{\varphi}_2)\right)\\
&=-2i e^{-i \vartheta}   \lambda_1 \lambda_2^{1/2} \int_C  \Tr  \left(\dot{\varphi}_2  \wedge \dot{A}_{1}^{1,0}  - \dot{\varphi}_1 \wedge \dot{A}^{0,1}_{2}  \right) 
\end{align*}
\end{proof}
\subsection{Moment maps}

Now, we remind the reader that if $\mathcal{G}_{\C}$ acts on $(\mathcal{A}_I^{\C}, I, \Omega_{I, \vartheta})$ the moment map is a map $$M_{I, \theta}: \mathcal{A}_I^{\C} \to \mathrm{Lie}(\mathcal{G}_{\C})^*.$$ 
We usually pair this with $Z \in  \mathrm{Lie}(\mathcal{G}_{\C})$, and then $M_{I, \theta}$ satisfies
 \[d \langle M_{I, \theta}, Z \rangle = \iota_{\rho(Z)} \Omega_{I, \theta}, \]
where $\langle , \rangle $ is the trace pairing on the Lie algebra and its dual. Here, this is $\langle M , Z \rangle  = \int_{C} \Tr(Z M)$.

\begin{prop}\label{prop:M}
 \[M_{I, \theta} = 2i e^{-i \theta}\lambda_1 \lambda_2^{1/2}  \delbar_E \varphi \]
 is a moment map for the $\mathcal{G}_{\C}$ action on the holomorphic symplectic space $(\mathcal{A}_I^{\C}, I, \Omega_{I, \vartheta})$.
\end{prop}
\begin{cor}
A moment map for the $\mathcal{G}$ action on the symplectic space $(\mathcal{A}^{h_0}, \omega_{J, \theta})$ is 
\begin{align*} 
\mu_{J, \vartheta} &= \mathrm{Re} (2i e^{-i \vartheta} \lambda_1 \lambda_2^{-1/2} \delbar_E \varphi) = -2i \lambda_1 \lambda_2^{-1/2} ( \cos \theta D  \star \Phi + \sin \vartheta D \Phi)
\end{align*}
\end{cor}
\begin{proof}[Proof of Proposition \ref{prop:M}]
Here, we follow the details in \cite{Neitzkecourse}. Given $Z \in \mathrm{Lie}(\mathcal{G}_{\C})$, then 
\begin{align*} 
M_{I, \theta, Z}: \mathcal{A}_I^{\C} & \to \mathbb{R}\\
(\delbar_E, \varphi) &\mapsto \langle M_{I, \theta}(\delbar_E, \varphi), Z \rangle  =2i e^{-i \theta}\lambda_1 \lambda_2^{1/2}  \int_C \Tr(Z  \delbar_E \varphi) \end{align*}
Hence, we compute: 
\begin{align*}
d M_{I, \vartheta, Z}(\dot{A}^{0,1}, \dot{\varphi}) &= 2i e^{-i \vartheta}\lambda_1 \lambda_2^{1/2}  \int_C \Tr(Z ( \delbar_E \dot \varphi + [\dot{A}^{0,1}, \varphi]))
\end{align*}
Meanwhile, $Z$ generates the following vector field on $\mathcal{A}_I^{\C}$:
\[ \rho(Z) = (-\delbar_E Z, [Z, \varphi]).\]
Hence, 
\begin{align*}
\Omega_{I, \vartheta}(\rho(Z), (\dot{A}^{0,1}, \dot{\varphi}))
&=-2i e^{-i \vartheta}   \lambda_1 \lambda_2^{1/2} \int_C  \Tr  \left(\dot \varphi \wedge (-\delbar_E Z) -[Z, \varphi] \wedge \dot{A}^{0,1}  \right) \\
&=d M_{I, \vartheta, Z}(\dot{A}^{0,1}, \dot{\varphi}),
\end{align*}
using integration by parts and expanding the Lie bracket. 
\end{proof}

Similarly, we compute that the moment map for the $\mathcal{G}$ action on $\mathcal{A}^{h_0}$. We first note that the symplectic form $\omega_I$ is 
\begin{align*}
\omega_I((\dot{A}_1, \dot \Phi_1), (\dot{A}_2, \dot \Phi_2))&= -g \left((\dot{A}_1, \dot \Phi_1), I(\dot{A}_2, \dot{\Phi}_2)\right)\\
&=  \lambda_1 \mathrm{Re} \left(\int_C \Tr \left( \lambda_2 \dot{A}_1 \wedge \star^2\dot{A}_2 + \dot \Phi_1 \wedge -\star^2  \dot\Phi_2 \right) \right)\\
&= \lambda_1  \left(\int_C \Tr \left(-\lambda_2 \dot{A}_1 \wedge \dot{A}_2 +\dot \Phi_1 \wedge  \dot\Phi_2 \right) \right).
\end{align*}
In the last line, we dropped the ``$\mathrm{Re}$'' because it is automatically real.

\begin{prop}
A moment map for the $\mathcal{G}$ action on the symplectic space $(\mathcal{A}^{h_0}, \omega_I)$ is 
\[ \mu_{I}(D, \Phi) = \lambda_1 \left(-\lambda_2F_D + \Phi \wedge \Phi \right)\]
\end{prop}
\begin{rem}
Since $\Phi \wedge \Phi = - [\varphi, \varphi^{*_{h_0}}]$, then we could instead write
\[ \mu_I(\delbar_E, \varphi, h_0) = -\lambda_1  \lambda_2 \left(F_D + \lambda_2^{-1} [\varphi, \varphi^{*_{h_0}}]\right) \]
\end{rem}
\begin{proof}
Take $$\mu_{I}(D, \Phi) = \lambda_1 \left(-\lambda_2F_D + \Phi \wedge \Phi \right).$$
Then $$\de \mu_{I}(\dot{A}, \dot{\Phi}) =  \lambda_1 \left(- \lambda_2D \dot{A} + \dot{\Phi} \wedge \Phi + \Phi \wedge \dot{\Phi} \right)=\lambda_1 \left(-\lambda_2 D \dot{A} + [\Phi \wedge \dot{\Phi}]\right).$$
We then compute
\begin{align*}
\mu_{I, Z}(D, \Phi) &= \lambda_1\int_C \Tr\left(Z(-\lambda_2F_D + \Phi \wedge \Phi)\right),
\end{align*}
hence 
\begin{align*}
d\mu_{I, Z}(\dot{A}, \Phi) &= \lambda_1 \int_C \Tr \left(Z (-\lambda_2 D \dot{A} + [\Phi \wedge \dot{\Phi}] \right).
\end{align*}
In the other direction, note  $\rho(Z)=(-D Z, [Z, \Phi])$. We compute:
\begin{align*}
\omega_I(\rho(Z), (\dot{A}, \dot{\Phi})) &=  \lambda_1  \left(\int_C \Tr \left(\lambda_2 DZ  \wedge \dot{A} + [Z, \Phi] \wedge  \dot\Phi \right) \right)\\
&=  \lambda_1  \left(\int_C \Tr \left( -Z \lambda_2 D \dot{A} + Z[ \Phi \wedge  \dot\Phi]\right) \right)\\
&=d \mu_{I, Z}(\dot{A}, \dot{\Phi})
\end{align*}
\end{proof}

\subsection{Levels of \texorpdfstring{hyperK\"ahler}{hyperKahler} moment map}
Based on Duisermaat--Heckman, when we express the integral of $\omega_I$ in terms of the level of the hyperkahler quotient at $\mu_\R$ and when we express $\Omega_I$ in terms of the level of the hyperkahler quotient at $\mu_\C$, the coefficients should be the same. 
\begin{prop}\label{prop:distributions}Define
\footnote{This is in keeping with the fact $\mathrm{pdeg}(E) = \frac{i}{2 \pi} \int_C \Tr F_D$.}
\[ F^\perp_D =F_D - \frac{1}{\mathrm{rk}(E)} \mathrm{Tr} F_{\mathrm{D}}  \mathrm{Id}_E.
\]
The value of the moment maps 
\begin{align*}
M_{I, \vartheta}(\delbar_E, \varphi)&= 2i e^{-i \vartheta} \lambda_1 \lambda_2^{1/2} \delbar_E \varphi\\
\mu_I(\delbar_E, \varphi) &= -\lambda_1 \left(\lambda_2 F^\perp_D + [\varphi, \varphi^{*_{h_0}}]   \right)
\end{align*}
at $p \in D$ are n
\begin{align*}
M_{I, \vartheta}(\delbar_E, \varphi) &= 2 i e^{-i \theta} \lambda_1 \lambda_2^{1/2} \delbar_E \varphi \\
&=  2 i e^{-i \theta} \lambda_1 \lambda_2^{1/2} \begin{pmatrix} -m& * \\ 0 & m \end{pmatrix} \pi \delta_p d \zbar \wedge d z\\
\mu_I(\delbar_E, \varphi) &=-\lambda_1 (\lambda_2 F^\perp_D  +[\varphi, \varphi^*_{h_0}])\\ \nonumber
&= -\lambda_1 \lambda_2  \begin{pmatrix} \alpha - \frac{1}{2} & \\ & \frac{1}{2}- \alpha \end{pmatrix} \pi \delta_p d \zbar \wedge d z \\
\end{align*}
\end{prop}
\begin{rem} Our convention in the rest of the paper corresponds to $\e^{i \vartheta}=1$ and $\lambda_1=\lambda_2=1$.
Thus, the levels of $M_I$ and $\mu_I$ at $p \in D$ in the frame where $F_p = \langle e_2 \rangle$ are respectively
\begin{align*}
 &2 i  \begin{pmatrix} -m_p& * \\ 0 & m_p \end{pmatrix} \pi \delta_p d \zbar \wedge d z
 & -\begin{pmatrix} \alpha_p - \frac{1}{2} & \\ & \frac{1}{2}- \alpha_p \end{pmatrix} \pi \delta_p d \zbar \wedge d z \\
\end{align*}
We computed that the integral of $\omega_I$ over $S_J$ was $4 \pi^2 K_J$. From the expressions for the levels of $M_I$ and $\mu_I$, by analogy with Duistermaat--Heckman as discussed in Section \ref{sec:affinelinearity}, we would have expected that the integral of $\Omega_I$ over $S_J$ would be $\frac{4 \pi^2}{2i}M_J$. At this time, we cannot explain the difference.
\end{rem}
\begin{proof}
We work out that $\partial_\zbar \frac{1}{z-p} =\pi \delta_p$.
Then, we compute that near $p \in D$,
for $\varphi= \begin{pmatrix}-m_p & * \\ 0 & m_p \end{pmatrix} \frac{dz}{z-p} + \mathrm{hol}$,
\begin{align*} \delbar_E \varphi &= \begin{pmatrix}-m_p & * \\ 0 & m_p \end{pmatrix} \pi \delta_p d\zbar \wedge d z \\
&= \begin{pmatrix}-m_p & * \\ 0 & m_p \end{pmatrix} \pi \delta_p(-2 i dx \wedge dy)
\end{align*}
We first note that distributional part of $F_D$ at $p \in D$ is as follows:
\begin{align*} F_D &= \delbar (h^{-1} \del h)\\
& \sim \delbar \del \begin{pmatrix} \alpha_p \log ((z-p) (\zbar-\overline{p})) & \\  & (1-\alpha_p)\log ((z-p)(\zbar-\overline{p})) \end{pmatrix}  \\
&= \begin{pmatrix} \alpha_p \log ((z-p)(\zbar-\overline{p})) & \\  & (1-\alpha_p)\log ((z-p)(\zbar-\overline{p})) \end{pmatrix}  d \zbar \wedge \del_{\zbar} \frac{1}{z-p} dz\\
&=  \begin{pmatrix} \alpha_p & \\ & 1-\alpha_p \end{pmatrix}   \pi \delta_p d \zbar \wedge d z
\end{align*}
Hence, the distributional part of $F_D^\perp$ at $p$ is 
\begin{align*}F_D^\perp &\sim \begin{pmatrix} \alpha_p- \frac{1}{2}& \\ & \frac{1}{2}-\alpha_p \end{pmatrix}   \pi \delta_p  d \zbar \wedge d z 
\end{align*}
since $\mathrm{Tr} F_D$ had a distributional piece supported at $D$ encoding the induced parabolic weight on $\det E$.
The result follows.
\end{proof}

\section{Locations of homology spheres in non-strongly parabolic Hitchin moduli spaces}\label{sec:homologyspheres} 
The $2$-homology $H_2(\cM, \Z)$ is $5$-dimensional and there is a basis $e_1, e_2, e_3, e_4, e_5$ in which the intersection form agrees with the intersection form in \eqref{eq:intersectionform}. (As discussed in Section \ref{sec:Dehntwists}, this basis is not unique, and any 4d Dehn twist maps one such basis to another.) 
When $\mathbf{m}=\mathbf{0}$, there is a single singular fiber---the nilpotent cone---consisting of five spheres in an affine $D_4$ arrangement, as described in Section \ref{sec:nilpotentcone}. These give a basis of $H_2(\cM,\Z)$.  In this section, we describe the location of homology spheres when $\mathbf{m} \neq \mathbf{0}$.

\subsection{Factorization}
 There is a $\Z^2$ lattice $\Gamma' \to \cB'$ such that $\cM'=\Gamma' \otimes_{\Z} \R/2\pi \Z$.\footnote{To recover this lattice, pick a point $\beta \in \cB'$ and a point in the fiber $\theta \in \cM'_\beta$. The Gauss-Manin connection gives us a notion of parallel transport. \ldots } 
 Let $\mathbb{D}$ be a large disk in $\cB$ containing all the points of the discriminant.
  Fix a point $\beta_0 \in \cB- \mathbb{D}$, and fix $\gamma_0$, a loop  based at $\beta_0$ going once around $\mathbb{D}$ counterclockwise, as shown in Figure \ref{fig:hurwitz}.
  \begin{figure}[ht]
\includegraphics[height=2.0in]{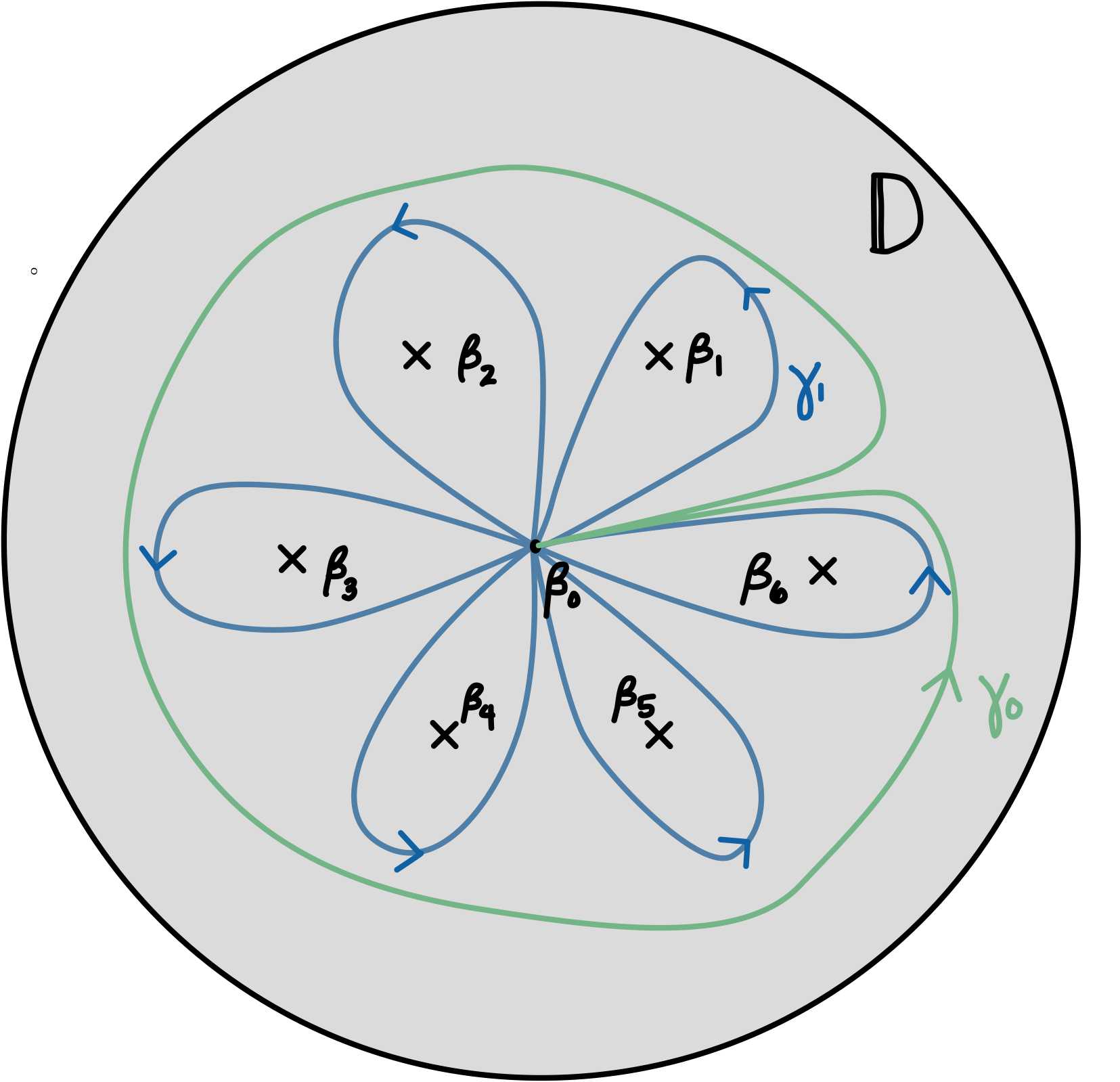}
\caption{\label{fig:hurwitz} }
\end{figure}
   We describe the generic situation  where the discriminant locus consists of six distinct points $\beta_1, \ldots, \beta_6$.
Given a point $\beta_i$ in the discriminant locus, let $\gamma_i$ be a loop which goes around $\beta_i$ counterclockwise, encircling no other point of the discriminant locus, as shown in Figure \ref{fig:hurwitz}.
We say that the indices are ``cyclically ordered'' if $[\gamma_6 \circ \gamma_5 \circ \gamma_4 \circ \gamma_3 \circ \gamma_2 \circ \gamma_1]=[\gamma_0]$, again as shown in Figure \ref{fig:hurwitz}.

Now, let $\rho$ be the $SL(2,\Z)$-valued monodromy of the lattice:
\[\rho: \pi_1(\cB', \beta_0) \to SL(2,\Z) \qquad \gamma \mapsto \rho(\gamma). \]
Then each of the six singular fibers is of type $I_1$; by the Picard-Lefshetz theorem, $A_i=\rho(\gamma_i)$ are given by Dehn twists, i.e. $$A_i \sim\begin{pmatrix} 1 & 1 \\ 0 & 1 \end{pmatrix}.$$ For other singular fiber types, a different conjugacy class is prescribed.  Since the total monodromy corresponding to a fiber of Kodaira type $I_0^*$ is $-\mathrm{Id}$, for the loop $\gamma_0$ going counterclockwise around $\mathbb{D}$, we have $\rho(\gamma_0)=-\mathrm{Id}$. This imposes a single relation on the $A_i$: $A_6 A_5 A_4 A_3 A_2 A_1=-\mathrm{Id}$.

The following theorem states that we can find loops $[\gamma_1], \cdots, [\gamma_6]$ such that the indices are cyclically ordered and $A_i = \rho(\gamma_i)$ are particularly nice by making a finite number of Hurwitz moves. Given one product
 $A_k \cdots A_1$, the product $A_k' \cdots A_1'$ is said to be obtained by applying a ``Hurwitz move''
if either (1) $A_i' = A_{i+1}$ and $A'_{i+1} = A_{i+1}^{-1} A_i A_{i+1}$ or (2) $A_i'=A_{i} A_{i+1} A_{i}^{-1}$ and $A_{i+1}'=g_i$.
We observe that (2) is the inverse of (1). The Hurwitz moves on the $A_i$ arise from the analogous moves on the loops $\gamma_i$, as shown in Figure \ref{fig:hurwitzmoves}.
At the level of the loops, in (1) we conjugate $\gamma_i$ by $\gamma_{i+1}$, observing that $\gamma_{i+1} ^{-1} \circ \gamma_i \circ \gamma_{i+1}$ goes once around $\beta_i$ counterclockwise; we leave all the remaining loops alone; after conjugation $\{1, 2, \cdots i+1, i, \cdots k\}$ are cyclically ordered, so we reindex, exchanging $i \leftrightarrow i+1$; (2) is similar.
  \begin{figure}[ht]
\includegraphics[width=7.0in]{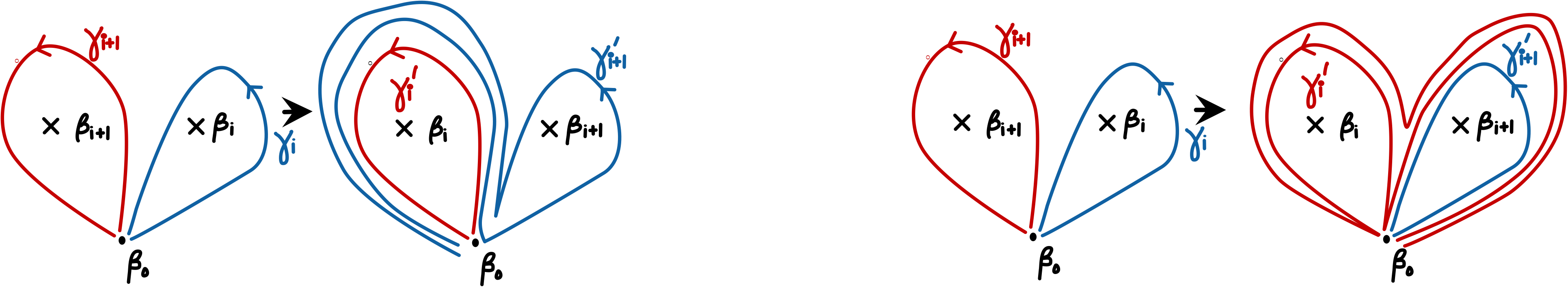}
\caption{Hurwitz moves (1) and (2) arise from the following analog of Hurwitz moves on the loops, shown respectively on the left and right.  \label{fig:hurwitzmoves} }
\end{figure}

\begin{thm}[Cadavid--Velez's Theorem 19 in \cite{CadavidVelez}]\label{thm:factorization}
Given initial loops, after a finite number of Hurwitz moves we produce loops $[\gamma_1], \cdots [\gamma_6]$ 
satisfying the same conditions and such that  $A_6=A_4=A_2=A$ and $A_5=A_3=A_1=B$
\[
	A = \begin{pmatrix}
 1 & 1 \\ 0 & 1	
 \end{pmatrix}
	\quad \text{and} \quad 
	B = \begin{pmatrix}
 1 & 0 \\ -1 & 1	
 \end{pmatrix} = W^{-1} A W \qquad W=\begin{pmatrix} 0 & 1 \\ -1 & 0 \end{pmatrix}, 
\]
i.e. we have the factorization $ABABAB=-\mathrm{Id}$.
\end{thm}

\subsection{Vanishing Cycles and Homology Spheres}

We now can discuss the homology spheres. Since the fiber above $\beta_i$ is of type $I_1$,
there is a unique (up to a sign) simple eigenvalue $\lambda_i = \begin{pmatrix} l_1\\l_2 \end{pmatrix}$ of $A_i$ such that $l_1$ and $l_2$ are relatively prime.  In any fiber of $\cM'$ above any $\beta$ nearby $\beta_i$, $\lambda_i$ corresponds to a circle; In the fiber of $\cM'$ above $\beta_i$, it corresponds to the ordinary double point. Consequently, we call $S^1_{\lambda_i}$ the vanishing cycle at $\beta_i$. (It is defined up to orientation.) We call the disk formed as the $S^1$-fibration over a path terminates at $\beta_i$ a ``Lefschetz thimble'', as shown in Figure \ref{fig:lefschetzthimble}.
\begin{figure}[ht] 
\begin{centering} 
\includegraphics[height=1.5in]{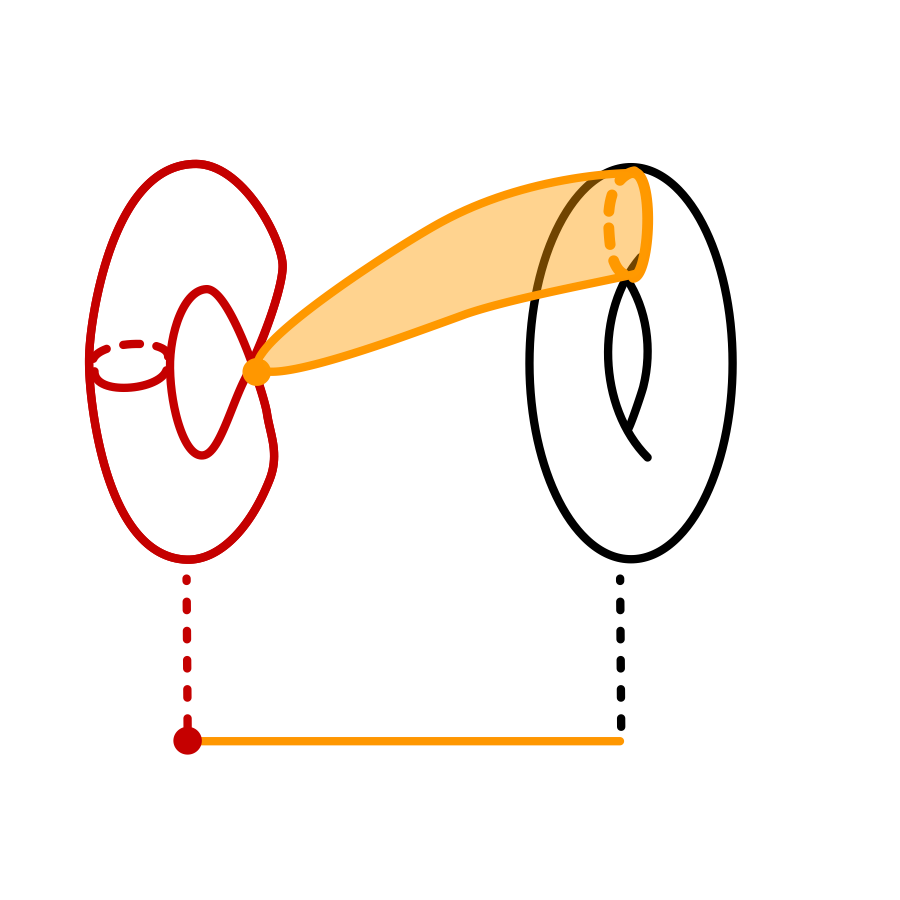}
\caption{\label{fig:lefschetzthimble}  Lefschetz thimble formed by vanishing cycle.}
\end{centering}
\end{figure}
Given a path $\gamma_1$ from $\beta_i \in \mathcal{B}^{\mathrm{sing}}$ to $\beta \in \mathcal{B}'$ and a path $\gamma_2$ from $\beta$ to $\beta_j \in \mathcal{B}^{\mathrm{sing}}$, and respective oriented vanishing cycles $S^1_{\lambda_i}$ and $S^1_{\lambda_j}$, observe that the two Lefschetz thimbles form a two-sphere if $S^1_{\lambda_i}=S^1_{\lambda_j}$ in the fiber over $\beta$, as shown in the left image of Figure \ref{fig:homologysphere}; if they do not match, no two-sphere is formed, as shown in the right image of Figure \ref{fig:homologysphere}.
\begin{figure}[ht] 
\begin{centering} 
\includegraphics[height=1.5in]{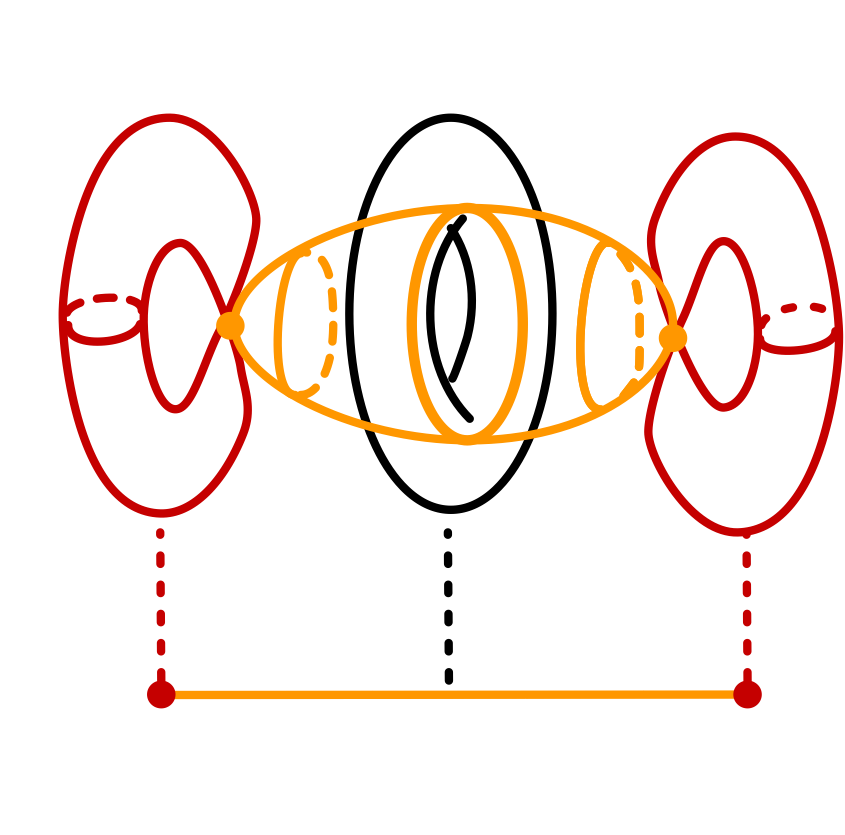}  \hspace{.5in} \includegraphics[height=1.5in]{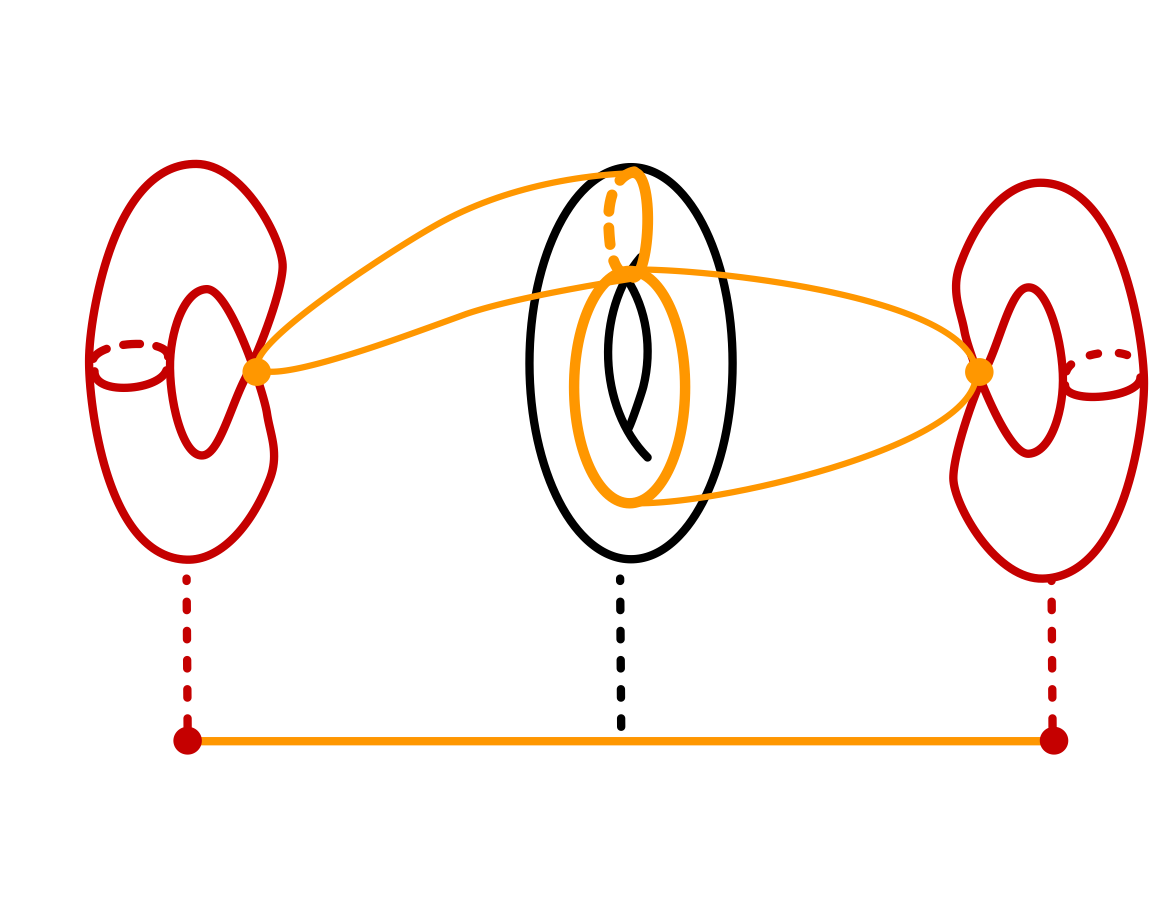} 
\caption{\label{fig:homologysphere}  (\textsc{Left}) Homology sphere formed from Lefschetz thimbles, since vanishing cycles coincide.  (\textsc{Right}) No homology sphere formed from Lefschetz thimbles, since vanishing cycles do not coincide.}
\end{centering}
\end{figure}

\medskip

Now, we can consider the location of the homology spheres in the basis of Theorem \ref{thm:factorization}. Since the eigenvalue associated to $A=\begin{pmatrix} 1 & 1 \\ 0 & 1 \end{pmatrix}$ is $\lambda_A=\begin{pmatrix} 1 \\ 0 \end{pmatrix}$ and the eigenvalue associated to $B=\begin{pmatrix} 1 & 0 \\ -1 & 1 \end{pmatrix}$ is $\lambda_B=\begin{pmatrix} 0 \\1 \end{pmatrix}$, we see that the vanishing cycles at $\beta_2, \beta_4, \beta_6$ are all $S^1_{\lambda_A}$ while the vanishing cycles at $\beta_1, \beta_3, \beta_5$ are all $S^1_{\lambda_B}$. The moduli space $\cM$ is homotopy equivalent to $T^2$ with $3$ disks into $S^1_{\lambda_A}$ and $3$ disks into $S^1_{\lambda_B}$, as shown in Figure \ref{fig:homotopy}. 
\begin{figure}[!ht]
\begin{centering} 
\includegraphics[height=1.5in]{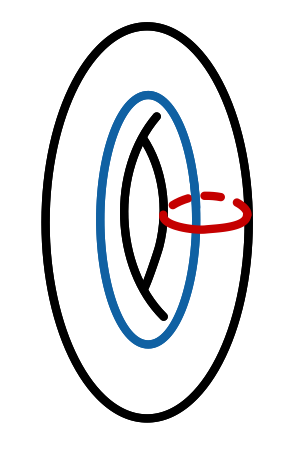}
\caption{\label{fig:homotopy} The moduli space $\mathcal{M}$ is homotopy equivalent to $T^2$ with $3$ disks glued into $S^1_{\lambda_A}$ and $3$ disks glued into $S^1_{\lambda_B}$, shown in blue and red.}
\end{centering}
\end{figure}
Note that there is no simple relationship between a homology basis coming from the $ABABAB=-\mathrm{Id}$ factorization and a homology basis consisting of five spheres in an affine $D_4$ configuration.  This is expected, since the affine $D_4$ Coxeter group acts on the homology basis.

\begin{rem}[Once-punctured torus]
A $3$-parameter subfamily of the $12$-parameter space of ALG-$D_4$ metrics is conjecturally realized by $SU(2)$-Hitchin moduli spaces on the once-punctured torus.
In this case, the $I_0^*$ splits as three $I_2$ fibers. We refer the reader to \cite[p. 14-16]{DAHA} for a discussion of the homology spheres in this case. 
\end{rem}  

 \bibliography{ALG}{}
\bibliographystyle{math}
\end{document}